\documentclass[10pt]{article}
\usepackage[all,2cell]{xy} \UseAllTwocells \SilentMatrices
\usepackage{latexsym,amsfonts,amssymb}
\usepackage{amsmath,amsthm,amscd}
\usepackage{hyperref,psfrag}
\usepackage{diagcat}
\usepackage{color}
\usepackage{etoolbox}
\usepackage[dvips]{epsfig}
\usepackage{psfrag}
\usepackage[vcentermath]{youngtab}

\usepackage{graphicx}

\usepackage{a4wide}
\usepackage{geometry}

\usepackage{epigraph,wrapfig}

\normalfont\upshape

\usepackage{fancyhdr}
\theoremstyle{definition}
\newtheorem{thm}{Theorem}[section]
\newtheorem{cor}[thm]{Corollary}
\newtheorem{conj}[thm]{Conjecture}
\newtheorem{lem}[thm]{Lemma}
\newtheorem{rem}[thm]{Remark}
\newtheorem{prop}[thm]{Proposition}
\newtheorem{defn}[thm]{Definition}
\newtheorem{example}[thm]{Example}

\newtheorem*{thm*}{Theorem}

\numberwithin{equation}{section}

\newcommand{\B}{\scriptstyle b+1}

\usepackage{bbm}
\def\o{\otimes}

\def\C{{\mathbbm C}}
\def\N{{\mathbbm N}}

\def\Z{{\mathbbm Z}}
\def\Q{{\mathbbm Q}}

\def\F{{\mathbbm F}}

\def\1{{\mathbbm{1}}}

\newcommand{\Hom}{{\rm Hom}}
\newcommand{\HOM}{{\rm HOM}}

\newcommand{\END}{{\rm END}}
\newcommand{\Tab}{{\rm Tab}}

\newcommand{\gr}{\mathrm{Gr}} 

\def\dif{\partial}

\def\lra{{\longrightarrow}}
   
\def\dmod{{\mbox{-}\mathrm{mod}}}   

\def\Id{\mathrm{Id}}
\def\mc{\mathcal}
\def\mf{\mathfrak}
\def\shuffle{\,\raise 1pt\hbox{$\scriptscriptstyle\cup{\mskip
               -4mu}\cup$}\,}

\newcommand{\refequal}[1]{\xy {\ar@{=}^{#1}
(-1,0)*{};(1,0)*{}};
\endxy}

\def\Tab{\mathrm{Tab}}
\def\e{\varepsilon}

\newcommand{\sym}{\mathrm{\Lambda}}

\newcommand{\nh}{\mathrm{NH}}

\newcommand{\mH}{\mathrm{H}} 
\newcommand{\RHOM}{\mathbf{R}\mathrm{HOM}}

\newcommand{\cwbubble}[2]{
\begin{DGCpicture}
\DGCcoupon*(-0.4,-0.4)(.4,.4){ }
\DGCbubble(0,0){0.35}
\DGCdot*<{0.2,L}
\DGCdot*.{0.1,R}[r]{#2}
\DGCcoupon*(-.4,-.4)(.4,.4){\small{#1}}
\end{DGCpicture}
}

\newcommand{\ccwbubble}[2]{
\begin{DGCpicture}
\DGCcoupon*(-0.4,-0.4)(.4,.4){ }
\DGCbubble(0,0){0.35}
\DGCdot*>{0.2,L}
\DGCdot*.{0.1,R}[r]{#2}
\DGCcoupon*(-.4,-.4)(.4,.4){\small{#1}}
\end{DGCpicture}
}

\newcommand{\bigcwbubble}[2]{
\begin{DGCpicture}
\DGCcoupon*(-0.8,-0.8)(.8,.8){ }
\DGCbubble(0,0){0.5}
\DGCdot*<{0.25,L}
\DGCdot*.{0.25,R}[r]{#2}
\DGCcoupon*(-.6,-.6)(.6,.6){\small{#1}}
\end{DGCpicture}
}

\newcommand{\bigccwbubble}[2]{
\begin{DGCpicture}
\DGCcoupon*(-0.8,-0.8)(.8,.8){ }
\DGCbubble(0,0){0.5}
\DGCdot*>{0.25,L}
\DGCdot*.{0.25,R}[r]{#2}
\DGCcoupon*(-.6,-.6)(.6,.6){\small{#1}}
\end{DGCpicture}
}

\newcommand{\cwcapbubcup}[4]{
\begin{DGCpicture}
\DGCcoupon*(-0.3,-0.1)(1.3,2.1){ }
\ifstrequal{#4}{no}{}{
\DGCcoupon*(0,.8)(.3,1.4){#4}}
\ifstrequal{#1}{no}{}{
\DGCstrand(0,0)(1,0)/d/
\DGCdot*>{0.25,2}
\ifstrequal{#1}{$0$}{}{
\ifstrequal{#1}{$1$}{\DGCdot{0.25,1}}{
\DGCdot{0.25,1}[r]{\mbox{\scriptsize #1}}}}}
\ifstrequal{#3}{no}{}{
\DGCstrand/d/(0,2)(1,2)
\DGCdot*<{1.75,2}
\ifstrequal{#3}{$0$}{}{
\ifstrequal{#3}{$1$}{\DGCdot{1.75,1}}{
\DGCdot{1.75,1}[r]{\mbox{\scriptsize #3}}}}}
\ifstrequal{#2}{no}{}{
\DGCbubble(1,1){.3}
\DGCdot*>{1.2,L}
\DGCcoupon*(.65,.65)(1.35,1.35){\small{#2}}}
\end{DGCpicture}
}

\newcommand{\ccwcapbubcup}[4]{
\begin{DGCpicture}
\DGCcoupon*(-0.3,-0.1)(1.3,2.1){ }
\ifstrequal{#4}{no}{}{
\DGCcoupon*(0,.8)(.3,1.4){#4}}
\ifstrequal{#1}{no}{}{
\DGCstrand(0,0)(1,0)/d/
\DGCdot*<{0.25,2}
\ifstrequal{#1}{$0$}{}{
\ifstrequal{#1}{$1$}{\DGCdot{0.25,1}}{
\DGCdot{0.25,1}[r]{\mbox{\scriptsize #1}}}}}
\ifstrequal{#3}{no}{}{
\DGCstrand/d/(0,2)(1,2)
\DGCdot*>{1.75,2}
\ifstrequal{#3}{$0$}{}{
\ifstrequal{#3}{$1$}{\DGCdot{1.75,1}}{
\DGCdot{1.75,1}[r]{\mbox{\scriptsize #3}}}}}
\ifstrequal{#2}{no}{}{
\DGCbubble(1,1){.3}
\DGCdot*<{1.2,L}
\DGCcoupon*(.65,.65)(1.35,1.35){\small{#2}}}
\end{DGCpicture}
}

\newcommand{\RIII}[5]{
\begin{DGCpicture}[scale=0.85]
\DGCcoupon*(-0.3,-0.3)(2.3,2.3){}
\ifstrequal{#5}{no}{}{
\DGCcoupon*(2.1,.95)(2.3,1.15){#5}}
\DGCstrand(0,0)(2,2)
\DGCdot*>{2}
\ifstrequal{#4}{$0$}{}{
\ifstrequal{#4}{$1$}{\DGCdot{1.7}}{
\DGCdot{1.7}[r]{\mbox{\scriptsize #4}}}}
\DGCstrand(2,0)(0,2)
\DGCdot*>{2}
\ifstrequal{#2}{$0$}{}{
\ifstrequal{#2}{$1$}{\DGCdot{1.7}}{
\DGCdot{1.7}[r]{\mbox{\scriptsize #2}}}}
\ifstrequal{#1}{L}{\DGCstrand(1,0)(0,1)(1,2)}{\DGCstrand(1,0)(2,1)(1,2)}
\DGCdot*>{2}
\ifstrequal{#3}{$0$}{}{
\ifstrequal{#3}{$1$}{\DGCdot{1.7}}{
\DGCdot{1.7}[r]{\mbox{\scriptsize #3}}}}
\end{DGCpicture}
}

\newcommand{\twolines}[3]{
\begin{DGCpicture}
\DGCcoupon*(-.3,-.3)(1.3,1.3){}
\ifstrequal{#3}{no}{}{\DGCcoupon*(1.1,.6)(1.4,.9){#3}}
\DGCstrand(0,0)(0,1)
\DGCdot*>{1}
\ifstrequal{#1}{$0$}{}{
\ifstrequal{#1}{$1$}{\DGCdot{.4}}{
\DGCdot{.4}[r]{\mbox{\scriptsize #1}}}}
\DGCstrand(1,0)(1,1)
\DGCdot*>{1}
\ifstrequal{#2}{$0$}{}{
\ifstrequal{#2}{$1$}{\DGCdot{.4}}{
\DGCdot{.4}[r]{\mbox{\scriptsize #2}}}}
\end{DGCpicture}
}

\newcommand{\twolinesD}[3]{
\begin{DGCpicture}
\DGCcoupon*(-.3,-.3)(1.3,1.3){}
\ifstrequal{#3}{no}{}{\DGCcoupon*(-.4,.6)(-.1,.9){#3}}
\DGCstrand(0,0)(0,1)
\DGCdot*<{0}
\ifstrequal{#1}{$0$}{}{
\ifstrequal{#1}{$1$}{\DGCdot{.6}}{
\DGCdot{.6}[r]{\mbox{\scriptsize #1}}}}
\DGCstrand(1,0)(1,1)
\DGCdot*<{0}
\ifstrequal{#2}{$0$}{}{
\ifstrequal{#2}{$1$}{\DGCdot{.6}}{
\DGCdot{.6}[r]{\mbox{\scriptsize #2}}}}
\end{DGCpicture}
}

\newcommand{\crossing}[5]{
\begin{DGCpicture}
\DGCcoupon*(-.3,-.3)(1.3,1.3){}
\ifstrequal{#5}{no}{}{
\DGCcoupon*(1,.4)(1.3,.7){#5}}
\DGCstrand(0,0)(1,1)
\DGCdot*>{1}
\ifstrequal{#2}{$0$}{}{
\ifstrequal{#2}{$1$}{\DGCdot{.7}}{
\DGCdot{.7}[r]{\mbox{\scriptsize #2}}}}
\ifstrequal{#3}{$0$}{}{
\ifstrequal{#3}{$1$}{\DGCdot{.3}}{
\DGCdot{.3}[r]{\mbox{\scriptsize #3}}}}
\DGCstrand(1,0)(0,1)
\DGCdot*>{1}
\ifstrequal{#1}{$0$}{}{
\ifstrequal{#1}{$1$}{\DGCdot{.7}}{
\DGCdot{.7}[r]{\mbox{\scriptsize #1}}}}
\ifstrequal{#4}{$0$}{}{
\ifstrequal{#4}{$1$}{\DGCdot{.3}}{
\DGCdot{.3}[r]{\mbox{\scriptsize #4}}}}
\end{DGCpicture}
}

\newcommand{\crossingD}[5]{
\begin{DGCpicture}
\DGCcoupon*(-.3,-.3)(1.3,1.3){}
\ifstrequal{#5}{no}{}{\DGCcoupon*(-.3,.4)(0,.7){#5}}
\DGCstrand(0,0)(1,1)
\DGCdot*<{0}
\ifstrequal{#2}{$0$}{}{
\ifstrequal{#2}{$1$}{\DGCdot{.3}}{
\DGCdot{.3}[r]{\mbox{\scriptsize #2}}}}
\ifstrequal{#3}{$0$}{}{
\ifstrequal{#3}{$1$}{\DGCdot{.7}}{
\DGCdot{.7}[r]{\mbox{\scriptsize #3}}}}
\DGCstrand(1,0)(0,1)
\DGCdot*<{0}
\ifstrequal{#1}{$0$}{}{
\ifstrequal{#1}{$1$}{\DGCdot{.3}}{
\DGCdot{.3}[r]{\mbox{\scriptsize #1}}}}
\ifstrequal{#4}{$0$}{}{
\ifstrequal{#4}{$1$}{\DGCdot{.7}}{
\DGCdot{.7}[r]{\mbox{\scriptsize #4}}}}
\end{DGCpicture}
}

\newcommand{\oneline}[2]{
\begin{DGCpicture}
\DGCcoupon*(-.3,-.1)(0.3,2.1){}
\DGCstrand(0,0)(0,2)
\DGCdot*>{2}
\ifstrequal{#2}{no}{}{\DGCcoupon*(.1,1.4)(.4,1.7){#2}}
\ifstrequal{#1}{$0$}{}{
\ifstrequal{#1}{$1$}{\DGCdot{1}}{
\DGCdot{1}[r]{\mbox{\scriptsize #1}}}}
\end{DGCpicture}
}

\newcommand{\onelineD}[2]{
\begin{DGCpicture}
\DGCcoupon*(-.3,-.1)(0.3,2.1){}
\DGCstrand(0,0)(0,2)
\DGCdot*<{0}
\ifstrequal{#2}{no}{}{\DGCcoupon*(-.4,1.4)(-.1,1.7){#2}}
\ifstrequal{#1}{$0$}{}{
\ifstrequal{#1}{$1$}{\DGCdot{1}}{
\DGCdot{1}[r]{\mbox{\scriptsize #1}}}}
\end{DGCpicture}
}

\newcommand{\onelineshort}[2]{
\begin{DGCpicture}
\DGCcoupon*(-.3,-.1)(0.3,1.1){}
\DGCstrand(0,0)(0,1)
\DGCdot*>{1}
\ifstrequal{#2}{no}{}{\DGCcoupon*(.1,.7)(.4,1){#2}}
\ifstrequal{#1}{$0$}{}{
\ifstrequal{#1}{$1$}{\DGCdot{.5}}{
\DGCdot{.5}[r]{\mbox{\scriptsize #1}}}}
\end{DGCpicture}
}

\newcommand{\onelineDshort}[2]{
\begin{DGCpicture}
\DGCcoupon*(-.3,-.1)(0.3,1.1){}
\DGCstrand(0,0)(0,1)
\DGCdot*<{0}
\ifstrequal{#2}{no}{}{\DGCcoupon*(-.4,.7)(-.1,1){#2}}
\ifstrequal{#1}{$0$}{}{
\ifstrequal{#1}{$1$}{\DGCdot{.5}}{
\DGCdot{.5}[r]{\mbox{\scriptsize #1}}}}
\end{DGCpicture}
}

\newcommand{\curl}[5]{
\begin{DGCpicture}
\ifstrequal{#1}{L}{
\DGCcoupon*(-2,-.8)(.3,1.8){}
\DGCstrand(0,-.5)(0,.25)/u/(-1.5,.5)/d/(0,.75)/u/(0,1.5)
\ifstrequal{#4}{no}{}{\DGCcoupon*(-1.4,0)(-.4,1){\small{#4}}}
\ifstrequal{#5}{$0$}{}{\ifstrequal{#5}{$1$}{\DGCdot{.5,4}}{\DGCdot{.5,4}[l]{\mbox{\scriptsize #5}}}}
\ifstrequal{#3}{no}{}{\DGCcoupon*(-1.2,1.1)(-.9,1.4){#3}}
}{
\DGCcoupon*(-.3,-.8)(2,1.8){}
\DGCstrand(0,-.5)(0,.25)/u/(1.5,.5)/d/(0,.75)/u/(0,1.5)
\ifstrequal{#4}{no}{}{\DGCcoupon*(.4,0)(1.4,1){\small{#4}}}
\ifstrequal{#5}{$0$}{}{\ifstrequal{#5}{$1$}{\DGCdot{.5,4}}{\DGCdot{.5,4}[r]{\mbox{\scriptsize #5}}}}
\ifstrequal{#3}{no}{}{\DGCcoupon*(.9,1.1)(1.2,1.4){#3}}
}
\ifstrequal{#2}{D}{\DGCdot*<{-0.25} \DGCdot*<{1.25} \DGCdot*<{.75,2}}{\DGCdot*>{-0.25} \DGCdot*>{1.25} \DGCdot*>{.75,2}}
\end{DGCpicture}
}

\newcommand{\cappy}[4]{
\begin{DGCpicture}
\DGCcoupon*(-.3,-.1)(1.3,.8){}
\DGCstrand(0,0)(1,0)/d/
\ifstrequal{#1}{CCW}{\DGCdot*<{0,2}}{\DGCdot*>{0,1}}
\ifstrequal{#2}{$0$}{}{\ifstrequal{#2}{$1$}{\DGCdot{.3}}{\DGCdot{.3}[r]{\mbox{\scriptsize #2}}}}
\ifstrequal{#3}{no}{}{
\DGCbubble(1.7,.4){0.3}
\ifstrequal{#3}{CCW}{\DGCdot*>{.7,L}}{\DGCdot*<{.7,L}}
\DGCcoupon*(1.35,0)(2.05,.8){\small{$1$}}}
\ifstrequal{#4}{no}{}{\DGCcoupon*(.8,.5)(1.2,.8){#4}}
\end{DGCpicture}
}

\newcommand{\cuppy}[4]{
\begin{DGCpicture}
\DGCcoupon*(-.3,0.2)(1.3,1.1){}
\DGCstrand(0,1)/d/(1,1)/u/
\ifstrequal{#1}{CCW}{\DGCdot*>{1,1}}{\DGCdot*<{1,2}}
\ifstrequal{#2}{$0$}{}{\ifstrequal{#2}{$1$}{\DGCdot{.7}}{\DGCdot{.7}[r]{\mbox{\scriptsize #2}}}}
\ifstrequal{#3}{no}{}{
\DGCbubble(1.7,.6){0.3}
\ifstrequal{#3}{CCW}{\DGCdot*>{.9,L}}{\DGCdot*<{.9,L}}
\DGCcoupon*(1.35,.2)(2.05,1){\small{$1$}}}
\ifstrequal{#4}{no}{}{\DGCcoupon*(.8,.2)(1.2,.5){#4}}
\end{DGCpicture}
}

\newcommand{\bottomcurl}[5]{
\begin{DGCpicture}
\DGCcoupon*(-.3,-.1)(1.3,1.6){}
\DGCstrand(0,0)(1,1)/u/(0,1)/d/(1,0)/d/
\ifstrequal{#1}{CW}{\DGCdot*>{1.3}}{\DGCdot*<{1.3}}
\ifstrequal{#2}{no}{}{\DGCcoupon*(0,0.65)(1,1.3){\mbox{\scriptsize #2}}}
\ifstrequal{#3}{yes}{\DGCdot{.3,2}}{}
\ifstrequal{#4}{no}{}{
\DGCbubble(1.5,0.6){.3}
\ifstrequal{#4}{CCW}{\DGCdot*>{.8,L}}{\DGCdot*<{.8,L}}
\DGCcoupon*(1.3,0.4)(1.7,0.8){\small{$1$}}}
\ifstrequal{#5}{no}{}{\DGCcoupon*(1.1,0.9)(1.5,1.5){#5}}
\end{DGCpicture}
}

\newcommand{\crossingR}[4]{
\begin{DGCpicture}
\DGCcoupon*(-.3,-.3)(1.3,1.3){}
\DGCstrand(0,0)(1,1)
\DGCdot*>{1}
\ifstrequal{#1}{no}{}{\DGCdot{.65}}
\DGCstrand(1,0)(0,1)
\DGCdot*<{0}
\ifstrequal{#2}{no}{}{\DGCdot{.65}}
\ifstrequal{#3}{no}{}{
\DGCbubble(1.2,.5){0.3}
\DGCdot*<{.7,R}
\DGCcoupon*(0.9,0.2)(1.5,.8){\small{$1$}}}
\ifstrequal{#4}{no}{}{
\DGCcoupon*(-.3,.2)(.3,.8){#4}}
\end{DGCpicture}
}

\newcommand{\crossingL}[4]{
\begin{DGCpicture}
\DGCcoupon*(-.3,-.3)(1.3,1.3){}
\DGCstrand(0,0)(1,1)
\DGCdot*<{0}
\ifstrequal{#1}{no}{}{\DGCdot{.35}}
\DGCstrand(1,0)(0,1)
\DGCdot*>{1}
\ifstrequal{#2}{no}{}{\DGCdot{.35}}
\ifstrequal{#3}{no}{}{
\DGCbubble(-.2,.5){0.3}
\DGCdot*<{.7,L}
\DGCcoupon*(-.5,0.2)(.1,.8){\small{$1$}}}
\ifstrequal{#4}{no}{}{
\DGCcoupon*(.7,.2)(1.3,.8){#4}}
\end{DGCpicture}
}

\newstrandstyle{Green}{type=ribbon,style=solid,ribboncolor=green,color=green}

\title{$p$-DG cyclotomic nilHecke algebras}

\author{Mikhail Khovanov, You Qi, Joshua Sussan}

\date{\today}

\begin{document}
%

\maketitle

\begin{abstract}
We categorify a tensor product of two Weyl modules for quantum $\mathfrak{sl}_2$ at a prime root of unity.
\end{abstract}

\setcounter{tocdepth}{2} \tableofcontents

\section{Introduction}

\subsection{Motivation}
Quantum groups for generic values of $q$, are algebras defined over the ring $\Z[q,q^{-1}]$. The categories of representations of these algebras are braided monoidal categories. Such categories lead to invariants of knots, links and tangles known as Witten-Reshetikhin-Turaev (WRT) invariants. In order to obtain a $3$-manifold invariant, one needs to specialize $q$ to a root of unity $\zeta$.  The corresponding quantum groups are algebras over the cyclotomic ring $\Z[\zeta]$. The representation theory of quantum groups at a root of unity plays a pivotal role in understanding 3-dimensional topological quantum field theory and 2-dimensional conformal field theory. Of particular importance is the concept of the \emph{fusion product} of \emph{tilting modules}.
 
Since the discovery of categorified quantum link invariants (see, e.g., \cite{KhJones}), there have been substantial developments in the categorification of the quantum algebraic invariants defined at a generic value of the quantum parameter. To this day, it remains a challenging problem to lift the 3-manifold invariants defined over $\Z[\zeta]$ to categorical invariants.

As an approach to this problem, the subject of hopfological algebra was introduced in \cite{Hopforoots}, with the goal of categorifying the WRT invariants at a prime root of unity. The generic ground algebra $\Z[q,q^{-1}]$ admits a straightforward categorification via chain complexes of $\Z$-graded vector spaces, the $q$-parameter coming from the grading shift on the categorical level. In the hopofological setting, we are forced into working with chain complexes of graded vector spaces whose differential satisfies $\dif^p\equiv 0$ over ground fields of characteristic $p>0$. To apply hopfological algebra in the pursuit of 
categorification of quantum group structures defined over the
cyclotomic ring\footnote{Here we usually consider a slightly larger ring $\mathbb{O}_p := \mathbb{Z}[q]/(\Psi_p(q^2))$ as the Grothendieck ring of $p$-complexes whose differentials have degree $2$. The degree choice is made to match with previous representation theoretical constructions and does not cause essential differences from working over the genuine cyclotomic ring $\Z[\zeta_p]$.} $\mathbb{O}_p$, one should look for $p$-nilpotent derivations on certain algebras categorifying quantum groups and their module categories for generic values of $q$.

Denote the small quantum group for $\mf{sl}_2$ over $ \mathbb{O}_p $ by
$ \dot{u}_{\mathbb{O}_p}(\mf{sl}_2) $.
The first example of the program outlined in \cite{Hopforoots} was given in \cite{KQ} where the authors categorified one half of $ \dot{u}_{\mathbb{O}_p}(\mf{sl}_2) $.  
The nilHecke algebra over a field of characteristic $p$ was equipped with a $p$-nilpotent derivation.  It is proved that the direct sum of the compact derived categories of $p$-DG modules over all such $p$-DG nilHecke algebras categorify the positive or negative half of this small quantum group.  One halves of quantum groups usually play special roles in (categorical) representation theory. One reason is that, as finite-dimensional representations of $\mf{sl}_2$ can be realized as quotients of certain Verma modules, simple categorical representations of quantum $\mf{sl}_2$ can be realized as specific categorical quotients of the nilHecke algebras known as the \emph{cyclotomic quotients}. The story is similar at a prime root of unity. Simple $\mf{sl}_2$-modules specialize at $\mathbb{O}_p$ to the (dual) Weyl modules. The Weyl modules of highest weight in the range $\{0,1,\dots, p-1\}$ are also categorified by the corresponding cyclotomic nilHecke algebras equipped with the natural quotient $p$-differential.

In \cite{EQ1,EQ2}, the one-half categorification in \cite{KQ} for $\mf{sl}_2$ is further extended to a categorification of the entire quantum $\mf{sl}_2$, for both the small version $ \dot{u}_{\mathbb{O}_p}(\mf{sl}_2) $ and an infinite dimensional version (the BLM form \cite{BLM}) $\dot{U}_{\mathbb{O}_p}(\mf{sl}_2)$ respectively. This is realized by equipping Lauda's $2$-category $\mc{U}$ and its Karoubi envelope $\dot{\mc{U}}$ with $p$-nilpotent differentials.  The resulting $p$-DG $2$-category $\mc{U}$ categorifies $\dot{u}_{\mathbb{O}_p}(\mf{sl}_2)$, while the $p$-DG $2$-category $\dot{\mc{U}}$ categorifies $\dot{U}_{\mathbb{O}_p}(\mf{sl}_2)$.
However, unlike at a generic $q$, there is some subtle hopfological behavior at a root of unity. The $p$-DG version of Lauda's category $\mc{U}$ is homotopy equivalent to its Karoubi envelope $\dot{\mc{U}}$ equipped with a compatible $p$-differential. But the homotopy equivalence fails to descend to the derived level. Rather, on the derived level, there is a fully-faithful embedding of the categorified small quantum group into the categorified BLM form. The two versions of quantum groups at a prime root of unity are also related by Lusztig's quantum Frobenius map, which is explained in \cite{QYFrob}. A similar feature also appears in the current work, as we will see in Section \ref{subsec-cat-simples} and \ref{examplesubsection}.

The next step is then to study the categorical representation theory of categorified quantum groups at a prime root of unity. Categorical representations, in particular, certain categorical tensor products of quantum group representations have been proposed and studied by Webster in \cite{Webcombined}. Recently, there has been minor progress on this front in \cite{QiSussan}, where a unique differential on the algebras defined by Webster, which is compatible with the categorical quantum group action, is determined. Furthermore, the second lowest weight space of $V_1^{\otimes m}$, where $V_1$ is the natural $2$-dimensional representation of $\mf{sl}_2$, has been categorified. This is done by equipping the quiver algebra Koszul dual to the zigzag algebra with a $p$-differential. The differential can be regarded as arising from the identification of the quiver algebra with the corresponding block of a Webster algebra.  A braid group action is also exhibited in this case, providing some evidence of connections to quantum topology.

The current work will serve as the starting point of a series of works on categorifying tensor product representations of quantum $\mf{sl}_2$ at a prime root of unity. Our goal is to set up a general framework for constructing such tensor product categorifications at a prime root of unity, as well as to gain better understanding of earlier works such as Webster \cite{Webcombined} and Hu-Mathas \cite{HuMathas}. In this work, we will focus on the particular case of categorifying a tensor product of two Weyl modules at a root of unity, since it is the foundation for constructing ``\emph{categorical fusion products}'' of ``\emph{categorified tilting modules}''. In a forthcoming paper, we will categorify the representation $V_1^{\otimes m}$ and study the braid group invariants arising from this construction.

\subsection{Summary of contents}
Let $V_r$ denote the Weyl module of $\dot{U}_{\mathbb{O}_p}(\mf{sl}_2)$ of rank $(r+1)$.
Fix two integers $r,s\in \N$ and let $l=r+s$. The main goal of this paper is to construct a categorification of the tensor product representation  $V_r \otimes_{\mathbb{O}_p} V_s$.  In this paper, we concentrate on this specific tensor product representation, since the simplicity of its canonical basis makes its categorification more accessible.
We now give a brief preview of the contents of each section and how the construction is carried out.

In Section \ref{sec-small-qgroup} and Section \ref{sec-hopfo-algebra}, we fix some notation and collect the necessary background material on the representation theory of quantum $\mf{sl}_2$ at a prime root of unity and hopfological algebra of $p$-differential graded ($p$-DG) algebras. These sections are used as tool boxes for the remainder of the paper, and may be safely skipped on a first reading by the reader and referred back to later when needed.

In Section \ref{sec-double}, we develop some basic tools for constructing tensor-product categorifications. The motivating question addressed here is, starting from a collection of algebras categorifying an irreducible representation, how can one build certain tensor product categorifications directly out of these algebras. The question may be too general to admit a simple universal answer, but when the algebras are Frobenius, a framework for constructing potential tensor-product algebras is established in that section. In particular, we consider faithful ($p$-DG) modules of the Frobenius algebras, and take their ($p$-DG) endomorphism algebras as candidates for tensor-product categorifications. Faithfulness of the modules turns out to be a key condition to require, and it implies the double-centralizer property many other works have relied on (see, e.g.~\cite{Webcombined, HuMathas}). Under the assumption on the ($p$-DG) functors that they preserve the additive full subcategory generated by the faithful modules over the Frobenius algebra, the new ($p$-DG) module categories over the endomorphism algebras are acted on by some natural ($p$-DG) functors extending the initial ones. The basic idea goes back to the construction of some special indecomposable projective modules for a maximal singular block of category $\mathcal{O}(\mathfrak{gl}_m)$ given in \cite[Section 3.1.3]{BFK}, and is related to the work of Losev-Webster \cite{LoWeb}. We will see another instance of our construction in a sequel to this paper on categorifying the $m$-fold tensor $V_1^{\otimes m}$.

Section \ref{sec-cat-sl(2)} is devoted to reviewing two seemingly different forms of the ($p$-DG) $2$-categories $(\mc{U},\dif)$, due respectively to Khovanov-Lauda \cite{Lau1, KL3} and Rouquier \cite{Rou2}. In a remarkable work of Brundan \cite{Brundan2KM}, it is shown that the definitions of Khovanov-Lauda and Rouquier are equivalent. We then readily extend Brundan's theorem to the $p$-DG setting (Theorem \ref{equiv-of-pdgRoug-pdgLauda}). 

Section \ref{sec-nilHecke} introduces the ($p$-DG) cyclotomic nilHecke algebra $\nh_n^l$.  This is a quotient of the nilHecke algebra $\nh_n$ of rank $n$ by an ideal depending upon the natural number $l$.  Taking the sum of the $p$-DG module categories $\oplus_{n=0}^l (\nh_n^l,\dif)\dmod$ provides, imprecisely, a categorification of the Weyl module $V_l$.  These algebras are symmetric Frobenius, and serve as the input data for the categorical framework of Section \ref{sec-double}. In \cite{EQ1}, the lowest weight $\mf{sl}_2$-representation $V_l$ is instead categorified via Rouquier's universal ($p$-DG) cyclotomic quotient. The version there is known to be Morita equivalent to the cyclotomic nilHecke constructions, but relies on certain implicit ($p$-DG) functors realizing the Morita equivalence. Here we strictify the functors acting on cyclotomic nilHecke algebras, and show that the $p$-DG $2$-category in Rouquier's definition acts directly on the direct sum of $p$-DG nilHecke module categories.  A technical caveat is that we do not pass right away to derived categories of $p$-DG cyclotomic nilHecke algebras for categorifying $V_l$, but use the abelian category of $p$-DG $\nh_n^l$-modules as an intermediate stepping stone.

Sections \ref{sec-cyclic-mod} and \ref{sec-2-tensor} constitute the technical heart of the current work\footnote{Before reading these two Sections, a diagrammatically oriented reader who is familiar with Webster's work \cite{Webcombined} may wish to proceed directly to Sections \ref{sec-Web} and \ref{sec-main-thm} after going through Sections \ref{sec-double}--\ref{sec-nilHecke}. This order gives the reader an alternative motivation for considering certain cyclically generated ($p$-DG) nilHecke modules in Sections \ref{sec-cyclic-mod} and \ref{sec-2-tensor}.}. In Section \ref{sec-cyclic-mod}, we recall the cellular structure on cyclotomic nilHecke algebras due to Hu-Mathas \cite{HuMathas}. We utilize the cellular structure on $\nh_n^l$ to exhibit a natural collection of cyclic right $p$-DG modules over $\nh_n^l$. A special collection of cyclic modules $p$-DG modules $e_\lambda G(\lambda)$ over $\nh_n^l$ are introduced, where $\lambda$'s are parameterized by certain partitions of $l$ into zeros and ones. The construction mimics those $G(\lambda)$'s appearing in Hu-Mathas \cite{HuMathas}, but is further truncated by an idempotent $e_\lambda\in \nh_n^l$.
We then show explicitly that the generating functors $\mc{E}$ and $\mc{F}$ for the category $\mc{U}$ acting on $\oplus_{n=0}^l(\nh_n^l,\dif)\dmod$ preserve the additive subcategory generated by these special modules $e_\lambda G(\lambda)$. This establishes the necessary conditions for the entire setup to fit into the framework of Section \ref{sec-double}.
The ($p$-DG) \emph{two-tensor quiver Schur algebra $S_n(r,s)$ } is defined as (Definition \ref{def-two-tensor-algebra}) 
\[
S_n(r,s):=\END_{\nh_n^l}\left(\bigoplus_{\lambda} e_\lambda G(\lambda)\right),
\]
where $(r,s)\in \N^2$ is the decomposition of $l$ fixed at the beginning of this section. By construction, the $p$-DG $2$-category $(\mc{U},\dif)$ acts on $\oplus_{n=0}^l(S_n(r,s),\dif)\dmod$.

The two-tensor quiver Schur algebra is related to a special case of Webster's diagrammatic tensor algebra by performing a ``thick'' idempotent truncation, which is shown in Section \ref{sec-Web}. However, the two-tensor quiver Schur algebra has fewer generating idempotents than in Webster's definition, which allows an easier identification of Grothendieck groups in this particular case. While Webster's setup has certain advantages (for example functors for tangles are naturally defined) we focus on these subcategories of nilHecke algebras since it is a little easier to calculate the $p$-DG Grothendieck groups.  We expect the categories defined in this work are Morita equivalent to Webster algebras but not $p$-DG Morita equivalent.

We then establish the following in Section \ref{sec-main-thm}:
\begin{thm*}[\ref{mainthm}]
There is an action of the derived $p$-DG Lauda category on $\oplus_{n=0}^l  \mathcal{D}(S_n(r,s))$. The action induces an identification of the Grothendieck groups 
\begin{equation*}
K_0(\bigoplus_{n=0}^l  \mathcal{D}^c(S_n(r,s))) \cong V_r \otimes_{\mathbb{O}_p} V_s
\end{equation*}
with the tensor product of the quantum $\mf{sl}_2$ Weyl modules $V_r$ and $V_s$ at a primitive $p$th root of unity.
\end{thm*}
Finally, a comment on the $p$-DG stratified structure on the two-tensor quiver Schur algebra is briefly discussed in Section \ref{subsec-strat}, and some further investigations are sketched out in the final Section \ref{subset-future}.

\subsection{Further comments}
A more general $m$-tensor quiver Schur algebra can be defined when trying to categorify a more general tensor product $V_{r_1} \otimes_{\mathbb{O}_p} \cdots \otimes_{\mathbb{O}_p} V_{r_m}$. See Section \ref{subset-future} for a brief introduction.  In future work we will study the $p$-DG Grothendieck groups of the module categories for these quiver Schur algebras along certain categorical actions. However, the case of just two tensor factors is relatively simpler. One reason is that the canonical basis of Lusztig \cite{Lus4} has a very explicit description in the two-tensor factor case. For general $m$ the canonical basis has only an inductive construction.

A categorification of tensor products of quantum $\mathfrak{sl}_2$ for generic $q$ was given in \cite{FKS} by considering certain categories of Harish-Chandra bimodules.  It would be interesting to understand the $p$-differential Lie theoretically.

Quiver Schur algebras were first introduced by Hu and Mathas \cite{HuMathas} in the graded case.  They were interested in putting a $\Z$-grading on the classical Schur algebras via an identification of cyclotomic Hecke algebras with cyclotomic KLR algebras due to Brundan and Kleshchev \cite{BK}. The quiver Schur algebras are close relatives of Webster's diagrammatic algebras in type $A$. Webster categorified tensor products of finite-dimensional irreducible representations for quantum groups at generic values of $q$ \cite{Webcombined} using his algebras. He also proved that special cases of the Webster algebras are graded Morita equivalent to the quiver Schur algebras of \cite{HuMathas}. However, as the current work shows, in the presence of a $p$-differential, the Morita equivalence does not necessarily pass to the derived category level. The perspectives of both Hu-Mathas and Webster will prove useful towards categorified representation theory at prime roots of unity, as can be seen from this work and subsequent future works (e.g.~on categorifying $V_1^{\otimes m}$ at prime roots of unity).

\paragraph{Acknowledgements.}
The authors would like to thank Catharina Stroppel and Daniel Tubbenhauer for their helpful and enlightening discussions on cellular algebras. They would like to thank Ben Elias for his helpful suggestions and allowing them to use some diagrams in his joint works with the second author. They would also like to thank the referee for the detailed comments and corrections.

M.~K.~is partially supported by the NSF grants DMS-1406065 and  DMS-1664240.
Y.~Q.~is partially supported by the NSF grant DMS-1763328. 
J.~S.~is supported by NSF grant DMS-1407394, PSC-CUNY Award 67144-0045, and Simons Foundation Collaboration Grant 516673.

\section{The quantum group at prime roots of unity}\label{sec-small-qgroup}

In this section, we collect some basic facts and fix some notation on quantum groups at a prime root of unity. It will serve as the decategorified story of this work, and the interested reader may want to skip this part first and refer back to it later when needed.

Throughout, we assume $\N$ to be the set of non-negative integers (containing zero).

\subsection{The small quantum \texorpdfstring{$\mathfrak{sl}_2$}{sl(2)}}
Let $N$ be an odd natural number or $2$,  and take $ \zeta_{2 N} $ to be a primitive $2 N$th root of unity.  The quantum group $ u_{\Q(\zeta_{2 N})}(\mathfrak{sl}_2) $, which we will denote simply by $ u_{\Q(\zeta_{2 N})} $, is the $\Q(\zeta_{2 N})$-algebra generated by $ E, F, K^{\pm 1} $ subject to relations:
\begin{enumerate}
\item[(1)] $ KK^{-1} = K^{-1}K=1$,
\item[(2)] $ K^{\pm 1}E = \zeta_{2 N}^{\pm 2} EK^{\pm 1} $, \quad $ K^{\pm 1} F = \zeta_{2 N}^{\mp 2} FK^{\pm 1} $,
\item[(3)] $ EF-FE = \frac{K-K^{-1}}{\zeta_{2 N}-\zeta_{2 N}^{-1}} $,
\item[(4)] $ E^{N} = F^{N} = 0 $.
\end{enumerate}

The quantum group is a Hopf algebra whose comultiplication map 
$$ \Delta \colon u_{\Q(\zeta_{2 N})} \longrightarrow u_{\Q(\zeta_{2 N})} \otimes_{\Q(\zeta_{2 N})} u_{\Q(\zeta_{2 N})} $$
is given on generators by
\begin{equation}
\label{comultform}
\Delta(E) = K^{-1} \otimes E + E \otimes 1, \quad\quad \Delta(F) = 1 \otimes F + F \otimes K, \quad\quad \Delta(K^{\pm 1}) = K^{\pm 1} \otimes K^{\pm 1}.
\end{equation}
The comultiplication defined in equation \eqref{comultform} is related to the standard comultiplication (see for example \cite{KhoThesis}) by the anti-isomorphism $E \mapsto F, F \mapsto E, K \mapsto K$.

For the purpose of categorification, it is more convenient to use the idempotented quantum group $ \dot{u}_{\Q[\zeta_{2 N}]}(\mathfrak{sl}_2) $, which we will denote simply by $ \dot{u}_{\Q[\zeta_{2 N}]} $. It is a non-unital $\Q[\zeta_{2 N}]$-algebra generated by $ E, F $ and idempotents $1_m$ ($m \in \Z$), subject to relations:
\begin{enumerate}
\item[(1)] $ 1_m 1_n=\delta_{m, n}1_m$,
\item[(2)]$ E1_m=1_{m+2}E $, \quad $ F1_{m+2} = 1_{m} F $,
\item[(3)] $ EF1_{m}-FE1_{m} = [m]1_m $,
\item[(4)] $ E^{N} = F^{N} = 0 $.
\end{enumerate}
Here $ [m] =\sum_{i=0}^{|m|-1} \zeta_{2 N}^{1-|m|+2i} $ is the quantum integer specialized at $\zeta_{2 N}$.

The quantum group $\dot{u}_{\Q[\zeta_{2 N}]}$ has an integral lattice subalgebra which we now recall. For any integer $n\in \{0,1,\dots, N-1\}$, let $ E^{(n)} = \frac{E^n}{[n]!} $, and $ F^{(n)} = \frac{F^n}{[n]!} $.
The elements $ E^{(n)}$, $F^{(n)}$ ($1\leq n \leq N-1$), and $ 1_{m}$ ($m \in \Z$) generate an algebra over the ring of cyclotomic integers $ \mathcal{O}_{2 N} = \mathbb{Z}[\zeta_{2 N}] $. Denote this integral form by $ \dot{u}_{\zeta_{2 N}} $.

Now let $ N=p $ be prime.  Introduce the auxiliary ring
\begin{equation}\label{eqn-aux-ring}
\mathbb{O}_p = \mathbb{Z}[q]/(\Psi_p(q^2)),
\end{equation}  where $ \Psi_p(q) $ is the $p$-th cyclotomic polynomial.
We can define in a similar fashion an integral form  $ {u}_{\mathbb{O}_p} $ and its dotted version $ \dot{u}_{\mathbb{O}_p} $ for the small quantum $\mf{sl}_2$ over $ \mathbb{O}_p $. All the necessary changes only occur for the specialization of quantum numbers: instead of at the root of unity $\zeta_{2l}$, we will set $[n]_{\mathbb{O}_p}:=q^{n-1}+q^{n-3}+\cdots + q^{1-n}$ to be understood as an element of $\mathbb{O}_p$.

\begin{defn}\label{def-small-sl2}
The $\mathbb{O}_p$-integral idempotented quantum algebra $\dot{u}_{\mathbb{O}_p}$ is generated by $ E, F $ and idempotents $1_m$ ($m\in \Z$), subject to the relations:
\begin{enumerate}
\item[(1)] $ 1_m 1_n=\delta_{m, n}1_m$ for any $m,n\in \Z$,
\item[(2)]$ E1_m=1_{m+2}E $, \quad $ F1_{m+2} = 1_{m} F $,
\item[(3)] $ EF1_{m}-FE1_{m} = [m]_{\mathbb{O}_p}1_m $,
\item[(4)] $ E^p = F^p = 0 $.
\end{enumerate}
\end{defn}

Let the lower half of $ u_{\mathbb{O}_p} $ be the subalgebra generated by the $ F^{(n)} $ ($0\leq n \leq p-1$) and denote it by $ u_{\mathbb{O}_p}^- $. Likewise, write the upper half as $u_{\mathbb{O}_p}^+$.
For more details see \cite[Section 3.3]{KQ}. 

\subsection{The BLM integral form}
We next recall the Beilinson-Lusztig-MacPherson \cite{BLM} (BLM) integral form of quantum $\mf{sl}_2$ at a prime root of unity. The small quantum sits inside the BLM form as an (idempotented) Hopf algebra.

\begin{defn}\label{def-BLM}
The non-unital associative quantum algebra $\dot{U}_{\mathbb{O}_p}(\mathfrak{sl}_2)$, or just $\dot{U}_{\mathbb{O}_p}$, is the $\mathbb{O}_p$-algebra generated by a family of orthogonal idempotents $\{1_n|n \in \Z\}$, a family of raising operators $E^{(a)}$ and a family of lowering operators $F^{(b)}$ ($a,b\in \N$), subject to the following relations.
\begin{itemize}
\item[(1)] $1_m 1_n=\delta_{m,n}1_m$ for any $m,n \in \Z$.
\item[(2)] $ E^{(a)}1_m=1_{m+2a}E^{(a)} $,  $ F^{(a)}1_{m+2a} = 1_{m} F^{(a)}$, for any $a\in \N$ and $m\in \Z$.
\item[(3)] For any $a,b\in \N$ and $m\in \Z$, 
\begin{equation}\label{eqn-divided-power}
E^{(a)}E^{(b)}1_m={a+b \brack a}_{\mathbb{O}_p}E^{(a+b)}1_m, \quad \quad \quad F^{(a)}F^{(b)}1_m={a+b\brack a}_{\mathbb{O}_p}F^{(a+b)}1_m.
\end{equation}
\item[(4)] The divided power $E$-$F$ relations, which state that
\begin{subequations}
\begin{align}\label{eqn-higher-Serre-1}
E^{(a)}F^{(b)}1_m=\sum_{j=0}^{\mathrm{min}(a,b)}{a-b+m\brack j}_{\mathbb{O}_p}F^{(b-j)}E^{(a-j)}1_m, \\
F^{(a)}E^{(b)}1_m=\sum_{j=0}^{\mathrm{min}(a,b)}{a-b-m\brack j}_{\mathbb{O}_p}E^{(b-j)}F^{(a-j)}1_m. \label{eqn-higher-Serre-2}
\end{align}
\end{subequations}
\end{itemize}
The elements $E^{(a)}1_m$, $F^{(a)}1_{m}$ will be referred to as the \emph{divided power elements}.
\end{defn}

There is a containment of algebras $\dot{u}_{\mathbb{O}_p} \subset \dot{U}_{\mathbb{O}_p}$ by identifying $E^{(n)}1_m$, $F^{(n)}1_m$ ($1\leq n\leq p-1$) with the elements of the same name in $\dot{U}_{\mathbb{O}_p}$.


\subsection{Representations}
Let $ {V}_l $ be the Weyl module for $ \dot{U}_{\mathbb{O}_p} $. It has a basis $ \lbrace v_0, v_1, \ldots, v_{l} \rbrace $ such that
\begin{equation}
\label{irreddef}
1_{n} v_i = \delta_{n,l-2i} v_i \quad \quad  F v_i =[i+1] v_{i+1}\quad\quad E v_i = [l-i+1] v_{i-1}.
\end{equation}
The Weyl module $V_l$ is irreducible when $l \leq p-1$.

On $ V_r \otimes_{\mathbb{O}_p} V_s $ there is an action of $ u_{\mathbb{O}_p} $ given via the comultiplication map $ \Delta $.  
The standard basis of $V_r \otimes V_s$ is given by $\{ v_i \otimes v_j | 0 \leq i \leq r, 0 \leq j \leq s \} $.

The canonical basis (due independently to Kashiwara \cite{Kas91} and Lusztig \cite{Lus4}) is more natural from the categorical perspective we pursue here.  For a detailed study of the canonical basis for quantum $\mathfrak{sl}_2$ see \cite{KhoThesis}.
The basis in Proposition \ref{canbasisprop} is related to that of 
\cite[Section 1.3.2]{KhoThesis} using the isomorphism $E \mapsto E, F \mapsto F, K \mapsto K^{-1}$, $q \mapsto q^{-1}$ and a map $V_l \rightarrow V_l$ where 
$v_i \mapsto v_{l-i}$.

\begin{prop}
\label{canbasisprop}
There is a basis $\{v_b \diamond v_d | 0 \leq b \leq r, 0 \leq d \leq s   \}$
of $V_r \otimes_{\mathbb{O}_p} V_s$ which is given by
\begin{equation} \label{canformulaeq}
v_b \diamond v_d :=
\begin{cases}
F^{(d)} (v_b \otimes v_0) = \sum_{j=0}^d q^{j(j+c)}{b+j \brack j}_{\mathbb{O}_p} v_{b+j} \otimes v_{d-j}  & \text{ if } b \leq c ,\\
E^{(a)}(v_r \otimes v_d)=\sum_{j=0}^a q^{j(j+b)}{c+j \brack j}_{\mathbb{O}_p} v_{b+j} \otimes v_{d-j} & \text{ if } b \geq c ,
\end{cases}
\end{equation}
where $a=r-b$ and $c=s-d$.
\end{prop}

\begin{proof}
Using  equation \eqref{comultform} and induction, one can deduce
\begin{equation}
\label{gencomultform}
\Delta(E^{(t)}) = \sum_{j=0}^t q^{-j(t-j)} E^{(t-j)} K^{-j} \otimes E^{(j)} \quad \quad
\Delta(F^{(t)}) = \sum_{j=0}^t q^{-j(t-j)}  F^{(t-j)} \otimes F^{(j)} K^{t-j}.
\end{equation}
Then one can directly see that the set $\{v_b \diamond v_d | 0 \leq b \leq r, 0 \leq d \leq s   \}$ relates to the usual tensor basis by an upper triangular matrix with diagonal entries $1$.
It follow that it
is a basis. 
\end{proof}

\begin{rem}
\label{canbasisremark}
Note that the elements of the canonical basis in Proposition \ref{canbasisprop} can be written as
\begin{equation*}
v_b \diamond v_d =
\begin{cases}
F^{(d)} E^{(a)} (v_r \otimes v_0) & \text{ if } b \leq c \\
E^{(a)} F^{(d)} (v_r \otimes v_0) & \text{ if } b \geq c
\end{cases}
\end{equation*}
where $a=r-b$ and $c=s-d$.
\end{rem}

From the relations in the quantum group (c.~f.~equations \eqref{eqn-divided-power}, \eqref{eqn-higher-Serre-1} and \eqref{eqn-higher-Serre-2}), along with the description of the canonical basis in Proposition \ref{canbasisprop}, it is clear that the structure coefficients of the canonical basis under the operators $E$ and $F$ are elements of $\N[q,q^{-1}]$.

\section{Elements of hopfological algebra}
\label{sec-hopfo-algebra}
In this Section, we gather some necessary background material on hopfological algebra as developed in \cite{QYHopf}. 

\subsection{\texorpdfstring{$p$}{p}-DG derived categories}
As a matter of notation for the rest of the paper, the undecorated tensor product symbol $\otimes$ will always denote tensor product over the ground field $\Bbbk$. All of our algebras will be graded so $A\dmod$ will denote the category of graded $A$-modules.

\begin{defn}\label{def-p-DGA}Let $\Bbbk$ be a field of positive characteristic $p$. A $p$-DG algebra $A$ over $\Bbbk$ is a $\Z$-graded $\Bbbk$-algebra equipped with a degree-two\footnote{In general one could define the degree of $\dif_A$ to be one. We adopt this degree only to match earlier grading conventions in categorification. One may adjust the gradings of the algebras we consider so as to make the degree of $\dif_A$ to be one, but we choose not to do so.} endomorphism $\dif_A$, such that, for any elements $a,b\in A$, we have
\[
\dif_A^p(a)=0, \quad \quad \dif_A(ab)=\dif_A(a)b+a\dif_A(b).
\]
\end{defn}

Compared with the usual DG case, the lack of the sign in the Leibniz rule is because of the fact that the Hopf algebra $\Bbbk[\dif]/(\dif^p)$ is a genuine Hopf algebra, not a Hopf super-algebra.

As in the DG case, one has the notion of left and right $p$-DG modules.

\begin{defn}\label{def-p-DG-module}Let $(A,\dif_A)$ be a $p$-DG algebra. A left $p$-DG module $(M,\dif_M)$ is a $\Z$-graded $A$-module endowed with a degree-two endomorphism $\dif_M$, such that, for any elements $a\in A$ and $m\in M$, we have
\[
\dif_M^p(m)=0, \quad \quad \dif_M(am)=\dif_A(a)m+a\dif_M(m).
\]
Similarly, one has the notion of a right $p$-DG module.
\end{defn}

It is readily checked that the category of left (right) $p$-DG modules, denoted $(A,\dif)\dmod$ ($(A^{op},\dif)\dmod$), is abelian, with morphisms being those grading-preserving $A$-module maps that also commute with differentials. When no confusion can be caused, we will drop all subscripts in differentials.

\begin{defn} \label{def-null-homotopy}
Let $M$ and $N$ be two $p$-DG modules. A morphism $f:M\lra N$ in $(A,\dif)\dmod$ is called \emph{null-homotopic} if there is an $A$-module map $h$ of degree $2-2p$ such that
\[
f=\sum_{i=0}^{p-1}\dif_N^{i}\circ h \circ \dif_M^{p-1-i}.
\]
\end{defn}

It is an easy exercise to check that null-homotopic morphisms form an ideal in $(A, \dif)\dmod$. The resulting quotient category, denoted $\mathcal{H}(A)$, is called the \emph{homotopy category} of left $p$-DG modules over $A$, and it is a triangulated category.

The simplest $p$-DG algebra is the ground field $\Bbbk$ equipped with the trivial differential, whose homotopy category is denoted $\mc{H}(\Bbbk)$\footnote{This is usually known as the graded stable category of $\Bbbk[\dif]/(\dif^p)$. }. Modules over $(\Bbbk,\dif)$ are usually referred to as \emph{$p$-complexes}. In general, given any $p$-DG algebra $A$, one has a forgetful functor
\begin{equation}
\mathrm{For}:\mathcal{H}(A)\lra \mc{H}(\Bbbk)
\end{equation}
by remembering only the underlying $p$-complex structure up to homotopy of any $p$-DG module over $A$. A morphism between two $p$-DG modules $f:M\lra N$ (or its image in the homotopy category) is called a \emph{quasi-isomorphism} if $\mathrm{For}(f)$ is an isomorphism in $\mc{H}(\Bbbk)$. Denoting the class of quasi-isomorphisms in $\mathcal{H}$ by $\mathcal{Q}$, we define the $p$-DG derived category of $A$ to be
\begin{equation}\label{eqn-derived-cat}
\mathcal{D}(A):=\mathcal{H}(A)[\mathcal{Q}^{-1}],
\end{equation}
the localization of $\mathcal{H}(A)$ at quasi-isomorphisms. By construction, $\mathcal{D}(A)$ is triangulated.

\subsection{Hopfological properties of \texorpdfstring{$p$}{p}-DG modules}
Many constructions in the usual homological algebra of DG-algebras translate over into the $p$-DG context without any trouble. For a starter, it is easy to see that the homotopy category of the ground field coincides with the derived category: $\mc{D}(\Bbbk)\cong \mc{H}(\Bbbk)$. We will see a few more illustrations of the similarities in what follows.

We first recall the following definitions.

\begin{defn}\label{defpdgsummand}
Let $A$ be a $p$-DG algebra, and $M$ be a $p$-DG module over $A$. A \emph{$p$-DG direct summand} $N$ of $M$ is a direct summand of $M$ as an $A$-module, such that the natural inclusion and projection maps between $M$ and $N$ commute with the $p$-differentials.
\end{defn}

\begin{defn}\label{def-finite-cell} Let $A$ be a $p$-DG algebra, and $K$ be a (left or right) $p$-DG module. 
\begin{enumerate}
\item[(1)] The module $K$ is said to satisfy \emph{property P} if there exists an increasing, possibly infinite, exhaustive $\dif_K$-stable filtration $F^\bullet$, such that each subquotient $F^{\bullet}/F^{\bullet -1}$ is isomorphic to a direct sum of $p$-DG direct summands of $A$. 
\item[(2)]The module $K$ is called a \emph{finite cell module}, if it satisfies property P, and as an $A$-module, it is finitely generated (necessarily projective by the property-P requirement).
\end{enumerate}
\end{defn}

Property-P modules are the analogues of K-projective modules in usual homological algebra. For instance, the morphism spaces from a property-P module to any $p$-DG module coincide in both the homotopy and derived categories.

It is a theorem \cite[Theorem 6.6]{QYHopf} that there are always sufficiently many property-P modules: for any $p$-DG module $M$, there is a surjective quasi-isomorphism 
\begin{equation}\label{eqn-bar-resolution}
\mathbf{p}(M)\lra M
\end{equation} 
of $p$-DG modules, with $\mathbf{p}(M)$ satisfying property P.  We will usually refer to such a property-P replacement $\mathbf{p}(M)$ for $M$ as a \emph{bar resolution}. The proof of its existence is similar to that of the usual (simplicial) bar resolution for DG modules over DG algebras. 

In a similar vein, finite cell modules play the role of finitely-generated projective modules in usual homological algebra.

\subsection{\texorpdfstring{$p$}{p}-DG functors}
A $p$-DG bimodule $_AM_B$ over two $p$-DG algebras $A$ and $B$ is a $p$-DG module over $A\otimes B^{\mathrm{op}}$. One has the associated tensor and (graded) hom functors
\begin{equation}\label{eqn-def-tensor}
M\otimes_B(-): (B,\dif)\dmod\lra (A,\dif)\dmod, \quad X\mapsto M\otimes_B X,
\end{equation}
\begin{equation}\label{eqn-def-hom}
\HOM_A(M,-): (A,\dif)\dmod\lra (B,\dif)\dmod, \quad Y\mapsto \HOM_A(M,Y),
\end{equation}
which form an adjoint pair of functors. In fact, we have the following enriched version of the adjunction.

\begin{lem}\label{lem-adjunction-tensor-hom}
Let $A$, $B$ be $p$-DG algebras and $M$ a $p$-DG bimodule over $A\otimes B^{\mathrm{op}}$. Then, for any $p$-DG $A$-module $Y$ and $B$-module $X$, there is an isomorphism of $p$-complexes
\[
\HOM_A(M\otimes_B X, Y)\cong \HOM_B(X,\HOM_A(M,Y)).
\]
\end{lem}
\begin{proof}
See \cite[Lemma 8.5]{QYHopf}.
\end{proof}

The functors descend to derived categories once appropriate property-P replacements are utilized. For instance, the derived tensor functor is given as the composition
 \begin{equation}
M\otimes^{\mathbf{L}}_B(-):\mc{D}(B)\lra \mc{D}(A), X\mapsto M\otimes_B\mathbf{p}_B(X)
\end{equation}
where $\mathbf{p}_B(X)$ is a bar resolution for $X$ as a $p$-DG module over $B$. 
Likewise, the derived $\HOM$, denoted $\RHOM_A(M,-)$, is given by the functor
\begin{equation}
\RHOM_A(M,-) :\mc{D}(A)\lra\mc{D}(B),\quad Y\mapsto  \HOM_{A}(\mathbf{p}_A(M),Y).
\end{equation}
The functors form an adjoint pair, in the sense that
\begin{equation}
\Hom_{\mc{D}(A)}(M\otimes_B^{\mathbf{L}}X,Y)\cong \Hom_{\mc{D}(B)}(X,\RHOM_A(M,Y)).
\end{equation}

One useful result about such functors is the following theorem, whose proof can be found in \cite[Section 8]{QYHopf}.

\begin{thm}\label{thm-qis-functors}
Let $f:{M_1}\lra M_2$ be a quasi-isomorphism of $p$-DG bimodules. Then $f$ descends to an isomorphism of the induced derived tensor product functors. \hfill$\square$
\end{thm}

\subsection{Grothendieck groups}
We next recall the notion of compact modules, which takes place in the derived category.

\begin{defn} \label{def-compact-mod} Let $A$ be a $p$-DG algebra. A $p$-DG module $M$ over $A$ is called \emph{compact} (in the derived category $\mc{D}(A)$) if and only if, for any family of $p$-DG modules $N_i$ where $i$ takes value in some index set $I$, the natural map
\[
\bigoplus_{i\in I}\Hom_{\mc{D}(A)}(M,N_i)\lra \Hom_{\mc{D}(A)}(M,\bigoplus_{i\in I} N_i)
\]
is an isomorphism of $\Bbbk$-vector spaces.
\end{defn}

The strictly full subcategory of $\mc{D}(A)$ consisting of compact modules will be denoted by $\mc{D}^c(A)$.  It is triangulated and will be referred to as the \emph{compact derived category}.

As in the DG case, in order to avoid trivial cancellations in the Grothendieck group, one should restrict the class of objects used to define $K_0(A)$. It turns out that the correct condition is that of compactness. 
So we let $K_0(A):=K_0(\mc{D}^c(A))$.
What we gain as dividend in the current situation is that, since $\mathcal{D}(A)$ is a ``categorical module'' over $\mc{D}(\Bbbk)$, the abelian group $K_0(A)$ naturally has a module structure over the auxiliary cyclotomic ring  at a $p$th root of unity, which was defined in equation \eqref{eqn-aux-ring} of the previous section:
\begin{equation}\label{eqn-aux-ring-2}
\mathbb{O}_p \cong K_0(\mc{D}^c(\Bbbk)).
\end{equation}
The Grothendieck group $K_0(A)$ will be the primary algebraic invariant of the triangulated category $\mc{D}(A)$ that will interest us in this work.

A class of examples for which the Grothendieck group is relatively easy to compute is the following. The notion is introduced in \cite[Section 2]{EQ1}.

\begin{defn}\label{def-positive-p-dga}
 A $p$-DG algebra $A$ is called \emph{positive} if the following three conditions hold:
\begin{itemize}
\item[(1)] $A$ is supported on non-negative degrees: $A= \oplus_{k\in \N}A^k$, and it is finite dimensional in each degree.
\item[(2)] The homogeneous degree zero part $A^0$ is semisimple.
\item[(3)] The differential $\dif_A$ acts trivially on $A^0$.
\end{itemize}
\end{defn}

\begin{thm}
\label{thm-K-group-positive}Let $A$ be a positive $p$-DG algebra. Then there is an isomorphism of Grothendieck groups
\[
K_0(A)\cong K_0^\prime(A)\o_{\Z[q,q^{-1}]}\mathbb{O}_p, 
\]
where $K_0^\prime(A)$ stands for the usual Grothendieck group of graded projective $A$-modules.
\end{thm}
\begin{proof}
See \cite[Corollary 2.18]{EQ1}.
\end{proof}

\begin{rem}[Grading shift]
Let $M$ be a $p$-DG module and $i\in \Z$ be an integer. In what follows, we will abuse notation by writing $q^i M$ for a $p$-DG module $M$ to stand for $M$ with grading shifted up by $i$, i.e., the homogeneous degree $j$ part of $q^iM$, where $j\in \Z$, is set to be the degree $j-i$ part of $M$. More generally, if $g(q)=\sum_{i} a_iq^i\in \N[q,q^{-1}]$, then we will write
\[
g(q)M:=\bigoplus_{i} (q^iM)^{\oplus a_i}.
\]
On the level of Grothendieck groups, the symbol of $g(q)M$ is equal to the multiplication of $g(q)$, regarded here as an element of $\mathbb{O}_p$, with the symbol of $M$.

Likewise, if $\mc{E}$ is a $p$-DG functor given by tensoring with the $p$-DG bimodule $E$ over $(A,B)$,
\[
\mc{E}:(B,\dif)\dmod\lra (A,\dif)\dmod,\quad \quad M\mapsto E\otimes_B M,
\]
then we will write $q^i \mc{E}$ as the functor represented by $q^iE$.
\end{rem}

\section{A double centralizer property}\label{sec-double}

In this section, we analyze the double centralizer property in the context of categorical representation theory, and investigate the $p$-DG analogue.
For two finite-dimensional algebras $B$ and $A$, we will often have a $(B,A)$-bimodule $M$.  Sometimes $M$ will be denoted by ${}_B M, M_A,$ or ${}_B M_A$ depending on what type of module we are emphasizing it to be.

\subsection{Extension of categorical actions I}\label{subsec-double-centralizer}
Consider the following situation. Let $A$ be a finite-dimensional (graded) algebra over $\Bbbk$, and $M$ be a finite-dimensional (graded) right $A$-module. Define the (graded) algebra 
\begin{equation}
B=\END_A(M).
\end{equation}
Then $M$ is naturally a left $B$-module equipped with the action that, for any $b\in B$ and $m\in M$,
\[
b\cdot m:= b(m).
\]
Since the left $B$-action commutes with the right $A$-action on it, $M$ is in fact a $(B,A)$-bimodule. In this way one associates with $M$ two natural functors on the left module categories
\begin{equation}\label{def-JW-functor}
\mc{J}: A\dmod\lra B\dmod,\quad X\mapsto M\otimes_A X,
\end{equation}
\begin{equation}\label{def-Soergel-functor}
\mc{V}: B\dmod\lra A\dmod,\quad Y\mapsto \HOM_B(M,Y),
\end{equation}
with $\mc{V}$ being right adjoint to $\mc{J}$, which also gives rise to a natural transformation of functors
\[
\mc{I}d_{A\dmod}\Rightarrow \mc{V}\circ\mc{J}: A\dmod \lra A\dmod.
\]

\begin{rem}
In our context to follow, it will be too much to ask for both functors to be exact. Instead we will usually require $\mc{V}$ to be exact. In other words, $M$ will usually be a projective left $B$-module. Under some additional assumptions, the functor $\mc{V}$  plays the role of a generalized \emph{Soergel functor}, while the functor $\mc{J}$ (and its cousin $\mc{I}$ to be defined later on) plays the role of a \emph{projection functor}.
\end{rem}

Our next goal is to determine a condition under which categorical constructions on the $A$-module level will automatically translate into categorical constructions on the $B$-module level. To do so, let us consider the following situation. Let $A_i$, $(i=1,2,3)$, be three finite-dimensional algebras, and $M_i$, $(i=1,2,3)$, be right modules over respective $A_i$'s. Set $B_i:=\END_{A_i}(M_i)$ for each $i$. Suppose we now have functors between $A_i$-module categories given by bimodules $_{A_2}{E_1}_{A_1}$, $_{A_3}{E_2}_{A_2}$:
\[
\mc{E}_1:A_1\dmod\lra A_2\dmod,\quad X\mapsto E_1 \otimes_{A_1} X ,
\]
\[
\mc{E}_2:A_2\dmod\lra A_3\dmod,\quad Y\mapsto E_2 \otimes_{A_2} Y .
\]
Together with the generalized Soergel and projection functors, we have a diagram:
\begin{equation}
\begin{gathered}
\xymatrix{
A_1\dmod \ar[r]^{\mc{E}_1}\ar@<-0.7ex>[d]_{\mc{J}_1} & A_2\dmod \ar[r]^{\mc{E}_2}\ar@<-0.7ex>[d]_{\mc{J}_2} & A_3\dmod\ar@<-0.7ex>[d]_{\mc{J}_3}\\
B_1\dmod \ar@<-0.7ex>[u]_{\mc{V}_1} & B_2\dmod \ar@<-0.7ex>[u]_{\mc{V}_2} & B_3\dmod \ar@<-0.7ex>[u]_{\mc{V}_3}
}
\end{gathered} \ .
\end{equation}
Composing functors gives rise to
\[
\mc{E}_1^{\prime}:= \mc{J}_2\circ \mc{E}_1\circ \mc{V}_1,\quad
\mc{E}_2^{\prime}:= \mc{J}_3\circ \mc{E}_2\circ \mc{V}_2
\]
and their composition
\begin{equation}
\mc{E}_2^\prime\circ \mc{E}_1^{\prime}: B_1\dmod \lra B_3\dmod.
\end{equation}
It is a natural question to ask whether this functor agrees with the composition
\begin{equation}
\mc{J}_3\circ \mc{E}_2\circ \mc{E}_1\circ \mc{V}_1:B_1\dmod\lra B_3\dmod.
\end{equation}

The two functors will indeed agree if
\begin{equation}\label{eqn-compo-id}
\mc{I}d_{A_2\dmod}\cong\mc{V}_2\circ \mc{J}_2.
\end{equation}
In this case, any natural transformation of the functor
\[
\mc{E}_2\circ \mc{E}_1 = (E_2\otimes_{A_2}E_1)\otimes_{A_1}(-): A_1\dmod\lra A_3\dmod,
\]
which arises as an $(A_3,A_1)$-bimodule homomorphism
\[
x: (E_2\otimes_{A_2}E_1) \lra (E_2\otimes_{A_2} E_1),
\]
will also induce a natural transformation of the functor $\mc{E}_2^\prime \circ \mc{E}_1^\prime\cong \mc{J}_3\circ\mc{E}_2\circ\mc{E}_1\circ \mc{V}_1$.

\begin{defn}\label{def-extension-categorical-action}
Let $A_i$, $i\in I$, be a family of algebras, and for each $i$ choose a right $A_i$-module $M_i$. Set $B_i=\END_{A_i}(M_i)$ so that each $M_i$ is a $(B_i, A_i)$-bimodule. The categories of left modules $A_i\dmod$ and $B_i\dmod$ are connected via the generalized \emph{Soergel functor} $\mc{V}_i$ and \emph{projection functor} $\mc{J}_i$:
\[
\mc{J}_i: A_i\dmod\lra B_i\dmod,\quad X\mapsto M_i\otimes_{A_i} X,
\]
\[
\mc{V}_i: B_i\dmod\lra A_i\dmod,\quad Y\mapsto \HOM_{B_i}(M_i,Y).
\]
Suppose that, in addition, there are functors, one for each pair of $(i,j)\in I^2$, acting on $A_i$-module categories, which are given by bimodules ${_{A_i}{E_{ij}}_{A_j}}$:
\[
\mc{E}_{ij}:A_j\dmod\lra A_i\dmod,\quad X\mapsto E_{ij} \otimes_{A_j} X.
\]
We will say that \emph{the categorical action on $\oplus_{i\in I} A_i\dmod$ by the bimodules $E_{ij}$'s  extends to $\oplus_{i\in I} B_i\dmod$} if the natural transformation of functors (induced by adjunctions $\mc{I}d_{A_j\dmod}\Rightarrow \mc{V}_j\circ \mc{J}_j$)
\begin{equation}
\mc{J}_i\circ \mc{E}_{ij}\circ \mc{E}_{jk}\circ \mc{V}_k\Rightarrow \mc{J}_i\circ \mc{E}_{ij}\circ\mc{V}_j\circ\mc{J}_j\circ \mc{E}_{jk}\circ \mc{V}_k:B_k\dmod\lra B_i\dmod
\end{equation}
is always an isomorphism, for any $i,j,k\in I$.
\end{defn}

Let us now analyze when the condition \eqref{eqn-compo-id} holds. This happens if and only if the adjunction map of functors
\begin{equation}
\mc{I}d_{A_2\dmod}\Rightarrow \mc{V}_2\circ \mc{J}_2
\end{equation}
is an isomorphism. In other words, for any left $A_2$-module $Y$, there needs to be an isomorphism
\begin{equation}
Y\lra\HOM_{B_2}(M_2,M_2\otimes_{A_2}Y) .
\end{equation}
Taking $Y=A_2$, we need the $(A_2,A_2)$-bimodule homomorphism
\begin{equation}
A_2\lra\HOM_{B_2}(M_2,M_2)\quad a\mapsto (m\mapsto ma)
\end{equation}
to be an isomorphism of bimodules.
\begin{lem}\label{lem-compo-identity-I}
Let $M$ be a $(B,A)$-bimodule which is projective as a left $B$-module. Suppose there is an isomorphism of $(A,A)$-bimodules
\[
A\cong \HOM_{B}(M,M).
\]
Then the natural transformation of functors $\mc{I}d_{A\dmod}\Rightarrow \mc{V}\circ \mc{J}$ is an isomorphism.
\end{lem}
\begin{proof}
Let $A^m\lra A^n\lra Y\lra 0 $ be a projective presentation of $Y$, where $m,n\in \N\cup \{\infty\}$. Since tensor product is right exact, we have, by applying $M\otimes_A(-)$ to the above, an exact sequence
\[
M^m\lra M^n\lra M\otimes_A Y\lra 0.
\]
Then, as $M$ is projective over $B$, $\HOM_B(M,-)$ is exact, and we have the sequence
\[
\HOM_B(M,M)^m\lra \HOM_B(M,M)^n\lra \HOM_B(M,M\otimes_A Y)\lra 0
\]
is exact. Using the assumption $A\cong \HOM_{B}(M,M)$, we deduce, by the classical Five Lemma, that $Y\cong \HOM_B(M,M\otimes_A Y)$. The result follows.
\end{proof}

Together with the definition of the $B$-algebras, the discussion leads us to consider the \emph{double centralizer condition}, which requires that, for a $(B,A)$-bimodule ${_BM_A}$, one has
\begin{equation}\label{eqn-double-centralizer}
B=\END_A(M_A),\quad \quad A=\END_B (_B M).
\end{equation}

We summarize the above discussion into the following theorem.

\begin{thm}\label{thm-action-extension}
Under the conditions of Definition \ref{def-extension-categorical-action}, the categorical action on $\oplus_{i\in I} A_i\dmod$ by the bimodules $E_{ij}$'s extends to a categorical action on $\oplus_{i\in I} B_i\dmod$ if, for each $i\in I$, the double centralizer property
\[
B_i=\END_{A_i}(M_{A_i}),\quad \quad A_i=\END_{B_i} (_{B_i}M),
\]
holds on the $(B_i,A_i)$-bimodules ${_{B_i}{M_i}_{A_i}}$, and $M_i$ is projective as a $B_i$-module. \hfill$\square$
\end{thm}

In general, we are not aware of the most natural conditions for a bimodule to satisfy \eqref{eqn-double-centralizer}. We will, however, move on to understand when the double centralizer condition holds for some particular class of algebras and bimodules.

\subsection{Self-injective algebras}\label{subsec-self-injective}
In this section, we will recall a class of examples for which the double centralizer property \eqref{eqn-double-centralizer} holds. The result is well-known in the ungraded case, see, for instance, Curtis-Reiner \cite{CurRei}. For a more modern account that applies to some more general cases, see K\"onig-Slung\aa rd-Xi \cite{KSX}. We record the proofs here as well as some consequences for the sake of completeness.

Throughout this subsection, let $A$ be a (graded) finite-dimensional \emph{self-injective algebra}, i.e, the left or right regular module $A$ is injective. We will take $M$ to be a (graded) finitely-generated faithful right $A$-module.

\begin{lem}\label{lem-faithful-A-mod}
$A$ is a direct summand of $M^{\oplus r}$, for some $r>0$.
\end{lem}
\begin{proof}
Write $A$ as a direct sum of indecomposable injective submodules. Each indecomposable injective $I_N\subset A$ is the injective envelope of its (simple) socle $N$. By the assumption, $N$, regarded as a submodule in the socle of $A$, embeds into $M$ since it is simple and $A$ acts faithfully on $M$. By injectivity of $I_N$, the embedding of $N$ into $M$ extends to an embedding of $I_N$ into $M$. Hence $I_N$ is a direct summand of $M$, and the result follows.
\end{proof}

Write $M_A=P_A\oplus N_A$, where $P_A$ is the direct sum of projective-injective indecomposable summands of $M_A$, and $N_A$ is a complementary module. By the previous lemma, every indecomposable projective summand of $A$ figures in $P_A$ as a direct summand and vice versa.

A main goal of this section is to establish the following double centralizer property \eqref{eqn-double-centralizer} in this particular context, using ideas of \cite{CurRei}.

\begin{thm}\label{thm-double-centralizer}
Let $A$ be a finite-dimensional graded self-injective algebra, and $M_A$ a faithful finitely-generated graded right $A$-module. Set $B=\END_A(M_A)$ so that $M$ is a $(B,A)$-bimodule. Then there is a canonical isomorphism of graded algebras:
\[
A=\END_B({_BM}), \quad a\mapsto (m\mapsto ma)
\]
for any $a\in A$ and $m\in M$.
\end{thm}

We first note the following Morita equivalence result, allowing us to exchange $M_A$ with $M_A^{\oplus r}$ for any positive integer $r$.

\begin{lem}\label{lem-Morita-reduction}
Let $A$ be a ring and $M_A$ be a right $A$-module. Then $A$ and $\END_A(M)$ satisfy the double centralizer property if and only if $A$ and $\END_A(M^{\oplus r})$ satisfy the double centralizer property for any positive integer $r$.
\end{lem}
\begin{proof}
Let $B=\END_A(M)$. Then $\END_A(M^{\oplus r})$ can be identified with an $r\times r$ matrix algebra $\mathrm{M}(r,B)$ with coefficients in $B$. These algebras are naturally Morita equivalent, and an equivalence is given by tensor product with the standard column module:
\[
 B\dmod\lra\mathrm{M}(r,B)\dmod ,\quad {_BK}\mapsto B^{\oplus r}\otimes_B K.
\]
Under this equivalence, the module ${M}$ is naturally sent to $M^{\oplus r}$. Therefore, if one of $B$ or $M(r,B)$ satisfies the double centralizer property, so that $\END_B(M)\cong A$ or $\END_{\mathrm{M}(r,B)}(M^{\oplus r})\cong A$, then so does the other, as the Morita equivalence preserves endomorphism spaces.
\end{proof}

\begin{proof}[Proof of Theorem \ref{thm-double-centralizer}.]
By Lemma \ref{lem-faithful-A-mod}, there is an $r\in \N_{>0}$ such that $M^{\oplus r}$ contains $A$ as a direct summand. Lemma \ref{lem-Morita-reduction} allows us to replace $M$ by $M^{\oplus r}$. Thus let us assume from the start that there is a decomposition $M_A\cong A\oplus N_A$. Then we may formally identify $B=\END_A(M_A)$ as a matrix algebra:
\[
B\cong
\left(
\begin{matrix}
\END_A(A) & \HOM_A(N,A)\\
 \HOM_A(A,N) & \END_A(N)
\end{matrix}
\right) \ .
\] 
Let $e_A\in B$ be the idempotent corresponding to the identity of $\END_A(A)$, and likewise for $e_N$. Then for any $(a,n)\in A\oplus N$
\[
e_A(a,n) = (a, 0) \quad \quad e_N(a, n)=(0,n).
\]
Given any $x\in \END_B(M)$ and assuming $(a,n)x=(a_x,n_x)$, we have
\[
(a,0)x=(e_A(a,n))x=e_A((a, n)x)=(a_x,0),\quad \quad (0,n)x=(e_N(a,n))x=e_N((a, n)x)=(0,n_x).
\]
Therefore, any $x\in \END_B(M)$ acts on $A\oplus N$ componentwise. Furthermore, since the map $a\mapsto a_x$ is left linear over $\END_A(A)\cong A$, it follows that there is an $x_0\in A$ such that $a_x=ax_0$ for all $a\in A$. 

The southwestern corner of the matrix description of $B$ can be identified with $\HOM_A(A,N)\cong N$. For any $n\in N$, let us define an element $b_n\in B$ sitting in the southwestern corner by
\[
b_n(a,n^\prime)=(0,na)
\]
for all $(a,n^\prime)\in A\oplus N$. Then we have,
\[
(b_n(a,n^\prime))x=b_n((a,n^\prime)x),
\]
which, in turn, is equal to
\[
(0,(na)_x)=(0,na_x)=(0,nax_0).
\]
Taking $a=1_A$ shows that $n_x=nx_0$. Thus, for all $(a,n)\in A\oplus N$, we have found an $x_0\in A$ such that
\[
(a,n )x=(ax_0,n x_0)=(a,n )x_0.
\]
The desired claim now follows.
\end{proof}

Let $A$, $B$ and $M$ be as in Theorem \ref{thm-double-centralizer}. We have the following.

\begin{cor}\label{cor-center-double-centralizer}
The centers of the algebras $A$ and $B$ are canonically isomorphic.
\end{cor}
\begin{proof}
Denote the center of $A$ and $B$ by $z(A)$ and $z(B)$ respectively.

The action of the center $z(A)$ on $M$ commutes with that of $A$, and therefore $z(A)$ is a subalgebra of $B=\END_A(M)$, as the $A$, and thus the $z(A)$, action is faithful on $M$. As $z(A)$ is a commutative subalgebra of $A$, it commutes with the action of $B$, and therefore $z(A)\subset z(B)$. The reverse inclusion holds by interchanging the roles of $A$ and $B$, and the result follows.
\end{proof}

To conclude this subsection, we will prove an exactness result about the Soergel functor $\mc{V}$.

\begin{lem}\label{lem-proj-inj}
As a left module over $B$, the module $M$ in Theorem \ref{thm-double-centralizer} is projective and injective.
\end{lem}
\begin{proof}
First, we show that $M$ is a projective left $B$-module. By Lemma \ref{lem-faithful-A-mod}, there is an $r$ such that $M^{\oplus r}$ contains $A$ as a direct summand. Write $M^{\oplus r} \cong A\oplus M^\prime$, where $M^{\prime}$ is some complementary module. Then
\[
{M}^{\oplus r} \cong \HOM_A(A_A,M^{\oplus r})\cong \HOM_A(eM^{\oplus r}, M^{\oplus r})\cong \HOM_A(M^{\oplus r}, M^{\oplus r})e,
\]
where $e$ is the idempotent in $\END_A(M^{\oplus r})$ given as a composition of right $A$-modules
\[
e:M^{\oplus r}\lra A \lra M^{\oplus r}.
\]
It follows that $M^{\oplus r}$ is a projective left module over the ring $\END_A(M^{\oplus r})$. Since $B$ and $\END_A(M^{\oplus r})$ are Morita equivalent, the claim follows.

If $J_A$ is any indecomposable summand of $A$ as a right module, then $J_A$ is projective and injective as a right $A$-module. 
Now, as a right $B$-module,
\[
B_B=\HOM_A(M_A,M_A)\cong \HOM_A(M_A,J_A)\oplus \HOM_A(M_A,J^\prime_A),
\]
where $J^\prime_A$ is some complementary module of $J_A$ in $M_A$. It follows that $\HOM_A(M_A,J_A)$ is a projective right $B$-module. By taking the vector space dual, we have that
 $$\HOM_A(M_A,J_A)^*=\HOM_{\Bbbk}(\HOM_A(M_A,J_A),\Bbbk)$$
 is an injective left $B$-module. We would like to use this to show that $_BM$ is injective as a left $B$-module. Notice that the vector space dual of $J_A$, $J_A^*=\HOM_\Bbbk(J_A,\Bbbk)$, carries a natural left $A$-module structure, and is, in fact, an injective and projective left $A$-module. We have an isomorphism of left $B$-modules 
\begin{align*}
{_BM_A}\otimes_A J_A^* & \cong ({_BM_A}\otimes_A J_A^*)^{**}\cong \HOM_{\Bbbk}({_BM_A}\otimes_AJ_A^*,\Bbbk)^*\\
&\cong \HOM_A({_BM_A},\HOM_\Bbbk(J_A^*,\Bbbk))^*\cong \HOM_A({_BM_A},J_A)^*.
\end{align*}
It follows that ${_BM_A}\otimes_AJ_A^*$ is an injective $B$-module. Furthermore, since $J_A^*$ is an indecomposable summand of $A$ as a left module, and $A$ decomposes into a direct sum of such $J_A^*$'s with various multiplicities,
we thus have that $ {_BM}\cong {_BM_A}\otimes_A A$ is a direct sum of modules of the form ${_BM_A}\otimes_A J_A^*$. The injectivity of $_BM$ follows.
\end{proof}

The following consequence of the lemma is then immediate.

\begin{cor}
The generalized Soergel functor 
\[
\mc{V}: B\dmod\lra A\dmod,\quad Y\mapsto \HOM_B(M,Y)
\]
is exact. \hfill $\square$
\end{cor}

The functor $\mc{V}$, however, does not necessarily send projective $B$-modules to projective $A$-modules, and thus does not descend to Grothendieck groups of finitely generated projective modules.

\subsection{Frobenius algebras}\label{subsec-Frob-algebra}
We next would like to give a criterion under which the generalized Soergel functor $\mc{V}:B\dmod\lra A\dmod$ is always fully-faithful on projective modules. To do this, we use the following lemma.

\begin{lem}\label{lem-two-duals-isomorphic}
Let ${_BM_A}$ be a $(B,A)$-bimodule on which $A$ and $B$ are centralizers of each other. Then there is a canonical isomorphism
\[
\HOM_{B}({_BM},B)\cong \HOM_A({M_A},A)
\]
as $(A,B)$-bimodules.
\end{lem}
\begin{proof}
Let us introduce notations ${^\vee M}:=\HOM_{B}({_BM},B)$ (left dual) and $M^\vee:=\HOM_A({M_A},A)$ (right dual). There are natural evaluation maps
\[
 M\otimes_A {^\vee M}\lra B,\quad x\otimes_A g \mapsto \langle x,g\rangle_B,
\]
\[
{M^\vee }\otimes_B M\lra A,\quad f\otimes_B x \mapsto \langle f,x\rangle_A,
\]
which are respectively $(B,B)$ and $(A,A)$-bimodule homomorphisms. 

Consider the canonical map
\[
{^\vee M}\otimes_B M\lra \END_B(M),\quad g\otimes_Bx\mapsto (y\mapsto \langle y,g \rangle_B x).
\]
Together with the double commutant property $\END_B(M)\cong A$, we have obtained an $(A,A)$-module homomorphism
\begin{equation}\label{eqn-can-map-1}
{^\vee M}\otimes_B M\lra A, \quad g\otimes_B x\mapsto (y\mapsto \langle y,g \rangle_B x= ya_{g,x}),
\end{equation}
where $a_{g,x}\in A$ is the unique element associated with $g\otimes_B x$.

Likewise, we have another canonical left $(B,B)$-module map
\begin{equation}\label{eqn-can-map-2}
M\otimes_AM^\vee \lra B,\quad x\otimes_A f\mapsto b_{x,f},
\end{equation}
where $b_{x,f}\in B$ is the unique element such that
\[
b_{x,f}y=x\langle f,y \rangle_A
\]
for all $y\in M$. 

By the tensor-hom adjunction, we have isomorphisms
\[
\HOM_A({^\vee M}\otimes_B M,A)\cong \HOM_B({^\vee M}, \HOM_A(M,A))=\HOM_B({^\vee M}, M^\vee),
\]
\[
\HOM_B(M\otimes_AM^\vee , B)\cong \HOM_A(M^\vee, \HOM_B(M,B))= \HOM_A(M^\vee, {^\vee M}).
\]
The images of the canonical maps \eqref{eqn-can-map-1} and \eqref{eqn-can-map-2} are uniquely determined by 
\begin{equation}
\phi: {^\vee M}\lra M^\vee,  \quad \quad y \langle \phi(g),x\rangle_A=\langle y, g\rangle_B x,
\end{equation}
\begin{equation}
 \psi: M^\vee\lra M^\vee, \quad \quad \langle x, \psi(f)\rangle_By=x\langle f,y \rangle_A.
\end{equation}
for any $x,y\in M$.

Now, given any $x,y\in M$, $g\in {^\vee M}$ and $f\in M^\vee$, we have
\[
\langle x, \psi\circ \phi(g) \rangle_B y= x \langle \phi(g),y \rangle_A =\langle x, g\rangle_B y,
\]
and likewise
\[
y \langle  \phi\circ \psi(f), x \rangle_A =  \langle y, \psi(f) \rangle_B x =y \langle f,x\rangle_A .
\]
By the faithfulness of the actions, we obtain
\[
\langle x,\psi\circ\phi (g)\rangle_B =\langle x,g\rangle_B \quad \quad
\langle  \phi\circ \psi(f), x \rangle_A=\langle  f, x \rangle_A.
\]
Thus we have equalities $\psi \circ \phi (g)=g$ and $\phi\circ \psi (f)=f$. It follows that $\psi$ and $\phi$ are inverse of each other, and the lemma follows.
\end{proof}

In what follows we will specialize to the case when $A$ is a Frobenius algebra. Recall that a \emph{Frobenius algebra} is a finite-dimensional algebra equipped with a non-degenerate trace map $\epsilon : A\lra \Bbbk$.
The (graded) vector space dual $A^*=\HOM_\Bbbk(A,\Bbbk)$ is naturally an $(A,A)$-bimodule via the action
\begin{equation}
(a\cdot f\cdot b)(c):=f(bca),\quad f\in A^*,\quad a,b,c\in A.
\end{equation} 
The trace map induces an isomorphism $A\lra A^*$, $a\mapsto a\cdot \epsilon$, as left $A$-modules. The \emph{Nakayama endomorphism} $\alpha: A\lra A$ is characterized by
\begin{equation}
\epsilon(ab)=\epsilon(\alpha(b)a),
\end{equation}
for any $a,b\in A$. Using the Nakayama automorphism, we define a twisted bimodule structure on $A^{\beta}$ by keeping the left regular module structure on $A$, while twisting the right regular $A$-module structure through $\beta=\alpha^{-1}$. Then the isomorphism ${_AA}\cong {_AA^*}$ can be strengthened to an $(A,A)$-bimodule isomorphism between $A^{\beta}$ and $A^*$:
\[
A^{\beta} \cong A^*, \quad  a\mapsto a\cdot \epsilon.
\]
This is true because, for any $a,b,c\in A$, we have
\[
a\cdot 1\cdot b = a \beta(b) \mapsto ((a \beta(b))\cdot \epsilon)(c)=\epsilon(ca\beta(b))=\epsilon(bca)=(a\cdot\epsilon\cdot b)(c).
\]

\begin{thm}\label{thm-exactness-Soergel}
Let $A$ be a Frobenius algebra, $M_A$ be a faithful $A$ module and $B=\END_A(M_A)$. Then the Soergel functor $\mc{V}$ is fully-faithful on projective $B$-modules.
\end{thm}
\begin{proof}
It suffices to prove the result for the left regular projective module $B$, since each indecomposable $B$-module is a direct summand of $B$.

By the previous lemma and the argument in the proof of Lemma \ref{lem-proj-inj}, we have an identification of $(A,B)$-bimodules
\[
\mc{V}(B)=\HOM_B({_BM},B)\cong \HOM_A(M_A,A) \cong (M\otimes_A A^*)^*.
\]
Since $A$ is Frobenius, we have, by the above discussion, that $A^*\cong A^{\beta}$,
where $\beta=\alpha^{-1}$ is the inverse Nakayama automorphism. 
Therefore, the isomorphism continues
\[
\mc{V}(B)\cong (M^{\beta})^*={^{\beta}M^*},
\]
where $^\beta M^*$ indicates that the left $A$-module structure on $M^*$ is twisted by the inverse Nakayama automorphism $\beta$. 

Notice that twisting the entire $A$-module category by an automorphism of $A$ induces an automorphism of the $A$-module category, and thus it preserves $\HOM$-spaces:
\begin{align*}
\HOM_A({^\beta M_1},{^\beta M_2}) & \cong \HOM_A({^\beta A}\otimes_A M_1,{^\beta M_2})\cong \HOM_A(M_1,\HOM_A({^\beta A},{^\beta M_2}))\\
& \cong \HOM_A(M_1,M_2).
\end{align*}
Using this equality, we have isomorphisms
\begin{align*}
\HOM_A({^{\beta} M^*},{{^{\beta}} M^*})& \cong \HOM_A(M^*,M^*)
\cong \HOM_A(M^*,\HOM_\Bbbk(M,\Bbbk))\\
& \cong \HOM_\Bbbk(M\otimes_A M^*,\Bbbk)\cong \HOM_A(M_A,M_A)\cong B,
\end{align*}
where the tensor-hom adjunctions are repeatedly used. The theorem now follows.
\end{proof}

\subsection{Extension of categorical actions II}
\label{subsec-ext-cat-act-ii}
Let us consider, back in the setting of Theorem \ref{thm-action-extension}, the following question: under what conditions on the functor $\mc{E}_{ij}$ does the induced functor
\[
\mc{E}_{ij}^\prime: B_j\dmod\stackrel{\mc{V}_j}{\lra}A_j\dmod \stackrel{\mc{E}_{ij}}{\lra} A_i\dmod \stackrel{\mc{J}_i}{\lra} B_i\dmod
\]
descend to a map of Grothendieck groups of finitely-generated projective modules?

One immediate problem arises from the fact that the functor $\mc{J}_i$ is not always exact, and therefore the composition $\mc{E}_{ij}^\prime$ will usually not be exact. 
To avoid this problem, we will, in this subsection, replace the functor $\mc{J}_i$ by a better-behaved adjoint to the Soergel functor $\mc{V}_i$.

We start by observing that, in the setting of Section \ref{subsec-double-centralizer}, if $M$ is a $(B,A)$-bimodule that is projective as a $B$-module, then we have two incarnations of the Soergel functor:
\begin{equation}
\HOM_B(M,-)\cong {^\vee M}\otimes_B (-): B\dmod \lra A\dmod,
\end{equation}
provided that $M$ is finitely-generated. Here ${^\vee M}:= \HOM_B({_BM},B)$ is an $(A,B)$-bimodule that is $B$-projective. This allows us to consider the right adjoint of $\mc{V}:B\dmod\lra A\dmod$, which is given by
\begin{equation}
\mc{I}: A\dmod\lra B\dmod, \quad X\mapsto \HOM_A({^\vee M}, X).
\end{equation}

\begin{lem}\label{lem-double-centralizer-dual-mod}
Let $A$ be a self-injective algebra, $M_A$ be a finite-dimensional faithful right $A$-module, and set $B=\END_A(M_A)$. Then $(A,B)$ satisfies the double centralizer property on the left dual module ${^\vee M}=\HOM_B({_B M},B)$.
\end{lem}
\begin{proof}
By Lemma \ref{lem-proj-inj}, ${_B M}$ is projective, and thus
\[
\HOM_B({^\vee M},B)\cong M.
\]
We first compute
\begin{align*}
\HOM_A({^\vee M}, {^\vee M}) & = \HOM_A({^\vee M}, \HOM_B(M,B))\cong \HOM_B(M\otimes_A{^\vee M},B)\\
&\cong \HOM_A(M_A,\HOM_B({^\vee M}, B))
\cong \HOM_A(M,M)\cong B.
\end{align*}
On the other hand, since both $M$ and ${^\vee M}$ are finitely-generated $B$-projective, we have
\begin{align*}
\HOM_B({^\vee M},{^\vee M}) & ={^\vee M}\otimes_B\HOM_B({^\vee M}, B)  \cong  {^\vee M}\otimes_B M \\ & \cong \HOM_B(M,B)\otimes_B M \cong \HOM_B(M,M)\cong A.
\end{align*}
The lemma follows.
\end{proof}

\begin{lem}\label{lem-comp-identity-II}
Let $A$ be a self-injective algebra, and $M_A$ be a finite-dimensional faithful right $A$-module. Then the composition of functors
\[
\mc{V}\circ \mc{I}: A\dmod\lra A\dmod,\quad \quad X\mapsto \HOM_B(M,\HOM_A({^\vee M}, X))
\]
is equivalent to the identity functor on $A\dmod$.
\end{lem}
\begin{proof}
Choose an injective presentation of $X$ as a left $A$-module:
\[
0\lra X\lra A^m\lra A^n,
\]
where $m,n\in \N\cup \{\infty\}$. Since $\HOM_A({^\vee M},-)$ is left exact, we have that
\[
0\lra \HOM_A({^\vee M}, X)\lra  \HOM_A({^\vee M} , A)^m \lra  \HOM_A({^\vee M}, A)^n
\]
is exact. By the previous lemma, since ${^\vee M}$ satisfies the double centralizer property, we have that the left and right duals of $^\vee M$ coincide (Lemma \ref{lem-two-duals-isomorphic}):
\[
\HOM_A({^\vee M}, A)\cong \HOM_B({^\vee M}, B).
\]
Thus the above sequence becomes 
\[
0\lra  \HOM_A({^\vee M}, X) \lra M^m\lra M^n.
\]
Applying the (exact) Soergel functor to the sequence gives us
\[
0\lra  \HOM_B({ M}, \HOM_A({^\vee M}, X))\lra A^m\lra A^n.
\] 
The result follows from the Five Lemma.
\end{proof}

We can now state a variation of Theorem \ref{thm-action-extension}. Let $A_i$, $i\in I$, be a family of self-injective algebras, and $M_i$, one for each $i\in I$, be faithful right modules over the respective $A_i$'s. Suppose there are also functors between $A_i$-module categories given by bimodules $_{A_i}{E_{ij}}_{A_j}$ ($i,j\in I$)
\[
\mc{E}_{ij}:A_j\dmod\lra A_i\dmod,\quad X\mapsto E_{ij} \otimes_{A_j} X.
\]
We have a diagram:
\begin{equation}
\begin{gathered}
\xymatrix{
A_i\dmod \ar[r]^{\mc{E}_{ji}}\ar@<-0.7ex>[d]_{\mc{I}_i} & A_j\dmod \ar[r]^{\mc{E}_{kj}}\ar@<-0.7ex>[d]_{\mc{I}_j} & A_k\dmod\ar@<-0.7ex>[d]_{\mc{I}_k}\\
B_i\dmod \ar@<-0.7ex>[u]_{\mc{V}_i} & B_j\dmod \ar@<-0.7ex>[u]_{\mc{V}_j} & B_k\dmod \ar@<-0.7ex>[u]_{\mc{V}_k}
}
\end{gathered} \ .
\end{equation}
Composing functors gives rise to
\[
\mc{E}_{ij}^{!}:= \mc{I}_i\circ \mc{E}_{ij}\circ \mc{V}_j: B_j\dmod\lra B_i\dmod.
\]
Lemma \ref{lem-comp-identity-II} shows that the compositions are equal:
\begin{equation}
\mc{E}_{kj}^!\circ \mc{E}_{ji}^{!}=\left( \mc{E}_{kj}\circ \mc{E}_{ji}\right)^!:=\mc{I}_k \circ \mc{E}_{kj}\circ \mc{E}_{ji} \circ \mc{V}_i: B_i\dmod \lra B_k\dmod.
\end{equation}

The same reasoning as for Theorem \ref{thm-action-extension} proves the following.

\begin{cor}\label{cor-ext-cat-actions}
Under the above conditions, the categorical action on $\oplus_i A_i\dmod$ by the functors $\mc{E}_{ij}$'s extends to a categorical action on $\oplus_i B_i\dmod$ by the functors $\mc{E}_{ij}^!$'s. \hfill $\square$
\end{cor}

The functor $\mc{E}_{ij}^!$ behaves better on the class of projective $B_i$-modules, as we will see. In turn, the induced action on the Grothendieck groups of projective modules will be interesting in what follows. 

Recall that the additive envelope of a module $N$ in $A\dmod$, denoted $(N)_A$, consists of direct summands of finite direct sums of $N$. The additive envelope is equivalent to finitely generated projective modules over $\END_A(N)$.

Let us consider a simplified situation as follows, which will be applied later. Suppose also that the family of self-injective algebras $A_i$ are \emph{symmetric} Frobenius algebras, i.e., the trace pairing $\epsilon:A_i\lra \Bbbk$ satisfies
\[
\epsilon(ab)=\epsilon(ba)
\]
for all $a,b\in A_i$. The Nakayama automorphism and its inverse are both equal to the identity map on $A$, and $A\cong A^*$ as an $(A,A)$-bimodule.

For a symmetric Frobenius algebra $A$, we have
\[
\HOM_A(M_A,A)\cong (M\otimes_A A^*)^* \cong M^*.
\]
Thus if $M_A$ is a faithful $A$-module, we have an identification of $(A,B)$-bimodules (Lemma \ref{lem-two-duals-isomorphic})
\[
\HOM_B({_BM},B)\cong\HOM_A(M_A,A)\cong M^*.
\]

For the next theorem, assume we are in the setting of Corollary \ref{cor-ext-cat-actions}, and futhermore that $A_i$'s are symmetric Frobenius. The non-symmetric cases can be treated similarly with appropriate twists by the Nakayama automorphisms inserted.

\begin{thm}\label{thm-extension-symmetric-Frob}
If the functors $\mc{E}_{ij}$'s ($i,j\in I$) preserve the additive envelopes of the left $A_i$-modules $M_i^*$, then the extended functors $\mc{E}_{ij}^!$ preserve finitely-generated projective modules in $\oplus_{j\in }B_j\dmod$.
\end{thm}
\begin{proof}
It suffices to show that $\mc{E}_{ij}^!(B_j)$ is a finitely-generated projective $B_i$-module. As in the proof of Theorem \ref{thm-exactness-Soergel}, we have
\[
\mc{V}_j(B_j)=\HOM_{B_j}({_{B_j}M_j},B_j)\cong B_j\otimes_{B_j}{^\vee M_j}\cong M_j^*.
\]
Since $\mc{E}_{ij}$ sends $M_j^*$ inside the additive envelope of $M_i^*$, it now suffices to show that $\mc{I}_i$ takes $M_i^*$ to a finitely-generated projective $B_i$-module. But this is now automatic, since $M_i^*\cong {^\vee M_i}$ and
\[
\mc{I}_i(M_i^*)\cong \HOM_{A_i}({^\vee M_i}, {^\vee M_i})\cong B_i
\]
via Lemma \ref{lem-double-centralizer-dual-mod}. The theorem follows.
\end{proof}

\subsection{A \texorpdfstring{$p$}{p}-DG context}
In this subsection, we will establish the $p$-DG analogues of the results in the previous two subsections.

Let $(A,\dif)$ be a $p$-DG algebra whose underlying algebra $A$ is self-injective (Frobenius). In this case we will say that $(A,\dif)$ is a \emph{self-injective (Frobenius) $p$-DG algebra}.

We start with a simple case.

\begin{lem}\label{lem-p-dg-cofibrance}
Let $A$ be a self-injective $p$-DG algebra, and $M_A$ a right $p$-DG module containing $A$ as a $p$-DG direct summand. Then $_BM$ is cofibrant as a left $p$-DG module over $B$.
\end{lem}
\begin{proof}
Since $M$ contains $A$ as a summand, the action of $A$ on $M$ is faithful, so that the previous results apply.

As in the proof of Lemma \ref{lem-proj-inj}, we have
\[
M_A \cong \HOM_A(A,M)\cong \HOM_A(eM,M)\cong \END_A(M)e,
\]
where $e\in \END_A(M)$ is a $p$-DG idempotent such that $A\cong eM$ and $\dif(e)=0$.  Thus $M$ is a $p$-DG summand of the $p$-DG algebra $B=\END_A(M)$, and the result follows.
\end{proof}

Our next goal is to develop a $p$-DG analogue of Theorem \ref{thm-extension-symmetric-Frob}. To do this, let us consider an analogue of the additive envelope of a module in this situation.

\begin{defn}\label{def-filtered-envelope}
Let $A$ be a $p$-DG algebra, and $X$ be a left (or right) $p$-DG module. The \emph{filtered $p$-DG envelope of $X$} consists of $p$-DG direct summands of left (or right) $p$-DG modules which have a finite filtration that splits when forgetting the differential, whose subquotients are isomorphic to grading shifts of $p$-DG summands of $X$ as $p$-DG modules.
\end{defn}

\begin{lem}\label{lem-filtered-to-cofibrant}
Let $X$ be a left $p$-DG module over $A$, and $B=\END_A(X)$ be the $p$-DG endomorphism algebra. Then 
\[
\HOM_A(X,-):(A,\dif)\dmod\lra (B,\dif)\dmod
\]
sends the $p$-DG filtered envelope of $X$ to the $p$-DG filtered envelope of $B$. In particular, the image of the functor consists of cofibrant $p$-DG modules over $B$.
\end{lem}
\begin{proof}
It is clear that, if $Y$ is a filtered $p$-DG module over $A$ whose subquotients are isomorphic to $X$, then the image of $Y$ under $\HOM_A(X,-)$ is a filtered $p$-DG module whose subquotients are isomorphic to grading shifts of $B$, and thus is a property-P module. Furthermore, the functor preserves $p$-DG summands. Hence it takes the $p$-DG filtered envelope of $X$ to the $p$-DG filtered envelope of $B$, the latter consisting of cofibrant $B$-modules.
\end{proof}

\begin{example}
We would like to emphasize that the lemma above does not require the module $X$ to be cofibrant over $A$. This will play an important role in the story later in this paper. For instance, consider $A:=\END_\Bbbk(U)$, where $U$ is a finite-dimensional $p$-complex. Then $\END_A(U)\cong \Bbbk$ as a $p$-DG algebra. It is easy to check that the above functor
\[
\HOM_A(U,-): (A,\dif)\dmod \lra (\Bbbk,\dif)\dmod
\]
sends any filtered $p$-DG module $A$ of the form $U\otimes V$, where $V$ is any $p$-complex, to the $p$-complex $V$ itself, which is property-P over the ground field. Even when $U$ is acyclic, in which case $U$ is not cofibrant over $A$ (otherwise $\HOM_A(U,U)=\Bbbk$ would compute the endomorphism space of $U$ in the derived category), the image of the $p$-DG filtered envelope consists of cofibrant $(\Bbbk,\dif)$-modules.
\end{example}

We are now ready to state the $p$-DG analogue of Theorem \ref{thm-extension-symmetric-Frob}. Consider now the following situation. Let $A_i$, $i\in I$, be a collection of symmetric Frobenius $p$-DG algebras, $M_i$ be a collection of faithful $p$-DG modules over $A_i$, and $B_i:=\END_{A_i}(M_i)$ be the endomorphism $p$-DG algebras. In this case, the tensor-hom adjunctions are compatible with $p$-differentials (Lemma \ref{lem-adjunction-tensor-hom}).

\begin{prop} \label{prop-pdg-extension}
Suppose $\mc{E}_{ij}:(A_j,\dif)\dmod\lra (A_i,\dif)\dmod$, $i,j\in I$, are $p$-DG functors given by tensoring with $p$-DG bimodules $E_{ij}$ over $(A_i,A_j)$:
$$
\mc{E}_{ij}:(A_j,\dif)\dmod\lra (A_i,\dif)\dmod,\quad N\mapsto {E_{ij}}\otimes_{A_j} N,
$$ 
and the functor $\mc{E}_{ij}$ sends the filtered $p$-DG envelope of $M_j^*$ into that of $M_i^*$ for each $i,j\in I$. Then the extended $p$-DG functors $\mc{E}_{ij}^!$ preserve compact cofibrant $p$-DG modules in $\oplus_{i\in I}(B_i,\dif)\dmod$.
\end{prop}
\begin{proof}
The proof follows in a similar fashion as for Theorem \ref{thm-extension-symmetric-Frob}. One uses Lemma \ref{lem-filtered-to-cofibrant} to show that the composition functor
$\mc{E}_{ij}^!$ sends the cofibrant module $B_j$ to a direct summand of a property-P module in $(B_i,\dif)\dmod$.
\end{proof}

Finally, to conclude this section, let us record a special $p$-DG Morita equivalence result for later use.

\begin{prop}\label{prop-p-dg-Morita}
Let $M_1$ and $M_2$ be $p$-DG right modules over $A$, and let $M_1^\prime$ be in the filtered $p$-DG envelope of $M_1\oplus M_2$, whose associated graded has the form $M_1\oplus g(q)M_2$ for some $g(q)\in \N[q,q^{-1}]$. Set $B_1:=\END_A(M_1\oplus M_2)$ and $B_2:=\END_A( M_1^\prime\oplus M_2)$. Then there is a derived equivalence between $\mc{D}(B_1)$ and $\mc{D}(B_2)$.
\end{prop} 

\begin{proof}
In the absence of the differential, the algebras $B_1$ and $B_2$ are clearly Morita equivalent to each other. The Morita equivalence is given by tensor product with the $(B_1,B_2)$-bimodule (resp.~$(B_2,B_1)$-bimodule)
\[
N:=\HOM_A(M_1^\prime\oplus M_2, M_1 \oplus M_2) \quad \quad \quad \left(\textrm{resp.}~ N^\prime:= \HOM_A(M_1\oplus M_2, M_1^\prime \oplus M_2)~\right).
\]
It follows that $B_1$ and $B_2$ satisfy the double centralizer property on either of the bimodules.

Now the bimodules above carry natural $p$-DG structures. For the derived tensor product with these bimodules to descend to derived equivalences, it suffices to show that they are cofibrant as $p$-DG modules over both $B_1$ and $B_2$, though not necessarily cofibrant over $B_1\otimes B_2^{\mathrm{op}}$ or $B_2\otimes B_1^{\mathrm{op}}$ (see \cite[Proposition 8.8]{QYHopf} for some more general conditions). We will only show that $N$ is cofibrant, as the case of $N^\prime$ is entirely similar.

It is clear that $\HOM_A(M_1, M_1\oplus M_2)$ and $\HOM_A(M_2, M_1\oplus M_2)$ are cofibrant left $p$-DG modules over $B_1$, since they are clearly $p$-DG direct summands of $B_1\cong \HOM_A(M_1\oplus M_2, M_1\oplus M_2)$. It follows that $N$ is cofibrant over $B_1$ since it has a filtration whose subquotients are isomorphic to the previous direct summands. To show that it is cofibrant over $B_2$, we use that $\HOM_A(M_1^\prime\oplus M_2, M_1^\prime) $ and $\HOM_A(M_1^\prime\oplus M_2, M_2) $ are clearly cofibrant as direct summands of $B_2$. Now $\HOM_A(M_1^\prime\oplus M_2, M_1^\prime) $ has a filtration whose associated graded has the form
\[
\HOM_A(M_1^\prime\oplus M_2, M_1)\oplus g(q)\HOM_A(M_1^\prime\oplus M_2,  M_2).  
\]
Since the last summands are cofibrant and so is the whole module, the cofibrance of $\HOM_A(M_1^\prime\oplus M_2, M_1) $ follows from the usual ``two-out-of-three'' property.
\end{proof}

\begin{rem}
The proof above generalizes easily to the case when $M_1^\prime$ lies in the $p$-DG filtered envelope of $M_1\oplus M_2$, and the corresponding $M_1$ summands do not form an acyclic $p$-DG sub or quotient module. 
\end{rem}

\section{A categorification of quantum \texorpdfstring{$\mathfrak{sl}_2$}{sl(2)} at prime roots of unity}\label{sec-cat-sl(2)}

In this Section, we review various forms of categorified quantum $\mf{sl}_2$ at a prime root of unity. We will show that the recent work of \cite{Brundan2KM} can be extended to $p$-DG setting, so that there is a $p$-DG equivalence of categorified quantum $\mf{sl}_2$ defined respectively by Khovanov-Lauda \cite{Lau1, KL3} and Rouquier \cite{Rou2}.

\subsection{The \texorpdfstring{$p$}{p}-DG 2-category \texorpdfstring{$\mathcal{U}$}{U}}\label{sec-2-cat-U}
We begin this section by recalling the diagrammatic definition of the $2$-category $\mathcal{U}$ introduced by Lauda in \cite{Lau1}. Then we will recall a specific $p$-differential on $\mc{U}$ introduced in \cite{EQ1}.

\begin{defn}\label{def-u-dot} The $2$-category $\mathcal{U}$ is an additive graded $\Bbbk$-linear category whose objects $m$ are elements of the weight lattice of $\mathfrak{sl}_2$. The $1$-morphisms are (direct sums of grading shifts of) composites of the generating $1$-morphisms $\1_{m+2} \mathcal{E} \1_m$ and $\1_m \mathcal{F} \1_{m+2}$, for each $m \in \Z$. Each $\1_{m+2} \mathcal{E} \1_m$ will be drawn the same, regardless of the object $m$.

\begin{align*}
\begin{tabular}{|c|c|c|}
	\hline
	$1$-\textrm{Morphism Generator} &
	\begin{DGCpicture}
	\DGCstrand(0,0)(0,1)
	\DGCdot*>{0.5}
	\DGCcoupon*(0.1,0.25)(1,0.75){$^m$}
	\DGCcoupon*(-1,0.25)(-0.1,0.75){$^{m+2}$}
    \DGCcoupon*(-0.25,1)(0.25,1.15){}
    \DGCcoupon*(-0.25,-0.15)(0.25,0){}
	\end{DGCpicture}&
	\begin{DGCpicture}
	\DGCstrand(0,0)(0,1)
	\DGCdot*<{0.5}
	\DGCcoupon*(0.1,0.25)(1,0.75){$^{m+2}$}
	\DGCcoupon*(-1,0.25)(-0.1,0.75){$^m$}
    \DGCcoupon*(-0.25,1)(0.25,1.15){}
    \DGCcoupon*(-0.25,-0.15)(0.25,0){}
	\end{DGCpicture} \\ \hline
	\textrm{Name} & $\1_{m+2}\mathcal{E}\1_m$ & $\1_m\mathcal{F}\1_{m+2}$ \\
	\hline
\end{tabular}
\end{align*}

The weight of any region in a diagram is determined by the weight of any single region. When no region is labeled, the ambient weight is irrelevant.

The $2$-morphisms will be generated by the following pictures.
	
\begin{align*}
\begin{tabular}{|c|c|c|c|c|}
  \hline
  \textrm{Generator} &
  \begin{DGCpicture}
  \DGCstrand(0,0)(0,1)
  \DGCdot*>{0.75}
  \DGCdot{0.45}
  \DGCcoupon*(0.1,0.25)(1,0.75){$^m$}
  \DGCcoupon*(-1,0.25)(-0.1,0.75){$^{m+2}$}
  \DGCcoupon*(-0.25,1)(0.25,1.15){}
  \DGCcoupon*(-0.25,-0.15)(0.25,0){}
  \end{DGCpicture}&
  \begin{DGCpicture}
  \DGCstrand(0,0)(0,1)
  \DGCdot*<{0.25}
  \DGCdot{0.65}
  \DGCcoupon*(0.1,0.25)(1,0.75){$^{m}$}
  \DGCcoupon*(-1,0.25)(-0.1,0.75){$^{m-2}$}
  \DGCcoupon*(-0.25,1)(0.25,1.15){}
  \DGCcoupon*(-0.25,-0.15)(0.25,0){}
  \end{DGCpicture} &
  \begin{DGCpicture}
  \DGCstrand(0,0)(1,1)
  \DGCdot*>{0.75}
  \DGCstrand(1,0)(0,1)
  \DGCdot*>{0.75}
  \DGCcoupon*(1.1,0.25)(2,0.75){$^m$}
  \DGCcoupon*(-1,0.25)(-0.1,0.75){$^{m+4}$}
  \DGCcoupon*(-0.25,1)(0.25,1.15){}
  \DGCcoupon*(-0.25,-0.15)(0.25,0){}
  \end{DGCpicture} &
  \begin{DGCpicture}
  \DGCstrand(0,0)(1,1)
  \DGCdot*<{0.25}
  \DGCstrand(1,0)(0,1)
  \DGCdot*<{0.25}
  \DGCcoupon*(1.1,0.25)(2,0.75){$^m$}
  \DGCcoupon*(-1,0.25)(-0.1,0.75){$^{m-4}$}
  \DGCcoupon*(-0.25,1)(0.25,1.15){}
  \DGCcoupon*(-0.25,-0.15)(0.25,0){}
  \end{DGCpicture} \\ \hline
  \textrm{Degree}  & 2   & 2 & -2 & -2 \\
  \hline
\end{tabular}
\end{align*}

\begin{align*}
\begin{tabular}{|c|c|c|c|c|}
  \hline
  \textrm{Generator} &
  \begin{DGCpicture}
  \DGCstrand/d/(0,0)(1,0)
  \DGCdot*>{-0.25,1}
  \DGCcoupon*(1,-0.5)(1.5,0){$^m$}
  \DGCcoupon*(-0.25,0)(1.25,0.15){}
  \DGCcoupon*(-0.25,-0.65)(1.25,-0.5){}
  \end{DGCpicture} &
  \begin{DGCpicture}
  \DGCstrand/d/(0,0)(1,0)
  \DGCdot*<{-0.25,2}
  \DGCcoupon*(1,-0.5)(1.5,0){$^m$}
  \DGCcoupon*(-0.25,0)(1.25,0.15){}
  \DGCcoupon*(-0.25,-0.65)(1.25,-0.5){}
  \end{DGCpicture}&
  \begin{DGCpicture}
  \DGCstrand(0,0)(1,0)/d/
  \DGCdot*<{0.25,1}
  \DGCcoupon*(1,0)(1.5,0.5){$^m$}
  \DGCcoupon*(-0.25,0.5)(1.25,0.65){}
  \DGCcoupon*(-0.25,-0.15)(1.25,0){}
  \end{DGCpicture}&
  \begin{DGCpicture}
  \DGCstrand(0,0)(1,0)/d/
  \DGCdot*>{0.25,2}
  \DGCcoupon*(1,0)(1.5,0.5){$^m$}
  \DGCcoupon*(-0.25,0.5)(1.25,0.65){}
  \DGCcoupon*(-0.25,-0.15)(1.25,0){}
  \end{DGCpicture}  \\ \hline
  \textrm{Degree} & $1+m$ & $1-m$ & $1+m$ & $1-m$ \\
  \hline
\end{tabular}
\end{align*}
\end{defn}

Before giving the full list of relations for $\mc{U}$, let us introduce some abbreviated notation. For a product of $r$ dots on a single strand, we draw a single dot labeled by $r$. Here is the case when $r=2$.
\begin{align*}
\begin{DGCpicture}
\DGCstrand(0,0)(0,1)
\DGCdot{.35}
\DGCdot{.65}
\DGCdot*>{.95}
\end{DGCpicture}
~=~
\begin{DGCpicture}
\DGCstrand(0,0)(0,1)
\DGCdot{.5}[r]{$^2$}
\DGCdot*>{.95}
\end{DGCpicture}
\end{align*}

A \emph{closed diagram} is a diagram without boundary, constructed from the generators above. The simplest non-trivial closed diagram is a \emph{bubble}, which is a closed diagram without any other closed diagrams inside. Bubbles can be oriented clockwise or counter-clockwise.
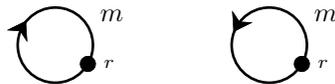
\begin{align*}
\begin{DGCpicture}
\DGCbubble(0,0){0.5}
\DGCdot*<{0.25,L}
\DGCdot{-0.25,R}[r]{$_r$}
\DGCdot*.{0.25,R}[r]{$m$}
\end{DGCpicture}
\qquad \qquad
\begin{DGCpicture}
\DGCbubble(0,0){0.5}
\DGCdot*>{0.25,L}
\DGCdot{-0.25,R}[r]{$_r$}
\DGCdot*.{0.25,R}[r]{$m$}
\end{DGCpicture}
\end{align*}

A simple calculation shows that the degree of a bubble with $r$ dots in a region labeled $m$ is $2(r+1-m)$ if the bubble is clockwise, and $2(r+1+m)$ if the bubble is counter-clockwise. Instead of keeping track of the number $r$ of dots a bubble has, it will be much more illustrative to keep track of the degree of the bubble, which is in $2\Z$. We will use the following shorthand to refer to a bubble of degree $2k$.

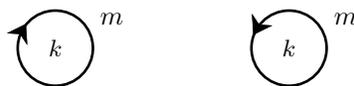
\begin{align*}
\bigcwbubble{$k$}{$m$}
\qquad \qquad
\bigccwbubble{$k$}{$m$}
\end{align*}
This notation emphasizes the fact that bubbles have a life of their own, independent of their presentation in terms of caps, cups, and dots.

Note that $m$ can be any integer, but $r\geq 0$ because it counts dots. Therefore, we can only construct a clockwise (resp. counter-clockwise) bubble of degree $k$ when $k \geq 1-m$ (resp. $k \geq 1+m$). These are called \emph{real bubbles}. Following Lauda, we also allow bubbles drawn as above with arbitrary $k \in \Z$. Bubbles with $k$ outside of the appropriate range are not yet defined in terms of the generating maps; we call these \emph{fake bubbles}. One can express any fake bubble in terms of real bubbles (see Remark \ref{rmk-inf-Grass-relation}).

Now we list the relations. Whenever the region label is omitted, the relation applies to all ambient weights.

\begin{itemize}
\item[(1)] {\bf Biadjointness and cyclicity relations.} These relations require that the generating endomorphisms of $\mc{E}$ and $\mc{F}$ are biadjoint to each other:
\begin{subequations} \label{biadjoint}
\begin{align} \label{biadjoint1}
\begin{DGCpicture}[scale=0.85]
\DGCstrand(0,0)(0,1)(1,1)(2,1)(2,2)
\DGCdot*>{0.75,1}
\DGCdot*>{1.75,1}
\end{DGCpicture}
~=~
\begin{DGCpicture}[scale=0.85]
\DGCstrand(0,0)(0,2)
\DGCdot*>{0.5}
\end{DGCpicture}
~=~
\begin{DGCpicture}[scale=0.85]
\DGCstrand(2,0)(2,1)(1,1)(0,1)(0,2)
\DGCdot*>{0.75}
\DGCdot*>{1.75}
\end{DGCpicture} \ ,
\qquad \qquad
\begin{DGCpicture}[scale=0.85]
\DGCstrand(0,0)(0,1)(1,1)(2,1)(2,2)
\DGCdot*<{0.75,1}
\DGCdot*<{1.75,1}
\end{DGCpicture}
~=~
\begin{DGCpicture}[scale=0.85]
\DGCstrand(0,0)(0,2)
\DGCdot*<{0.5}
\end{DGCpicture}
~=~
\begin{DGCpicture}[scale=0.85]
\DGCstrand(2,0)(2,1)(1,1)(0,1)(0,2)
\DGCdot*<{0.75}
\DGCdot*<{1.75}
\end{DGCpicture} \ ,
\end{align}

\begin{align} \label{biadjointdot}
\begin{DGCpicture}
\DGCstrand(0,0)(0,.5)(1,.5)/d/(1,0)/d/
\DGCdot*<{1}
\DGCdot{.3,1}
\end{DGCpicture}
~=~
\begin{DGCpicture}
\DGCstrand(0,0)(0,.5)(1,.5)/d/(1,0)/d/
\DGCdot*<{1}
\DGCdot{.3,2}
\end{DGCpicture} \ ,
\qquad \qquad \qquad \qquad
\begin{DGCpicture}
\DGCstrand(0,0)(0,.5)(1,.5)/d/(1,0)/d/
\DGCdot*>{1}
\DGCdot{.3,1}
\end{DGCpicture}
~=~
\begin{DGCpicture}
\DGCstrand(0,0)(0,.5)(1,.5)/d/(1,0)/d/
\DGCdot*>{1}
\DGCdot{.3,2}
\end{DGCpicture} \ ,
\end{align}

\begin{align} \label{biadjointcrossing}
\begin{DGCpicture}[scale=0.75]
\DGCstrand(0,0)(1,1)/u/(2,1)/d/(2,0)/d/
\DGCdot*<{1.5}
\DGCstrand(1,0)(0,1)/u/(3,1)/d/(3,0)/d/
\DGCdot*<{2.5}
\end{DGCpicture}
~=~
\begin{DGCpicture}[scale=0.75]
\DGCstrand(0,0)(0,1)(3,1)/d/(2,0)/d/
\DGCdot*<{2.5}
\DGCstrand(1,0)(1,1)(2,1)/d/(3,0)/d/
\DGCdot*<{1.5}
\end{DGCpicture} \ ,
\qquad \qquad
\begin{DGCpicture}[scale=0.75]
\DGCstrand(0,0)(1,1)/u/(2,1)/d/(2,0)/d/
\DGCdot*>{1.5}
\DGCstrand(1,0)(0,1)/u/(3,1)/d/(3,0)/d/
\DGCdot*>{2.5}
\end{DGCpicture}
~=~
\begin{DGCpicture}[scale=0.75]
\DGCstrand(0,0)(0,1)(3,1)/d/(2,0)/d/
\DGCdot*>{2.5}
\DGCstrand(1,0)(1,1)(2,1)/d/(3,0)/d/
\DGCdot*>{1.5}
\end{DGCpicture} \ .
\end{align}
\end{subequations}

\item[(2)] {\bf Positivity and Normalization of bubbles.} Positivity states that all bubbles (real or fake) of negative degree should be zero.
\begin{subequations} \label{negzerobubble}
\begin{align} \label{negbubble}
\bigcwbubble{$k$}{}~=~0~=~\bigccwbubble{$k$}{}
\qquad
\textrm{if}~k<0.
\end{align}

Normalization states that degree $0$ bubbles are equal to the empty diagram (i.e., the identity $2$-morphism of the identity $1$-morphism).

\begin{align} \label{zerobubble}
\bigcwbubble{$0$}{}~=~1~=~\bigccwbubble{$0$}{}~.
\end{align}
\end{subequations}

\item[(3)] {\bf NilHecke relations.} The upward pointing strands satisfy nilHecke relations. 
\begin{subequations} \label{NHrels}
\begin{align}
\begin{DGCpicture}
\DGCstrand(0,0)(1,1)(0,2)
\DGCdot*>{2}
\DGCstrand(1,0)(0,1)(1,2)
\DGCdot*>{2}
\end{DGCpicture}
=0\ , \quad \quad
\RIII{L}{$0$}{$0$}{$0$}{no} = \RIII{R}{$0$}{$0$}{$0$}{no}, 
\label{NHrelR3}
\end{align}
\begin{align}
\crossing{$0$}{$0$}{$1$}{$0$}{no} - \crossing{$0$}{$1$}{$0$}{$0$}{no} = \twolines{$0$}{$0$}{no} = \crossing{$1$}{$0$}{$0$}{$0$}{no} -
\crossing{$0$}{$0$}{$0$}{$1$}{no}. \label{NHreldotforce}
\end{align}
\end{subequations}
\item[(4)] {\bf Reduction to bubbles.} The following equalities hold for all $m \in \Z$.
\begin{subequations} \label{bubblereduction}
\begin{align}
\curl{R}{U}{$m$}{no}{$0$} = -\sum_{a+b=-m} \oneline{$b$}{$m$} \bigcwbubble{$a$}{}\ ,
\end{align}
\begin{align}
\curl{L}{U}{$m$}{no}{$0$} = \sum_{a+b=m} \bigccwbubble{$a$}{$m$} \oneline{$b$}{no} \ .
\end{align}
\end{subequations}
These sums only take values for $a,b \geq 0$. Therefore, when $m \neq 0$, either the right curl or the left curl is zero.
\item[(5)] {\bf Identity decomposition.} The following equations hold for all $m \in \Z$.
\begin{subequations} \label{IdentityDecomp}
\begin{align}
\begin{DGCpicture}
\DGCstrand(0,0)(0,2)
\DGCdot*>{1}
\DGCdot*.{1.25}[l]{$m$}
\DGCstrand(1,0)(1,2)
\DGCdot*<{1}
\DGCdot*.{1.25}[r]{$m$}
\end{DGCpicture}
~=~-~
\begin{DGCpicture}
\DGCstrand(0,0)(1,1)(0,2)
\DGCdot*>{0.25}
\DGCdot*>{1}
\DGCdot*>{1.75}
\DGCstrand(1,0)(0,1)(1,2)
\DGCdot*.{1.25}[l]{$m$}
\DGCdot*<{0.25}
\DGCdot*<{1}
\DGCdot*<{1.75}
\end{DGCpicture}
~+~
\sum_{a+b+c=m-1}~
\cwcapbubcup{$a$}{$b$}{$c$}{$m$} \ , \label{IdentityDecompPos}
\end{align}
\begin{align}
\begin{DGCpicture}
\DGCstrand(0,0)(0,2)
\DGCdot*<{1}
\DGCdot*.{1.25}[l]{$m$}
\DGCstrand(1,0)(1,2)
\DGCdot*>{1}
\DGCdot*.{1.25}[r]{$m$}
\end{DGCpicture}
~=~-~
\begin{DGCpicture}
\DGCstrand(0,0)(1,1)(0,2)
\DGCdot*<{0.25}
\DGCdot*<{1}
\DGCdot*<{1.75}
\DGCstrand(1,0)(0,1)(1,2)
\DGCdot*.{1.25}[l]{$m$}
\DGCdot*>{0.25}
\DGCdot*>{1}
\DGCdot*>{1.75}
\end{DGCpicture}
~+~
\sum_{a+b+c=-m-1}~
\ccwcapbubcup{$a$}{$b$}{$c$}{$m$} \ . \label{IdentityDecompNeg}
\end{align}
\end{subequations}
The sum in the first equality vanishes for $m \leq 0$, and the sum in the second equality vanishes for $m \geq 0$.

The terms on the right hand side form a collection of orthogonal idempotents. 
\end{itemize}

\begin{rem}[Infinite Grassmannian relations]\label{rmk-inf-Grass-relation}  This family of relations, which follows from the above defining relations, can be expressed most succinctly in terms of generating functions.

\begin{equation}\label{eqn-infinite-Grassmannian}
\left( \cwbubble{$0$}{}+t~\cwbubble{$1$}{}+t^2~\cwbubble{$2$}{}+\ldots \right) \cdot
\left( \ccwbubble{$0$}{}+t~\ccwbubble{$1$}{}+t^2~\ccwbubble{$2$}{}+\ldots \right)  =  1~.
\end{equation}

The cohomology ring of the ``infinite dimensional Grassmannian" is the ring $\Lambda$ of symmetric functions. Inside this ring, there is an analogous relation $\mathtt{e}(t)\mathtt{h}(t)=1$, where $\mathtt{e}(t) = \sum_{i \ge 0} (-1)^i \mathtt{e}_i t^i$ is the total Chern class of the tautological bundle, and $\mathtt{h}(t) = \sum_{i \ge 0} \mathtt{h}_i t^i$ is the total Chern class of the dual bundle. Lauda has proved that the bubbles in a single region generate an algebra inside $\mathcal{U}$ isomorphic to $\Lambda$.

Looking at the homogeneous component of degree $m$, we have the following equation.
\begin{align}
\sum_{a + b = m} \bigcwbubble{$a$}{} \bigccwbubble{$b$}{} = \delta_{m,0}. \label{infgrass}
\end{align}
Because of the positivity of bubbles relation, this equation holds true for any $m \in \Z$, and the sum can be taken over all $a,b \in \Z$.

Using these equations one can express all (positive degree) counter-clockwise bubbles in terms of clockwise bubbles, and vice versa.
Consequentially, all fake bubbles can be expressed in terms of real bubbles.
\end{rem}

\begin{defn}\label{def-special-dif}
Let $\dif$ be the derivation defined on the $2$-morphism generators of $\mathcal{U}$ as follows,
\[
\begin{array}{c}
\dif \left( \onelineshort{$1$}{no} \right) =  \onelineshort{$2$}{no} , \quad \quad \quad
\dif \left( \crossing{$0$}{$0$}{$0$}{$0$}{no} \right) =  \twolines{$0$}{$0$}{no} -2 \crossing{$1$}{$0$}{$0$}{$0$}{no} ,\\ \\
\dif \left( \onelineDshort{$1$}{no} \right) =  \onelineDshort{$2$}{no} , \quad \quad \quad
\dif \left( \crossingD{$0$}{$0$}{$0$}{$0$}{no} \right) = - \twolinesD{$0$}{$0$}{no} -2 \crossingD{$0$}{$1$}{$0$}{$0$}{no},
\end{array}
\]
\[
\begin{array}{c}
\dif \left( \cappy{CW}{$0$}{no}{$m$} \right)  =  \cappy{CW}{$1$}{no}{$m$}- \cappy{CW}{$0$}{CW}{$m$},\quad \quad \quad
\dif \left( \cuppy{CW}{$0$}{no}{$m$} \right) = (1-m) \cuppy{CW}{$1$}{no}{$m$},\\ \\
\dif \left( \cuppy{CCW}{$0$}{no}{$m$} \right) =  \cuppy{CCW}{$1$}{no}{$m$} + \cuppy{CCW}{$0$}{CW}{$m$},\quad \quad \quad
\dif \left( \cappy{CCW}{$0$}{no}{$m$} \right) = (m+1 ) \cappy{CCW}{$1$}{no}{$m$} ,
\end{array}
\]
and extended to the entire $\mc{U}$ by the Leibniz rule.
\end{defn}

\begin{thm}\label{thm-itsanalghom}
There is an isomorphism of $\mathbb{O}_p$-algebras
$$\dot{u}_{\mathbb{O}_p} \lra K_0(\mathcal{D}^c(\mathcal{U})),$$
sending $1_{m+2}E1_m $ to $[\1_{m+2}\mathcal{E}\1_m]$ and $1_m F 1_{m+2} $ to $ [\1_m \mathcal{F}\1_{m+2}]$ for any weight $m \in \Z$. 
\end{thm}
\begin{proof}
This is \cite[Theorem 6.11]{EQ1}.
\end{proof}

\begin{rem}\label{rmk-uniqueness-of-d}
Let us also record an important property of the differential established in \cite{EQ1}.
 In \cite[Definition 4.6]{EQ1}, a multi-parameter family of $p$-differentials is defined on $\mc{U}$ which preserve the defining relations of $\mc{U}$. This specific differential above is essentially the only one, up to conjugation by automorphisms of $\mc{U}$, that allows a \emph{fantastic filtration} to exist on $(\mc{U},\dif)$. The fantastic filtration in turn decategorifies to the quantum Serre relations for  $\mf{sl}_2$ at a prime root of unity.
\end{rem}

\begin{lem}\label{lemdifonhalfgenerators}
There is a unique $p$-DG structure on $\mc{U}$ determined by
\begin{subequations}\label{eqndifonhalfgenerators}
\begin{equation}
\dif \left( \onelineshort{$1$}{no} \right) =  \onelineshort{$2$}{no} , \quad \quad \quad
\dif \left( \crossing{$0$}{$0$}{$0$}{$0$}{no} \right) =  \twolines{$0$}{$0$}{no} -2 \crossing{$1$}{$0$}{$0$}{$0$}{no} ,
\end{equation}
\begin{equation}
\dif \left( \cuppy{CW}{$0$}{no}{$m$} \right) = (1-m) \cuppy{CW}{$1$}{no}{$m$},\quad \quad \quad
\dif \left( \cappy{CCW}{$0$}{no}{$m$} \right) = (m+1) \cappy{CCW}{$1$}{no}{$m$}.
\end{equation}
\end{subequations}
\end{lem}
\begin{proof}
In the definition \cite[Definition 4.6]{EQ1} of the multi-parameter family of $p$-differentials on $\mc{U}$, there are certain redundancies set forth by the constraints on the parameters (see \cite[equation (4.8)]{EQ1}). Indeed, once one fixes the differential on the upward pointing nilHecke generators, clockwise cups and counter-clockwise caps, then the differential is uniquely specified on the entire $\mc{U}$. Therefore, in order to obtain the differential in Definition \ref{def-special-dif}, it is enough to just specify it as in equation \eqref{eqndifonhalfgenerators}.
\end{proof}

\subsection{Thick calculus with the differential}\label{sec-thick-cal}
We begin by recording some notation concerning symmetric polynomials. Let $\mc{P}(a,b)$ be the set of partitions which fit into an $a \times b$ box (height $a$, width $b$).  An element $\mu=(\mu_1, \ldots, \mu_a) \in \mc{P}(a,b)$ with $ \mu_1 \geq \cdots \geq \mu_a \geq 0$ gives rise to a \emph{Schur polynomial} $ \pi_{\mu}$ defined as follows:
\begin{equation}\label{eqn-Schur-pol}
\pi_{\mu}:= \frac{\mathrm{det}(M_{\mu})}{\prod_{1 \leq i<j \leq a}(y_i-y_j)}
\hspace{.5in}
(M_{\mu})_{ij}:=y_i^{a+\mu_j-j}.
\end{equation}

The following are some special examples of Schur polynomials.
\begin{itemize}
\item If $\mu=(1^c)$ with $ c \leq a$ then $\pi_{\mu}=\sum_{1\leq i_1<\cdots <i_c \leq a}y_{i_1}\cdots y_{i_c}$ is the usual degree-$c$ elementary symmetric function.

\item If $\mu=(c)$ with $ c \leq b$ then $\pi_{\mu}=\pi_{\mu}=\sum_{1\leq i_1\leq \cdots \leq i_c \leq a}y_{i_1}\cdots y_{i_c}$ is the usual degree-$c$ complete symmetric function.

\item If $\mu=(c^a)$ with $ c \leq b $ then $\pi_{\mu}= y_1^c \cdots y_a^c$.
\end{itemize}

For a partition $\mu \in \mc{P}(a,b)$ we will form a complementary partition $ \hat{\mu} \in \mc{P}(b,a)$.
First define the sequence
\begin{equation*}
\mu^c=(b-\mu_a, \ldots, b-\mu_1).
\end{equation*}
Then set
\begin{equation*}
\hat{\mu}=(\hat{\mu}_1, \ldots, \hat{\mu}_b) 
\hspace{.5in}
\hat{\mu}_{j}=|\{i | \mu^c_i \geq j \}|
\end{equation*}
which we will refer to as the \emph{complementary partition} of $\mu$, with the underlying $\mc{P}(a,b)$ implicitly taken into account.

The $2$-category $\mc{U}$ introduced in the previous section has a ``thick enhancement'' $\dot{\mc{U}}$ defined by Khovanov, Lauda, Mackaay and Sto\v{s}i\'{c} (\cite{KLMS}). The thick calculus $\dot{\mc{U}}$ is the Karoubi envelope of $\mc{U}$, and thus is Morita equivalent to $\mc{U}$.  

To construct $\dot{\mc{U}}$ it suffices to adjoin to $\mc{U}$ the image of the idempotents $e_r\in \END_{\mc{U}}(\mc{E}^r)$, one for each $r\in \N$, where $e_r$ is diagrammatically given by
\begin{equation}
e_{r}:=
\begin{DGCpicture}[scale=0.9]
\DGCstrand(0,0)(4,4)
\DGCdot>{3.95}
\DGCstrand(4,0)(0,4)
\DGCdot>{3.95}
\DGCdot{3.5}[r]{$_{r-1}$}
\DGCstrand(3,0)(4,1)(1,4)
\DGCdot>{3.95}
\DGCdot{3.5}[r]{$_{r-2}$}
\DGCstrand(2,0)(4,2)(2,4)
\DGCdot>{3.95}
\DGCdot{3.5}[r]{$_{r-3}$}
\DGCcoupon*(2.2,3.6)(3.8,4){$\cdots$}
\DGCcoupon*(0.2,0)(1.8,0.4){$\cdots$}
\end{DGCpicture}
\ .
\end{equation}
As in \cite{KLMS}, we can depict the 1-morphism representing the pair $(\mc{E}^r,e_r)$ by an upward pointing ``thick arrow'' of thickness $r$ (When the thickness equals one, the above $1$-morphism agrees with the single black strand for $\mc{E}$.):
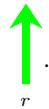
\[
\begin{DGCpicture}
\DGCstrand[Green](0,0)(0,1)[$^r$]
\DGCdot>{0.95}
\end{DGCpicture} \ .
\]
The new object of thickness $r$, for each $r\in \N$, is called the \emph{$r$th divided power} of $\mc{E}$, and is usually denoted as $\mc{E}^{(r)}$. The object $\mc{E}^{(r)}$ has symmetric polynomials in $r$ variables as its endomorphism algebra. We express the multiplication of this endomorphism algebra elements by vertically concatenating pictures labeled by symmetric polynomials:
\[
\begin{DGCpicture}
\DGCstrand[Green](0,-0.35)(0,1.35)[$^r$]
\DGCdot>{1.3}
\DGCcoupon(-0.3,0.25)(0.3,0.75){$g$}
\end{DGCpicture}\ , \quad \quad \quad \quad
\begin{DGCpicture}
\DGCstrand[Green](0,-0.35)(0,1.35)[$^r$]
\DGCdot>{1.3}
\DGCcoupon(-0.3,0.25)(0.3,0.75){$gh$}
\end{DGCpicture}
~=~
\begin{DGCpicture}
\DGCstrand[Green](0,-0.35)(0,1.35)[$^r$]
\DGCdot>{1.3}
\DGCcoupon(-0.3,0.5)(0.3,1){$g$}
\DGCcoupon(-0.3,-0.2)(0.3,0.3){$h$}
\end{DGCpicture}
\ .
\]

There are generating morphisms from $\mc{E}^r$ to $\mc{E}^{(r)}$ and backwards, depicted as
\begin{equation}
\begin{DGCpicture}
\DGCstrand[Green](0,0)(0,1)[`$_r$]
\DGCdot>{0.5}
\DGCstrand/u/(-1,-1)(0,0)/u/
\DGCstrand/u/(0.5,-1)(0,0)/u/
\DGCstrand/u/(1,-1)(0,0)/u/
\DGCcoupon*(-0.8,-0.8)(0.3,-0.6){$\cdots$}
\end{DGCpicture} 
\ ,
\quad \quad \quad \quad
\begin{DGCpicture}
\DGCstrand[Green](0,-1)(0,0)[$^r$]
\DGCdot>{-0.5}
\DGCstrand/d/(-1,1)(0,0)/d/
\DGCstrand/d/(0.5,1)(0,0)/d/
\DGCstrand/d/(1,1)(0,0)/d/
\DGCcoupon*(-0.6,0.8)(-0.1,0.6){$\cdots$}
\end{DGCpicture}
\ .
\end{equation}
More generally, we have generating morphisms between $\mc{E}^{(r)}\mc{E}^{(s)}$ and $\mc{E}^{(r+s)}$ which are drawn respectively as
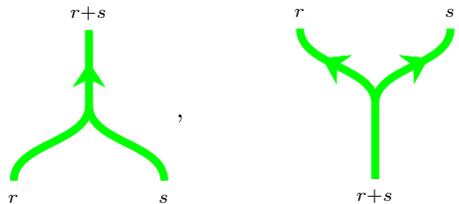
\[
\begin{DGCpicture}
\DGCstrand[Green](0,0)(0,1)[`$_{r+s}$]
\DGCdot>{0.45}
\DGCstrand[Green]/u/(-1,-1)(0,0)/u/[$^r$]
\DGCstrand[Green]/u/(1,-1)(0,0)/u/[$^s$]
\end{DGCpicture} \ ,
\quad \quad \quad \quad
\begin{DGCpicture}
\DGCPLstrand[Green](0,-1)(0,0)[$^{r+s}$]
\DGCstrand[Green]/u/(0,0)(-1,1)/u/[`$_r$]\DGCdot>{0.55}
\DGCstrand[Green]/u/(0,0)(1,1)/u/[`$_s$]\DGCdot>{0.55}
\end{DGCpicture} \ .
\]
Such thick diagrams, carrying symmetric polynomials, satisfy certain diagrammatic identities which are consequences of relations in the $2$-category $\mc{U}$. For instance, we have the following \emph{identity decomposition relation}:
\begin{equation} \label{eq-identitydecomp}
\begin{DGCpicture}[scale=0.9]
\DGCstrand[Green](0,0)(0,3)[$^r$`$_r$]
\DGCdot>{1.5}
\DGCstrand[Green](2,0)(2,3)[$^s$`$_s$]
\DGCdot>{1.5}
\end{DGCpicture}
=\sum_{\alpha\in \mc{P}(r,s) }(-1)^{|\hat{\alpha}|}
\begin{DGCpicture}[scale=0.9]
\DGCPLstrand[Green](0,0)(1,1.15)[$^r$]
\DGCPLstrand[Green](2,0)(1,1.15)[$^s$]
\DGCPLstrand[Green](1,1.15)(1,1.85)
\DGCdot>{1.5}
\DGCPLstrand[Green](1,1.85)(0,3)[`$_r$]
\DGCPLstrand[Green](1,1.85)(2,3)[`$_s$]
\DGCcoupon(1.15,0.35)(1.85,0.85){$_{\pi_{\hat{\alpha}}}$}
\DGCcoupon(0.15,2.15)(0.85,2.65){$_{\pi_\alpha}$}
\end{DGCpicture}~.
\end{equation}
Here $\pi_{\alpha}$ stands for the Schur polynomial corresponding to a partition $\alpha\in \mc{P}(r,s)$, while $\hat{\alpha}$ stands for the complementary partition in $\mc{P}(s,r)$. We refer the reader to \cite{KLMS} for the details.

In \cite{EQ2}, a $p$-differential is defined on $\dot{\mc{U}}$ extending that of $\mc{U}$:
\begin{subequations}
\begin{equation}\label{eqn-d-action-mod-generator}
\dif\left(~
\begin{DGCpicture}[scale=0.85]
\DGCPLstrand[Green](1,0)(1,1)[$^{r+s}$]
\DGCdot>{0.5}
\DGCPLstrand[Green](1,1)(0,2)[`$_r$]
\DGCPLstrand[Green](1,1)(2,2)[`$_s$]
\end{DGCpicture}
~\right)
=
-s\begin{DGCpicture}[scale=0.85]
\DGCPLstrand[Green](1,0)(1,1)[$^{r+s}$]
\DGCdot>{0.5}
\DGCPLstrand[Green](1,1)(0,2)[`$_r$]
\DGCPLstrand[Green](1,1)(2,2)[`$_s$]
\DGCcoupon(0.2,1.25)(0.8,1.75){$_{\pi_1}$}
\end{DGCpicture} \ ,
\quad \quad \quad
\dif\left(~
\begin{DGCpicture}[scale=0.85]
\DGCPLstrand[Green](0,0)(1,1)[$^r$]
\DGCPLstrand[Green](2,0)(1,1)[$^s$]
\DGCPLstrand[Green](1,1)(1,2)[`$_{r+s}$]
\DGCdot>{1.5}
\end{DGCpicture}
~\right)=
-r
\begin{DGCpicture}[scale=0.85]
\DGCPLstrand[Green](0,0)(1,1)[$^r$]
\DGCPLstrand[Green](2,0)(1,1)[$^s$]
\DGCPLstrand[Green](1,1)(1,2)[`$_{r+s}$]
\DGCdot>{1.5}
\DGCcoupon(1.2,0.25)(1.8,0.75){$_{\pi_1}$}
\end{DGCpicture} \ ,
\end{equation}

\begin{equation}\label{eqn-dif-thick-cup-cap-1}
\dif\left(~
\begin{DGCpicture}[scale =0.5]
\DGCstrand[Green]/u/(0,0)(0.75,1)(1.5,0)/d/
\DGCdot<{0.5,1}[r]{$\scriptstyle{r}$}
\DGCcoupon*(1.4,0.7)(1.8,1.1){$m$}
\end{DGCpicture}
~\right)
=
(m+r)
\begin{DGCpicture}[scale =0.5]
\DGCstrand[Green]/u/(0,0)(0.75,1)(1.5,0)/d/
\DGCdot<{0.5,1}[r]{$\scriptstyle{r}$}
\DGCcoupon*(1.4,0.8)(1.8,1){$m$}
\DGCcoupon(1.2,0.2)(1.7,0.5){$\scriptstyle{\pi_1}$}
\end{DGCpicture}
\ ,
\quad \quad
\dif\left(~
\begin{DGCpicture}[scale =0.5]
\DGCstrand[Green]/d/(0,0)(0.75,-1)(1.5,0)/u/
\DGCdot<{-0.5,1}[r]{$\scriptstyle{r}$}
\DGCcoupon*(1.4,-0.8)(1.8,-1){$m$}
\end{DGCpicture}
~\right)
=(r-m)
\begin{DGCpicture}[scale =0.5]
\DGCstrand[Green]/d/(0,0)(0.75,-1)(1.5,0)/u/
\DGCdot<{-0.5,1}[r]{$\scriptstyle{r}$}
\DGCcoupon*(1.4,-0.8)(1.8,-1){$m$}
\DGCcoupon(1.2,-0.2)(1.7,-0.5){$\scriptstyle{\pi_1}$}
\end{DGCpicture}
\ ,
\end{equation}

\begin{equation}\label{eqn-dif-thick-cup-cap-2}
\dif\left(~
\begin{DGCpicture}[scale =0.5]
\DGCstrand[Green]/u/(0,0)(0.75,1)(1.5,0)/d/
\DGCdot>{0.5,1}[r]{$\scriptstyle{r}$}
\DGCcoupon*(1.4,0.7)(1.8,1.1){$m$}
\end{DGCpicture}
~\right)
=
r~~~
\begin{DGCpicture}[scale =0.5]
\DGCstrand[Green]/u/(0,0)(0.75,1)(1.5,0)/d/
\DGCdot>{0.5,1}[r]{$\scriptstyle{r}$}
\DGCcoupon*(1.4,0.8)(1.8,1){$m$}
\DGCcoupon(1.2,0.2)(1.7,0.5){$\scriptstyle{\pi_1}$}
\end{DGCpicture}
-r~~~
\begin{DGCpicture}[scale =0.5]
\DGCstrand[Green]/u/(0,0)(0.75,1)(1.5,0)/d/
\DGCdot>{0.5,1}[r]{$\scriptstyle{r}$}
\DGCcoupon*(1.4,0.8)(1.8,1){$m$}
\DGCbubble(2.25,0.65){0.3}
\DGCdot<{0.65,L}
\DGCcoupon*(2,0.5)(2.5,0.8){$\scriptsize{1}$}
\end{DGCpicture}
\ ,
\end{equation}

\begin{equation}\label{eqn-dif-thick-cup-cap-3}
\dif\left(~
\begin{DGCpicture}[scale =0.5]
\DGCstrand[Green]/d/(0,0)(0.75,-1)(1.5,0)/u/
\DGCdot>{-0.5,1}[r]{$\scriptstyle{r}$}
\DGCcoupon*(1.4,-0.8)(1.8,-1){$m$}
\end{DGCpicture}
~\right)
=
r~~~
\begin{DGCpicture}[scale =0.5]
\DGCstrand[Green]/d/(0,0)(0.75,-1)(1.5,0)/u/
\DGCdot>{-0.5,1}[r]{$\scriptstyle{r}$}
\DGCcoupon*(1.4,-0.8)(1.8,-1){$m$}
\DGCcoupon(1.2,-0.2)(1.7,-0.5){$\scriptstyle{\pi_1}$}
\end{DGCpicture}
+r~~~
\begin{DGCpicture}[scale =0.5]
\DGCstrand[Green]/d/(0,0)(0.75,-1)(1.5,0)/u/
\DGCdot>{-0.5,1}[r]{$\scriptstyle{r}$}
\DGCcoupon*(1.4,-0.8)(1.8,-1){$m$}
\DGCbubble(2.25,-0.65){0.3}
\DGCdot<{-0.65,L}
\DGCcoupon*(2,-0.5)(2.5,-0.8){$\scriptsize{1}$}
\end{DGCpicture}
\ .
\end{equation}
Here $\pi_1$ stands for the first elementary symmetric function in the number of variables labeled by the thickness of the strand, and the clockwise ``bubble''
\begin{equation}\label{equ-degree-2-bubble}
\begin{DGCpicture}
\DGCbubble(2.25,0.65){0.5}
\DGCdot<{0.65,L}
\DGCcoupon*(2,0.5)(2.5,0.8){$\scriptsize{1}$}
\end{DGCpicture}
:=
\begin{DGCpicture}
\DGCbubble(0,0){0.5}
\DGCdot<{0,L}
\DGCdot{-0.25,R}[r]{$_m$}
\DGCdot*.{0.25,R}[r]{$m$}
\end{DGCpicture}
\end{equation}
agrees with the one defined in $\mc{U}$ in the previous subsection.
\end{subequations}

The thick cups and caps give rise to right and left adjoints of the $1$-morphism 
$\mathcal{E}^{(r)} \1_m$ in $\dot{\mc{U}}$. Taking into account of degrees, they are given by
\begin{equation}
\label{Fadjointformula}
(\mathcal{E}^{(r)} \1_m)_R = \1_{m} \mathcal{F}^{(r)} \{-r(m+r)  \} , 
\quad \quad \quad \quad
(\mathcal{E}^{(r)} \1_m)_L = \1_{m} \mathcal{F}^{(r)} \{r(m+r)  \} .
\end{equation}

By construction, in the non-$p$-DG setting, $\dot{\mathcal{U}}$ is Morita equivalent to ${\mathcal{U}}$, and they both categorify quantum $\mathfrak{sl}_2$ at generic $q$ values. However, unlike the abelian case, the $p$-DG derived categories are drastically different. 

There is a natural embedding of $p$-DG $2$-categories
\[
\mathcal{J}: (\mc{U},\dif)\lra (\dot{\mathcal{U}},\dif),
\]
which is given by tensor product with $\dot{\mc{U}}$ regarded as a $(\dot{\mathcal{U}},{\mathcal{U}})$-bimodule with a compatible differential. This functor, not surprisingly, induces an equivalence of abelian categories of $p$-DG modules, and further descends to an equivalence of the corresponding homotopy categories. However, after localization (inverting quasi-isomorphisms), it is no longer an equivalence, but instead is a fully-faithful embedding of derived categories:
\[
\mathcal{J}: \mathcal{D}(\mathcal{U}) \lra \mathcal{D}(\dot{\mathcal{U}}).
\]
The embedding also plays an important role when categorifying the quantum Frobenius map, see \cite{QYFrob}.

\begin{thm}
The derived embedding $\mc{J}$ categorifies the embedding of $\dot{u}_{\mathbb{O}_p}$ into the BLM form $\dot{U}_{\mathbb{O}_p}$ for quantum $\mathfrak{sl}_2$.
\end{thm}
\begin{proof}
See \cite{EQ2}. 
\end{proof}

In the absence of $\dif$, the Sto\v{s}i\'{c} formula in \cite{KLMS} gives rise to a direct sum decomposition of $1$-morphism $\mc{E}^{(a)}\mc{F}^{(b)}\1_m$ for various $a,b\in \N$ and $m\in \Z$. In the $p$-DG setting, the direct sum decomposition is replaced by a fantastic filtration on the corresponding $1$-morphisms.

\begin{prop}\label{prop-higher-Serre}
In the $p$-DG category $(\dot{\mc{U}},\dif)$, the following statements hold.
\begin{enumerate}
\item[(i)] The $1$-morphisms in the collection
\[
\{\mc{E}^{(a)}\mc{E}^{(b)}\1_m,~\mc{F}^{(a)}\mc{F}^{(b)}\1_m|a,b\in \N,~m\in \Z\}
\]
are equipped with a fantastic filtration, whose subquotients are isomorphic to grading shifts of $\mc{E}^{(a+b)}\1_m$ and $\mc{F}^{(a+b)}\1_m$ respectively. Consequently, the defining relation \eqref{eqn-divided-power} for $\dot{U}_{\mathbb{O}_p}$ holds in the Grothendieck group of $\mc{D}(\dot{\mc{U}})$. 
\item[(ii)] The $1$-morphisms in the collection
\[
\{\mc{E}^{(a)}\mc{F}^{(b)}\1_m,~\mc{F}^{(a)}\mc{E}^{(b)}\1_m|a,b\in \N,~m\in \Z\}
\]
admit natural $\dif$-stable fantastic filtrations. The direct summands in the St\v{o}si\'{c} formula consitute the associated graded pieces of the filtrations. Consequently, in the Grothendieck group of $\mc{D}(\dot{\mc{U}})$, the divided power $E$-$F$ relations (equation \eqref{eqn-higher-Serre-1} and \eqref{eqn-higher-Serre-2}) hold.
\end{enumerate}

\end{prop}
\begin{proof}
See \cite[Section 3 and Section 6]{EQ2}.
\end{proof}

\subsection{The \texorpdfstring{$p$}{p}-DG 2-category \texorpdfstring{$\mathcal{U}^R$}{UR}}
We now recall a similar version of the Khovanov-Lauda $2$-category introduced by Rouquier in \cite{Rou2} (for the special case of $\mathfrak{sl}_2$). The definition involves fewer generators and relations, and therefore is usually easier to examine on (potential) $2$-representations.

\begin{defn}\label{def-u-dot-Rouquier} The $2$-category $\mathcal{U}^R$ is an additive graded $\Bbbk$-linear category whose objects $m$ are elements of the weight lattice of $\mathfrak{sl}_2$. The $1$-morphisms are (direct sums of grading shifts of) composites of the generating $1$-morphisms $\1_{m+2} \mathcal{E} \1_m$ and $\1_m \mathcal{F} \1_{m+2}$, for all $m \in \Z$. Each $\1_{m+2} \mathcal{E} \1_m$ will be drawn the same, regardless of the object $m$.

\begin{align*}
\begin{tabular}{|c|c|c|}
	\hline
	$1$-\textrm{Morphism Generator} &
	\begin{DGCpicture}
	\DGCstrand(0,0)(0,1)
	\DGCdot*>{0.5}
	\DGCcoupon*(0.1,0.25)(1,0.75){$^m$}
	\DGCcoupon*(-1,0.25)(-0.1,0.75){$^{m+2}$}
    \DGCcoupon*(-0.25,1)(0.25,1.15){}
    \DGCcoupon*(-0.25,-0.15)(0.25,0){}
	\end{DGCpicture}&
	\begin{DGCpicture}
	\DGCstrand(0,0)(0,1)
	\DGCdot*<{0.5}
	\DGCcoupon*(0.1,0.25)(1,0.75){$^{m+2}$}
	\DGCcoupon*(-1,0.25)(-0.1,0.75){$^m$}
    \DGCcoupon*(-0.25,1)(0.25,1.15){}
    \DGCcoupon*(-0.25,-0.15)(0.25,0){}
	\end{DGCpicture} \\ \hline
	\textrm{Name} & $\1_{m+2}\mathcal{E}\1_m$ & $\1_m\mathcal{F}\1_{m+2}$ \\
	\hline
\end{tabular}
\end{align*}

The weight of any region in a diagram is determined by the weight of any single region. When no region is labeled, the ambient weight is irrelevant.

The $2$-morphisms will be generated by the following pictures.
	
\begin{align*}
\begin{tabular}{|c|c|c|c|c|c|}
  \hline
  \textrm{Generator} &
  \begin{DGCpicture}
  \DGCstrand(0,0)(0,1)
  \DGCdot*>{0.75}
  \DGCdot{0.45}
  \DGCcoupon*(0.1,0.25)(1,0.75){$^m$}
  \DGCcoupon*(-1,0.25)(-0.1,0.75){$^{m+2}$}
  \DGCcoupon*(-0.25,1)(0.25,1.15){}
  \DGCcoupon*(-0.25,-0.15)(0.25,0){}
  \end{DGCpicture}&
  \begin{DGCpicture}
  \DGCstrand(0,0)(1,1)
  \DGCdot*>{0.75}
  \DGCstrand(1,0)(0,1)
  \DGCdot*>{0.75}
  \DGCcoupon*(1.1,0.25)(2,0.75){$^m$}
  \DGCcoupon*(-1,0.25)(-0.1,0.75){$^{m+4}$}
  \DGCcoupon*(-0.25,1)(0.25,1.15){}
  \DGCcoupon*(-0.25,-0.15)(0.25,0){}
  \end{DGCpicture}& 
   \begin{DGCpicture}
  \DGCstrand/d/(0,0)(1,0)
  \DGCdot*<{-0.25,2}
  \DGCcoupon*(1,-0.5)(1.5,0){$^m$}
  \DGCcoupon*(-0.25,0)(1.25,0.15){}
  \DGCcoupon*(-0.25,-0.65)(1.25,-0.5){}
  \end{DGCpicture} & 
  \begin{DGCpicture}
  \DGCstrand(0,0)(1,0)/d/
  \DGCdot*<{0.25,1}
  \DGCcoupon*(1,0)(1.5,0.5){$^m$}
  \DGCcoupon*(-0.25,0.5)(1.25,0.65){}
  \DGCcoupon*(-0.25,-0.15)(1.25,0){}
  \end{DGCpicture} \\ \hline
  \textrm{Degree} & 2  & $-2$   & $1-m$  & $1+m$   \\
  \hline
\end{tabular}
\end{align*}
\end{defn}

We will need the following $2$-morphism defined in terms of the generating $2$-morphisms.
\begin{equation*}
  \begin{DGCpicture}
  \DGCstrand(0,0)(1,1)
      \DGCdot*<{0.05}
  \DGCstrand(1,0)(0,1)
  \DGCdot*>{0.95}
  \DGCcoupon*(-1,0.25)(-0.1,0.75){$^{m}$}
  \DGCcoupon*(-0.25,1)(0.25,1.15){}
  \DGCcoupon*(-0.25,-0.15)(0.25,0){}
  \end{DGCpicture}
  ~:=~
\begin{DGCpicture}
\DGCstrand(0,1)(1,1)/d/
\DGCstrand/u/(1,-0.5)(1,0)(2,1)(2,1.5)/u/
  \DGCdot*>{0.75}
\DGCstrand(2,0)(1,1)
  \DGCdot*>{0.75}
  \DGCstrand(0,1)(0,-0.5)/d/
   \DGCstrand/d/(2,0)(3,0)
  \DGCstrand(3,0)(3,1.5)
\DGCcoupon*(-1,0.25)(-0.1,0.75){$^{m}$}
  \end{DGCpicture}
\end{equation*}

Now we list the relations. Whenever the region label is omitted, the relation applies to all ambient weights.

\begin{itemize}
\item[(1)] {\bf Adjointness relations.} Only one-sided adjunction is required between $\mc{E}$ and $\mc{F}$:
\begin{subequations} \label{biadjointRouq}
\begin{align} \label{biadjoint1Rouq}
\begin{DGCpicture}[scale=0.85]
\DGCstrand(0,0)(0,2)
\DGCdot*>{1}
\end{DGCpicture}
~=~
\begin{DGCpicture}[scale=0.85]
\DGCstrand(2,0)(2,1)(1,1)(0,1)(0,2)
\DGCdot*>{0.35}
\DGCdot*>{1.75}
\end{DGCpicture}\ ,
\qquad \qquad \qquad
\begin{DGCpicture}[scale=0.85]
\DGCstrand(0,0)(0,1)(1,1)(2,1)(2,2)
\DGCdot*<{0.25,1}
\DGCdot*<{1.65,1}
\end{DGCpicture}
~=~
\begin{DGCpicture}[scale=0.85]
\DGCstrand(0,0)(0,2)
\DGCdot*<{1}
\end{DGCpicture} \ .
\end{align}
\end{subequations}

\item[(2)] {\bf NilHecke relations.} The upward pointing strands satisfy nilHecke relations. Note that, diagrammatically, far-away commuting elements become isotopy relations and are thus built in by default.
\begin{subequations} \label{NHrelsRouq}
\begin{align}
\begin{DGCpicture}
\DGCstrand(0,0)(1,1)(0,2)
\DGCdot*>{2}
\DGCstrand(1,0)(0,1)(1,2)
\DGCdot*>{2}
\end{DGCpicture}
=0\ , \quad \quad
\RIII{L}{$0$}{$0$}{$0$}{no} = \RIII{R}{$0$}{$0$}{$0$}{no},
\label{NHrelR3Rouq}
\end{align}
\begin{align}
\crossing{$0$}{$0$}{$1$}{$0$}{no} - \crossing{$0$}{$1$}{$0$}{$0$}{no} = \twolines{$0$}{$0$}{no} = \crossing{$1$}{$0$}{$0$}{$0$}{no} -
\crossing{$0$}{$0$}{$0$}{$1$}{no}. \label{NHreldotforceRouq}
\end{align}
\end{subequations}
\item[(3)] {\bf Isomorphism of $1$-morphisms.}  For all $m \in \Z$, there exist isomorphisms\footnote{As shown in \cite{Brundan2KM}, one does not need additional generators for constructing the inverses of \eqref{IdentityDecompRouquier}; the inverse maps can already be constructed from the generators listed
above.}
\begin{subequations} \label{IdentityDecompRouquier}
\begin{align}
 \begin{DGCpicture}
  \DGCstrand(0,0)(1,1)
      \DGCdot*<{0.05}
  \DGCstrand(1,0)(0,1)
  \DGCdot*>{0.95}
  \DGCcoupon*(-1,0.25)(-0.1,0.75){$^{m}$}
  \DGCcoupon*(-0.25,1)(0.25,1.15){}
  \DGCcoupon*(-0.25,-0.15)(0.25,0){}
  \end{DGCpicture}
~\bigoplus\left(
~\bigoplus_{j=0}^{m-1}~
\begin{DGCpicture}
\DGCstrand(0,0)(0,.05)(1,.05)/d/(1,0)/d/
\DGCdot*<{.55}
\DGCdot{.25,1}[r]{$_j$}
\DGCcoupon*(-.5,0.25)(-0.1,0.75){$^{m}$}
\end{DGCpicture}
~\right)
\colon
\mathcal{F} \mathcal{E} \1_m~ 
\longrightarrow 
\mathcal{E} \mathcal{F} \1_m
\bigoplus \left(
\bigoplus_{j=0}^{m-1} 
q^{m-1-2j}\1_m  \right),
\end{align}
\begin{align}
 \begin{DGCpicture}
  \DGCstrand(0,0)(1,1)
      \DGCdot*<{0.05}
  \DGCstrand(1,0)(0,1)
  \DGCdot*>{0.95}
  \DGCcoupon*(-1,0.25)(-0.1,0.75){$^{m}$}
  \DGCcoupon*(-0.25,1)(0.25,1.15){}
  \DGCcoupon*(-0.25,-0.15)(0.25,0){}
  \end{DGCpicture}
~\bigoplus\left(
\bigoplus_{j=0}^{-m-1}~
\begin{DGCpicture}
  \DGCstrand/d/(0,1)(0,.95)(1,.95)/u/(1,1)/u/
\DGCdot*<{0.45}
\DGCdot{.75,1}[r]{$_j$}
\DGCcoupon*(-.5,0.25)(-0.1,0.75){$^{m}$}
\end{DGCpicture}
\right)
\colon
\mathcal{F} \mathcal{E} \1_m
\bigoplus \left(
\bigoplus_{j=0}^{-m-1} 
q^{-m-1-2j}\1_m  \right)
\longrightarrow
\mathcal{E} \mathcal{F} \1_m.
\end{align}
The sum in the first equality vanishes for $m \leq 0$, and the sum in the second equality vanishes for $m \geq 0$.
\end{subequations} 
\end{itemize}

In a remarkable work, Brundan \cite{Brundan2KM} establishes an equivalence between the two versions of the $2$-categories defined by Khovanov-Lauda and Rouquier respectively.

\begin{thm}\label{thmBrundan}
 There is an equivalence of $2$-categories $\mathcal{U} \cong \mathcal{U}^R$. 
\end{thm} 
\begin{proof}
This is proven in \cite[Main Theorem]{Brundan2KM}.
\end{proof}

\begin{lem}\label{def-special-dif-Rouquier}
There is a unique derivation defined on the generating $2$-morphisms of $\mathcal{U}^R$ as follows.
\[
\begin{array}{c}
\dif \left( \onelineshort{$1$}{no} \right) =  \onelineshort{$2$}{no} , \quad \quad \quad
\dif \left( \crossing{$0$}{$0$}{$0$}{$0$}{no} \right) =  \twolines{$0$}{$0$}{no} -2 \crossing{$1$}{$0$}{$0$}{$0$}{no} ,\\ \\
\end{array}
\]
\[
\begin{array}{c}
\dif
\left(  
\begin{DGCpicture}
\DGCstrand(0,0)(0,.05)(1,.05)/d/(1,0)/d/
\DGCdot*<{.55}
\DGCcoupon*(-.5,0.25)(-0.1,0.75){$^{m}$}
\end{DGCpicture}
~~\right)
~=(m+1)~
\begin{DGCpicture}
\DGCstrand(0,0)(0,.05)(1,.05)/d/(1,0)/d/
\DGCdot*<{.55}
\DGCdot{.25,1}
\DGCcoupon*(-.5,0.25)(-0.1,0.75){$^{m}$}
\end{DGCpicture} \ , \quad \quad \quad
\dif
\left(
\begin{DGCpicture}
  \DGCstrand/d/(0,1)(0,.95)(1,.95)/u/(1,1)/u/
\DGCdot*<{.45}
\DGCcoupon*(-.5,0.25)(-0.1,0.75){$^{m}$}
\end{DGCpicture}
~~\right)
~=(1-m)~
\begin{DGCpicture}
  \DGCstrand/d/(0,1)(0,.95)(1,.95)/u/(1,1)/u/
\DGCdot*<{0.45}
\DGCdot{.75,1}
\DGCcoupon*(-.5,0.25)(-0.1,0.75){$^{m}$}
\end{DGCpicture} \ .
\end{array}
\]
The resulting 2-category $(\mc{U}^R,\dif)$ is $p$-DG equivalent to $(\mc{U},\dif)$.
\end{lem}
\begin{proof}
This is just transporting the $p$-DG structures via Brundan's Theorem \ref{thmBrundan}.
\end{proof}

Lemma \ref{def-special-dif-Rouquier} gives one a useful criterion to construct categorical actions of $(\mc{U},\dif)$.

\begin{thm}
\label{equiv-of-pdgRoug-pdgLauda}
To define a $p$-DG functor on $(\mc{U},\dif)$, it suffices to define a functor on $\mc{U}^R$ and define the action of $\dif$ on the image of the generators in Lemma \ref{def-special-dif-Rouquier}. 
\end{thm}
\begin{proof}
This is now just a consequence of Lemma \ref{lemdifonhalfgenerators} and Lemma \ref{def-special-dif-Rouquier}.
\end{proof}

\section{The nilHecke algebra}
\label{sec-nilHecke}
\subsection{Definitions}
\label{subsec-def-nilHecke}
Recall that $\Bbbk$ is a field of characteristic $p>0$.
Let $l $ and $n$ be integers such that $ l \geq n \geq 0$.
Define the nilHecke algebra $\nh_n$ to be the $\Bbbk$-algebra generated by $ y_1, \ldots, y_n$
and $\psi_1, \ldots, \psi_{n-1} $ with relations

\begin{equation}\label{eqn-NH-relation}
\begin{gathered}
y_iy_j=y_jy_i, \quad y_i\psi_j=\psi_jy_i~(i \neq j,j+1),\quad  y_i\psi_i-\psi_iy_{i+1}=1=\psi_iy_i-y_{i+1}\psi_i,\\
\psi_i^2=0,
\quad \quad
\psi_i\psi_j=\psi_j\psi_i~(|i-j|>1),\quad \quad \psi_{i}\psi_{i+1}\psi_1=\psi_{i+1}\psi_i\psi_{i+1}.
\end{gathered}
\end{equation}
The \emph{cyclotomic nilHecke algebra} $\nh_n^l$ is the quotient of the nilHecke algebra $\nh_n$ by the \emph{cyclotomic relation}
\begin{equation}\label{eqn-NH-cyclotomicrelation}
\begin{gathered}
y_1^l=0.
\end{gathered}
\end{equation}
The (cyclotomic) nilHecke algebra is a graded algebra where the degree of $y_i$ is $2$ and the degree of $\psi_i$ is $-2$.

The relations \eqref{eqn-NH-relation} translate into planar diagrammatic relations for the upward pointing strands in the 2-category $\mc{U}$ (see Section \ref{sec-2-cat-U}), with the orientation labels dropped:
\begin{gather}
\begin{DGCpicture}[scale=0.55]
\DGCstrand(1,0)(0,1)(1,2)
\DGCstrand(0,0)(1,1)(0,2)
\end{DGCpicture}
~= 0 \ ,  \quad \quad \quad \quad
\begin{DGCpicture}[scale=0.55]
\DGCstrand(0,0)(2,2)
\DGCstrand(1,0)(0,1)(1,2)
\DGCstrand(2,0)(0,2)
\end{DGCpicture}
~=~
\begin{DGCpicture}[scale=0.55]
\DGCstrand(0,0)(2,2)
\DGCstrand(1,0)(2,1)(1,2)
\DGCstrand(2,0)(0,2)
\end{DGCpicture}
\ ,\label{eqn-nilHecke-RII-RIII} \\
\begin{DGCpicture}
\DGCstrand(0,0)(1,1)
\DGCdot{0.25}
\DGCstrand(1,0)(0,1)
\end{DGCpicture}
-
\begin{DGCpicture}
\DGCstrand(0,0)(1,1)
\DGCdot{0.75}
\DGCstrand(1,0)(0,1)
\end{DGCpicture}
~=~
\begin{DGCpicture}
\DGCstrand(0,0)(0,1)
\DGCstrand(1,0)(1,)
\end{DGCpicture}
~=~
\begin{DGCpicture}
\DGCstrand(1,0)(0,1)
\DGCdot{0.75}
\DGCstrand(0,0)(1,1)
\end{DGCpicture}
-
\begin{DGCpicture}
\DGCstrand(1,0)(0,1)
\DGCdot{0.25}
\DGCstrand(0,0)(1,1)
\end{DGCpicture} \ .
\end{gather}
The cyclotomic relation means that a black strand carrying $l$ consecutive dots and appearing to the left of the rest of a diagram annihilates the entire picture:
\begin{equation}\label{eqn-cyclotomic}
\begin{DGCpicture}
\DGCstrand(1,0)(1,1)
\DGCdot{0.5}[ur]{$_l$}
\DGCcoupon*(1.25,0.25)(1.75,0.75){$\cdots$}
\end{DGCpicture}
~=~0.
\end{equation}

There is a graded anti-automorphism on the (cyclotomic) nilHecke algebras $ * \colon \nh_n^l \rightarrow \nh_n^l$ defined by $\psi_i^*=\psi_i$
and $ y_i^*=y_i$. Diagrammatically, it is interpreted as flipping a diagram upside down about a horizontal axis.

Let us recall some special elements of (cyclotomic) nilHecke algebras that correspond to symmetric group elements. Fix a reduced decomposition of $w \in \mf{S}_n$,
$w=s_{i_1} \cdots s_{i_r}$.
This gives rise to an element $\psi_w= \psi_{i_1} \cdots \psi_{i_r} \in \nh_n^l$ 
which is independent of the expression for $w$ by the second group of relations in \eqref{eqn-NH-relation}. For instance, if $w_0\in \mf{S}_n$ is the usual longest element with respect to the usual Coxeter length function\footnote{The length $\ell(w)$ of an element in a reduced expression $w=s_{i_1}\cdots s_{i_r}$ is defined to be $r$.} $\ell:\mf{S}_n\lra \N$, then the corresponding (cyclotomic) nilHecke element is unambiguously depicted as the following $n$-stranded element:
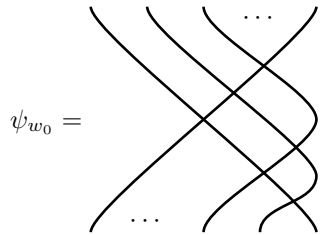
\[
\psi_{w_0}=
\begin{DGCpicture}[scale=0.75]
\DGCstrand(0,0)(4,4)
\DGCstrand(4,0)(0,4)
\DGCstrand(3,0)(4,1)(1,4)
\DGCstrand(2,0)(4,2)(2,4)
\DGCcoupon*(2.2,3.6)(3.8,4){$\cdots$}
\DGCcoupon*(0.2,0)(1.8,0.4){$\cdots$}
\end{DGCpicture} \ .
\]
The element is symmetric with respect to the $*$ anti-automorphism. 

\subsection{Idempotents}
Let $w_0$ be the longest element in the symmetric group $\mf{S}_n$.  This gives rise to an indecomposable idempotent $e_n \in \nh_n^l$
\begin{subequations}
\begin{equation}
e_n := y_1^{n-1} \cdots y_n^0 \psi_{w_0}=\begin{DGCpicture}[scale=0.75]
\DGCstrand(0,0)(4,4)
\DGCstrand(4,0)(0,4)
\DGCdot{3.65}[ur]{$_{_{n-1}}$}
\DGCstrand(3,0)(4,1)(1,4)
\DGCdot{3.65}[ur]{$_{_{n-2}}$}
\DGCstrand(2,0)(4,2)(2,4)
\DGCdot{3.65}[ur]{$_{_{n-3}}$}
\DGCcoupon*(2.2,3.3)(3.8,3.7){$\cdots$}
\DGCcoupon*(0.2,0)(1.8,0.4){$\cdots$}
\end{DGCpicture} \ .
\end{equation}
In the notation of Section \ref{sec-thick-cal}, we will also depict the above idempotent as an unoriented thick strand
\begin{equation}\label{eqn-thick-idemp}
e_n=
\begin{DGCpicture}
\DGCstrand[Green](0,0)(0,1)[$^n$`{\ }]
\end{DGCpicture} \ .
\end{equation}
\end{subequations}
Define
\begin{equation}\label{eqn-half-diag}
\begin{DGCpicture}
\DGCstrand[Green](0,0)(0,0.5)
\DGCdot.{0.25}[r]{$_n$}
\DGCstrand/u/(0,0.5)(-1,1.5)/u/
\DGCstrand/u/(0,0.5)(0.5,1.5)/u/
\DGCstrand/u/(0,0.5)(1,1.5)/u/
\DGCcoupon*(-0.8,1.1)(0.3,1.3){$\cdots$}
\end{DGCpicture}
:=
\begin{DGCpicture}
\DGCstrand(0,0)(0,1.5)
\DGCstrand(0.5,0)(0.5,1.5)
\DGCstrand(1.5,0)(1.5,1.5)
\DGCcoupon(-0.25,0.5)(1.75,1){$\psi_{w_0}$}
\DGCcoupon*(0.6,0.1)(1.4,0.4){$\cdots$}
\DGCcoupon*(0.6,1.1)(1.4,1.4){$\cdots$}
\end{DGCpicture} \ ,
\quad \quad \quad \quad
\begin{DGCpicture}
\DGCstrand[Green](0,0)(0,0.5)
\DGCdot.{0.25}[r]{$_n$}
\DGCstrand/u/(-1,-1)(0,0)/u/
\DGCstrand/u/(0.5,-1)(0,0)/u/
\DGCstrand/u/(1,-1)(0,0)/u/
\DGCcoupon*(-0.8,-0.8)(0.3,-0.6){$\cdots$}
\end{DGCpicture}
:=
\begin{DGCpicture}
\DGCstrand(0,0)(0,1.5)
\DGCstrand(0.5,0)(0.5,1.5)
\DGCstrand(1.5,0)(1.5,1.5)
\DGCcoupon(-0.25,0.5)(1.75,1){$e_n$}
\DGCcoupon*(0.6,0.1)(1.4,0.4){$\cdots$}
\DGCcoupon*(0.6,1.1)(1.4,1.4){$\cdots$}
\end{DGCpicture} \ .
\end{equation}
Then one can show that
\[
\psi_{w_0}=
\begin{DGCpicture}
\DGCstrand[Green](0,0)(0,0.5)
\DGCdot.{0.25}[r]{$_n$}
\DGCstrand/u/(-1,-1)(0,0)/u/
\DGCstrand/u/(0.5,-1)(0,0)/u/
\DGCstrand/u/(1,-1)(0,0)/u/
\DGCstrand/u/(0,0.5)(-1,1.5)/u/
\DGCstrand/u/(0,0.5)(0.5,1.5)/u/
\DGCstrand/u/(0,0.5)(1,1.5)/u/
\DGCcoupon*(-0.8,-0.8)(0.3,-0.6){$\cdots$}
\DGCcoupon*(-0.8,1.1)(0.3,1.3){$\cdots$}
\end{DGCpicture}
\ .
\]
The diagram is thick in the middle, indicating that it is a morphism factoring through the image of the idempotent. In this notation, the diagram is the concatenation its two halves.
This follows from the simple computation that
\begin{equation}\label{eqn-psiepsi}
\psi_{w_0} e_n = \psi_{w_0}y_1^{n-1}\cdots y_{n-1}^1 y_n^0 \psi_{w_0}= \psi_{w_0}.
\end{equation}
Analogously, there are thick \emph{splitters} and \emph{mergers}, for any pair $(a,b)\in \N^2$, that are built out of the idempotents:
\begin{equation}
\begin{DGCpicture}[scale=0.85]
\DGCPLstrand[Green](1,0)(1,1)[$^{a+b}$]
\DGCPLstrand[Green](1,1)(0,2)[`$_a$]
\DGCPLstrand[Green](1,1)(2,2)[`$_b$]
\end{DGCpicture}
:=
\begin{DGCpicture}[scale=0.85]
\DGCstrand(0,0)(2,2)(2,2.3)[$^1$]
\DGCstrand(0.5,0)(2.5,2)(2.5,2.3)[$^2$]
\DGCstrand(1.5,0)(3.5,2)(3.5,2.3)[$^b$]
\DGCstrand(2.5,0)(0,2)(0,2.3)[$^{b+1}$]
\DGCstrand(3.5,0)(1,2)(1,2.3)[$^{b+a}$]
\DGCcoupon*(0.8,0.1)(1.4,0.3){$\cdots$}
\DGCcoupon*(2.5,1.55)(3.2,1.7){$\cdots$}
\DGCcoupon*(2.6,0.1)(3.2,0.3){$\cdots$}
\DGCcoupon*(0.4,1.55)(1,1.7){$\cdots$}
\DGCcoupon(1.8,1.8)(3.7,2.2){$e_b$}
\DGCcoupon(-0.2,1.8)(1.2,2.2){$e_{a}$}
\end{DGCpicture} \ ,
\quad \quad \quad
\begin{DGCpicture}[scale=0.85]
\DGCPLstrand[Green](0,0)(1,1)[$^a$]
\DGCPLstrand[Green](2,0)(1,1)[$^b$]
\DGCPLstrand[Green](1,1)(1,2)[`$_{a+b}$]
\end{DGCpicture} 
:=
\begin{DGCpicture}[scale=0.85]
\DGCstrand(0,0)(0,2.3)[$^1$]
\DGCstrand(0.5,0)(0.5,2.3)[$^2$]
\DGCstrand(1.5,0)(1.5,2.3)[$^b$]
\DGCstrand(2.5,0)(2.5,2.3)[$^{b+1}$]
\DGCstrand(3.5,0)(3.5,2.3)[$^{b+a}$]
\DGCcoupon*(0.7,0.1)(1.3,0.3){$\cdots$}
\DGCcoupon*(2.7,1.55)(3.3,1.7){$\cdots$}
\DGCcoupon*(2.7,0.1)(3.3,0.3){$\cdots$}
\DGCcoupon*(0.7,1.55)(1.3,1.7){$\cdots$}
\DGCcoupon(-0.2,0.9)(3.7,1.4){$e_{a+b}$}
\end{DGCpicture} \ .
\end{equation} 

Let ${\bf i}=(i_1, \ldots, i_r)$ be a tuple of natural numbers such that $i_1+ \cdots + i_r=n$.
Then we set the idempotent
\begin{equation}\label{eqn-sequence-idempotent}
e_{{\bf i}} = e_{i_1} \otimes \cdots \otimes e_{i_r}.
\end{equation}
This corresponds to putting the diagrams for $e_{i_1}$,..., $e_{i_r}$ side by side next to one another:
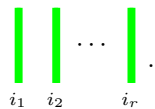
\[
\begin{DGCpicture}
\DGCstrand[Green](0,0)(0,1)[$^{i_1}$]
\DGCstrand[Green](0.5,0)(0.5,1)[$^{i_2}$]
\DGCstrand[Green](1.5,0)(1.5,1)[$^{i_r}$]
\DGCcoupon*(0.6,0.1)(1.4,0.9){$\cdots$}
\end{DGCpicture} \ .
\]

Rotating a diagram $180^\circ$ turns $e_n$ into a quasi-idempotent
\begin{equation}
e_n' := \psi_{w_0} \cdot y_1^0 \cdots y_n^{n-1}.
\end{equation}
To obtain a genuine idempotent one needs to correct the element with the sign $(-1)^{n(n-1)/2}$.

Let $a+b=n$.
For $\mu \in P(a,b) $ we define a minimal idempotent $ e^{\mu}_{(a,b)} $ of $\nh_n^l$ as follows.
First set
\begin{equation}
\psi_{a,b} = (\psi_b \cdots \psi_{a+b-1}) \cdots (\psi_2 \cdots \psi_{a+1})(\psi_1 \cdots \psi_{a-1}),
\end{equation}
which is diagrammatically depicted as
\begin{equation}
\psi_{a,b} =
\begin{DGCpicture}[scale=0.85]
\DGCstrand(0,0)(2,2)[$^1$]
\DGCstrand(0.5,0)(2.5,2)[$^2$]
\DGCstrand(1.5,0)(3.5,2)[$^a$]
\DGCstrand(2.5,0)(0,2)[$^{a+1}$]
\DGCstrand(3.5,0)(1,2)[$^{a+b}$]
\DGCcoupon*(0.8,0.1)(1.4,0.3){$\cdots$}
\DGCcoupon*(2.5,1.55)(3.2,1.7){$\cdots$}
\DGCcoupon*(2.6,0.1)(3.2,0.3){$\cdots$}
\DGCcoupon*(0.4,1.55)(1,1.7){$\cdots$}
\end{DGCpicture} \ .
\end{equation}
Then we define
\begin{equation}\label{eqn-general-idempotent}
e^{\mu}_{(a,b)} = (-1)^{|\hat{\mu}|}(\pi_{\mu}(y_1, \cdots, y_a)) (e_{(a,b)}) (\psi_{a,b})  (e_{a+b}) (\pi_{\hat{\mu}}(y_{a+1},\cdots,y_{a+b})).
\end{equation}
The idempotent is diagrammatically depicted as (c.f.~equation \eqref{eq-identitydecomp} and also \cite{KLMS})
\begin{equation}
(-1)^{|\hat{\mu}|}
\begin{DGCpicture}[scale=0.9]
\DGCPLstrand[Green](0,0)(1,1.15)[$^a$]
\DGCPLstrand[Green](2,0)(1,1.15)[$^b$]
\DGCPLstrand[Green](1,1.15)(1,1.85)
\DGCPLstrand[Green](1,1.85)(0,3)[`$_a$]
\DGCPLstrand[Green](1,1.85)(2,3)[`$_b$]
\DGCcoupon(1.15,0.35)(1.85,0.85){$_{\pi_{\hat{\mu}}}$}
\DGCcoupon(0.15,2.15)(0.85,2.65){$_{\pi_\mu}$}
\end{DGCpicture} \ .
\end{equation}
It is understood that we may let elements of $\nh_{n}^l$ act on the top and bottom of this diagram.

\subsection{\texorpdfstring{$p$}{p}-DG structure}
The cyclotomic nilHecke algebra $\nh^l_n$ has a $p$-DG structure inherited from that of $\nh_n$:
\begin{equation}
\partial(y_i)=y_i^2 \hspace{1in}
\partial(\psi_i)=-y_i \psi_i - \psi_i y_{i+1},
\end{equation}
which is diagrammatically expressed as
\[
\dif\left(~
\begin{DGCpicture}
\DGCstrand(0,0)(0,1)
\DGCdot{0.5}
\end{DGCpicture}
~\right)=
\begin{DGCpicture}
\DGCstrand(0,0)(0,1)
\DGCdot{0.5}[r]{$_2$}
\end{DGCpicture} \ ,
\quad \quad \quad \quad
\dif\left(~
\begin{DGCpicture}
\DGCstrand(1,0)(0,1)
\DGCstrand(0,0)(1,1)
\end{DGCpicture}
~\right)
=-
\begin{DGCpicture}
\DGCstrand(1,0)(0,1)
\DGCdot{0.75}
\DGCstrand(0,0)(1,1)
\end{DGCpicture}
-
\begin{DGCpicture}
\DGCstrand(1,0)(0,1)
\DGCdot{0.25}
\DGCstrand(0,0)(1,1)
\end{DGCpicture} \ .
\]

\begin{prop}
\label{pdgidempotprop}
Assume $ {\bf i}=(i_1, \ldots, i_r)$ with $i_1+\cdots+i_r=n$.  Then the projective $\nh_n^l$-module $e_{{\bf i}}\nh_n^l$ is a right $p$-DG module over $\nh_n^l$.
\end{prop}

\begin{proof}
For simplicity we assume $r=1$.  The general case follows similarly.
Using \cite[equation (2.4)]{EQ2}, one has
\begin{equation}
\dif(\psi_{w_0})=\sum_{i=1}^n (i-n)y_i \psi_{w_0}+\sum_{i=1}^n (1-i)\psi_{w_0}y_i.
\end{equation}
Since $e_n=y_1^{n-1}\cdots y_{n-1}^1y_n^0 \psi_{w_0}$, the differential action on the monomial $y_1^{n-1}\cdots y_{n-1}^1y_n^0$ cancels with the first term above, and results in
\begin{equation}
\label{partialofidempotent}
\partial(e_n)=-e_n \sum_{i=1}^n (i-1) y_i.
\end{equation}
Thus $e_n \nh_n^l$ is stable under $\partial$.
\end{proof}

We will need later the easily verified fact below:
\begin{equation}
\label{enprime}
\partial(e_n^\prime)=-\sum_{i=1}^n (n-i) y_i e_n^\prime,
\end{equation}
which follows from a similar computation as in Proposition \ref{pdgidempotprop}.

\begin{prop}
For $0 \leq n \leq p-1$,
the $p$-DG module $e_{n}\nh_n^l$ is compact and cofibrant.
\end{prop}

\begin{proof}
One repeats the arguments in the proof of \cite[Proposition 3.26]{KQ} by making the following replacements
\[
\begin{array}{ccccc}
\nh_n  \mapsto \nh_n^l, &&&&
\mathcal{P}_n^+  \mapsto e_n  \nh_n^l, \\
\sym_n  \mapsto  \mathrm{H}^*(\gr(n,l)), &&&&
\nh_n'  \mapsto (\nh_n^l)' ,
\end{array}
\]
where $\mathrm{H}^*(\gr(n,l))$ is the cohomology of the Grassmannian of $n$-dimensional subspaces in $\mathbb{C}^l$ and $(\nh_n^l)'$ is the space of traceless
$n! \times n!$ matrices over $ \mathrm{H}^*(\gr(n,l))$.
\end{proof}

\subsection{A categorification of simples}\label{subsec-cat-simples}
We first review a categorification of the irreducible representation $V_l$ of quantum $\mathfrak{sl}_2$ for a generic value of $q$ using cyclotomic nilHecke algebras due to Kang and Kashiwara \cite{KK}.  
Then we enhance it to a categorification of cyclically generated modules over the small quantum group $\dot{u}_{\mathbb{O}_p}$ and over the BLM quantum group $\dot{U}_{\mathbb{O}_p}$.

For any $a \in \Z_{\geq 0}$ there is an embedding $\nh_n^l \hookrightarrow \nh_{n+a}^l$ given by 
\begin{equation*}
y_i \mapsto y_i \quad (i=1,\ldots,n), \quad  \quad \quad  \quad \psi_i \mapsto \psi_i \quad (i=1, \ldots, n-1).
\end{equation*}
We use this embedding to produce functors between categories of nilHecke modules.

\begin{defn}\label{def-E-and-F}
There is an induction functor
\begin{subequations}\label{eqn-F(a)}
\begin{equation}
\mf{F}^{(a)} \colon (\nh_n^l,\partial) \dmod \longrightarrow (\nh^l_{n+a},\partial) \dmod
\end{equation}
given by tensoring with the $p$-DG bimodule over $(\nh_n^l,\nh_{n+a}^l)$
\begin{equation}
e_{(1^n,a)} \nh_{n+a}^l, \quad \quad \partial(e_{(1^n,a)}):=-\sum_{i=1}^{a}(1-i)e_{(1^n,a)}y_{n+i}.
\end{equation}
\end{subequations}
Notice that the differential action on the bimodule generator $e_{(1^n,a)}$ arises from the differential action on the idempotent $e_a$ defined in \eqref{pdgidempotprop}.

There is a restriction functor
\begin{subequations}\label{eqn-E(a)}
\begin{equation}
\mf{E}^{(a)} \colon (\nh_{n+a}^l, \partial) \dmod \longrightarrow (\nh_n^l,\partial) \dmod
\end{equation}
given by tensoring with the $p$-DG bimodule over $(\nh_{n+a}^l,\nh_n^l)$
\begin{equation}
\nh_{n+a}^l e_{(1^n,a)}^\star, \quad \quad \partial( e_{(1^n,a)}^\star):=\sum_{i=1}^a (2n+i-l)y_{n+i}e_{(1^n,a)}^\star.
\end{equation}
\end{subequations}
Here the differential action on the bimodule generator $e_{(1^n,a)}^\star$ is twisted from the natural $\dif$-action on $e_a^\prime$ (see equation \eqref{enprime}) by the symmetric function $(a-l+2n)(y_{n+1}+\cdots+y_{n+a})$.

For simplicity, set $\mf{F}=\mf{F}^{(1)}$ and $\mf{E}=\mf{E}^{(1)}$.
\end{defn}

\begin{rem}
In the absence of the $p$-DG structure, the functors $\mf{E}$ and $\mf{F}$ give rise to an action of Lauda's $2$-category $\mc{U}$ on 
$ \oplus_{n = 0}^l \nh_n^l \dmod$. See, for instance, \cite{Rou2} and \cite{KK}.  
\end{rem}

There is an adjunction map of the functors $\cap \colon \mf{F} \mf{E}\Rightarrow \Id$ given by the following bimodule homomorphism
\begin{equation}\label{eqn-rep-cap}
   \nh_n^le_{(1^{n-1},1)}^\star \otimes_{\nh^l_{n-1}}e_{(1^{n-1},1)} \nh_n^l \longrightarrow  \nh_n^l ,\quad \quad \quad \alpha \otimes \beta \mapsto \alpha \beta,
\end{equation}
and similarly an adjunction map $\cup:\Id\Rightarrow \mf{E} \mf{F}$ arising from
\begin{equation}\label{eqn-rep-cup}
  \nh_n^l \longrightarrow  \nh_{n+1}^le_{(1^n,1)}^\star, \quad \quad \quad \alpha \mapsto \alpha e_{(1^n,1)}^\star.
\end{equation}
There is a ``dot'' natural transformation 
\begin{equation}\label{eqn-rep-dot}
Y \colon \mf{E}\Rightarrow \mf{E}, \quad \quad \quad \left(\nh_n^l \longrightarrow \nh_n^l \quad  \quad \alpha \mapsto \alpha y_n\right).
\end{equation}
There is also a ``crossing''
\begin{equation}\label{eqn-rep-cross}
\Psi \colon \mf{E} \mf{E}  \Rightarrow \mf{E} \mf{E}, \quad \quad \quad \left(\nh_n^l \lra \nh_n^l \quad \quad \alpha \mapsto \alpha \psi_{n-1}\right).
\end{equation}

\begin{thm}
\label{catofVl}
For any $l\in \N$, there is a $2$-representation of the $p$-DG $2$-category $(\mathcal{U},\dif)$ on 
$ \oplus_{n=0}^l (\nh_n^l,\dif)\dmod $ defined as follows.

On $0$-morphisms we have
\begin{equation*}
{m} \mapsto 
\begin{cases}
(\nh_n^l,\dif)\dmod & \textrm{ if } m = l-2n \\
0 & \textrm{ otherwise }.
\end{cases}
\end{equation*}

On $1$-morphisms we have
\begin{equation*}
\mathcal{E} \1_{m} \mapsto
\begin{cases}
\mf{E} & \textrm{ if } m = l-2n \\
0 & \textrm{ otherwise }
\end{cases}
\quad \quad \quad
\mathcal{F} \1_{m} \mapsto
\begin{cases}
\mf{F} & \textrm{ if } m = l-2n \\
0 & \textrm{ otherwise }.
\end{cases}
\end{equation*}

On $2$-morphisms we have
\begin{equation*}
  \begin{DGCpicture}
  \DGCstrand(0,0)(0,1)
  \DGCdot*>{0.75}
  \DGCdot{0.45}
  \DGCcoupon*(0.1,0.25)(1,0.75){$^m$}
  \DGCcoupon*(-1,0.25)(-0.1,0.75){$^{m+2}$}
  \DGCcoupon*(-0.25,1)(0.25,1.15){}
  \DGCcoupon*(-0.25,-0.15)(0.25,0){}
  \end{DGCpicture}
  ~\mapsto~
  ~Y~
  \quad \quad \quad 
  \begin{DGCpicture}
  \DGCstrand(0,0)(1,1)
  \DGCdot*>{0.75}
  \DGCstrand(1,0)(0,1)
  \DGCdot*>{0.75}
  \DGCcoupon*(1.1,0.25)(2,0.75){$^m$}
  \DGCcoupon*(-1,0.25)(-0.1,0.75){$^{m+4}$}
  \DGCcoupon*(-0.25,1)(0.25,1.15){}
  \DGCcoupon*(-0.25,-0.15)(0.25,0){}
  \end{DGCpicture} 
 ~ \mapsto~
 ~\Psi~
  \end{equation*}

\begin{equation*}
\begin{DGCpicture}
  \DGCstrand/d/(0,0)(1,0)
  \DGCdot*<{-0.25,2}
  \DGCcoupon*(1,-0.5)(1.5,0){$^m$}
  \DGCcoupon*(-0.25,0)(1.25,0.15){}
  \DGCcoupon*(-0.25,-0.65)(1.25,-0.5){}
  \end{DGCpicture}
  ~\mapsto~
  ~\cup~
  \quad \quad \quad
  \begin{DGCpicture}
  \DGCstrand(0,0)(1,0)/d/
  \DGCdot*<{0.25,1}
  \DGCcoupon*(1,0)(1.5,0.5){$^m$}
  \DGCcoupon*(-0.25,0.5)(1.25,0.65){}
  \DGCcoupon*(-0.25,-0.15)(1.25,0){}
  \end{DGCpicture}
  ~\mapsto~
  ~\cap~.
\end{equation*}

Furthermore, for $0 \leq l \leq p-1$, there is an isomorphism
$K_0(\oplus_{n=0}^l \mathcal{D}^c(\nh_n^l)) \cong V_l $
as (irreducible) modules over $\dot{u}_{\mathbb{O}_p}$.
\end{thm}

\begin{proof}
By \cite[Theorem 5.2]{KK} there are isomorphisms of functors
\begin{align*}
&\rho \colon \mf{E} \mf{F}  \bigoplus \left( \bigoplus_{r=0}^{l-2n-1} q^{l-2n-1-2r} \Id \right) \Rightarrow \mf{F} \mf{E}, \quad \quad \text{ if } l-2n \geq 0 \\
&\rho \colon \mf{E} \mf{F}  \Rightarrow \mf{F} \mf{E} 
\bigoplus \left( \bigoplus_{r=0}^{-l+2n-1} q^{ l-2n+1+2r }\Id \right), \quad \quad \text{ if } l-2n \leq 0. 
\end{align*}
which extend to a representation of Rouquier's category \cite{Rou2} on 
$\oplus_{n=0}^l \nh_n^l\dmod$.
The main theorem of \cite{Brundan2KM} implies that this extends to a representation of the Lauda category $\mathcal{U}$ (ignoring $\partial$).

Now we check how the derivation acts on $Y$, $\Psi$, $\cap$, and $\cup$.

\begin{align*}
\partial(Y) (\alpha) &= \partial(Y(\alpha)) - Y(\partial(\alpha)) \\
&= \partial(\alpha y_n) - (\partial \alpha) y_n \\
&= (\partial \alpha) y_n + \alpha (\partial y_n) - (\partial \alpha) y_n \\
&= \alpha y_n^2.
\end{align*}

\begin{align*}
\partial(\Psi)(\alpha) &= \partial(\Psi \alpha) - \Psi(\partial \alpha) \\
&= \partial(\alpha \psi_{n-1}) - (\partial \alpha) \psi_{n-1} \\
&= (\partial \alpha) \psi_{n-1} - \alpha y_{n-1} \psi_{n-1} - \alpha \psi_{n-1} y_n -(\partial \alpha) \psi_{n-1} \\
&= \alpha(-y_{n-1} \psi_{n-1} - \psi_{n-1} y_n).
\end{align*}

\begin{align*}
\partial(\cap)(\alpha \otimes \beta) &= \partial(\cap(\alpha \otimes \beta)) - \cap(\partial(\alpha \otimes \beta)) \\
&= \partial(\alpha \beta) - \partial(\alpha) \beta - \alpha \partial(\beta)
+(l-2n+1) \alpha y_n \beta \\
&= (l-2n+1) \alpha y_n \beta.
\end{align*}

\begin{align*}
\partial(\cup)(\alpha)&=\partial(\cup(\alpha))-\cup(\partial(\alpha)) \\
&=\partial(\alpha)-(l-2n-1) \alpha y_{n+1}- \partial(\alpha) \\
&=-(l-2n-1) \alpha y_{n+1}.
\end{align*}

Thus we have
\begin{equation}
\label{derongenerators}
\partial(Y) = Y^2, \quad \quad 
\partial(\Psi)=-Y \Psi - \Psi Y, \quad \quad 
\partial(\cap) = (m+1) \cap Y, \quad \quad
\partial(\cup) = (1-m) Y \cup
\end{equation}
implying that the representation of $\mathcal{U}^R$ on $\oplus_{n=0}^l \nh_n^l \dmod$ is actually a representation of the $p$-DG category $(\mathcal{U}^R,\partial)$.
Theorem \ref{equiv-of-pdgRoug-pdgLauda} implies that the action of 
$(\mathcal{U}^R,\partial)$ extends to an action of $(\mathcal{U},\partial)$.

The isomorphism 
$K_0(\oplus_{n=0}^l \mathcal{D}^c(\nh_n^l)) \cong V_l $
is \cite[Theorem 6.15]{EQ1}.
\end{proof}

\begin{rem}\label{rmk-restriction-on-l}
The restriction on $l$ in the last part of Theorem \ref{catofVl} is essential when passing to derived categories $\mc{D}^c(\nh_n^l)$. The reason is that, for $n\geq p$, one can show in a similar fashion as in \cite[Section 3]{KQ} that $\nh_n^l$ is always acyclic whenever $n\geq p$. Therefore, the sum $\oplus_{n=0}^l \mc{D}^c(\nh_n^l)$ only categorifies the subspace generated by the highest weight vector in the Weyl module $V_l$ for $\dot{U}_{\mathbb{O}_p}$. For a categorification of the Weyl module $V_l$ itself, see Section \ref{examplesubsection}.
\end{rem}

\begin{rem}
The biadjointness of the functors $\mf{E}$ and $\mf{F}$ follows from \cite{Rou2, KK} in conjunction with \cite{Brundan2KM}.
Kashiwara \cite{KaBi} showed directly that the functors $\mf{E}$ and $\mf{F}$ are biadjoint.
\end{rem}

Above we have given an algebraic construction of a categorical action of $(\mc{U},\dif)$, using induction and restriction
bimodules (with suitable differentials), on some $2$-representations. In \cite[Section 6.3]{EQ1}, such an action has also been exhibited in a diagrammatic approach, which we recall in the remainder of this section.  We will also point out the relationship between the algebraic and diagrammatic constructions. This comparison explains the (seemingly unexpected) formulas in \eqref{eqn-F(a)} and \eqref{eqn-E(a)}.

Consider the \emph{cyclotomic quotient category} $\mathcal{V}_l$ to be the quotient category of $\1_l\mc{U}$ by morphisms in the two-sided ideal which is right monoidally generated by
\begin{itemize}
\item[(1)] Any morphism that contains the following subdiagram on the far left:
$$\begin{DGCpicture}
\DGCstrand(0,0)(0,1)
\DGCdot*<{0.5}
\DGCcoupon*(-0.4,0.4)(-0.7,0.8){$l$}
\end{DGCpicture}~.
$$
\item[(2)] All positive degree bubbles on the far left region labeled $l$.
\end{itemize}
Here by ``two-sided'' we mean concatenating diagrams vertically from top and bottom to those in the relations, while by ``right monoidally generated'' we mean composing pictures from $\mc{U}$ to the right of those generators. Schematically we depict elements in the ideal as follows.

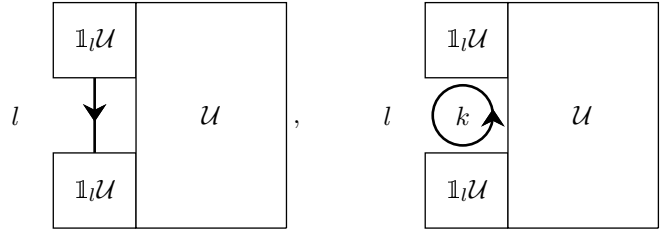
\[
\begin{DGCpicture}
\DGCstrand(-0.55,1)(-0.55,2)
\DGCdot*<{1.5}
\DGCcoupon(0,0)(-1.1,1){$\1_l\mc{U}$}
\DGCcoupon(0,2)(-1.1,3){$\1_l\mc{U}$}
\DGCcoupon(2,0)(0,3){$\mc{U}$}
\DGCcoupon*(-1.2,1)(-2,2){$l$}
\end{DGCpicture}~, \quad \quad
\begin{DGCpicture}
\DGCbubble(-0.6,1.5){0.4}
\DGCdot*>{1.5,R}
\DGCcoupon*(-0.4,1.3)(-0.8,1.7){$k$}
\DGCcoupon(0,0)(-1.1,1){$\1_l\mc{U}$}
\DGCcoupon(0,2)(-1.1,3){$\1_l\mc{U}$}
\DGCcoupon(2,0)(0,3){$\mc{U}$}
\DGCcoupon*(-1.2,1)(-2,2){$l$}
\end{DGCpicture}~.
\]

One implication of these relations is that
\begin{equation}\label{cyclotomiccurl}
0 = \curl{L}{U}{$l$}{no}{$0$} = \sum_{a+b=l}
~\bigccwbubble{$a$}{$l$}\oneline{$b$}{no} .
\end{equation}
Moreover, every bubble with $a>0$ is also in the ideal, so that
\begin{equation}\label{eqn-cyclotomic-rel}
\begin{DGCpicture}
\DGCstrand(0,0)(0,1)
\DGCdot{.5}[r]{$^l$}
\DGCdot*>{1}
\DGCcoupon*(-0.2,0.4)(-0.8,0.8){$l$}
\end{DGCpicture}~=0.
\end{equation}
Therefore, it follows that the cyclotomic nilHecke algebra $\nh_n^l$ maps onto $\END_{\mc{V}_l}(\1_l \mc{E}^n\1_{l-2n})$. Further, Lauda  \cite[Section 7]{Lau1} has proven that this map is an isomorphism:
\begin{equation}\label{eqn-iso-to-nilHecke}
\nh_n^l\cong \END_{\mc{V}_l}(\1_l\mc{E}^n\1_{l-2n})
=\left\{
\begin{DGCpicture}
\DGCstrand(0,0)(0,1.5)
\DGCdot*>{1.5}
\DGCstrand(0.5,0)(0.5,1.5)
\DGCdot*>{1.5}
\DGCstrand(1.5,0)(1.5,1.5)
\DGCdot*>{1.5}
\DGCcoupon(-0.1,0.4)(1.6,1.1){$\nh_n^l$}
\DGCcoupon*(-0.7,0.4)(-0.1,1.1){$_l$}
\DGCcoupon*(1.6,0.4)(2.4,1.1){$_{l-2n}$}
\DGCcoupon*(0.6,0.1)(1.4,0.3){$\cdots$}
\DGCcoupon*(0.6,1.2)(1.4,1.4){$\cdots$}
\end{DGCpicture}
\right\} \ .
\end{equation}
Since the ideal used to define $\mc{V}_l$ is clearly $\dif$-stable, $\mc{V}_l$ carries a natural quotient $p$-differential which is still denoted $\dif$. In this way, $(\mc{V}_l,\dif)$ is a $p$-DG module-category over $(\mc{U},\dif)$. The isomorphism \eqref{eqn-iso-to-nilHecke} is an isomorphism of $p$-DG algebras by Definition \ref{def-special-dif}.

The assignment on 0-morphisms in Theorem \ref{catofVl} can now be seen as reading off the weights on the far right for the diagrams in equation \eqref{eqn-iso-to-nilHecke}.

To see the necessity of twisiting the differential on the restriction functor $\mc{E}$, it is readily seen that the restriction bimodule $\nh_n^le_{(1^{n-1},1)}^\star$, regarded as a functor
\[
\mc{E}=(-)\otimes_{\nh_n^l}\left(\nh_n^le_{(1^{n-1},1)}^\star\right): (\nh_n^l,\dif)\lra (\nh_{n-1}^l,\dif),
\]
may be identified with the space of diagrams
\begin{equation}
\nh_n^le_{(1^{n-1},1)}^\star\cong 
\left\{
\begin{DGCpicture}
\DGCstrand(0,0)(0,1.5)
\DGCdot*>{1.5}
\DGCstrand(0.5,0)(0.5,1.5)
\DGCdot*>{1.5}
\DGCstrand(1.5,0.5)(1.5,1.5)
\DGCdot*>{1.5}
\DGCstrand/d/(1.5,0.5)(2,0.5)/u/(2,1.5)/u/
\DGCcoupon(-0.1,0.5)(1.6,1){$\nh_n^l$}
\DGCcoupon*(-0.7,0.4)(-0.1,1.1){$_l$}
\DGCcoupon*(2.1,0.4)(3.2,1.1){$_{l-2n+2}$}
\DGCcoupon*(0.6,0.1)(1.4,0.3){$\cdots$}
\DGCcoupon*(0.6,1.2)(1.4,1.4){$\cdots$}
\end{DGCpicture}
\right\}
\end{equation}
whose diagrammatic generator satisfies the differential formula
\begin{equation}
\dif\left(~
\begin{DGCpicture}
\DGCstrand(0,0)(0,1.5)
\DGCdot*>{1.5}
\DGCstrand(0.5,0)(0.5,1.5)
\DGCdot*>{1.5}
\DGCstrand(1.5,0.5)(1.5,1.5)
\DGCdot*>{1.5}
\DGCstrand/d/(1.5,0.5)(2,0.5)/u/(2,1.5)/u/
\DGCcoupon*(2.1,0.4)(3.2,1.1){$_{l-2n+2}$}
\DGCcoupon*(0.6,0.1)(1.4,0.3){$\cdots$}
\DGCcoupon*(0.6,1.2)(1.4,1.4){$\cdots$}
\end{DGCpicture}
~\right)
=(2n-l-1)
\begin{DGCpicture}
\DGCstrand(0,0)(0,1.5)
\DGCdot*>{1.5}
\DGCstrand(0.5,0)(0.5,1.5)
\DGCdot*>{1.5}
\DGCstrand(1.5,0.5)(1.5,1.5)
\DGCdot*>{1.5}
\DGCdot{1}
\DGCstrand/d/(1.5,0.5)(2,0.5)/u/(2,1.5)/u/
\DGCcoupon*(2.1,0.4)(3.2,1.1){$_{l-2n+2}$}
\DGCcoupon*(0.6,0.1)(1.4,0.3){$\cdots$}
\DGCcoupon*(0.6,1.2)(1.4,1.4){$\cdots$}
\end{DGCpicture} \ .
\end{equation}
On the other hand, the induction functor has an obvious diagrammatic interpretation by identifying the bimodule
$_{\nh_{n-1}^l}\left({e_{(1^{n-1},1)}\nh_n^l}\right)_{\nh_n^l}$ as the space in equation \eqref{eqn-iso-to-nilHecke}, but the left $\nh_{n-1}^l$ acts only through the first $n-1$ strands on the top.

The bimodule homomorphism $\cap$ is now literally given by the ``cap''. More precisely, given elements $\alpha \in \nh_{n}^le^\star_{(1^{n-1},1)}$ and $\beta\in e_{(1^{n-1},1)}\nh_n^l$, the bimodule homomorphism $\cap$ is diagrammatically given by
\[
\left(~
\begin{DGCpicture}
\DGCstrand(0,0)(0,1.5)
\DGCdot*>{1.5}
\DGCstrand(0.5,0)(0.5,1.5)
\DGCdot*>{1.5}
\DGCstrand(1.5,0.5)(1.5,1.5)
\DGCdot*>{1.5}
\DGCstrand/d/(1.5,0.5)(2,0.5)/u/(2,1.5)/u/
\DGCcoupon(-0.1,0.5)(1.6,1){$\alpha$}
\DGCcoupon*(-0.7,0.4)(-0.1,1.1){$_l$}
\DGCcoupon*(2.1,0.4)(3.2,1.1){$_{l-2n+2}$}
\DGCcoupon*(0.6,0.1)(1.4,0.3){$\cdots$}
\DGCcoupon*(0.6,1.2)(1.4,1.4){$\cdots$}
\end{DGCpicture}
~,~
\begin{DGCpicture}
\DGCstrand(0,0)(0,1.5)
\DGCdot*>{1.5}
\DGCstrand(0.5,0)(0.5,1.5)
\DGCdot*>{1.5}
\DGCstrand(1.5,0)(1.5,1.5)
\DGCdot*>{1.5}
\DGCcoupon(-0.1,0.5)(1.6,1){$\beta$}
\DGCcoupon*(-0.7,0.4)(-0.1,1.1){$_l$}
\DGCcoupon*(1.7,0.4)(2.4,1.1){$_{l-2n}$}
\DGCcoupon*(0.6,0.1)(1.4,0.3){$\cdots$}
\DGCcoupon*(0.6,1.2)(1.4,1.4){$\cdots$}
\end{DGCpicture}~\right)
\mapsto
\begin{DGCpicture}
\DGCstrand(0,0)(0,2)
\DGCdot*>{2}
\DGCstrand(0.5,0)(0.5,2)
\DGCdot*>{2}
\DGCstrand(1.25,1.75)(1.25,2)
\DGCdot*>{2}
\DGCstrand(1.25,0)(1.25,0.25)
\DGCstrand/d/(1.25,1.35)(1.65,1.35)/u/(1.65,1.35)/u/(1.65,1.75)(2.05,1.75)(2.05,1.35)/d/(1.25,0.75)/d/
\DGCdot*<{1.5,1}
\DGCcoupon(-0.1,1.25)(1.35,1.75){$\alpha$}
\DGCcoupon(-0.1,0.25)(1.35,0.75){$\beta$}
\DGCcoupon*(-0.7,0.4)(-0.1,1.6){$_l$}
\DGCcoupon*(2.1,0.4)(3,1.6){$_{l-2n}$}
\DGCcoupon*(0.6,0)(1.15,0.2){$\cdots$}
\DGCcoupon*(0.6,1.8)(1.15,2){$\cdots$}
\end{DGCpicture}
=
\begin{DGCpicture}
\DGCstrand(0,0)(0,1.5)
\DGCdot*>{1.5}
\DGCstrand(0.5,0)(0.5,1.5)
\DGCdot*>{1.5}
\DGCstrand(1.5,0)(1.5,1.5)
\DGCdot*>{1.5}
\DGCcoupon(-0.1,0.5)(1.6,1){$\alpha\beta$}
\DGCcoupon*(-0.7,0.4)(-0.1,1.1){$_l$}
\DGCcoupon*(1.7,0.4)(2.4,1.1){$_{l-2n}$}
\DGCcoupon*(0.6,0.1)(1.4,0.3){$\cdots$}
\DGCcoupon*(0.6,1.2)(1.4,1.4){$\cdots$}
\end{DGCpicture} \ ,
\]
i.e., we ``cap'' off $\alpha$ by the element
\[
\begin{DGCpicture}
  \DGCstrand(0,0)(1,0)/d/
  \DGCdot*<{0.25,1}
  \DGCcoupon*(1.1,0)(2,0.5){$^{l-2n}$}
  \DGCcoupon*(-0.25,0.5)(1.25,0.65){}
  \DGCcoupon*(-0.25,-0.15)(1.25,0){}
  \end{DGCpicture}
\]
on the upper right corner, and multiply the resulting diagram with $\beta$. This is evidently the $p$-DG bimodule homomorphism
\[
\nh_{n}^le^\star_{(1^{n-1},1)}\otimes_{\nh_{n-1}^l} e_{(1^{n-1},1)}\nh_n^l\lra \nh_n^l,\quad\quad \alpha\otimes \beta\mapsto \alpha\beta
\]
utilized in the statement of Theorem \ref{catofVl}.

The cup bimodule homomorphism $\cup$ admits a similar diagrammatic description in $\mc{V}_l$ as well. It is given by
\begin{align*}
\nh_n^l  \lra \nh_{n+1}^le_{(1^n,1)}^\star , \quad \quad \quad 
\begin{DGCpicture}
\DGCstrand(0,0)(0,1.5)
\DGCdot*>{1.5}
\DGCstrand(0.5,0)(0.5,1.5)
\DGCdot*>{1.5}
\DGCstrand(1.5,0)(1.5,1.5)
\DGCdot*>{1.5}
\DGCcoupon(-0.1,0.5)(1.6,1){$\alpha$}
\DGCcoupon*(-0.7,0.4)(-0.1,1.1){$_l$}
\DGCcoupon*(1.7,0.4)(2.4,1.1){$_{l-2n}$}
\DGCcoupon*(0.6,0.1)(1.4,0.3){$\cdots$}
\DGCcoupon*(0.6,1.2)(1.4,1.4){$\cdots$}
\end{DGCpicture}
 \mapsto
\begin{DGCpicture}
\DGCstrand(0,0)(0,1.5)
\DGCdot*>{1.5}
\DGCstrand(0.5,0)(0.5,1.5)
\DGCdot*>{1.5}
\DGCstrand(1.5,0)(1.5,1.5)
\DGCdot*>{1.5}
\DGCstrand/d/(2,1.5)(2,1.2)(2.5,1.2)/u/(2.5,1.5)/u/
\DGCdot*<{1.5,1}
\DGCcoupon(-0.1,0.5)(1.6,1){$\alpha$}
\DGCcoupon*(-0.7,0.4)(-0.1,1.1){$_l$}
\DGCcoupon*(1.8,0.4)(2.7,1){$_{l-2n}$}
\DGCcoupon*(0.6,0.1)(1.4,0.3){$\cdots$}
\DGCcoupon*(0.6,1.2)(1.4,1.4){$\cdots$}
\end{DGCpicture} \ ,
\end{align*}
i.e., appending the corresponding cup to the far right of any diagram $\alpha \in \nh_n^l$. The differential action on the map is then a consequence of Lemma \ref{def-special-dif-Rouquier}.

\section{Some cyclic modules}
\label{sec-cyclic-mod}
In this section, we will study a collection of combinatorially defined nilHecke modules introduced in the work of Hu-Mathas \cite{HuMathasKLR, HuMathas} under the action of $p$-differentials. These modules will be utilized to define an analogue of $p$-DG quiver Schur algebras in the current work.

\subsection{Cellular structure}
\begin{defn}\label{def-multipartition}
A \emph{$\nh^l_n$-multipartition} (or simply a \emph{partition} for short) is an $l$-tuple $\mu=(\mu^{1},\ldots,\mu^{l})$ 
such that $ \mu^i \in \{0, 1\} $ and $ \mu^1 + \cdots + \mu^l=n$. We may also think of a partition as a sequence of empty slots and boxes.

The set of $\nh_n^l$-partitions will be denoted by $\mathcal{P}_n^l$. 
\end{defn}

\begin{example}\label{eg-23partitions}
As an example, the following partitions constitute the full list of all $\nh_{2}^3$-multipartitions:
\[
\left(~\yng(1)~,~\yng(1)~,~\emptyset~\right),\quad \quad
\left(~\yng(1)~,~\emptyset~,~\yng(1)~\right),\quad \quad
\left(~\emptyset~,~\yng(1)~,~\yng(1)~\right).
\]
They correspond to the numerical notation of $(1,1,0)$, $(1,0,1)$ and $(0,1,1)$ respectively.
\end{example}

\begin{defn}\label{def-order-on-nh-partition}
For two elements $\lambda, \mu \in \mathcal{P}_n^l$, declare $\lambda \geq \mu$ if
\begin{equation*}
\lambda^1+\cdots+\lambda^k \geq \mu^1+\cdots+\mu^k 
\end{equation*}
for $k=1, \ldots, l$. We will say $\lambda> \mu$ if $\lambda \geq \mu$ and $\lambda \neq \mu$. This defines a partial order on the set of partitions $\mc{P}_n^l$ called the \emph{dominance order}.
\end{defn}

Combinatorially, when regarded as partitions, $\lambda \geq \mu$ if $\mu$ can be obtained from $\lambda$ by a sequence of moves which exchange a box and an empty space immediately right to the box. It is easily seen that there is always a unique partition that is minimal with respect to the dominance order, namely the one with all boxes on the right. We will denote the unique minimal element under this partial ordering by $\lambda_0$.

 The following is an example of incomparable partitions:
 \[
\left(~\emptyset~, ~\yng(1)~,~\yng(1)~,~\emptyset~\right),\quad \quad
\left(~\yng(1)~,~\emptyset~,~\emptyset~,~\yng(1)~\right),
\]
which are both greater than
\[
\left(~\emptyset~, ~\yng(1)~,~\emptyset~,~\yng(1)~\right).
\]

We next introduce the notion of tableaux and a partial order on them as well.

\begin{defn}\label{def-tableau}
Given a partition $\mu \in \mc{P}_n^l$, suppose $ \mu^{j_1}=\cdots=\mu^{j_n}=1$ and $ j_1 < \cdots < j_n$.
A \emph{tableau of shape $\mu$} is a bijection 
\begin{equation*}
\mf{t} \colon \{j_1, \ldots, j_n \} \longrightarrow \{1, \ldots, n \}.
\end{equation*}
Denote the set of $\mu$-tableaux by $\mathrm{Tab}(\mu)$, and write for any $\mf{t}\in \mathrm{Tab}(\mu)$ that $\mathrm{shape}(\mf{t})=\mu$.
\end{defn}

Given a tableau, we may think of it as a filling of its underlying partition labeled by the set of natural numbers $\{1,\dots, n\}$. This, in turn, gives us a sequence of subtableaux in order of which the tableaux is built up by adding at the $k$th step the box labeled by $k$ ($1\leq k \leq n$).

\begin{example}
For the partition $\mu:=\left(~\yng(1)~,~\yng(1)~,~\emptyset~\right)$, we have its set of tableaux equal to
\[
\mathrm{Tab}(\mu)=\left\{
\left(~\young(1)~,~\young(2)~,~\emptyset~\right),~\left(~\young(2)~,~\young(1)~,~\emptyset\right)
\right\}.
\]
In these examples, the corresponding tableaux can be regarded as built up in two steps:
\[
 \left(~\young(1)~,~\emptyset~,~\emptyset~\right)\rightarrow
\left(~\young(1)~,~\young(2)~,~\emptyset~\right) , 
\quad \quad
\left(~\emptyset~,~\young(1)~,~\emptyset~\right)\rightarrow
\left(~\young(2)~,~\young(1)~,~\emptyset~\right) .
\]
Another example of the process can be read from
\[
\left(~\emptyset~, ~\young(1)~,~\emptyset~,~\emptyset~\right)\rightarrow \left(~\emptyset~, ~\young(1)~,~\young(2)~,~\emptyset~\right)\rightarrow \left(~\young(3)~, ~\young(1)~,~\young(2)~,~\emptyset~\right).
\]
\end{example}

\begin{defn}\label{def-standard-tableaux}
\begin{enumerate}
\item[(1)] For a partition $\mu$ let $\mf{t}^{\mu}$ be the tableau given by
\begin{equation*}
\mf{t}^{\mu}(j_k)=k \hspace{.5in} k=1,\ldots, n.
\end{equation*}
We will refer to the tableau as the \emph{standard tableau} of shape $\mu$. 
\item[(2)] 
Any tableaux $\mf{t}\in \mathrm{Tab}(\mu)$ can be identified with a unique permutation $w_\mf{t}\in \mf{S}_n$ which acts on $\mf{t}^{\mu}$ by $w_\mf{t} \mf{t}^{\mu}=\mf{t}$.
We will call $w_\mf{t}$ the \emph{permutation determined by $\mf{t}$}.
\end{enumerate}
\end{defn}

\begin{rem}[On notation]\label{ntntableauxsymmetricgroup}
The set $\Tab(\mu)$ constitutes an $\mf{S}_n$-set with a simple transitive action. Therefore, part (2) of Definition \ref{def-standard-tableaux} relies on the fact we assign to the identity element $e\in \mf{S}_n$ the maximal tableaux $\mf{t}^\mu$. 

In what follows, when talking about tableaux of a fixed partition $\mu$, we will abuse notation and use the tableau $\mf{t}$ and the permutation associated with it $w_\mf{t}\in \mf{S}_n$ interchangeably. In this notation $e$ will always stand for the standard tableau $\mf{t}^\mu$. We will also use the usual Coxeter length function on the symmetric group $\ell: \mf{S}_n\lra \N$, and transport it to $\Tab(\mu)$ under this identification.
\end{rem}


\begin{example}\label{eg-two-tableaux}
The tableaux $\left(~\young(1)~,~\young(2)~,~\emptyset~\right)$ is the standard one of its shape, while $\left(~\young(2)~,~\young(1)~,~\emptyset~\right)$ is non-standard. The corresponding permutations are the identity and non-identity element of the symmetric group $\mf{S}_2$.
\end{example}

\begin{defn}
Given a tableau $\mf{t}$, it is regarded as a bijection $\mf{t}: \{j_1,\dots, j_n\} \lra \{1,\dots, n\}$. Let $\mf{t}_{\downarrow k}$ be the subtableau defined by
\begin{equation*}
\mf{t}_{\downarrow k} \colon = \mf{t}\big|_{\mf{t}^{-1}\{1, \ldots, k \}} : \mf{t}^{-1}\{1, \ldots, k \} \lra \{1,\dots, k  \}.
\end{equation*}
In other words, $\mf{t}_{\downarrow k}$ is the subtableau of $\mf{t}$ built in the first $k$ steps, and it is a filling of a partition in the set $\mathcal{P}_k^l$.
\end{defn}

We are now ready to introduce a partial order on tableaux.

\begin{defn}\label{def-order-on-nh-tableaux}
Let $\mf{s}$ be a $\mu$-tableau and $\mf{t}$ a $\lambda$-tableau.  We write $ \mf{s} \geq \mf{t} $ if we have the inequalities of the underlying partitions of subtableaux (Definition \ref{def-order-on-nh-partition})
\begin{equation*}
\mathrm{shape}(\mf{s}_{\downarrow k}) \geq \mathrm{shape}( \mf{t}_{\downarrow k}), \hspace{.5in} \text{ for all }k=1, \ldots, n.
\end{equation*}
Moreover, if $\mf{s}\geq \mf{t}$ and $\mf{s}\neq \mf{t}$, we then write $\mf{s}> \mf{t}$.
\end{defn}

\begin{rem}
We make several simple notes.
\begin{enumerate}
\item[(1)] For later use it is convenient to observe that a $\mu$-tableau $\mf{s}$ is greater than or equal to the the standard tableau $\mf{t}^{\lambda}$ of a partition $\lambda$, if each element in the filling $\mu$ corresponding to $\mf{s}$ appears to the left of the same element in  $\mf{t}^\lambda$.
\item[(2)] One may easily verify that, $\mu \geq\lambda$ if and only if the standard filling $\mf{t}^\mu$ of $\mu$ is greater than or equal to the standard filling $\mf{t}^\lambda$ of $\lambda$.
\item[(3)] It is clear from definition that $ \mf{t}^{\mu} \geq \mf{s}$ for all $\mu$-tableaux $\mf{s}$.
\end{enumerate} 
\end{rem}

\begin{defn}
The degree of a $\mu$-tableau $\mf{t}$ is defined by
\begin{equation*}
deg(\mf{t})=nl-(j_1+\cdots+j_n)-2\ell(w_\mf{t}),
\end{equation*}
where $\ell(w_\mf{t})$ is the length of the permutation $w_\mf{t}$.
\end{defn}

\begin{rem}
If $l=n$, then $deg(\mf{t})=\frac{n(n-1)}{2}-2\ell(w_\mf{t})$.
\end{rem}

We now consider certain important elements in the cyclotomic nilHecke algebra. Recall from Definition~\ref{def-standard-tableaux} that if  $\mf{t}$ is a tableau of shape $\mu$, then $w_\mf{t}$ is a permutation, which in turn defines a nilHecke element $\psi_\mf{t}:=\psi_{w_\mf{t}}\in \nh_n^l$ (see the end of Section \ref{subsec-def-nilHecke}). For instance, for the standard and non-standard tableaux in Example \ref{eg-two-tableaux}, we obtain their corresponding nilHecke elements :
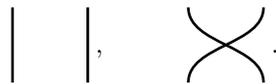
\[
\begin{DGCpicture}
\DGCstrand(1,0)(1,1)
\DGCstrand(0,0)(0,1)
\end{DGCpicture} \ , 
\quad \quad \quad
\begin{DGCpicture}
\DGCstrand(1,0)(0,1)
\DGCstrand(0,0)(1,1)
\end{DGCpicture} \ .
\]

\begin{defn}\label{def-y-mu-psi-mu}
Again suppose that for the partition $\mu$ that $ \mu^{j_1}=\cdots=\mu^{j_n}=1$ and $ j_1 < \cdots < j_n$.
For two $\mu$-tableaux $\mf{s}$ and $\mf{t}$ let 

\begin{equation}\label{eqn-psi-st}
y^{\mu}:=y_1^{l-j_1} \cdots y_n^{l-j_n}, \quad \quad \quad \psi_{\mf{st}}^{\mu}:=\psi_\mf{s}^* y^{\mu} \psi_\mf{t}.
\end{equation}
(Recall that $\psi_\mf{s}^*$ is the image of $\psi_\mf{s}$ under the anti-automorphism  $ * \colon \nh_n^l \rightarrow \nh_n^l$ defined by $\psi_i^*=\psi_i$.)
When the composition $ \mu $ is clear from context we will often abbreviate $ \psi_{\mf{st}}^{\mu} $ by $ \psi_{\mf{st}}$.
\end{defn}
Note that when $\mf{t}$ and $\mf{s}$ are the standard filling $t^\mu$ of the tableau $\mu$, so that the corresponding permutations are both $e\in \mf{S}_n$, then
$\psi_{\mf{st}}^{\mu}=\psi_{ee}^\mu=y^{\mu}$.

\begin{thm}
The set $ \{ \psi^{\mu}_{\mf{st}} | \mf{s},\mf{t} \in \mathrm{Tab}(\mu), \mu \in \mathcal{P}_n^l \}$ is a graded cellular basis of $\nh_n^l$.
More precisely
\begin{enumerate}
\item[(i)] The degree of $\psi_{\mf{st}}$ is the sum of the degrees of $\mf{s}$ and $\mf{t}$.
\item[(ii)] For $\mu \in \mathcal{P}_n^l$ and $\mf{s},\mf{t} \in \mathrm{Tab}(\mu)$, there are scalars $r_{\mf{tv}}(x)$ which do not depend on $\mf{s}$ such that,
\begin{equation*}
\psi_{\mf{st}}x = \sum_{ \mf{v} \in \mathrm{Tab}(\mu)} r_{\mf{tv}}(x) \psi_{\mf{sv}} \hspace{.15in} \text{ mod } (\nh_n^l)^{>\mu},
\end{equation*}
where 
\begin{equation*}
(\nh_n^l)^{>\mu} = \Bbbk \langle \psi_{ab}^{\lambda} | \lambda > \mu, \text{ and } a,b \in \mathrm{Tab}(\lambda) \rangle.
\end{equation*}
\item[(iii)] The anti-automorphism $ * \colon \nh_n^l \longrightarrow \nh_n^l$ sends $\psi_{\mf{st}}$ to $\psi_{\mf{st}}^*=\psi_{\mf{ts}}$.
\end{enumerate}
\end{thm}
\begin{proof}
This can be found in \cite[Theorem 5.8 and 6.11]{HuMathasKLR}. See also \cite{Lige}.
\end{proof}

The next result shows that $\nh_n^l$ is a symmetric Frobenius algebra.  
\begin{prop}
\label{symm}
 There is a non-degenerate homogeneous trace $\tau \colon \nh_n^l \longrightarrow \Bbbk$ of degree $-2n(l-n)$.
\end{prop}
\begin{proof}
See \cite[Theorem 6.17]{HuMathasKLR}.
\end{proof}

\begin{rem}
The result in \cite{HuMathasKLR} is much more general than what we need in the current work.  The special case of $\nh_n^l$ follows from the fact that it is isomorphic to an $n! \times n!$ matrix algebra over $\mathrm{H}^*(\gr(n,l))$, and the algebra $\mathrm{H}^*(\gr(n,l))$ is Frobenius with a trace that maps a generator in the top degree to $1$ and everything else to $0$.  It would be nice to match it up with the graphical trace defined in \cite[Section 3.4]{Webcombined}.
\end{rem}

\begin{prop}
\label{idealispdg}
The two-sided ideal $(\nh_n^l)^{> \mu}$ is preserved by the differential.
\end{prop}

\begin{proof}
Let $ \lambda$ be a partition such that $ \lambda > \mu$.
Assume $ \mu^{j_1}=\cdots=\mu^{j_n}=1$ with $ j_1 < \cdots < j_n$ and
$ \lambda^{k_1}=\cdots=\lambda^{k_n}=1$ with $ k_1 < \cdots < k_n$.
Then $y^{\lambda}=y_1^{l-k_1} \cdots y_n^{l-k_n}$.  

By definition, we have
\begin{align*}
\partial(\psi_\mf{s}^* y_1^{l-k_1} \cdots y_n^{l-k_n} \psi_\mf{t})&=
\sum_{i=1}^n (l+1-k_i) \psi_\mf{s}^* y_i (y_1^{l-k_1} \cdots y_i^{l-k_i} \cdots y_n^{l-k_n}) \psi_\mf{t} +
\partial(\psi_\mf{s}^*) y^{\lambda} \psi_\mf{t} +
\psi_\mf{s}^* y^{\lambda} \partial(\psi_\mf{t}) \\
&= \sum_{i=1}^n (l+1-k_i) \psi_\mf{s}^* y_i (y^{\lambda}) \psi_\mf{t} +
\partial(\psi_\mf{s}^*) y^{\lambda} \psi_\mf{t} +
\psi_\mf{s}^* y^{\lambda} \partial(\psi_\mf{t}).
\end{align*}
Since $y^{\lambda}$ is in the two-sided ideal $ (\nh_n^l)^{> \mu} $, it is clear that each term in the summation above lies in $ (\nh_n^l)^{> \mu} $. This finishes the proof.
\end{proof}

\subsection{Specht modules}
The modules over the cyclotomic nilHecke algebra which we are about to define are known as Specht modules.  These modules have been considered for the classical cyclotomic Hecke algebras in earlier literature.  Their graded lifts (as modules over graded KLR algebras) have been constructed by Brundan, Kleshchev, and Wang~\cite{BKW}.  Here we once again specialize their general KLR construction to nilHecke case corresponding to $\mf{sl}_2$-categorification.

Let $\mu \in \mathcal{P}_n^l$ and $ \mf{t} \in \mathrm{Tab}(\mu)$.  
One defines the right Specht module $S^{\mu}_{\mf{t}}$ in terms of the cellular structure 
on $\nh_n^l$ defined in the previous section.
Let $S_\mf{t}^{\mu}$ be the submodule of $ \nh_n^l / (\nh_n^l)^{> \mu} $ generated by the coset $\psi_{\mf{t}e}^{\mu} + (\nh_n^l)^{> \mu}$.
For the ease of notation, we will usually write $e$ for $\mf{t}^{\mu}$, so that, for instance, $\psi_{\mf{t}^\mu\mf{t}^\mu}^\mu$ will be denoted $\psi^\mu_{ee}$ as well.

\begin{prop}
\label{spechtispdg}
The Specht module $S_e^{\mu}$ is a right $p$-DG module, where $e$ the identity tableau.
\end{prop}

\begin{proof}
It suffices to show that $\partial(\psi_{ee}^{\mu}) \in (\nh_n^l)^{> \mu}$, and the proposition then follows from Proposition \ref{idealispdg}.

Note that $\psi_{ee}^{\mu}=y^{\mu}$.  Then 
$ \partial(y^{\mu})=\sum_i c_i y^{\mu} y_i$ for some coefficients $c_i$.
Thus $ \partial(y^{\mu}) \in (\nh_n^l)^{\geq \mu} $.
For a fixed $i$, if the exponents of $y^{\mu} y_i$ are strictly decreasing, then 
$y^{\mu} y_i$ is clearly in $(\nh_n^l)^{> \mu}$.  If the exponents are not strictly decreasing, then we may rewrite $y^{\mu} y_i$ in the cellular basis as a sum where the monomial part in each term has strictly decreasing exponents.  The monomial part of each term must have degree at least two more than the degree of $y^{\mu}$.
Since $\partial(y^{\mu}) \in (\nh_n^l)^{\geq \mu}$, for degree reasons it must actually be in $(\nh_n^l)^{> \mu}$.
\end{proof}

\begin{rem}
To see that the above proposition is not true for an arbitrary tableaux consider the case $l=n=2$.
Then there is only one partition; call it $\mu$.  There are two tableaux: the identity $e$ and the transposition $s_1$.
Then $S^{\mu}_e$ has basis $\{ y_1, y_1 \psi_1 \}$ and as in the proposition is obviously a $p$-DG module.
However $S^{\mu}_{s_1}$ has a basis $\{ \psi_1 y_1, \psi_1 y_1 \psi_1 \} $ and is clearly not stable under $\partial$.
\end{rem}

The module $S_\mf{t}^{\mu} $ has a basis
\begin{equation*}
\{ \psi_{\mf{ts}}^{\mu} + (\nh_n^l)^{> \mu}  | \mf{s} \in \mathrm{Tab}(\mu) \}.
\end{equation*}

\begin{prop}
\label{spechtsareiso}
For any $ \mf{t} \in \mathrm{Tab}(\mu)$ there is an isomorphism of $\nh_n^l$-modules 
$S^{\mu}_\mf{t} \cong S^{\mu}_e \langle -2\ell(w_\mf{t}) \rangle$.
\end{prop}

\begin{proof}
This is obvious.  There is an isomorphism that sends an element $\psi_{\mf{ts}}$ to $ \psi_{e\mf{s}}$.  
\end{proof}

\begin{rem}
Proposition \ref{spechtsareiso} cannot be a statement of $p$-DG modules since there is only one preferred representative in the isomorphism class which actually even carries a $p$-DG structure.
\end{rem}

Since $S^{\mu}_e$ is the preferred Specht module which carries a $p$-DG structure, we write $S^{\mu} := S^{\mu}_{e} $.
We now give another realization of this Specht module.
Let $ v^{\mu} $ be a formal basis vector spanning a right $\Bbbk[y_1,\ldots,y_n]$-module where $v^{\mu}y_i=0$ for $i=1,\ldots,n$.  We endow it with a $p$-DG structure by setting $\partial(v^{\mu})=0$.
Define the module
\begin{equation*}
\widetilde{S}^{\mu} = \Bbbk v^{\mu} \otimes_{\Bbbk[y_1,\ldots,y_n]} \nh_n^l.
\end{equation*}

\begin{prop}
\label{spechtinducediso}
The $\nh_n^l$-module $\widetilde{S}^{\mu}$ is isomorphic to the Specht module $S^{\mu}$.
Furthermore, this is an isomorphism of $p$-DG modules.
\end{prop}

\begin{proof}
The isomorphism $\widetilde{S}^{\mu} \cong S^{\mu}$ maps
the generator $v^{\mu} \in \widetilde{S}^{\mu} $ to the generator 
$\psi_{ee}^{\mu} + (\nh_n^l)^{> \mu}$ of $S^{\mu}$.
By definition $\partial(v^{\mu})=0$ and by Proposition \ref{spechtispdg},
$\partial(\psi_{ee}^{\mu} + (\nh_n^l)^{> \mu})=0$.  Thus it is an isomorphism of $p$-DG modules.
\end{proof}

\begin{prop}\label{prop-unique-Specht}
The Specht module $S^{\mu}$ is irreducible. In particular, up to isomorphism and grading shifts, the isomorphism class of the $p$-DG module $S^\mu$ is independent of the partition $\mu\in \mc{P}_n^l$.
\end{prop}

\begin{proof}
By Proposition \ref{spechtinducediso}, $S^{\mu} \cong \Bbbk v^{\mu} \otimes_{\Bbbk[y_1,\ldots,y_n]} \nh_n^l $ where $\Bbbk v^{\mu}$ is the trivial module over the polynomial algebra $\Bbbk[y_1,\ldots,y_n]$.
This shows that the Specht module has a basis
$ \{ v^{\mu} \otimes \psi_w | w \in \mf{S}_n \}$.

It is now straightforward to show directly that the Specht module is irreducible.  
Indirectly, the cyclotomic nilHecke algebra is isomorphic to an $n! \times n!$ matrix algebra over $\mathrm{H}^*(\gr(n,l))$.  Thus it has a unique simple module of dimension $n!$.
Since the basis for the Specht module given above has $n!$ elements, it follows that $S^{\mu}$ must be irreducible.
\end{proof}

\subsection{The modules \texorpdfstring{$G(\lambda)$}{G(lambda)}}

Once again, fix a partition $\lambda$ with $ \lambda^{j_1}=\cdots=\lambda^{j_n}=1$ and $ j_1 < \cdots < j_n$.
Recall that
\begin{equation*}
y^{\lambda} = y_1^{l-j_1} \cdots y_n^{l-j_n}.
\end{equation*}

As in \cite{HuMathas} we define for each partition a cyclic module.
\begin{defn}\label{def-G-lambda}
Let $\lambda\in \mc{P}_n^l$ be a partition. Define the module
\begin{equation*}
G(\lambda):=q^{ -nl+ j_1 + \cdots + j_n }y^{\lambda} \nh_n^l.
\end{equation*}
\end{defn}

The next definition is used for describing a basis of $G(\lambda)$.

\begin{defn}
Let $\lambda$ and $\mu$ be two partitions.  We define a subset of $\mu$-tableaux by
\begin{equation*}
\mathrm{Tab}^{\lambda}(\mu) := \{ \mf{t} \in \mathrm{Tab}(\mu) | \mf{t} \geq \mf{t}^{\lambda} \}.
\end{equation*}
\end{defn}

\begin{prop}
\label{basisG}
The module $G(\lambda)$ has a basis
\begin{equation*}
\{ \psi_{\mf{ts}} | \mf{t} \in \mathrm{Tab}^{\lambda}(\mu), \mf{s} \in \mathrm{Tab}(\mu), \mu \in \mathcal{P}_n^l \}.
\end{equation*}
\end{prop}
\begin{proof}
See \cite[Theorem 4.9]{HuMathas}.
\end{proof}
Since $ \partial(y^{\lambda})=y^{\lambda}a$ for some $ a \in \Bbbk[y_1,\ldots,y_n]$,
the module $G(\lambda)$ is naturally a right $p$-DG module.

\begin{cor}\label{cor-Spect-equals-G}
If $\lambda=(1^n 0^{l-n})$, then there is an isomorphism of $p$-DG modules
$$G(\lambda) \cong q^{-nl+j_1+\cdots+j_n} S^{\lambda}. $$  
\end{cor}
\begin{proof}
By Proposition \ref{basisG}, since $\lambda$ is the maximal partition in the dominance order, the $p$-DG module $G(\lambda)$ has a basis labeled by $\mathrm{Tab}(\lambda)$. Thus it has dimension $n!$ and must be the Specht module up to grading shift. The claim now follows from the uniqueness of the $p$-DG structure on Specht modules (Proposition \ref{prop-unique-Specht}).
\end{proof}

\begin{prop}
\label{spechtfiltofG}
The module $G(\lambda)$ has a filtration whose subsequent quotients are isomorphic to Specht modules.
More specifically choose a partial order on $ \cup_{\mu} \mathrm{Tab}^{\lambda}(\mu) = \{ \mf{t}_1, \ldots, \mf{t}_m   \} $ such that $ \mf{t}_i \geq \mf{t}_j$ implies that $i \leq j$.  Suppose the tableau $\mf{t}_i$ corresponds to a partition $\nu_i$.  Then $ G(\lambda)$ has a filtration
\begin{equation*}
G(\lambda)=G_m \supset G_{m-1} \supset \cdots \supset G_0=0
\end{equation*}
such that 
\begin{equation*}
G_i/G_{i-1} \cong q^{\ell(w_i) -nl+(j_1+\cdots+j_n)}S^{\nu_i}
\end{equation*}
where $w_i $ is the permutation in $\mf{S}_n$ associated to the tableau $\mf{t}_i$.
\end{prop}
\begin{proof}
This is \cite[Corollary 4.11]{HuMathas}.
\end{proof}

\begin{prop}
The Specht filtration of $G(\lambda)$ in Proposition \ref{spechtfiltofG} is a filtration of $p$-DG modules.
\end{prop}

\begin{proof}
Following \cite[Corollary 4.11]{HuMathas}, let $ \cup_{\mu} \mathrm{Tab}^{\lambda}(\mu) = \{ \mf{t}_1, \ldots, \mf{t}_m   \} $.  Suppose that 
$ \mf{t}_r \in \mathrm{Tab}(\nu_r)$.  

Define 
\begin{equation*}
G_i = \Bbbk \{ \psi_{w_{j}}^* y^{\nu_j} \psi_{w_\mf{t}} | j \leq i ,~ \mf{t}_j \in \mathrm{Tab}(\nu_j) \}
\end{equation*}
where we recall from Definition \ref{def-standard-tableaux} that to a tableau $\mf{t}_j$ we associate an element $w_{j}:=w_{\mf{t}_j} \in \mf{S}_n$.

Now we compute the derivation on a basis vector $\psi_{w_{j}}^* y^{\nu_j} \psi_{w_\mf{t}}$ from above:
\begin{equation}
\label{deronbasisvecGi}
\partial(\psi_{w_{j}}^* y^{\nu_j} \psi_{w_\mf{t}}) = \partial(\psi_{w_{j}}^*) y^{\nu_j} \psi_{w_\mf{t}} + \psi_{w_{j}}^* \partial(y^{\nu_j}) \psi_{w_\mf{t}} + \psi_{w_{j}}^* y^{\nu_j} \partial(\psi_{w_\mf{t}}).
\end{equation}
Now we analyze each term in \eqref{deronbasisvecGi}.

First, note that $ \psi_{w_{j}}^* y^{\nu_j} \in G_i$.  By \cite[Corollary 4.11]{HuMathas}, $G_i$ is a right $\nh_n^l$-module.
Since $\partial(\psi_{w_\mf{t}}) \in \nh_n^l$ it follows that $\psi_{w_{j}}^* y^{\nu_j} \partial(\psi_{w_\mf{t}}) \in G_i$.

Next, 
\begin{equation*}
\psi_{w_{j}}^* \partial(y^{\nu_j}) \psi_{w_\mf{t}} = \sum_k  d_k \psi_{w_{j}}^* y^{\nu_j} y_k \psi_{w_\mf{t}}
\end{equation*}
for some constants $d_k\in \F_p$.
Once again, since $\psi_{w_{j}}^* y^{\nu_j} \in G_i$ and $ y_k \psi_{w_\mf{t}} \in \nh_n^l$
it follows that 
\begin{equation*}
\sum_k  d_k \psi_{w_{j}}^* y^{\nu_j} y_k \psi_{w_\mf{t}} \in G_i.
\end{equation*}

Finally we consider $\partial(\psi_{w_{j}}^*)$.
Let $ \psi_{w_{j}}^*=\psi_{i_1} \cdots \psi_{i_r}$. 
Then it is easy to see that
\begin{equation*}
\partial(\psi_{w_j}^*)=\sum_m c_m \psi_{i_1} \cdots \hat{\psi}_{i_m} \cdots \psi_{i_r} +
\sum_m d_m \psi_{w_j}^* y_m
\end{equation*}
for constants $c_m$ and $d_m$, where $\hat{\psi}_{i_m}$ means omitting the factor $\psi_{i_m}$.
If the expression $ \psi_{i_1} \cdots \hat{\psi}_{i_m} \cdots \psi_{i_r} $ is non-zero,
then $s_{i_1} \cdots \hat{s}_{i_m} \cdots s_{i_r}$ is smaller than $s_{i_1} \cdots s_{i_r}$ in the Bruhat order of $\mf{S}_n$.  It follows that the tableau corresponding to
$s_{i_1} \cdots \hat{s}_{i_m} \cdots s_{i_r}$ is greater than the tableau corresponding to $s_{i_1} \cdots s_{i_r}$ (see \cite[Theorem 3.8]{Mathas}).
Thus
\begin{equation*}
\sum_m c_m \psi_{i_1} \cdots \hat{\psi}_{i_m} \cdots \psi_{i_r} y^{\nu_j} \in G_i.
\end{equation*}
By the arguments from earlier it follows that 
\begin{equation*}
\sum_m d_m \psi_{w_j}^* y^{\nu_j} y_m \in G_i.
\end{equation*}
This finishes the proof of the proposition.
\end{proof}

\begin{prop}\label{prop-G-self-dual}
There is a non-degenerate bilinear form on $G(\lambda)$. In particular, the graded dual of $G(\lambda)$ is isomorphic to itself.
\end{prop}

\begin{proof}
This is proved in \cite[Theorem 4.14]{HuMathas}.  It utilizes the Frobenius structure of $\nh_n^l$.
\end{proof}


\begin{defn}\label{def-partition-idemp}
For $\mu=(0^{a_1},1^{b_1},\dots ,0^{a_k},1^{b_k})\in \mc{P}_n^l$, we associate with it the sequence ${\bf b}:=(b_1,\dots, b_k)$ and define the idempotent $e_\mu\in \nh_n^l$ by (see equation \eqref{eqn-sequence-idempotent})
$$
e_\mu:=e_{\bf b}=
\begin{DGCpicture}
\DGCstrand[Green](0,0)(0,1)[$^{b_1}$]
\DGCstrand[Green](0.5,0)(0.5,1)[$^{b_2}$]
\DGCstrand[Green](1.5,0)(1.5,1)[$^{b_k}$]
\DGCcoupon*(0.6,0.1)(1.4,0.9){$\cdots$}
\end{DGCpicture} 
\ .
$$
\end{defn}

\begin{prop}
\label{genYindecomp}
Suppose $\lambda=(0^{a_1} 1^{b_1} \ldots 0^{a_r} 1^{b_r})\in \mc{P}_n^l$, so that $ \sum_{i=1}^r (a_i+b_i) =l$ and $\sum_{i=1}^r b_i =n$.
Then the right module $e_{\lambda} y^{\lambda} \nh_n^l$ is a $p$-DG submodule of $G(\lambda)$.
\end{prop}

\begin{proof}
For notational convenience set $a_0=b_0=0$.  By definition
\begin{align*}
y^{\lambda}&=\prod_{i=0}^{r-1} y_{b_0+\cdots+b_i+1}^{l-(a_1+b_1+\cdots+a_i+b_i+a_{i+1}+1)}
\cdots
y_{b_0+\cdots+b_{i+1}}^{l-(a_1+b_1+\cdots+a_{i+1}+b_{i+1})} 
=y^\tau y^{\gamma} ,
\end{align*}
where
$$y^{\gamma}=\prod_{i=0}^{r-1}y_{b_0+\cdots+b_i+1}^{b_{i+1}-1} \cdots  y_{b_0+\cdots+b_{i+1}}^0, \quad \quad \quad
y^\tau =\prod_{i=0}^{r-1} \left( y_{b_0+\cdots+b_i+1}
\cdots
y_{b_0+\cdots+b_{i+1}} \right)^{l-(a_1+b_1+\cdots+a_{i+1}+b_{i+1})}.
$$
Observe that the element $y^\tau$ lies in the center of $\nh_{b_1}\times \cdots \times \nh_{b_r}$.

The idempotent $e_{\lambda}$ may be written as
\begin{equation*}
e_{\lambda}=\prod_{i=0}^{r-1} (y_{b_0+\cdots+b_i+1}^{b_{i+1}-1} \cdots y_{b_0+\cdots+b_{i+1}}^0) \psi_w  =y^\gamma \psi_w,   
\end{equation*}
where $w \in \mf{S}_{b_1} \times \cdots \times \mf{S}_{b_r} $ is the longest element in the parabolic subgroup. 

Now we are able to write
\begin{equation*}
e_{\lambda} y^{\lambda}=y^\gamma \psi_w y^{\gamma} y^\tau =y^\gamma y^\tau \psi_w y^{\gamma} = y^{\lambda} \psi_w y^\gamma.
\end{equation*}
Thus $e_{\lambda} y^{\lambda} \nh_n^l$ is a submodule of $G(\lambda)$.

By \eqref{partialofidempotent}, $e_{\lambda}$ is a $\partial$-stable idempotent, i.e., $\partial(e_{\lambda})=e_{\lambda} x$ for $x$ some linear polynomial element in $y$'s.  Thus $e_{\lambda} y^{\lambda}$ generates a $p$-DG ideal.  Therefore 
$e_{\lambda} y^{\lambda} \nh_n^l$ is a $p$-DG submodule of $G(\lambda)$.
\end{proof}

\section{Two-tensor quiver Schur algebra}
\label{sec-2-tensor}
\subsection{Quiver Schur algebra}
In this section we recall the definition of ($p$-DG) quiver Schur algebra for $\mf{sl}_2$ in the sense of Hu and Mathas~\cite{HuMathas}. Although it will only play an auxiliary role for the current work, it will be studied in more detail in a sequel to this paper.
\begin{defn}\label{def-Schur-algebra}
The \emph{(graded) quiver Schur algebra} is by definition
\begin{equation}
S_n(l):=\END_{\nh_n^l}\left(\bigoplus_{\lambda\in \mc{P}_n^l} G(\lambda)\right).
\end{equation}
The algebra $S_n(l)$ inherits a $p$-DG structure from the $p$-DG structures on $\nh_n^l$ and the $G(\lambda)$.  If $ f \in S_n(l)$ then set
\begin{equation*}
(\partial f)(x)=\partial(f(x))-f(\partial x).
\end{equation*}
\end{defn}

We now recall a basis of this algebra from \cite[Section 4.2]{HuMathas}.
For $\mf{t} \in \mathrm{Tab}^{\mu}(\lambda), \mf{s} \in \mathrm{Tab}^{\nu}(\lambda)$, there is a map
\begin{equation*}
\Psi_{\mf{ts}}^{\mu \nu} \colon G(\nu) \longrightarrow G(\mu) 
\end{equation*}
defined by
\begin{equation*}
\Psi_{\mf{ts}}^{\mu \nu}(y^{\nu})=\psi^{\lambda}_{\mf{ts}}.
\end{equation*}
Note that there is a hidden dependence of $\Psi_{\mf{ts}}^{\mu \nu}$ on $\lambda$ because $\mf{t}$ and $\mf{s}$ are $\lambda$-tableaux for some $\lambda$.  When we want to stress the dependence on $\lambda$ we write
$\Psi^{\mu \nu}_{\mf{ts} \lambda}$ instead.

\begin{defn}
\begin{enumerate}
\item[(1)] Let $S_n^{> \lambda}(l)$ be the $\Bbbk$-vector space spanned by
$\{ \Psi_{\mf{ts}}^{\mu \nu} \}$ for $\mf{t} \in \Tab^{\mu}(\gamma)$ and $\mf{s} \in \Tab^{\nu}(\gamma)$ with $\gamma > \lambda$.
\item[(2)] Let $S_n^{\geq \lambda}(l)$ be the $\Bbbk$-vector space spanned by
$\{ \Psi_{\mf{ts}}^{\mu \nu} \}$ for $\mf{t} \in \Tab^{\mu}(\gamma)$ and $\mf{s} \in \Tab^{\nu}(\gamma)$ with $\gamma \geq \lambda$.
\end{enumerate}
Here the order on partitions is the dominance order (Definition \ref{def-order-on-nh-partition}).
\end{defn}

\begin{prop}
\label{cellularschuralgprop}
The algebra $S_n(l)$ is a graded cellular algebra with cellular basis
\begin{equation*}
\left\{ \Psi_{\mf{ts}}^{\mu \nu} \big| \mf{t} \in \Tab^{\mu}(\lambda),~ \mf{s} \in \Tab^{\nu}(\lambda),~ \lambda \in \mathcal{P}_n^l \right\}.
\end{equation*}
\end{prop}
\begin{proof}
This is \cite[Theorem 4.19]{HuMathas}.
\end{proof}

The modules $G(\lambda)$ are generally decomposable.
For each $\lambda \in \mathcal{P}_n^l$ there exists a unique module $Y(\lambda)$ such that $Y(\lambda)$ is a summand of $G(\lambda)$ of multiplicity one and does not appear in a decomposition of $G(\mu)$ for any $\mu > \lambda$.  It is difficult to explicitly construct each $Y(\lambda)$.  As a consequence it is not clear that $Y(\lambda)$ is a $p$-DG module.  This prevents us from understanding the Grothendieck group of compact $p$-DG $S_n(l)$-modules in a straightforward way.

\begin{example}
\label{n=1generall}
When $n=1$, the quiver Schur algebra $S_1(l)$ is isomorphic to $A_l^!$ which is the algebra Koszul dual to the zigzag algebra.  This algebra has been studied with its $p$-DG structure in \cite{QiSussan}.  It is shown that there is a braid group action on the derived category of compact modules.

Let $l$ be a natural number greater than or equal to two, and $Q_l$ be the following quiver:
\begin{equation}\label{quiver-Q}
\xymatrix{
 \overset{1}{\circ} &
 \cdots \ltwocell{'}&
 \overset{i-1}{~\circ~}\ltwocell{'}&
 \overset{i}{\circ}\ltwocell{'}&
 \overset{i+1}{~\circ~}\ltwocell{'}&
 \cdots \ltwocell{'}&
 \overset{l}{\circ}\ltwocell{'}
 }
\end{equation}

Let $\Bbbk Q_l$ be the path algebra associated to $Q_l$ over the ground field. We use, for instance, the symbol $(i|j|k)$, where $i,j,k$ are vertices of the quiver $Q_l$, to denote the path which starts at a vertex $i$, then goes through $j$ (necessarily $j=i\pm 1$) and ends at $k$.
The composition of paths is given by
\begin{equation}
(i_i|i_2|\cdots|i_r)\cdot (j_1|j_2|\cdots|j_s)=
\left\{
\begin{array}{ll}
(i_i|i_2|\cdots|i_r|j_2|\cdots|j_s) & \textrm{if $i_r=j_1$,}\\
0 & \textrm{otherwise,}
\end{array}
\right.
\end{equation}
where $i_1,\dots, i_r$ and $j_1, \dots, j_s$ are sequences of neighboring vertices in $Q_l$.

The algebra $A_l^!$ is the quotient of the path algebra $\Bbbk Q_l$ by the relations
\[
(i|i-1|i)=(i|i+1|i)~~ (i=2,\dots, l-1),\quad \quad  (1|2|1)=0.
\]

The identification of $A_l^!$ with $S_1(l)$ is obtained by associating $(i+1|i)$ to the morphism from $G(0^{i-1} 1 0^{l-i})$ to $ G(0^i 1 0^{l-i-1})$ given by mapping $y_1^{l-i} $ to $y_1^{l-i}$.
The element $(i|i+1)$ corresponds to the morphism 
from $ G(0^i 1 0^{l-i-1})$ to $G(0^{i-1} 1 0^{l-i})$ given by mapping $y_1^{l-i-1}$ to $y_1^{l-i}$.

This realization of $A_l^!$ endows the algebra with a $p$-DG structure.
On $\nh_1^l\cong \Bbbk[y_1]/(y_1^l)$ the differential $ \partial $ is given by $\partial(y_1)=y_1^2$.  It is clear that each $G(0^{i-1} 1 0^{l-i})$ is a $p$-DG submodule of $\nh_1^l$.
This gives $A_l^!$ the structure of a $p$-DG algebra where the differential is given by
\[
\partial(i|i+1)=(i|i+1|i|i+1), \quad \quad \partial(i|i-1)=0.
\]
\end{example}

\begin{example}
\label{n=2l=3}
The quiver Schur algebra $S_2(3)$ is the first example of a quiver Schur algebra which is not basic.
There are three partitions in this case.
Let 
\begin{equation*}
\lambda=(~\yng(1)~,~\yng(1)~,~\emptyset~) \hspace{.2in} \mu=(~\yng(1)~,~\emptyset~,~\yng(1)~) \hspace{.2in} \zeta=(~\emptyset~,~\yng(1)~,~\yng(1)~).
\end{equation*}
Then we have
\begin{equation*}
G(\lambda)=y_1^2 y_2 \nh_2^3.
\end{equation*}
It is spanned by
\begin{equation*}
\lbrace \psi^{\lambda}_{e,e} =y_1^2 y_2,  \psi^{\lambda}_{e,s_1}= y_1^2 y_2 \psi_1 \rbrace.
\end{equation*}
It is clear that (up to a shift)  $G(\lambda) \cong S^{\lambda}$, and so it is simple implying that $Y(\lambda)=G(\lambda)$.

Next consider
\begin{equation*}
G(\mu)=y_1^2 \nh_2^3.
\end{equation*}
It is spanned by
\begin{equation*}
\lbrace \psi^{\lambda}_{e,e}=y_1^2 y_2,  \psi^{\lambda}_{e,s_1}=y_1^2 y_2 \psi_1,
\psi^{\mu}_{e,e}=y_1^2,  \psi^{\mu}_{e,s_1}=y_1^2 \psi_1 \rbrace.
\end{equation*}
It is easy to check that the endomorphism algebra of $G(\mu)$ is non-negatively graded and its degree zero piece is one-dimensional.  Thus $Y(\mu)=G(\mu)$.

The remaining cyclic module is
\begin{equation*}
G(\zeta)=y_1 \nh_2^3.
\end{equation*}
It is spanned by
\begin{align*}
\lbrace 
&\psi^{\lambda}_{e,e}=y_1^2 y_2,  \psi^{\lambda}_{e,s_1}=y_1^2 y_2 \psi_1,
\psi^{\lambda}_{s_1,e}= \psi_1 y_1^2 y_2, \psi^{\lambda}_{s_1,s_1}= \psi_1 y_1^2 y_2 \psi_1, \\
&\psi^{\mu}_{e,e}=y_1^2,  \psi^{\mu}_{e,s_1}=y_1^2 \psi_1, 
\psi^{\zeta}_{e,e}=y_1, \psi^{\zeta}_{e,s_1}=y_1 \psi_1
\rbrace.
\end{align*}
There is a decomposition 
\begin{equation*}
G(\zeta) \cong e_2 G(\zeta) \oplus (1-e_2)G(\zeta).
\end{equation*}
It is readily checked that $e_2G(\zeta)\cong Y(\zeta)$ is the indecomposable projective module over $\nh_2^3$ and is closed under the $p$-differential action. While $e_2 G(\zeta)$ is a $p$-DG submodule, this decomposition does not respect $\partial$ so we merely have a short exact sequence of $p$-DG modules
\begin{equation*}
0\lra e_2G(\zeta) \longrightarrow G(\zeta) \longrightarrow (1-e_2)G(\zeta)\lra 0.
\end{equation*}
One can also identify $(1-e_2)G(\zeta)$ with the Specht module $G(\lambda)\cong Y(\lambda)$.

Due to the decomposability of $G(\zeta)$, the endomorphism algebras
$\END_{\nh_2^3}(G(\lambda) \oplus G(\mu) \oplus G(\zeta))$ and 
$\END_{\nh_2^3}(Y(\lambda) \oplus Y(\mu) \oplus Y(\zeta))$ are not isomorphic, but they are $p$-DG Morita equivalent.

It is straightforward to show that $\END_{\nh_2^3}(Y(\lambda) \oplus Y(\mu) \oplus Y(\zeta)) \cong A_3^!$ from Example \ref{n=1generall} where 
$(2|1)$ corresponds to the homomorphism $\Psi_{ee \lambda}^{\mu \lambda}$.
The element $(1|2)$ corresponds to the homomorphism
$\Psi_{ee \lambda}^{\lambda \mu}$.
The path $(3|2)$ corresponds to the composition of the homomorphism
${\Psi}_{ee \mu}^{\zeta \mu}$ with projection onto the summand $Y(\zeta)$.
Finally the element $(2|3)$ corresponds to inclusion of $Y(\zeta)$ into $G(\zeta)$ composed with $\Psi_{ee \mu}^{\mu \zeta}$.
\end{example}

In the next subsection we will consider a collection of $\lambda$'s where we have a full understanding of $Y(\lambda)$'s, leading in turn to a categorification of $V_r \otimes_{\mathbb{O}_p} V_s$ at a prime root of unity.

\subsection{A category of nilHecke modules}
In this section we fix $l,n\in \N$ and two other natural numbers $r,s$ such that $r+s=l$. We will abbreviate $\HOM_{\nh_n^l}$ simply as $\HOM$ for the ease of notation.

\begin{defn}
\label{speciallambda}
Let $\mathcal{P}_n^{r,s}$ be the subset of all partitions $\lambda \in
\mathcal{P}_n^{l}$ of the form $\lambda=(0^a 1^b 0^c 1^d)$ with 
\begin{equation*}
r+s=l, \quad a+b=r, \quad c+d=s, \quad b+d=n.
\end{equation*}
We will think of such a sequence as a partition
\[
(\underbrace{~\emptyset~,\dots, ~\emptyset~}_a,\underbrace{~\yng(1)~,\dots, ~\yng(1)~}_b|\underbrace{~\emptyset~,\dots, ~\emptyset~}_c,\underbrace{~\yng(1)~,\dots, ~\yng(1)~}_d).
\]
\end{defn}

The collection of partitions is a subset of $\mc{P}_n^l$, and thus inherits the partial order on $\mc{P}_n^l$ (Definition \ref{def-order-on-nh-partition}). Notice that the unique minimal element $\lambda_0\in \mc{P}_n^l$ always lies inside $\mc{P}_n^{r,s}$, and is also the minimal element here.
It is the element of $\mathcal{P}_n^{r,s}$ where the entries labeled by $1$ are as far to the right as possible. Recall that if $\lambda=(0^a 1^b 0^c 1^d)\in \mc{P}_n^{r,s}$, we have (Definitions \ref{def-y-mu-psi-mu} and \ref{def-partition-idemp})
\begin{equation}
y^\lambda=y_1^{b+c+d-1}y_2^{b+c+d-2}\cdots y_b^{c+d} y_{b+1}^{d-1}y_{b+2}^{d-2}\cdots y_{b+d}^0, \quad \quad
e_\lambda = e_{(b,d)}=
e_b\otimes e_d.
\end{equation}

\begin{defn}\label{def-rs-additive-subcat}
Let $ \mc{NH}_n^{r,s} $ be the filtered $p$-DG envelope (Definition \ref{def-filtered-envelope}) inside $(\nh_n^l ,\dif)\dmod$ generated by the collections of $p$-DG modules $e_{\lambda}G(\lambda)$ with $\lambda=(0^a 1^b 0^c 1^d)\in \mc{P}_n^{r,s}$.
\end{defn}

The definition is equivalent to the following more ($p$-DG) ring-theoretic description.

\begin{defn}\label{def-two-tensor-algebra}
Let $n,l$ be two natural numbers, and $r,s\in \N$ such that $r+s=l$. We define the \emph{two-tensor quiver Schur algebra}  to be
\begin{equation*}
S_n(r,s) := \END_{\nh_n^l}\left(\bigoplus_{\lambda \in \mathcal{P}_n^{r,s}}
e_{\lambda} G(\lambda)\right).
\end{equation*}
The algebras are equipped with $p$-DG structures coming from the $p$-differentials on the $e_{\lambda}G(\lambda)$'s. 
\end{defn}

\begin{rem}\label{rmk-two-tensor-algebra}
We collect some simple facts about the definitions here.
\begin{enumerate}
\item[(1)]It is clear, by construction, that the objects of $\mc{NH}_{n}^{r,s}$ consist of finitely-generated, cofibrant left $p$-DG modules over $S_n(r,s)$. We will thus use both descriptions for the convenience of the context.
\item[(2)]In the definition of the category $\mc{NH}_n^{r,s} $, it is crucial for categorification purposes that we take as a generating set of objects $e_{\lambda}G(\lambda)$ rather than just $G(\lambda)$, as we will see in Example \ref{l=3n=2}.
\item[(3)] The simple module over $\nh_n^l$ is isomorphic to $G(\lambda)$ where $\lambda \in \mathcal{P}_n^l$ is the partition for which all of the boxes are to the left and all the empty slots are to the right.
One can readily check that the category $\mc{NH}_{n}^{r,s} $ does not contain the unique irreducible $(\nh_n^l,\dif)$-module unless $r=n$, or $n=0$, or $n=l$.
\end{enumerate}
\end{rem}

Our next goal is to explicitly construct the generating indecomposable objects of $\mc{NH}_{n}^{r,s} $. We will mostly ignore the differentials and study the indecomposability as modules over $\nh_n^l$ in this subsection unless otherwise specified. The $p$-DG structure and the compatibility with the $\mf{E}$, $\mf{F}$-actions will be studied in the next subsection.

\begin{lem}
\label{Yindecomspecial}
Suppose $d=0$ and let $\lambda = (0^a 1^b 0^c)$.
Then 
\begin{equation*}
e_{\lambda}G(\lambda)=e_b y^{\lambda} \nh_b^{a+b+c} = e_b y_1^{b+c-1} \cdots y_b^{c} \nh_b^{a+b+c}.
\end{equation*}
is an indecomposable summand of $G(\lambda)$.
Furthermore 
\begin{equation*}
Y(\lambda)\cong e_\lambda G(\lambda)
\end{equation*}
\end{lem}

\begin{proof}
It is easy to see that $ e_b y^{\lambda} \nh_b^{a+b+c} \subset G(\lambda)$ 
since $e_b y^{\lambda}=y^{\lambda} e_b^*$.

We next show that $e_b y^{\lambda} \nh_b^{a+b+c}$ is a summand of $G(\lambda)$.  
Clearly
\begin{equation}
\label{Glambdadecopm}
G(\lambda)=e_b y^{\lambda} \nh_b^{a+b+c} + (1-e_b) y^{\lambda} \nh_b^{a+b+c}.
\end{equation}
Note that $e_b y^{\lambda} \nh_b^{a+b+c} \subset e_b \nh_b^{a+b+c}$
and $(1-e_b) y^{\lambda} \nh_b^{a+b+c} \subset (1-e_b) \nh_b^{a+b+c}$.
Since 
\begin{equation*}
e_b \nh_b^{a+b+c} \cap (1-e_b) \nh_b^{a+b+c} = 0,
\end{equation*} 
the decomposition in \eqref{Glambdadecopm} is direct.

Let $ \mu \geq \lambda$ and $ \mf{s}, \mf{t} \in \mathrm{Tab}(\mu)$.
In order for the map $ \Psi := \Psi^{\lambda \lambda}_{\mf{st} \mu} $ to have a non-positive degree, either $\mf{s}$ or $\mf{t}$ cannot be the identity permutation.

Define $ \tilde{\Psi} \colon e_b y^{\lambda} \nh_b^{a+b+c} \longrightarrow e_b y^{\lambda} \nh_b^{a+b+c}$ by
\begin{align*}
\tilde{\Psi}(e_b y^{\lambda}) &= e_b \Psi(e_b y^{\lambda}) \\
&= y_1^{b-1} \cdots y_b^0 \psi_{w_0} \Psi(y^{\lambda} e_b^*) \\
&= y_1^{b-1} \cdots y_b^0 \psi_{w_0} \psi_{\mf{s}}^* y^{\mu} \psi_{\mf{t}} \psi_{w_0} y_1^{b-1} \cdots y_b^0.
\end{align*}
If $\mf{s}$ and $\mf{t}$ are both not the identity permutations then the above quantity is zero so there are no negative degree maps.
Clearly if $\mf{s}$ and $\mf{t}$ are both the identity permuations then $ \tilde{\Psi}$ is simply the identity map on $ e_b y^{\lambda} \nh_b^{a+b+c}$.
Thus there are no negative degree endomorphisms of $e_b y^{\lambda} \nh_b^{a+b+c}$ and the only degree zero map is the identity.
Therefore $e_b y^{\lambda} \nh_b^{a+b+c}$ is indecomposable.

Now we show that $e_b y^{\lambda} \nh_b^{a+b+c}=Y(\lambda)$.
Let us assume that $e_b y^{\lambda} \nh_b^{a+b+c} \subset G(\mu)$ for some $\mu\in \mc{P}_n^l$.
Clearly $e_b y^{\lambda} \psi_{w_0} \in e_b y^{\lambda} \nh_b^{a+b+c} $.  Now we calculate that
\begin{equation*}
e_b y^{\lambda} \psi_{w_0} = y_1^{b+c-1} \cdots y_b^{c} \psi_{w_0}.
\end{equation*}
We know from the basis of $G(\mu)$ (Proposition \ref{basisG}), that this basis element is not in $G(\mu)$ for $\mu > \lambda$.
Thus $e_b y^{\lambda} \nh_b^{a+b+c}$ is not contained in any $G(\mu)$ for $\mu > \lambda$.
\end{proof}

\begin{prop}
\label{abc1cgeqb}
Let $ \lambda=(0^a 1^b 0^c 1^d)$ with $ c \geq b$.
Then $ e_{\lambda} G(\lambda)$ is an indecomposable submodule of $G(\lambda)$.
\end{prop}
\begin{proof}
The case $d=0$ follows from Lemma \ref{Yindecomspecial}, so assume $d \geq 1$.

First we consider the case $d=1$.
Let $ Y'=  e_{(b,1)} y^{\lambda} \nh_{b+1}^{a+b+c+1} $.
A routine calculation shows $e_{(b,1)} y^{\lambda} = y^{\lambda} {e}^*_{(b,1)} $
(recall that the star of an idempotent means to reflect its graphical depiction across a horizontal axis).  This shows $Y'$ is a submodule of $G(\lambda)$ and in fact a summand since $ (1-e_{(b,1)}) y^{\lambda} \nh_{b+1}^{a+b+c+1} $ is a complementary summand.

Consider an endomorphism $\Psi := \Psi^{\lambda \lambda}_{\mf{t},\mf{s}, \mu}$ of $G(\lambda)$.
Let $\tilde{\Psi}$ be the restriction of $\Psi$ to $Y'$ followed by projection onto $Y'$.  That is
$ \tilde{\Psi}(e_{(b,1)} y^{\lambda})= e_{(b,1)} \Psi(e_{(b,1)} y^{\lambda})$.

Write $\mf{t}=t't''$  with $t'$ a shortest coset representative in $\mf{S}_{b+1}/\mf{S}_b \times \mf{S}_1$ and $t'' \in \mf{S}_b \times \mf{S}_1$.  Write $\mf{s}=s's''$ in a similar manner.

One then calculates
\begin{equation*}
\tilde{\Psi}(e_{(b,1)} y^{\lambda})
= e_{(b,1)} \Psi(e_{(b,1)} y^{\lambda})
= e_{(b,1)} \psi_{t''} \psi_{t'} y^{\mu} \psi_{s'} \psi_{s''} {e}^*_{(b,1)}.
\end{equation*}
By analyzing the form of the idempotent $e_{(b,1)}$ it is clear that this map is zero if $t''$ or $s''$ is not the identity permutation.  Thus assume $\mf{t}=t'$ and $\mf{s}=s'$.  Therefore
\begin{equation*}
\mf{t}, \mf{s} \in \{e, s_b, s_{b-1} s_b, \ldots, s_1 \cdots s_b \}.
\end{equation*}
Suppose $\mf{s}=s_i \cdots s_b$.  The filling $\mf{s}$ of $\mu$ has entries $1, \ldots, i-1, b+1, i, \ldots b$, read from left to right.
This filling must be greater than or equal to the standard filling of $\lambda$ given in \eqref{stdfillingoflambda}.

\begin{equation}
\label{stdfillingoflambda}
(\underbrace{~\emptyset~,\dots, ~\emptyset~}_a,\underbrace{~\Yboxdim{14.5pt}\young(1)~,\dots, ~\Yboxdim{14.5pt}\young(b)~}_b|\underbrace{~\emptyset~,\dots, ~\emptyset~}_c,~
\Yboxdim{14.5pt}\young(\B)
~).
\end{equation}
The tableau $\mf{s}$ of $\mu$ must have
boxes labeled $1,2, \ldots, i-1$ each moved at least one to the left from their original positions as prescribed by the standard filling of $\lambda$ since $b+1$ will be moved to the left of the bar in \eqref{stdfillingoflambda}.  
Note that if $i=1$ nothing needs to be moved.  
Now one calculates
\begin{align*}
deg(y^{\mu})-deg(y^{\lambda}) & \geq 2(i-1)+2(c+b-i+2) \\
&= 2c+2b+2 \\
& \geq 4b+2.
\end{align*}
The minimal degree of $\psi_{\mf{t}} $ or $\psi_{\mf{s}}$ is $-2b$.  
Thus the degree of $ \tilde{\Psi}$ is greater than or equal to two.
Thus it is impossible to decompose $Y'$.

The proof for $d >1$  is similar.  We provide some of the details.

Let $ Y'=  e_{(b,d)} y^{\lambda} \nh_{b+d}^{a+b+c+d} $.
Just as in the proof of Lemma \ref{abc1cgeqb}, let $\Psi = \Psi^{\lambda \lambda}_{\mf{t},\mf{s},\mu}$ and $\tilde{\Psi}$ be the restriction of $\Psi$ to $Y'$ composed with projection onto $Y'$.

Just as in the case $d=1$ we may assume $\mf{t},\mf{s}$ are shortest length coset representatives in $\mf{S}_{b+d}/\mf{S}_b \times \mf{S}_d$.

Consider $\mf{t}$ for example.  We encode it by the tuple $(r_1, \ldots, r_d)$ with
$1 \leq r_1 < \cdots < r_d \leq b+d$.
This tuple tells us how to shuffle the $d$ boxes into the entire collection of $b+d$ boxes.

Suppose $r_j=i+1<b$ and $r_{j+1} \geq b$.  This means for the partition $\mu$ that we have to push each of the first $i$ boxes in the first collection over at least one spot each.  We have to repeat this $j$ times.
In the second collection we have to move each of the first $j$ boxes to the left at least
$b+c+j-i$ spots.  Thus
\begin{equation*}
deg (y^{\mu}) - deg (y^{\lambda})  \geq 2ij +2j(b+c+j-i) \geq 2j(2b+j).
\end{equation*}
The minimal degree of $\psi_\mf{t}$ or $\psi_\mf{s}$ is $-2jb$.  Thus
\begin{equation*}
deg(\tilde{\Psi}) \geq 2j(2b+j)-4jb=2jb.
\end{equation*}
Since $j \geq 1$, $\tilde{\Psi}$ is a positive degree endomorphism.
Thus $Y'$ is indecomposable.

By \eqref{partialofidempotent}, $\partial(e_{\lambda} y^{\lambda})=e_{\lambda} y^{\lambda} x$ for some linear polynomial $x \in \nh_n^l$, so $e_{\lambda} G(\lambda)$ is a $p$-DG submodule of $G(\lambda)$.
\end{proof}

\begin{prop}
\label{propY0bcd}
When $a=0$, the module $e_{(b,d)} G(1^b 0^c 1^d)$ is indecomposable.
Furthermore, as a graded vector space there is an isomorphism
\begin{equation*}
\END\left(e_{(b,d)} G(1^b 0^c 1^d)\right) \cong \mathrm{H}^*(\gr(d,c+d)),
\end{equation*}
where $\mathrm{H}^*(\gr(d,c+d))$ is the cohomology of $d$-planes in $\C^{c+d}$.
\end{prop}

\begin{proof}
Let $\lambda=(1^b 0^c 1^d)$ and $Y'=e_{(b,d)} G(\lambda)$.

Consider $\Psi=\Psi_{\mf{s},\mf{t},\mu}^{\lambda \lambda} \in \END(G(\lambda))$.
Let $\tilde{\Psi}$ be the restriction of $\Psi$ to $Y'$ composed with projection onto $Y'$.  
It is easy to see that $\tilde{\Psi}$ is trivial unless $\mf{s},\mf{t} \in \mf{S}_{b+d} / \mf{S}_b \times \mf{S}_d$ are shortest length coset representatives.
But then $\mf{s}$ and $\mf{t}$ cannot fill any possible $\mu$ unless $\mf{s}=\mf{t}=e$, the identity coset.
Thus $deg(\tilde{\Psi})=0$ only if $\tilde{\Psi}=\Id$ and otherwise the degree is positive.  Thus $Y'$ is indecomposable.

The possible $\mu$ which could arise in $\Psi^{\lambda \lambda}_{e,e,\mu}$ are those
$\mu \in \mathcal{P}_{b+d}^{b+c+d}$ obtained from shuffling 
the $c$ entries labeled $0$ and the $d$ entries labeled $1$ in 
$\lambda=(1^b 0^c 1^d)$.
There $\binom{c+d}{d}$ such shuffles and it easy to compute the degree of the resulting $\Psi^{\lambda \lambda}_{e,e,\mu}$ which then immediately gives that the graded vector space $\END(e_{(b,d)} G(1^b 0^c 1^d))$ is isomorphic to $\mathrm{H}^*(\gr(d,c+d))$.
\end{proof}

For the next construction, we will need the following results. Recall from Definition \ref{def-E-and-F} that we have defined an action of the $p$-DG functors $\mf{E}$ (restriction) and $\mf{F}$ (induction) on $\oplus_{n=0}^l(\nh_n^l,\dif)\dmod$.

\begin{lem}\label{lem-fG-equals-FY}
For any $\lambda=(0^a 1^b 0^{c+d})\in \mc{P}_n^{r,s}$, there is an isomorphism of $\nh_n^l$-modules 
\[
\mf{E}^{(a)} Y(1^{a+b} 0^{c+d})\cong e_{\lambda} G(\lambda).
\]
\end{lem}

\begin{proof}
By Proposition \ref{basisG} and Corollary \ref{cor-Spect-equals-G}, the indecomposable module $Y(1^{a+b} 0^{c+d})\cong G(1^{a+b}0^{c+d})$ is simple with a basis 
$$ \{ y_1^{l-1} \cdots y_{a+b}^{l-a-b} \psi_w | w \in \mf{S}_{a+b} \}.$$ 
Applying $\mf{E}^{(a)}$ to this module has the effect of multiplying the basis elements on the right by the idempotent
$ e_{(1^b,a)}^\star$ which annihilates some of the basis elements.  In fact, a basis for the restricted module $\mf{E}^{(a)} Y(1^{a+b} 0^{c+d})$ consists of
\[
\{ y_1^{l-1} \cdots y_{a+b}^{l-a-b} \psi_{w} e_{(1^b,a)}^\star 
| w \in \mf{S}_{a+b}/  ( \mf{S}_1^{\times b} \times \mf{S}_a) \},
\]
where it is understood that $w\in \mf{S}_{a+b}/(\mf{S}_1^{\times b} \times \mf{S}_a)$ ranges over the shortest coset representatives.

On the other hand, using Proposition \ref{basisG}, we easily deduce that a basis of 
$e_{\lambda} G(\lambda)$ is 
\[
\{ e_b (y_1 \cdots y_b)^{l-a-b} y_1^{k_1} \cdots y_b^{k_b} \psi_w
| w \in \mf{S}_b, a+b-1 \geq k_1 > \cdots > k_b \geq 0.
\}
\]
In particular, both spaces $\mf{E}^{(a)} Y(1^{a+b} 0^{c+d})$ and $e_{\lambda} G(\lambda)$ have the same dimension $\binom{a+b}{b} b!$.

Define a map $ \Phi: e_{\lambda} G(\lambda) \longrightarrow \mf{E}^{(a)} Y(1^{a+b} 0^{c+d})$ by, for any $ X= e_by^\lambda x\in e_by^\lambda \nh_n^l$,
\begin{equation}
\Phi(X):=
y_1^{a+b-1} \cdots y_a^{b} e_{(1^a,b)}\psi_{b,a} e_{(1^b,a)}^\star (y_{b+1}\cdots y_{b+a})^{l-a-b}  X,
\end{equation}
where we recall that $\psi_{b,a}$ is the composition of nilHecke generators associated to the longest minimal coset representatives in $
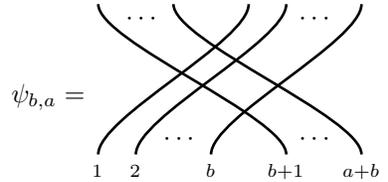
\mf{S}_{a+b} / \mf{S}_b \times \mf{S}_a$:
\[
\psi_{b,a}=
\begin{DGCpicture}
\DGCstrand(0,0)(2,2)[$^1$]
\DGCstrand(0.5,0)(2.5,2)[$^2$]
\DGCstrand(1.5,0)(3.5,2)[$^b$]
\DGCstrand(2.5,0)(0,2)[$^{b+1}$]
\DGCstrand(3.5,0)(1,2)[$^{a+b}$]
\DGCcoupon*(0.8,0.1)(1.4,0.3){$\cdots$}
\DGCcoupon*(2.6,1.7)(3.2,1.9){$\cdots$}
\DGCcoupon*(2.6,0.1)(3.2,0.3){$\cdots$}
\DGCcoupon*(0.3,1.7)(0.9,1.9){$\cdots$}
\end{DGCpicture} \ .
\]
The map is evidently a map of right $\nh_b^l$-modules. It is diagrammatically described by
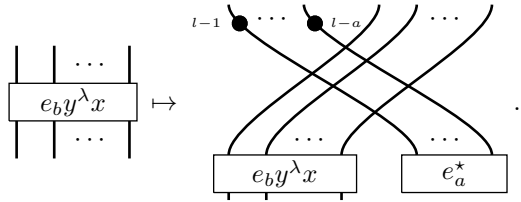
\[
\begin{DGCpicture}
\DGCstrand(0,-0.25)(0,1.25)
\DGCstrand(0.5,-0.25)(0.5,1.25)
\DGCstrand(1.5,-0.25)(1.5,1.25)
\DGCcoupon(-0.1,0.25)(1.6,0.75){$e_by^\lambda x$}
\DGCcoupon*(0.60,0.85)(1.4,1.15){$\cdots$}
\DGCcoupon*(0.6,-0.15)(1.4,0.15){$\cdots$}
\end{DGCpicture}
~\mapsto~
\begin{DGCpicture}
\DGCstrand/u/(0,-0.6)/u/(0,0)(2,2)
\DGCstrand/u/(0.5,-0.6)/u/(0.5,0)(2.5,2)
\DGCstrand/u/(1.5,-0.6)/u/(1.5,0)(3.5,2)
\DGCstrand(2.5,0)(0,2)
\DGCdot{1.75}[dl]{$_{_{a+b-1}}$}
\DGCstrand(3.5,0)(1,2)
\DGCdot{1.75}[dl]{$_{_{b}}$}
\DGCcoupon*(0.8,0.1)(1.4,0.3){$\cdots$}
\DGCcoupon*(2.0,1.3)(3.1,1.5){$\cdots$}
\DGCcoupon*(2.6,0.1)(3.2,0.3){$\cdots$}
\DGCcoupon*(0.3,1.7)(0.9,1.9){$\cdots$}
\DGCcoupon(-0.2,-0.5)(1.7,0){$e_b y^\lambda x$}
\DGCcoupon(1.6,1.55)(3.7,1.95){$e_b$}
\DGCcoupon(2.3,-0.5)(3.7,0){$u_a e_{a}^\star$}
\end{DGCpicture} \ ,
\]
where we have abbreviated $u_a:= (y_{b+1}\cdots y_{b+a})^{l-a-b}$.

Let us use this description to explain why the map has its image inside 
$$\mf{E}^{(a)}G(1^{a+b}0^{c+d})=y_1^{l-1}\cdots y_{a+b}^{l-a-b}\nh_{a+b}^l e^\star_{(1^b,a)}.$$

It suffices to take $x=1\in \nh_b^l$. Denote, until the end of the proof, the elements on southwest-northeast bound $b$ strands by 
\begin{equation}
v_b:= y_1^{b-1}y_2^{b-2}\cdots y_b^{0}, \quad \quad
\psi_0:= \psi_{w_0}, 
\end{equation} 
Then we have
\begin{align*}
e_by^\lambda &= e_b y_1^{l-a-1}y_2^{l-a-2}\cdots y_b^{l-a-b}=
(y_1y_2\cdots y_b)^{l-a-b}e_b y_1^{b-1}y_2^{b-2}\cdots y_b^0\\
& = y_1^{l-a-1}y_{2}^{l-a-2}\cdots y_b^{l-a-b}\psi_{0} y_1^{b-1}y_2^{b-2}\cdots y_b^0= y^\lambda \psi_{0} v_b.
\end{align*}

The factor $(y_1\cdots y_b)^{l-a-b}$ inside $y^\lambda$ combines with $u_a=(y_{b+1}\cdots y_{b+a})^{l-a-b}$ into the symmetric function
$(y_1\cdots y_{a+b})^{l-a-b}$, which can then slide through the element $\psi_{b,a}$. This results in
\begin{equation}
\Phi(e_by^\lambda)=
\begin{DGCpicture}
\DGCstrand/u/(0,-0.6)/u/(0,0)(2,2)
\DGCstrand/u/(0.5,-0.6)/u/(0.5,0)(2.5,2)
\DGCstrand/u/(1.5,-0.6)/u/(1.5,0)(3.5,2)
\DGCstrand(2.5,0)(0,2)
\DGCdot{1.75}[dl]{$_{_{a+b-1}}$}
\DGCstrand(3.5,0)(1,2)
\DGCdot{1.75}[dl]{$_{_{b}}$}
\DGCcoupon*(0.8,0.1)(1.4,0.3){$\cdots$}
\DGCcoupon*(2.0,1.3)(3.1,1.5){$\cdots$}
\DGCcoupon*(2.6,0.1)(3.2,0.3){$\cdots$}
\DGCcoupon*(0.3,1.7)(0.9,1.9){$\cdots$}
\DGCcoupon(-0.2,-0.5)(1.7,0){$y^\lambda \psi_{0} v_b$}
\DGCcoupon(1.6,1.55)(3.7,1.95){$e_b$}
\DGCcoupon(2.3,-0.5)(3.7,0){$u_a e_{a}^\star$}
\end{DGCpicture} 
=
\begin{DGCpicture}
\DGCstrand/u/(0,-0.6)/u/(0,0)(2,2)
\DGCstrand/u/(0.5,-0.6)/u/(0.5,0)(2.5,2)
\DGCstrand/u/(1.5,-0.6)/u/(1.5,0)(3.5,2)
\DGCstrand(2.5,0)(0,2)
\DGCdot{1.75}[dl]{$_{_{l-1}}$}
\DGCstrand(3.5,0)(1,2)
\DGCdot{1.75}[dl]{$_{_{l-a}}$}
\DGCcoupon*(0.8,0.1)(1.4,0.3){$\cdots$}
\DGCcoupon*(2.0,1.3)(3.1,1.5){$\cdots$}
\DGCcoupon*(2.6,0.1)(3.2,0.3){$\cdots$}
\DGCcoupon*(0.3,1.7)(0.9,1.9){$\cdots$}
\DGCcoupon(-0.2,-0.5)(1.7,0){$v_b \psi_{0}v_b$}
\DGCcoupon(1.6,1.5)(3.7,1.95){$ y^\lambda \psi_0 $}
\DGCcoupon(2.3,-0.5)(3.7,0){$e_{a}^\star$}
\end{DGCpicture} 
=
\begin{DGCpicture}
\DGCstrand/u/(0,-0.6)/u/(0,0)(2,2)
\DGCstrand/u/(0.5,-0.6)/u/(0.5,0)(2.5,2)
\DGCstrand/u/(1.5,-0.6)/u/(1.5,0)(3.5,2)
\DGCstrand(2.5,0)(0,2)
\DGCdot{1.75}[dl]{$_{_{l-1}}$}
\DGCstrand(3.5,0)(1,2)
\DGCdot{1.75}[dl]{$_{_{l-a}}$}
\DGCcoupon*(0.8,0.1)(1.4,0.3){$\cdots$}
\DGCcoupon*(2.0,1.3)(3.1,1.5){$\cdots$}
\DGCcoupon*(2.6,0.1)(3.2,0.3){$\cdots$}
\DGCcoupon*(0.3,1.7)(0.9,1.9){$\cdots$}
\DGCcoupon(-0.2,-0.5)(1.7,0){$\psi_0 v_b$}
\DGCcoupon(1.6,1.55)(3.7,1.95){$ y^\lambda $}
\DGCcoupon(2.3,-0.5)(3.7,0){$e_{a}^\star$}
\end{DGCpicture} \ .
\end{equation}
The last equality holds by iterated application of the Reidmeister-III move for nilHecke algebras (the second equation in equation \eqref{eqn-nilHecke-RII-RIII}) and the equality that 
$\psi_0 v_b \psi_0 = \psi_0$ (see equation \eqref{eqn-psiepsi}).
Thus we have confirmed that
\begin{equation}
\Phi(e_by^\lambda)=~
\begin{DGCpicture}
\DGCstrand(0,0)(2,2)
\DGCstrand(0.5,0)(2.5,2)
\DGCstrand(1.5,0)(3.5,2)
\DGCstrand(2.5,0)(0,2)
\DGCstrand(3.5,0)(1,2)
\DGCcoupon*(0.8,0.1)(1.4,0.3){$\cdots$}
\DGCcoupon*(2.5,1.55)(3.2,1.7){$\cdots$}
\DGCcoupon*(2.6,0.1)(3.2,0.3){$\cdots$}
\DGCcoupon*(0.4,1.55)(1,1.7){$\cdots$}
\DGCcoupon(-0.2,-0.5)(1.7,0){$\psi_0v_b$}
\DGCcoupon(2.3,-0.5)(3.7,0){$e_{a}^\star$}
\DGCcoupon(-0.2,1.75)(3.7,2.2){$_{y_1^{l-1}y_2^{l-2}\cdots y_{a+b}^{l-a-b}}$}
\end{DGCpicture}  \in y_1^{l-1}\cdots y_{a+b}^{l-a-b}\nh_{a+b}^l e^\star_{(1^b,a)} \ .
\end{equation}

It also follows from the above computation that
\begin{equation}
\Phi(e_by^\lambda \psi_{0})=
\begin{DGCpicture}
\DGCstrand(0,0)(2,2)
\DGCstrand(0.5,0)(2.5,2)
\DGCstrand(1.5,0)(3.5,2)
\DGCstrand(2.5,0)(0,2)
\DGCstrand(3.5,0)(1,2)
\DGCcoupon*(0.8,0.1)(1.4,0.3){$\cdots$}
\DGCcoupon*(2.5,1.55)(3.2,1.7){$\cdots$}
\DGCcoupon*(2.6,0.1)(3.2,0.3){$\cdots$}
\DGCcoupon*(0.4,1.55)(1,1.7){$\cdots$}
\DGCcoupon(-0.2,-0.5)(1.7,0){${\psi_{0}}$}
\DGCcoupon(2.3,-0.5)(3.7,0){$e_{a}^\star$}
\DGCcoupon(-0.2,1.75)(3.7,2.2){$_{y_1^{l-1}y_2^{l-2}\cdots y_{a+b}^{l-a-b}}$}
\end{DGCpicture}  \ .
\end{equation}
This is the lowest degree element of $\mf{E}^{(a)}G(1^{a+b}0^{c+d})$, and is readily seen to generate the entire module. Hence the map $\Phi$ is surjective, and it follows that $\Phi$ is an isomorphism since both modules have the same dimension.
\end{proof}

\begin{lem}\label{lem-eG-equals-EY}
For any $\lambda=(0^a1^b0^c1^d)\in \mc{P}_n^{r,s}$, there is an isomorphism of $\nh_n^l$-modules 
\[
e_{\lambda}G(\lambda) \cong 
\mf{F}^{(d)}Y(0^a 1^b 0^{c+d}).
\]
\end{lem}
\begin{proof}
By Lemma \ref{Yindecomspecial}, $Y(0^a1^b0^{c+d})\cong e_bG(0^a1^b0^{c+d})$. We have a sequence of isomorphisms
\begin{align*}
\mf{F}^{(d)} Y(0^a1^{b}0^{c+d}) 
& \cong \mf{F}^{(d)} (e_{b}y_1^{b+c+d-1} \cdots y_{b}^{c+d} \nh_{b}^l) \\
& \cong e_{(b,d)} y_1^{b+c+d-1} \cdots y_{b}^{c+d} \nh^l_{b+d} \\
& = e_{\lambda} (y_1^{b+c+d-1} \cdots y_{b}^{c+d}) (y_{b+1}^{d-1} 
\cdots y_{b+d}^0) \nh^l_{b+d} \\
& = e_{\lambda }G(\lambda).
\end{align*}
The third equality above needs some further explanation. It is clear that the right ideal $e_{\lambda }G(\lambda)$ is contained in the right ideal generated by the element $ e_{(b,d)} y_1^{b+c+d-1} \cdots y_{b}^{c+d} $. The reverse inclusion holds because of the identity
\begin{equation*}
e_{(b,d)} y_1^{b+c+d-1} \cdots y_{b}^{c+d}  =
e_{(b,d)} (y_1^{b+c+d-1} \cdots y_{b}^{c+d}) (y_{a+b+1}^{d-1} 
\cdots y_{b+d}^0) \psi_w
\end{equation*}
where $\psi_w$ is the element of $\nh_{b+d}^{l} $ arising from the longest element of the parabolic group $\mf{S}_1^{\times b}\times \mf{S}_{d}\subset \mf{S}_{b+d}$:
\[
\psi_w=\begin{DGCpicture}[scale=0.5]
\DGCcoupon*(-3,1)(-2,3){$\cdots$}
\DGCstrand(-5,0)(-5,4)
\DGCstrand(-4,0)(-4,4)
\DGCstrand(-1,0)(-1,4)
\DGCstrand(0,0)(4,4)
\DGCstrand(4,0)(0,4)
\DGCstrand(3,0)(4,1)(1,4)
\DGCstrand(2,0)(4,2)(2,4)
\DGCcoupon*(2.2,3.6)(3.8,4){$\cdots$}
\DGCcoupon*(0.2,0)(1.8,0.4){$\cdots$}
\end{DGCpicture} \ .
\]
The lemma follows.
\end{proof}

\begin{prop}
\label{indecomb>c}
Suppose $ b >c $.  Then
$\mf{E}^{(a)} (e_{(a+b,d)} G(1^{a+b} 0^c 1^d)) $ is indecomposable.
\end{prop}

\begin{proof}
Let $Y' = Y(1^{a+b} 0^{c+d})$ and $Y=\mf{E}^{(a)} \mf{F}^{(d)} Y'$.
Then, by the previous lemma and Proposition \ref{propY0bcd}, 
$$\mf{F}^{(d)}Y'\cong e_{(a+b,d)} G(1^{a+b} 0^c 1^d) \cong Y(1^{a+b}0^c1^d)   .$$
We will show that $\END(Y)$ is non-negatively graded and that the degree zero subspace is one dimensional.  We compute:
\begin{align*}
\END(Y)&\cong
\HOM(\mf{E}^{(a)} \mf{F}^{(d)} Y', \mf{E}^{(a)} \mf{F}^{(d)} Y' )  \\
&\cong \HOM(\mf{F}^{(d)} Y', \mf{F}^{(a)} \mf{E}^{(a)} \mf{F}^{(d)} Y')  q^{a(b+d-c)}  \\
&\cong \bigoplus_{j=0}^a \HOM(\mf{F}^{(d)} Y', \mf{E}^{(a-j)} \mf{F}^{(a-j)} \mf{F}^{(d)} Y'){a+b+d-c \brack j} q^{a(b+d-c)} \\
&\cong \bigoplus_{j=0}^a \HOM(\mf{F}^{(d)} Y', \mf{E}^{(a-j)} \mf{F}^{(a+d-j)} Y'){a+b+d-c \brack j} {a+d-j \brack d} q^{a(b+d-c)} \\
& \cong \bigoplus_{j=0}^a \HOM(\mf{F}^{(a-j)} \mf{F}^{(d)} Y',  \mf{F}^{(a+d-j)} Y'){a+b+d-c \brack j} {a+d-j \brack d} q^{a(b+d-c)} q^{(a-j)(2a+b+d-c-j)}  \\
&\cong \bigoplus_{j=0}^a \HOM(\mf{F}^{(a+d-j)} Y',  \mf{F}^{(a+d-j)} Y'){a+b+d-c \brack j} {a+d-j \brack d}^2 q^{(2a-j)(a+b+d-c-j)} 
\\
&\cong \bigoplus_{j=0}^a \END(Y(1^{a+b} 0^{c-a+j} 1^{a+d-j})){a+b+d-c \brack j} {a+d-j \brack d}^2 q^{(2a-j)(a+b+d-c-j)} .
\end{align*}
The second and fifth isomorphisms follow by taking the biadjoints of $\mf{E}^{(a)} $ and $\mf{E}^{(a-j)} $ using equation~\eqref{Fadjointformula}.
The third one follows from taking commutation of $\mf{F}^{(a)}$ and $\mf{E}^{(a)}$ (see Proposition \ref{prop-higher-Serre} (ii)).
The formula for multiplication of $\mf{F}^{(a-j)}$ and $\mf{F}^{(d)}$ produces the fourth and sixth isomorphism (Proposition \ref{prop-higher-Serre} (i)), while the last one is again a consequence of Lemma \ref{lem-eG-equals-EY}.

It follows from Proposition \ref{propY0bcd} again that $Y(1^{a+b} 0^{c-a+j} 1^{a+d-j})$ is indecomposable, and
$$\END(Y(1^{a+b} 0^{c-a+j} 1^{a+d-j})) \cong \mathrm{H}^* (\gr(a+d-j,c+d)).$$
Since the graded dimension of $\mathrm{H}^* (\gr(a+d-j,c+d))$ is an element of 
$1+q \N[q] $ we need only determine the smallest exponent of 
\begin{equation*}
{a+b+d-c \brack j} {a+d-j \brack d}^2 q^{(2a-j)(a+b+d-c-j)}.
\end{equation*}
The smallest exponent is
\begin{equation}
\label{functionofj}
-j(a+b+d-c-j)-2d(a-j)+(2a-j)(a+b+d-c-j)=2j^2+(2(c-b)-4a)j+2a(b-c+a).
\end{equation}
The minimum of this function (with respect to $j$) occurs when $j=a-\frac{c-b}{2}$.
Since $b>c$, it follows that $j=a-\frac{c-b}{2} \geq a$.
Since $j \in [0,a]$, the minimum of the function \eqref{functionofj} must occur at either $j=0$ or $j=a$.

When $j=0$, the function \eqref{functionofj} is equal to $2a(a+b-c) >0$.
When $j=a$, the value of \eqref{functionofj} is equal to zero.
Thus the graded dimension of $\END(Y)$ is an element of $1+q \N[q]$ so $Y$ is indecomposable.
\end{proof}

\begin{rem}
In general, we do not know of a construction of the indecomposable generating $p$-DG objects in $\mc{NH}_n^{r,s}$ directly from $\nh_n^l$ except in some very special cases.  For instance, for the case $a>0$, $b>c$, and $d=1$ it is possible to show that 
\begin{equation*}
Y(0^a 1^b 0^c 1) \cong 
(e_{(b,1)}^{(1^{b-c})} + \cdots + e_{(b,1)}^{(1^{b})}) G(0^a 1^b 0^c 1)
\end{equation*}
and that $e_{(b,1)}^{(1^{b-c})} + \cdots + e_{(b,1)}^{(1^{b})}$ is a $p$-DG idempotent (see the notation \eqref{eqn-general-idempotent}).
\end{rem}

We now state the following result on the classification of indecomposable objects in $\mc{NH}_n^{r,s}$. One should compare it with Remark \ref{canbasisremark}.

\begin{thm}
\label{classificationprop}
The indecomposable objects $Y(\lambda)$ in $\mc{NH}_{n}^{r,s} $ are parametrized by partitions in $ \mc{P}_n^{r,s}$. Such a $Y(\lambda)$, when $\lambda=(0^a1^b0^c1^d)$, is isomorphic to either of the following forms:
\begin{itemize}
\item $ \mf{F}^{(d)} \mf{E}^{(a)}Y(1^{a+b} 0^{c+d})\cong e_{(b,d)} G(0^a 1^b 0^c 1^d) \cong \mf{F}^{(d)}Y(0^a 1^b 0^{c+d}) $ 
if $b \leq c$;
\item $\mf{E}^{(a)} \mf{F}^{(d)} Y(1^{a+b} 0^{c+d}) \cong \mf{E}^{(a)} (e_{(a+b,d)} G(1^{a+b} 0^c 1^d))
 $ if $ b \geq c$.
\end{itemize}
If $b=c$ the modules are isomorphic.
\end{thm}

\begin{proof}
The fact that the modules in the statement of the theorem are indecomposable follows from Propositions \ref{abc1cgeqb} and \ref{indecomb>c}.

It is trivial that the object $e_{(b,d)} G(0^a 1^b 0^c 1^d)$ for $b \leq c$  is in $\mc{NH}_{n}^{r,s}$  since it is a generating object of the category.

Now we will show that $\mf{E}^{(a)} \mf{F}^{(d)} Y(1^{a+b} 0^{c+d}) $ is in $\mc{NH}_n^{r,s}$ for $b \geq c$.  By Proposition \ref{prop-higher-Serre}, the $p$-DG module
$\mf{E}^{(a)} \mf{F}^{(d)} Y(1^{a+b} 0^{c+d}) $ occurs as a subquotient in the fantastic filtration of
\begin{equation*}
\mf{F}^{(d)} \mf{E}^{(a)} Y(1^{a+b} 0^{c+d}) \cong \mf{F}^{(d)} Y(0^a 1^b 0^{c+d})
\cong e_{(b,d)} G(0^a 1^b 0^c 1^d),
\end{equation*}
where the last isomorphism follows from Lemma \ref{lem-eG-equals-EY}.

Lastly, when $b=c$, the fact that $\mf{E}^{(a)} \mf{F}^{(d)} Y(1^{a+b} 0^{b+d}) \cong
\mf{F}^{(d)} \mf{E}^{(a)} Y(1^{a+b} 0^{b+d}) $ follows from
the special case of Proposition \ref{prop-higher-Serre} that
\[
\mf{E}^{(a)} \mf{F}^{(d)}\1_{m}\cong \mf{F}^{(d)} \mf{E}^{(a)}\1_{m}
\]
whenever $m \leq d-a$. Now $m$ is the weight of the module $Y(1^{a+b}0^{c+d})$, which equals $l-2(a+b)=d-a$. The result follows.
\end{proof}

\begin{cor}
\label{pdgclassificationcor}
The indecomposable objects $Y(\lambda)$ in $\mathcal{NH}_{n}^{r,s} $ are $p$-DG modules.
\end{cor}

\begin{proof}
This follows from Theorem \ref{classificationprop} since 
$Y(1^{a+b}0^{c+d})$ is a $p$-DG module and $\mf{E}^{(a)}$, $\mf{F}^{(d)}$ are $p$-DG functors.
\end{proof}

\subsection{Extension of categorical actions III}
\label{subsec-ext-cat-act-iii}
Except for the last result, we have largely ignored the $p$-DG structure in the previous subsection, and, in particular, the twists involved in the definition of the functors $\mf{E}^{(a)}$ in Definition \ref{def-E-and-F}. We will take these matter up in the subsection, and check that the categorical framework of Section \ref{subsec-ext-cat-act-ii} is suitable for our special situation of two-tensor quiver Schur algebras.

Recall from Lemma \ref{lem-fG-equals-FY} that there is an isomorphism of right $\nh_b^l$-modules
\[
\Phi:e_bG(0^a1^b0^{c+d})\cong \mf{E}^{(a)}Y(1^{a+b}0^{c+d}) .
\] 
Our first goal is to upgrade this isomorphism into one of $p$-DG modules. This will also explain the necessity of twisting the differential for the functor $\mf{E}^{(a)}$ (Definition \ref{def-E-and-F}).

\begin{lem}\label{lem-fG-equals-FY-pdg}
For the partition $\lambda=(0^a1^b0^{c+d})$, there is a $p$-DG isomorphism
\[
\Phi: e_\lambda G(\lambda)\cong \mf{E}^{(a)}Y(1^{a+b}0^{c+d}).
\]
\end{lem}
\begin{proof}
Let us adopt the simplified notations used in the proof of Lemma \ref{lem-fG-equals-FY}. We will show that, as in the Lemma, the differential actions on the generator of $e_\lambda G(\lambda)$ and its image under $\Phi$ match with each other. To do this let us denote the image of the generator $e_\lambda y^\lambda$ under $\Phi$ by $A:=\Phi(e_\lambda y^\lambda)$, so that
\[
A=~
\begin{DGCpicture}
\DGCstrand(0,0)(2,2)
\DGCstrand(0.5,0)(2.5,2)
\DGCstrand(1.5,0)(3.5,2)
\DGCstrand(2.5,0)(0,2)
\DGCstrand(3.5,0)(1,2)
\DGCcoupon*(0.8,0.1)(1.4,0.3){$\cdots$}
\DGCcoupon*(2.5,1.55)(3.2,1.7){$\cdots$}
\DGCcoupon*(2.6,0.1)(3.2,0.3){$\cdots$}
\DGCcoupon*(0.4,1.55)(1,1.7){$\cdots$}
\DGCcoupon(-0.2,-0.5)(1.7,0){${\psi_{0}v_b}$}
\DGCcoupon(2.3,-0.5)(3.7,0){$e_{a}^\star$}
\DGCcoupon(-0.2,1.75)(3.7,2.2){$_{y_1^{l-1}y_2^{l-2}\cdots y_{a+b}^{l-a-b}}$}
\end{DGCpicture} 
~=~
\begin{DGCpicture}
\DGCstrand(0,-0.5)(0,2)
\DGCstrand(0.5,-0.5)(0.5,2)
\DGCstrand(1.5,-0.5)(1.5,2)
\DGCstrand(2.5,-0.5)(2.5,2)
\DGCstrand(3.5,-0.5)(3.5,2)
\DGCcoupon*(0.6,0.25)(1.4,0.45){$\cdots$}
\DGCcoupon*(2.6,0.25)(3.4,0.45){$\cdots$}
\DGCcoupon(2.2,0.5)(3.7,1){$1_{a}^\star$}
\DGCcoupon(2.2,-0.25)(3.7,0.25){$_{y_{b+1}^{0}\cdots y_{b+a}^{a-1}}$}
\DGCcoupon(-0.2,-0.25)(1.7,0.25){$v_b$}
\DGCcoupon(-0.2,1.55)(3.7,2.2){$_{(y_1y_2\cdots y_{a+b})^{l-a-b}e_{a+b}}$}
\end{DGCpicture} \ .
\]
Here, in the last diagram, the element $1_a^\star$ is the unit element on the last $a$ strands satisfying
\begin{equation}
 \dif (1_a^\star)=(a-l+2b)(y_{b+1}+\cdots+y_{b+a})1_a^\star.
\end{equation} 
This is the twisting used in defining $\mf{E}^{(a)}$ (Definition \ref{def-E-and-F}).

We now easily compute the differential action on $A$ to be
\begin{equation}
\dif(A) = \sum_{i=1}^b (l-a-2i+1)Ay_i.
\end{equation}
Comparing this with the differential action
\begin{equation}
\dif(e_\lambda y^\lambda)=\sum_{i=1}^b(l-a-2i+1) e_\lambda y^\lambda y_i,
\end{equation}
it is clear that $\Phi$ intertwines the differential actions. The lemma follows.
\end{proof}

\begin{example}We list a few examples illustrating the necessity of using the twisted restriction functor (Definition \ref{def-E-and-F}) in Lemma~\ref{lem-fG-equals-FY-pdg}.

\begin{enumerate}
\item[(1)]
We consider the case when $a=2$ and $b=2$.  Then there is an isomorphism 
$\mf{E}^{(2)}Y(1^4 0^{c+d}) \cong e_2 G(0^2 1^2 0^{c+d}) $ given by
\begin{align*}
y_1^{l-1} y_2^{l-2} y_3^{l-3} y_4^{l-4} \psi_2 \psi_3 \psi_1 \psi_2 \psi_w e_{(1^2,2)}^\star
& \mapsto 
e_2 (y_1 y_2)^{l-4} y_1 \psi_w \\
y_1^{l-1} y_2^{l-2} y_3^{l-3} y_4^{l-4} \psi_3 \psi_1 \psi_2 \psi_w e_{(1^2,2)}^\star
& \mapsto 
e_2 (y_1 y_2)^{l-4} y_1^2 \psi_w \\
y_1^{l-1} y_2^{l-2} y_3^{l-3} y_4^{l-4} \psi_1 \psi_2 \psi_w e_{(1^2,2)}^\star
& \mapsto 
e_2 (y_1 y_2)^{l-4} y_1^2 y_2 \psi_w \\
y_1^{l-1} y_2^{l-2} y_3^{l-3} y_4^{l-4} \psi_3 \psi_2 \psi_w e_{(1^2,2)}^\star
& \mapsto 
e_2 (y_1 y_2)^{l-4} y_1^3 \psi_w \\
y_1^{l-1} y_2^{l-2} y_3^{l-3} y_4^{l-4} \psi_2 \psi_w e_{(1^2,2)}^\star
& \mapsto 
e_2 (y_1 y_2)^{l-4} y_1^3 y_2 \psi_w \\
y_1^{l-1} y_2^{l-2} y_3^{l-3} y_4^{l-4} \psi_w e_{(1^2,2)}^\star
& \mapsto 
e_2 (y_1 y_2)^{l-4} y_1^3 y_2^2 \psi_w
\end{align*}
where $w \in \mf{S}_2$.

Note that the element
$X=y_1^{l-1} y_2^{l-2} y_3^{l-3} y_4^{l-4} \psi_2 \psi_3 \psi_1 \psi_2 e_{(1^2,2)}^\star$
is a generator for $\mf{E}^{(2)}Y(1^4 0^{c+d}) $ and that
$$ \partial(X) 
= 
(l-3)X y_1
+
(l-5)X y_2.$$
Similarly, $Z=e_2 (y_1 y_2)^{l-4} y_1 $ generates $e_2 G(0^2 1^2 0^{c+d})$ and
one computes that
$$\partial(Z)=(l-3)Zy_1+(l-5)Zy_2.$$
This confirms that the isomorphism in Lemma \ref{lem-fG-equals-FY-pdg} is an isomorphism of $p$-DG modules.

\item[(2)] Another interesting special case of Lemma \ref{lem-fG-equals-FY-pdg} occurs when $c=d=0$.  Then Specht module $e_{a+b}G(1^{a+b})=Y(1^{a+b})$ is indecomposable and projective.  Then the restriction functor $\mf{E}^{(a)}$ takes the $\nh_l^l$-module $Y(1^{a+b})$ to the indecomposable $\nh_b^l$-module $e_bG(0^a 1^b)$.
\end{enumerate}
\end{example}

In a similar vein as for Lemma \ref{lem-fG-equals-FY-pdg}, Lemma \ref{lem-eG-equals-EY} can also be made into a statement about isomorphism of $p$-DG modules.

\begin{lem}\label{lem-eG-equals-EY-pdg}
For any $\lambda=(0^a1^b0^c1^d)\in \mc{P}_n^{r,s}$, there is a $p$-DG isomorphism of $\nh_n^l$-modules 
\[
e_{\lambda}G(\lambda) \cong 
\mf{F}^{(d)}Y(0^a 1^b 0^{c+d}).
\]
\end{lem}

\begin{proof}
The proof is easy since the functors $\mf{F}^{(d)}$ utilize the natural differential on the idempotent $e_d$ (Definition \ref{def-E-and-F}).
\end{proof}

Our next goal is to show that the indecomposable modules appearing in Theorem \ref{classificationprop} are naturally contained in the $p$-DG envelope of the generating modules $e_\lambda G(\lambda)$, $\lambda\in \mc{P}_n^{r,s}$. This will be guaranteed by the next result together with Proposition \ref{prop-higher-Serre}. 

\begin{prop}\label{prop-pdg-eG}
For any $\lambda=(0^a1^b0^c1^d)\in \mc{P}_n^{r,s}$, there is a $p$-DG isomorphism
\[
e_\lambda G(\lambda)\cong \mf{F}^{(d)}\mf{E}^{(a)}Y(1^{a+b}0^{c+d}).
\]
\end{prop}
\begin{proof}
By Lemma \ref{lem-fG-equals-FY-pdg}, we have a $p$-DG isomorphism
\[
\mf{E}^{(a)}Y(1^{a+b}0^{c+d})\cong e_bG(0^a1^b0^{c+d})=e_b(y_1^{l-a-1}\cdots y_b^{l-a-b})\nh_b^l.
\]
Now apply the induction functor $\mf{F}^{(d)}$ to the above cyclic module, we have
\[
\mf{F}^{(d)}\mf{E}^{(a)}Y(1^{a+b}0^{c+d})\cong e_b(y_1^{l-a-1}\cdots y_b^{l-a-b})\nh_b^l\otimes_{\nh_b^{l}}e_{(1^b,d)}\nh_{b+d}^l
\cong e_{(b,d)}(y_1^{l-a-1}\cdots y_b^{l-a-b})\nh_{b+d}^l,
\] 
where we have used that the elements in $\nh_b^l$ commute with the idempotent $e_{(1^b,d)}$, and that $e_be_{(1^b,d)}=e_{(b,d)}$. 
We claim that the last module cyclically generated by $e_{(b,d)}(y_1^{l-a-1}\cdots y_b^{l-a-b})$ is isomorphic to $e_\lambda G(\lambda)$, which is generated by $e_{(b,d)}y_1^{l-a-1}\cdots y_b^{l-a-b}y_{b+1}^{d-1}\cdots y_{b+d}^0$. The inclusion that 
$$e_\lambda G(\lambda)\subset e_{(b,d)}(y_1^{l-a-1}\cdots y_b^{l-a-b})\nh_{a+b}^l $$
is clear. The converse follows from the fact that
\[
e_{(b,d)}y_1^{l-a-1}\cdots y_b^{l-a-b}=e_{(b,d)}y_1^{l-a-1}\cdots y_b^{l-a-b}y_{b+1}^{d-1}\cdots y_{b+d}^0\psi_{w},
\]
where $w\in \mf{S}_1^b\times \mf{S}_d$ is the longest element of the parabolic subgroup (c.f.~the proof of Lemma \ref{lem-eG-equals-EY}). The result now follows.
\end{proof}

We are now ready to verify the conditions of Section \ref{subsec-ext-cat-act-ii} are satisfied in our particular situation.

\begin{thm}
\label{EFact}
The $p$-DG functors $\mf{E}$ and $\mf{F}$ restrict to the $p$-DG category
$\oplus_{n=0}^l~ \mathcal{NH}_{n}^{r,s} $.
\end{thm}

\begin{proof}
This follows from the classification of the indecomposable objects of the category.

Case I: $ b < c$.
Then $Y(0^a 1^b 0^c 1^d) \cong \mf{F}^{(d)} \mf{E}^{(a)}Y(1^{a+b} 0^{c+d})$.  In this case
\begin{align*}
\mf{E}Y(0^a 1^b 0^c 1^d) &\cong \mf{E} \mf{F}^{(d)} \mf{E}^{(a)}Y(1^{a+b} 0^{c+d}) \\
&\cong [a+1] \mf{F}^{(d)} \mf{E}^{(a+1)}Y(1^{a+b} 0^{c+d})
\oplus g_1 \mf{F}^{(d-1)} \mf{E}^{(a)}Y(1^{a+b} 0^{c+d}) \\
&\cong [a+1] \mf{F}^{(d)} Y(0^{a+1} 1^{b-1} 0^{c+d}) \oplus 
g_1 \mf{F}^{(d-1)} Y(0^a 1^{b} 0^{c+d}) \\
&\cong [a+1]Y(0^{a+1} 1^{b-1} 0^c 1^d) \oplus g_1 Y(0^a 1^b 0^{c+1}1^{d-1})
\end{align*}
for some multiplicity $g_1\in \N[q,q^{-1}]$.

Next we have
\begin{equation}
\label{preserveeq1}
\mf{F} Y(0^a 1^b 0^c 1^d) \cong \mf{F} \mf{F}^{(d)} \mf{E}^{(a)} Y(1^{a+b} 0^{c+d}) 
\cong [d+1] \mf{F}^{(d+1)} \mf{E}^{(a)} Y(1^{a+b} 0^{c+d}). 
\end{equation}
Since $b<c$ then the objects in \eqref{preserveeq1} are isomorphic to
\begin{equation*}
[d+1] \mf{F}^{(d+1)} Y(0^a 1^b 0^{c+d}) \cong [d+1]Y(0^a 1^b 0^{c-1} 1^{d+1}).
\end{equation*}

Case II: $b > c$.
We compute
\begin{align*}
\mf{F}Y(0^a 1^b 0^c 1^d) &\cong \mf{F} \mf{E}^{(a)} \mf{F}^{(d)} Y(1^{a+b} 0^{c+d}) \\
&\cong [d+1] \mf{E}^{(a)} \mf{F}^{(d+1)} Y(1^{a+b} 0^{c+d}) 
\oplus g_2 \mf{E}^{(a-1)} \mf{F}^{(d)} Y(1^{a+b} 0^{c+d}) \\
&\cong [d+1]Y(0^a 1^b 0^{c-1} 1^{d+1}) \oplus g_2 Y(0^{a-1} 1^{b+1} 0^c 1^d)
\end{align*}
for some multiplicity $g_2\in \N[q,q^{-1}]$.

Next we have
\begin{equation}
\label{preserveeq2}
\mf{E} Y(0^a 1^b 0^c 1^d) \cong \mf{E} \mf{E}^{(a)} \mf{F}^{(d)} Y(1^{a+b} 0^{c+d}) 
\cong [a+1] \mf{E}^{(a+1)} \mf{F}^{(d)} Y(1^{a+b} 0^{c+d}). 
\end{equation}
Since $b>c$, the objects in \eqref{preserveeq2} are isomorphic to
\begin{equation*}
[a+1] \mf{E}^{(a+1)} Y(1^{a+b} 0^c 1^d) \cong [a+1] Y(0^{a+1} 1^{b-1} 0^c 1^d).
\end{equation*}

Case III: $b=c$.
In order to show $\mf{F} Y(0^a 1^b 0^b 1^d)$ decomposes as we would like, we proceed as in case II.
\begin{align*}
\mf{F}Y(0^a 1^b 0^b 1^d) &\cong \mf{F} \mf{E}^{(a)} \mf{F}^{(d)} Y(1^{a+b} 0^{b+d}) \\
&\cong [d+1] \mf{E}^{(a)} \mf{F}^{(d+1)} Y(1^{a+b} 0^{b+d}) 
\oplus g_3 \mf{E}^{(a-1)} \mf{F}^{(d)} Y(1^{a+b} 0^{b+d}) \\
&\cong [d+1]Y(0^a 1^b 0^{b-1} 1^{d+1}) \oplus g_3 Y(0^{a-1} 1^{b+1} 0^b 1^d)
\end{align*}
for some multiplicity $g_3\in \N[q,q^{-1}]$.

In order to show $\mf{E} Y(0^a 1^b 0^b 1^d)$ decomposes as we would like, we proceed as in case I.
\begin{align*}
\mf{E}Y(0^a 1^b 0^b 1^d) &\cong \mf{E} \mf{F}^{(d)} \mf{E}^{(a)}Y(1^{a+b} 0^{b+d}) \\
&\cong [a+1] \mf{F}^{(d)} \mf{E}^{(a+1)}Y(1^{a+b} 0^{b+d})
\oplus g_4 \mf{F}^{(d-1)} \mf{E}^{(a)}Y(1^{a+b} 0^{b+d}) \\
&\cong [a+1] \mf{F}^{(d)} Y(0^{a+1} 1^{b-1} 0^{b+d}) \oplus 
g_4 \mf{F}^{(d-1)} Y(0^a 1^{b} 0^{b+d}) \\
&\cong [a+1]Y(0^{a+1} 1^{b-1} 0^b 1^d) \oplus g_4 Y(0^a 1^b 0^{b+1}1^{d-1})
\end{align*}
for some multiplicity $g_4\in \N[q,q^{-1}]$.

Finally, the compatibility of the $\mf{E}$-and-$\mf{F}$ actions with the $p$-differential follows from the higher categorical relations (Proposition \ref{prop-higher-Serre}). The direct sum decomposition above will just be replaced by a fantastic filtration, whose associated pieces are isomorphic to the occurring summands.
\end{proof}

The following example explains the necessity of taking the idempotent truncation for the modules $G(\lambda)$ in the definition of the two-tensor algebra.

\begin{example}
\label{l=3n=2}
Consider the nilHecke algebra $\nh_2^3$.
There are three cyclic modules in this case: $G(110)$, $ G(101)$ and $G(011)$.
In Example \ref{n=2l=3} we saw that
$Y(110)=G(110)$ and $Y(101)=G(101)$.
However $G(011) \cong Y(011) \oplus Y(110)$ where $Y(011)=e_2G(011)$.

In order to categorify the second highest weight space of $V_1 \otimes V_2$ we would take the category $\mc{NH}_2^{1,2}$.  The weight space is two-dimensional and so we would like $\mc{NH}_2^{1,2}$ to have exactly two generating non-isomorphic indecomposable objects.
  
Taking $G(011)$ and $G(101)$ as the generating set for $\mc{NH}_2^{1,2}$ is problematic because these two objects actually contain three non-isomorphic indecomposable modules.  However, taking the truncated modules $e_{(0,2)}G(011)$ and $e_{(1,1)}G(101)=G(101)$ produces only two non-isomorphic indecomposable objects $Y(011)$ and $Y(101)$.
\end{example}

Finally, we collect two results about the $p$-DG structure of the module $e_\lambda G(\lambda)$ that will be useful later. The next result shows that the $Y(\lambda)$
corresponding to the minimal partition in $\mc{P}_n^{r,s}$ (see Definition \ref{speciallambda})
is the projective indecomposable module over $\nh_n^l$.

\begin{prop}\label{prop-projective-as-summand}
For the minimal partition $\lambda_0\in \mc{P}_n^{r,s}$, the module $Y(\lambda_0)$ is isomorphic to the indecomposable projective module of $\nh_n^l$.
\end{prop}

\begin{proof}
The minimal partition $\lambda_0=(0^a 1^b0^c1^d)$ has either $b=0$ or $c=0$. Suppose $b=0$.
By Theorem \ref{classificationprop} we have
\begin{equation*}
Y(0^a 0^c 1^d) \cong \mf{F}^{(d)} \mf{E}^{(a)}Y(1^{a} 0^{c+d}).
\end{equation*}
Note that $\mf{E}^{(a)}Y(1^{a} 0^{c+d})$ is a module over $\nh_0^l \cong \Bbbk$.
By Theorem \ref{classificationprop}, $\mf{F}^{(d)} \mf{E}^{(a)}Y(1^{a} 0^{c+d})$ is indecomposable and hence so is $\mf{E}^{(a)}Y(1^{a} 0^{c+d})$.
Thus $\mf{E}^{(a)}Y(1^{a} 0^{c+d}) \cong \Bbbk$.  Since $\mf{F}^{(d)}$ is an induction functor, $\mf{F}^{(d)} \Bbbk$ is projective and we already know from Theorem \ref{classificationprop} that it is indecomposable.

The case $c=0$ is similar.  One uses the fact that $\nh_l^l$ is a simple algebra since it is isomorphic to an $l! \times l!$ matrix algebra over $\Bbbk$.
\end{proof}

Next, recall from Proposition \ref{abc1cgeqb} that, if $b\leq c$, then the module $e_{(b,d)} G(0^a 1^b 0^c 1^d)$ is indecomposable. When $b\geq c$, the module has a filtered $p$-DG structure as follows.

\begin{prop}\label{prop-multiplicity-Y-in-G}
Let $b \geq c$.  Then 
$ e_{(b,d)} G(0^a 1^b 0^c 1^d) \cong \mf{F}^{(d)} \mf{E}^{(a)}Y(1^{a+b} 0^{c+d}) $
has a filtration with subquotients isomorphic to 
$Y(0^{a-j} 1^{b+j} 0^{c+j} 1^{d-j})$ for $j=0,\ldots,\mathrm{min}(a,d)$.
Furthermore $Y(0^{a-j} 1^{b+j} 0^{c+j} 1^{d-j})$ appears with multiplicity
${b-c \brack j}$.
\end{prop}

\begin{proof}
By Proposition \ref{prop-higher-Serre}, the functor $ \mf{F}^{(d)} \mf{E}^{(a)}$ has a $\partial$-stable filtration with graded subquotients $ \mf{E}^{(a-j)} \mf{F}^{(d-j)}$ appearing with multiplicities ${b-c \brack j}$ since $ \mf{E}^{(a-j)} \mf{F}^{(d-j)}$ is acting on a category of weight $\1_{a+b-c-d}$.
\end{proof}

\subsection{A basic algebra}
Let us consider the basic version of Definition \ref{def-two-tensor-algebra} in this subsection. We again fix $r,s\in \N$ such that $r+s=l$ and assume $a,b,c,d\in \N$ satisfying conditions of Definition \ref{speciallambda}.

Recall that, by Theorem \ref{classificationprop}, we have a full list of indecomposable $p$-DG modules $Y(\lambda)$ in $\mc{NH}_n^{r,s}$ among the list
\[
\left\{
 \mf{F}^{(d)} \mf{E}^{(a)}Y(1^{a+b} 0^{c+d}) \quad (b \leq c) , \quad \quad 
\mf{E}^{(a)} \mf{F}^{(d)} Y(1^{a+b} 0^{c+d})  \quad (b \geq c)
\right\}
\]
with the identification $ \mf{F}^{(d)} \mf{E}^{(a)}Y(1^{r} 0^{s})\cong 
\mf{E}^{(a)} \mf{F}^{(d)} Y(1^{r} 0^{s}) $ if $b=c$. Let us fix, once an for all, an arbitrary grading on the module $Y(1^{r}0^{s})$, which in turn determines a grading on each $Y(\lambda)$ via the $p$-DG functors $\mc{F}^{(d)}$ and $\mc{E}^{(a)}$.

\begin{defn}\label{def-basic-two-tensor-algebra}
Let $n,l$ be two natural numbers, and $(r,s)\in \N^2$ be a decomposition of $l$. We define the \emph{basic two-tensor Schur algebra} to be the graded endomorphism algebra
\begin{equation*}
S^b_n(r,s) := \END_{\nh_n^l}\left(\bigoplus_{\lambda \in \mathcal{P}_n^{r,s}} Y(\lambda)\right).
\end{equation*}
The algebra is equipped with the induced $p$-differentials on $Y(\lambda)$'s, and therefore is a $p$-DG algebra.
\end{defn}


\begin{thm}
\label{homspaceY1Y2}
Let $Y_1$ and $Y_2$ be non-isomorphic indecomposable objects of $\mathcal{NH}_{n}^{r,s} $ with the gradings chosen as above.  Then
$\HOM_{\nh_n^l}(Y_1,Y_2)$ is positively graded. Consequently, the $p$-DG algebra $S_n^b(r,s)$ is a positive $p$-DG algebra.
\end{thm}

\begin{proof}
Let $u \in \Z$ be a nonzero integer and set
\begin{equation*}
Y_1=Y(0^a 1^b 0^c 1^d), \hspace{.5in}
Y_2=Y(0^{a+u} 1^{b-u} 0^{c-u} 1^{d+u}), \hspace{.5in}
Y'=Y(1^{r} 0^{s}).
\end{equation*}
Note that any indecomposable object is of the form $Y_2$ because of the constraints
$a+b=r, c+d=s$, and $b+d=n$.

Assume $b \geq c$.  The case $b \leq c$ is similar.
As in the proof of Proposition \ref{indecomb>c}, we calculate
\begin{equation*}
\HOM(Y_1, Y_2) \cong
\HOM(\mf{E}^{(a)} \mf{F}^{(d)} Y', \mf{E}^{(a+u)} \mf{F}^{(d+u)} Y' )  
\end{equation*}
which is isomorphic to
\begin{equation}
\label{HOMY1Y2a}
\bigoplus_{j=0}^{\mathrm{min}(a,a+u)} \END(Y(1^{a+b} 0^{c-a+j-u} 1^{a+d-j+u})) M(j)
\end{equation}
where
\begin{equation*}
M(j)={a+b+d-c+u \brack j} {a+d-j+u \brack d} 
{a+d-j+u \brack d+u}
q^{(2a+u-j)(a+b+d-c-j+u)}.
\end{equation*}
The endomorphism algebra in \eqref{HOMY1Y2a} is isomorphic to the cohomology of a Grassmannian and in particular non-negatively graded.
The smallest power of $q$ occuring in $M(j)$ is
\begin{equation*}
\rho(j)=2j^2+(2(c-b)-4a-2u))j +((2a+u)(a+b-c+u)-au).
\end{equation*}
An analysis of $\rho(j)$ similar to the smallest exponent studied in the proof of Proposition \ref{indecomb>c} shows that $\rho(j)>0$.  Note that there are two cases here: $u>0$ and $u<0$. They are handled similarly.
\end{proof}

\begin{rem}We make two comments about the proof above.
\begin{enumerate}
\item[(1)] Alternatively, the positivity of the $p$-DG algebra $S_n^b(r,s)$ also follows the work of Hu-Mathas. Indeed, by \cite[Theorem C]{HuMathas}, it is proven that the bigger quiver Schur algebra $S_n(l)$ is graded Morita equivalent to a basic algebra which is Koszul, and therefore the basic algebra for $S_n(l)$ is positively graded. On the other hand, the modules $Y(\lambda)$, $\lambda\in \mc{P}_n^{r,s}$ constitute a subset of the indecomposable modules for defining the basic algebra for $S_n(l)$. It follows that, with respect to an implicit grading choice, the smaller endomorphism algebra $S_n^b(r,s)$ is also positively graded.
\item[(2)] The proof given above is more explicit and exhibits a natural choice of gradings for $Y(\lambda)$'s arising from the categorical quantum $\mf{sl}_2$ action. The calculations above are motivated from similar considerations appearing in \cite[Section 9.2]{Lau1}.
\end{enumerate}
\end{rem}

\subsection{A basis for truncated modules}
We now describe a basis of $e_{\lambda}G(\lambda)$ inherited from the basis for $G(\lambda)$ constructed by Hu-Mathas \cite{HuMathas}.

\begin{lem}
\label{spansetspecialG}
If $\lambda=(0^a1^b0^c1^d)\in \mc{P}_n^{r,s}$, then the module $e_{\lambda}G(\lambda)$ is spanned by
$B_\lambda:= \bigcup_{j=0}^{\rm{min}(a,d)} B_j$ where 

\begin{align*}
B_j=
\left\{
e_{(b,d)} \psi_{w_1}^*
y_1^{u_1} \cdots y_{b+j}^{u_{b+j}}
y_{b+j+1}^{v_1} \cdots y_{b+d}^{v_{d-j}}
\psi_{w_2}
\Bigg|
\begin{gathered}
\begin{array}{l}
 w_1 \in \mf{S}_{b+j} / \mf{S}_b \times \mf{S}_j, \\
 r+s > u_1 > \cdots > u_{b+j} > s-1, \\
 s > v_1 > \cdots > v_{d-j} > -1 \\
 w_2 \in \mf{S}_{b+d}
 \end{array}
\end{gathered}
\right \}.
\end{align*}
\end{lem}

\begin{proof}
Let $\lambda=(0^a 1^b 0^c 1^d)$.
Proposition \ref{basisG} provides a basis of $G(0^a 1^b 0^c 1^d)$ of the form
\begin{equation*}
\{ \psi_{\mf{st}} | \mf{s} \in \Tab^{\lambda}(\mu), \mf{t} \in \Tab(\mu), \mu \in \mathcal{P}_n^l \}.
\end{equation*}
Recall that $\mf{s} \in \Tab^{\lambda}(\mu)$ means that $\mf{s}$ is a filling of a partition $\mu$ greater than or equal to the standard filling $\mf{t}^\lambda$ of $\lambda$.  

Let $j$ be the number of the last $d$ entries that we move into the first $a+b$ spots.
Clearly $0 \leq j \leq \rm{min}(a,b)$.  This means that there will be $b+j$ entries in the first $r=a+b$ spots and $d-j$ entries in the last $s=c+d$ spots.

Assume for a moment that all of the $b+j$ entries in the first $r$ spots are as far to the right as possible and all of the remaining $d-j$ entries in the last $s$ spots are as far to the right as possible.

Since we are look for a spanning set of $e_{\lambda}G(\lambda)$, there are restrictions on a filling $\mf{s}$ of such a partition so that $e_{(b,d)} \psi_\mf{s}^*$ is non-zero.  Namely,
$w_1$ must be a minimal length representative in $\mf{S}_{b+j} /\mf{S}_b \times \mf{S}_j$.  

Next, we have freedom to choose which of the first $r$ spots contains the first $b+j$ entries.  This determines the exponents $u_1, \ldots, u_{b+j}$.
Similarly the remaining $d-j$ entries may occupy any of the last $s$ spots.
This determines the exponents $v_1, \ldots, v_{d-j}$.

Finally, $\mf{t} \in \Tab(\mu)$ may be anything which corresponds to $w_2 \in \mf{S}_{b+d}$.
\end{proof}

\begin{prop}
\label{basisspecialG}
The spanning set $B_\lambda$ in Lemma \ref{spansetspecialG} is a basis of
$e_{\lambda}G(\lambda)$.
\end{prop}

\begin{proof}
Via Lemma \ref{lem-eG-equals-EY} we know that $e_{(b,d)}G(0^a 1^b 0^c 1^d) \cong 
\mf{F}^{(d)}Y(0^a 1^b 0^{c+d})$.
From Hu and Mathas we know the number of composition factors in a Jordan-Holder series for this module.  We thus easily compute that it has dimension
\begin{equation*}
\binom{a+b}{a}
\binom{a+c+d}{d}
(b+d)!.
\end{equation*}
The number of elements in the spanning set in Lemma \ref{spansetspecialG} is
\begin{equation*}
\sum_{j=0}^{\mathrm{min}(a,d)} 
\binom{c+d}{d-j}
\binom{b+j}{j}
\binom{a+b}{b+j}
(b+d)!.
\end{equation*}
The proposition follows from
\begin{equation}
\label{combidentity}
\binom{a+b}{a}
\binom{a+c+d}{d}
=
\sum_{j=0}^{\rm{min}(a,d)} 
\binom{c+d}{c+j}
\binom{b+j}{j}
\binom{a+b}{j+b}.
\end{equation}
Equation \eqref{combidentity} is easily proved by induction on $c$ where the base case $c=0$
\begin{equation*}
\binom{a+d}{a} = \sum_{j=0}^a \binom{d}{j} \binom{a}{j}
\end{equation*}
is Vandermonde's identity.
\end{proof}

\section{\texorpdfstring{$p$}{p}-DG Webster algebras}
\label{sec-Web}
\subsection{Definitions}
\label{subsec-Def-Web}
We begin by recalling the definition of a particular Webster algebra. More general versions of these algebras, which are associated with arbitrary finite Cartan data, can be found in \cite{Webcombined}.

\begin{defn}
\label{def-Webster-algebra}
Set $\mathbf{r}=(r_1, \ldots, r_m)$ with each $r_i \in \N$, and let $n \in \N$ be in the range $0 \leq n \leq \sum_{i=1}^m r_i$.
The Webster algebra $W_n(\mathbf{r})$ in this case is an algebra with $m$ red strands of widths (labels) $r_1,\dots, r_m$ (read from left to right) and $n$ black strands. Far away generators commute with each other. The black strands are allowed to carry dots, and red strands are not allowed to cross each other. We depict the local generators of this algebra by
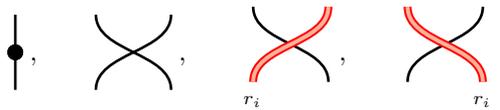
\[
\begin{DGCpicture}
\DGCstrand(0,0)(0,1)
\DGCdot{0.5}
\end{DGCpicture}
\ , \quad \quad
\begin{DGCpicture}
\DGCstrand(1,0)(0,1)
\DGCstrand(0,0)(1,1)
\end{DGCpicture}
\ , \quad \quad
\begin{DGCpicture}
\DGCstrand(1,0)(0,1)
\DGCstrand[Red](0,0)(1,1)[$^{r_i}$`{\ }]
\end{DGCpicture}
\ , \quad \quad
\begin{DGCpicture}
\DGCstrand(0,0)(1,1)
\DGCstrand[Red](1,0)(0,1)[$^{r_i}$`{\ }]
\end{DGCpicture}
\ .
\]
Multiplication in this algebra is vertical concatenation of diagrams.  When the colors of the boundary points of one diagram do not match the colors of the boundary points of  a second diagram, their product is taken to be zero.
The relations between the local generators are given by the usual nilHecke algebra relations among black strands
\begin{subequations}
\begin{gather}
\begin{DGCpicture}[scale=0.55]
\DGCstrand(1,0)(0,1)(1,2)
\DGCstrand(0,0)(1,1)(0,2)
\end{DGCpicture}
~= 0 \,  \quad \quad
\begin{DGCpicture}[scale=0.55]
\DGCstrand(0,0)(2,2)
\DGCstrand(1,0)(0,1)(1,2)
\DGCstrand(2,0)(0,2)
\end{DGCpicture}
~=~
\begin{DGCpicture}[scale=0.55]
\DGCstrand(0,0)(2,2)
\DGCstrand(1,0)(2,1)(1,2)
\DGCstrand(2,0)(0,2)
\end{DGCpicture}
\ , \\
\begin{DGCpicture}
\DGCstrand(0,0)(1,1)
\DGCdot{0.25}
\DGCstrand(1,0)(0,1)
\end{DGCpicture}
-
\begin{DGCpicture}
\DGCstrand(0,0)(1,1)
\DGCdot{0.75}
\DGCstrand(1,0)(0,1)
\end{DGCpicture}
~=~
\begin{DGCpicture}
\DGCstrand(0,0)(0,1)
\DGCstrand(1,0)(1,)
\end{DGCpicture}
~=~
\begin{DGCpicture}
\DGCstrand(1,0)(0,1)
\DGCdot{0.75}
\DGCstrand(0,0)(1,1)
\end{DGCpicture}
-
\begin{DGCpicture}
\DGCstrand(1,0)(0,1)
\DGCdot{0.25}
\DGCstrand(0,0)(1,1)
\end{DGCpicture} \ ,
\end{gather}
and local relations among red-black strands
\begin{equation}
\begin{DGCpicture}[scale=0.55]
\DGCstrand(1,0)(0,1)(1,2)
\DGCstrand[Red](0,0)(1,1)(0,2)[$^{r_i}$`{\ }]
\end{DGCpicture}
~=~
\begin{DGCpicture}[scale=0.55]
\DGCstrand(1,0)(1,2)
\DGCdot{1}[r]{$r_i$}
\DGCstrand[Red](0,0)(0,2)[$^{r_i}$`{\ }]
\end{DGCpicture}
\ , \quad \quad \quad
\begin{DGCpicture}[scale=0.55]
\DGCstrand(0,0)(1,1)(0,2)
\DGCstrand[Red](1,0)(0,1)(1,2)[$^{r_i}$`{\ }]
\end{DGCpicture}
~=~
\begin{DGCpicture}[scale=0.55]
\DGCstrand(0,0)(0,2)
\DGCdot{1}[l]{$r_i$}
\DGCstrand[Red](1,0)(1,2)[$^{r_i}$`{\ }]
\end{DGCpicture}
\ , \label{eqn-Webster-double-crossing}
\end{equation}
\begin{equation}
\begin{DGCpicture}[scale=0.5]
\DGCstrand(0,0)(2,2)
\DGCstrand(1,0)(0,1)(1,2)
\DGCstrand[Red](2,0)(0,2)[$^{r_i}$`{\ }]
\end{DGCpicture}
~=~
\begin{DGCpicture}[scale=0.5]
\DGCstrand(0,0)(2,2)
\DGCstrand(1,0)(2,1)(1,2)
\DGCstrand[Red](2,0)(0,2)[$^{r_i}$`{\ }]
\end{DGCpicture}
\ , \quad \quad
\begin{DGCpicture}[scale=0.5]
\DGCstrand(1,0)(0,1)(1,2)
\DGCstrand(2,0)(0,2)
\DGCstrand[Red](0,0)(2,2)[$^{r_i}$`{\ }]
\end{DGCpicture}
~=~
\begin{DGCpicture}[scale=0.5]
\DGCstrand(1,0)(2,1)(1,2)
\DGCstrand(2,0)(0,2)
\DGCstrand[Red](0,0)(2,2)[$^{r_i}$`{\ }]
\end{DGCpicture}
\ , 
\end{equation}
\begin{equation}
\begin{DGCpicture}
\DGCstrand(0,0)(1,1)
\DGCdot{0.25}
\DGCstrand[Red](1,0)(0,1)[$^{r_i}$`{\ }]
\end{DGCpicture}
~=~
\begin{DGCpicture}
\DGCstrand(0,0)(1,1)
\DGCdot{0.75}
\DGCstrand[Red](1,0)(0,1)[$^{r_i}$`{\ }]
\end{DGCpicture}
\ , \quad \quad
\begin{DGCpicture}
\DGCstrand(1,0)(0,1)
\DGCdot{0.25}
\DGCstrand[Red](0,0)(1,1)[$^{r_i}$`{\ }]
\end{DGCpicture}
~=~
\begin{DGCpicture}
\DGCstrand(1,0)(0,1)
\DGCdot{0.75}
\DGCstrand[Red](0,0)(1,1)[$^{r_i}$`{\ }]
\end{DGCpicture}
, 
\end{equation}
\begin{equation}
\begin{DGCpicture}[scale=0.5]
\DGCstrand(1,0)(3,2)
\DGCstrand(3,0)(1,2)
\DGCstrand[Red](2,0)(1,1)(2,2)[$^{r_i}$`{\ }]
\end{DGCpicture}
\ - \
\begin{DGCpicture}[scale=0.5]
\DGCstrand(1,0)(3,2)
\DGCstrand(3,0)(1,2)
\DGCstrand[Red](2,0)(3,1)(2,2)[$^{r_i}$`{\ }]
\end{DGCpicture}
\ = \
\sum_{a+b=r_i-1}~
\begin{DGCpicture}[scale=0.5]
\DGCstrand(1,0)(1,2)
\DGCdot{1}[l]{$a$}
\DGCstrand(3,0)(3,2)
\DGCdot{1}[r]{$b$}
\DGCstrand[Red](2,0)(2,2)[${^{r_i}}$`{\ }]
\end{DGCpicture}
\ ,
\end{equation}
together with the \emph{cyclotomic relation} that a black strand, appearing on the far left of any diagram, annihilates the entire picture:
\begin{equation}\label{eqn-cyclotomic-Web}
\begin{DGCpicture}
\DGCstrand(1,0)(1,1)
\DGCcoupon*(1.25,0.25)(1.75,0.75){$\cdots$}
\end{DGCpicture}
~=~0.
\end{equation}
\end{subequations}
When $m=2$ and $\mathbf{r}=(r,s)$, we will write $W_n(\mathbf{r})$ as $W_n(r,s)$.
\end{defn}

A family of differentials was introduced on $W_n(\mathbf{r})$ in \cite{KQ}. Among the family, a unique differential, up to conjugation by (anti)-automorphisms of $W_n(\mathbf{r})$, is determined in \cite{QiSussan}. This is the differential which is compatible with the natural categorical half-quantum $\mf{sl}_2$ action.  We give this choice in the lemma below.  The fact that $\partial$ is indeed a $p$-differential is a straightforward calculation.

\begin{lem}
\label{lem-dif-on-Web}
The Webster algebra has a $p$-DG structure given by
\begin{gather}
\label{diffonWeb}
\dif\left(~
\begin{DGCpicture}
\DGCstrand(0,0)(1,1)
\DGCstrand[Red](1,0)(0,1)[$^{r_i}$`{\ }]
\end{DGCpicture}
~\right)=0 , 
\quad \quad \quad
\dif\left(~
\begin{DGCpicture}
\DGCstrand(1,0)(0,1)
\DGCstrand[Red](0,0)(1,1)[$^{r_i}$`{\ }]
\end{DGCpicture}
~\right)=r_i
\begin{DGCpicture}
\DGCstrand(1,0)(0,1)
\DGCdot{0.75}
\DGCstrand[Red](0,0)(1,1)[$^{r_i}$`{\ }]
\end{DGCpicture}
\ .
\end{gather}
\end{lem}
\begin{proof}
See \cite[Section 4.2]{QiSussan}.
\end{proof}

To a sequence $\kappa$ of vertical lines read from left to right: red strand labeled by $r$, followed by $b$ black strands, followed by red strand labeled by $s$, followed by $d$ black strands, we associate an idempotent $\e(\kappa) \in W_n(r,s)$:
\begin{equation}\label{eqn-Webster-idem-1}
\e(\kappa)
=
\begin{DGCpicture}
\DGCstrand[Red](-1.5,-1)(-1.5,1)[$^r$]
\DGCPLstrand(-1,-1)(-1,1)[$^1$]
\DGCPLstrand(0.5,-1)(0.5,1)[$^{b-1}$]
\DGCPLstrand(1,-1)(1,1)[$^b$]
\DGCcoupon*(-0.8,-0.8)(0.3,0.8){$\cdots$}
\DGCstrand[Red](1.5,-1)(1.5,1)[$^s$]
\DGCPLstrand(2,-1)(2,1)[$^1$]
\DGCPLstrand(3.5,-1)(3.5,1)[$^{d-1}$]
\DGCPLstrand(4,-1)(4,1)[$^d$]
\DGCcoupon*(2.2,-0.8)(3.8,0.8){$\cdots$}
\end{DGCpicture} \ .
\end{equation}

\begin{rem}
By the cyclotomic condition, $\e(\kappa)=0$ if $b > r$. Similarly, one may show that $\e(\kappa)=0$ if $b+d>r+s$. This follows from relations in the Webster algebra and \cite[Lemma 4.3]{KK}.  Also see \cite{HoffLau} for details about nilpotency degrees of elements in the cyclotomic nilHecke algebra.
\end{rem}

A special sequence $\kappa_0$, having all the black strands to the right of the red strands, will play an important role later:
\begin{equation}
\e(\kappa_0)
=
\begin{DGCpicture}
\DGCstrand[Red](-1.5,-1)(-1.5,1)[$^r$]
\DGCstrand[Red](-1,-1)(-1,1)[$^s$]
\DGCPLstrand(-0.5,-1)(-0.5,1)[$^1$]
\DGCPLstrand(0,-1)(0,1)[$^{2}$]
\DGCPLstrand(1,-1)(1,1)[$^{n}$]
\DGCcoupon*(0.2,-0.8)(0.8,0.8){$\cdots$}
\end{DGCpicture} \ .
\end{equation}

\begin{lem}\label{lem-Web-big-block}
There is a $p$-DG isomorphism
\[
\begin{array}{ccc}
\nh_n^l & \lra & \e(\kappa_0)W_n(r,s)\e(\kappa_0),\\ && \\
\begin{DGCpicture}
\DGCstrand(0,0)(0,1)[$^1$]
\DGCstrand(1,0)(1,1)[$^n$]
\DGCcoupon(-0.1,0.3)(1.1,0.7){$x$}
\DGCcoupon*(0.1,0)(0.9,0.2){$\cdots$}
\DGCcoupon*(0.1,0.8)(0.9,1){$\cdots$}
\end{DGCpicture}
& \mapsto &
\begin{DGCpicture}
\DGCstrand[Red](-0.5,0)(-0.5,1)[$^r$]
\DGCstrand[Red](-0.25,0)(-0.25,1)[$^s$]
\DGCstrand(0,0)(0,1)[$^1$]
\DGCstrand(1,0)(1,1)[$^n$]
\DGCcoupon(-0.1,0.3)(1.1,0.7){$x$}
\DGCcoupon*(0.1,0)(0.9,0.2){$\cdots$}
\DGCcoupon*(0.1,0.8)(0.9,1){$\cdots$}
\end{DGCpicture} \ .
\end{array} 
\]
\end{lem}
\begin{proof}
See \cite[Proposition 5.31]{Webcombined}. The compatibility with $p$-differentials is clear from the definitions of the differentials on both sides.
\end{proof}

\begin{defn}\label{def-projective-Webster}
To any sequence $\kappa$ there is a projective right module $W_n(r,s)$-module
\begin{equation*}
Q(\kappa):=\e(\kappa) W_n(r,s).
\end{equation*}
The projective module $Q(\kappa)$ carries a $p$-DG structure inherited from that of $W_n(r,s)$, which makes it into a cofibrant $p$-DG right module since $\dif(\e(\kappa))=0$.
\end{defn}

\subsection{Connection to quiver Schur algebras}
\label{subsec-conn-Web}
Our next goal is to give a diagrammatic description of the two-tensor quiver Schur algebras as certain blocks of Webster algebras.

\begin{defn}
Let $\kappa$ be a sequence of vertical lines read from left to right:
red strand labeled by $r$, followed by $b$ black strands, followed by red strand labeled by $s$, followed by $d$ black strands. 
\begin{enumerate}
\item[(1)] The element  ${\theta}_{\kappa}$ is obtained from the diagram with minimal number of crossings and no dots which takes the sequence of black and red boundary points at the bottom of the diagram governed by $\kappa$ to the sequence at the top 
governed by the sequence $\kappa_0$.
\item[(2)] The element ${\theta}_{\kappa}^*$ is obtained by reflecting the diagram for  ${\theta}_{\kappa}$ in the horizontal axis at the bottom of the diagram.
\end{enumerate}
\end{defn}

\begin{example}
Suppose $r=2$, $s=1$ and $\kappa$ is the sequence: red strand labeled by $2$, followed by two black strands, followed by a red strand labeled by $1$.
Then
\[
{\theta}_{\kappa} =
\begin{DGCpicture}[scale=0.5]
\DGCstrand[Red](1,0)(1,2)[$^{2}$`{\ }]
\DGCstrand(2,0)(5,2)
\DGCstrand(3,0)(6,2)
\DGCstrand[Red]/u/(6,0)(2,2)/u/[$^{1}$`{\ }]
\end{DGCpicture}
\ , \quad \quad
{\theta}^*_{\kappa} =
\begin{DGCpicture}[scale=0.5]
\DGCstrand[Red](1,0)(1,2)[$^{2}$`{\ }]
\DGCstrand(5,0)(2,2)
\DGCstrand(6,0)(3,2)
\DGCstrand[Red]/u/(2,0)(6,2)/u/[$^{1}$`{\ }]
\end{DGCpicture}
\ .
\]
\end{example}

We start by constructing a collection of submodules of the $p$-DG Webster algebra that restricts to the direct sum of modules 
$$G:=\bigoplus_{\lambda\in \mc{P}_n^{r,s}} e_{\lambda}G(\lambda)$$ over $\nh_n^l$ as $p$-DG modules.

Let $\lambda=(0^a1^b0^c1^d) \in \mathcal{P}_n^{r,s}$ as in Definition \ref{speciallambda}.
To such a $\lambda$ we associate a sequence $\kappa_{\lambda}$ of vertical lines read from left to right:
red strand labeled by $r$, followed by $b$ black strands, followed by red strand labeled by $s$, followed by $d$ black strands.

Recall from the previous subsection that, for a multi-partition $ \lambda=(0^a1^b0^c1^d) \in \mathcal{P}_n^{r,s} $ we have attached to it an idempotent $\e(\kappa_\lambda)$ (equation \eqref{eqn-Webster-idem-1}) and the corresponding projective module $Q(\kappa_\lambda)$. Now, for any two natural numbers $b,d$, we construct a thick version of the idemptent of \eqref{eqn-Webster-idem-1} as follows. We first place a diagram for the thick idempotent $e_b$ (equation \eqref{eqn-thick-idemp}) to the right of a red strand labeled $r$, then, to the right, we place a second red strand labeled by $s$ and another idempotent $e_d$:
\begin{equation}\label{eqn-Webster-idem-2}
\e_{(b,d)}
=
\begin{DGCpicture}[scale=0.75]
\DGCstrand[Red](-1.5,-1)(-1.5,1)[`$_r$]
\DGCstrand[Green](0,0)(0,1)[`$_b$]
\DGCstrand/u/(-1,-1)(0,0)/u/
\DGCstrand/u/(0.5,-1)(0,0)/u/
\DGCstrand/u/(1,-1)(0,0)/u/
\DGCcoupon*(-0.8,-0.8)(0.3,-0.6){$\cdots$}
\DGCstrand[Red](1.5,-1)(1.5,1)[`$_s$]
\DGCstrand[Green](3,0)(3,1)[`$_d$]
\DGCstrand/u/(2,-1)(3,0)/u/
\DGCstrand/u/(3.5,-1)(3,0)/u/
\DGCstrand/u/(4,-1)(3,0)/u/
\DGCcoupon*(2.2,-0.8)(3.2,-0.6){$\cdots$}
\end{DGCpicture} \ .
\end{equation}
Then $\e_{(b,d)}$ generates a ($p$-DG) right projective module over $W_n(r,s)$.

\begin{defn}\label{def-thick-Webster-idempotent}
 Given $\lambda=(0^a1^b0^c1^d)\in \mc{P}_n^{r,s}$, we will associate with it $\e_\lambda:=\e_{(b,d)}$ and define the corresponding projective module 
 $$ Q(\lambda)=\e_{\lambda} Q(\kappa_{\lambda}) \subset Q(\kappa_{\lambda}). $$
The module $Q(\lambda)$ inherits a $p$-DG submodule structure from $Q(\kappa_{\lambda})$.
\end{defn}

Now, for any $b,d\in \N$, we consider
\[
\HOM_{W_n(r,s)}(Q(\kappa_0),Q(\lambda))\cong \e_{(b,d)}W_n(r,s)\e(\kappa_0).
\]
The right hand side is naturally a module over $\e({\kappa_0})W_n(r,s)\e(\kappa_0)\cong \nh_n^l$ (Lemma \ref{lem-Web-big-block}). Diagrammatically, the module consists of diagrams of the form
\begin{equation}
\e_{\lambda}W_n(r,s)\e(\kappa_0)
\cong
\left\{~
\begin{DGCpicture}
\DGCstrand[Red](0,-2)(0,0)[`$_r$]
\DGCstrand[Green](0.75,-0.5)(0.75,0)[`$_b$]
\DGCstrand(1,-2)(0.75,-0.5)
\DGCstrand(1.5,-2)(0.75,-0.5)
\DGCstrand[Red](0.5,-2)(1.5,0)[`$_s$]
\DGCstrand[Green](2,-0.5)(2,0)[`$_d$]
\DGCstrand(1.75,-2)(2,-0.5)
\DGCstrand(2.25,-2)(2,-0.5)
\DGCcoupon*(1.35,-1.3)(1.75,-1.4){$\cdots$}
\DGCcoupon(0.8,-1.8)(2.4,-1.5){$x$}
\end{DGCpicture}~\Bigg|
x\in \nh_n^l\right\}.
\end{equation}

In particular, for $\lambda=(0^a1^b0^c1^d)\in \mc{P}_n^{r,s}$, we have an isomorphism of right $\nh_n^l$-modules
\[
\e_{\lambda}W_n(r,s)\e(\kappa_0)\cong e_{\lambda}G(\lambda).
\]
This is given as follows. We ``sweep'' the thickness-$b$ strand to the right of the red strand labeled $s$, i.e., multiplying on top of the diagram the element ${\theta}_\kappa$. Then we simplify the diagrams obtained using relation \eqref{eqn-Webster-double-crossing}. Finally we utilize the isomorphism of Lemma \ref{lem-Web-big-block}:
\begin{equation}
\begin{DGCpicture}
\DGCstrand[Red](0,-2)(0,0)[`$_r$]
\DGCstrand[Green](0.75,-0.5)(0.75,0)[`$_b$]
\DGCstrand(1,-2)(0.75,-0.5)
\DGCstrand(1.5,-2)(0.75,-0.5)
\DGCstrand[Red](0.5,-2)(1.5,0)[`$_s$]
\DGCstrand[Green](2,-0.5)(2,0)[`$_d$]
\DGCstrand(1.75,-2)(2,-0.5)
\DGCstrand(2.25,-2)(2,-0.5)
\DGCcoupon*(1.35,-1.3)(1.75,-1.4){$\cdots$}
\DGCcoupon(0.8,-1.8)(2.4,-1.5){$x$}
\end{DGCpicture}
\mapsto
\begin{DGCpicture}
\DGCstrand[Red](0,-2)(0,0)[`$_r$]
\DGCstrand[Green](0.5,-.7)(1.25,0)[`$_b$]
\DGCstrand(1,-2)(0.5,-.7)
\DGCstrand(1.5,-2)(0.5,-.7)
\DGCstrand[Red]/u/(0.5,-2)(0.5,-1.7)/u/(1.25,-0.7)/u/(0.5,0)/u/[`$_s$]
\DGCstrand[Green](2,-0.7)(2,0)[`$_d$]
\DGCstrand(1.75,-2)(2,-.7)
\DGCstrand(2.25,-2)(2,-.7)
\DGCcoupon*(1.35,-1.3)(1.75,-1.4){$\cdots$}
\DGCcoupon(0.7,-1.8)(2.3,-1.5){$x$}
\end{DGCpicture}
~=~
\begin{DGCpicture}
\DGCstrand[Red](0,-2)(0,0)[`$_r$]
\DGCstrand[Green](1.25,-.7)(1.25,0)[`$_b$]
\DGCcoupon(0.75,-0.5)(1.75,-0.2){$_{y_1^{s}\cdots y_b^s}$}
\DGCstrand(1,-2)(1.25,-.7)
\DGCstrand(1.5,-2)(1.25,-.7)
\DGCstrand[Red](0.5,-2)(0.5,0)[`$_s$]
\DGCstrand[Green](2,-0.7)(2,0)[`$_d$]
\DGCstrand(1.75,-2)(2,-.7)
\DGCstrand(2.25,-2)(2,-.7)
\DGCcoupon*(1.45,-1.2)(1.85,-1.3){$\cdots$}
\DGCcoupon(0.8,-1.8)(2.4,-1.5){$x$}
\end{DGCpicture}
~\mapsto~
\begin{DGCpicture}
\DGCstrand[Green](1.25,-.7)(1.25,0)[`$_b$]
\DGCcoupon(0.75,-0.5)(1.75,-0.2){$_{y_1^{s}\cdots y_b^s}$}
\DGCstrand(1,-2)(1.25,-.7)
\DGCstrand(1.5,-2)(1.25,-.7)
\DGCstrand[Green](2,-0.7)(2,0)[`$_d$]
\DGCstrand(1.75,-2)(2,-.7)
\DGCstrand(2.25,-2)(2,-.7)
\DGCcoupon*(1.45,-1.2)(1.85,-1.3){$\cdots$}
\DGCcoupon(0.8,-1.8)(2.4,-1.5){$x$}
\end{DGCpicture} \ .
\end{equation}
The sweeping map is always an injection \cite[Lemma 5.25]{Webcombined}. Furthermore, it is a $p$-DG homomorphism since, by Lemma \ref{lem-dif-on-Web}, we have
\[
\dif\left(~
\begin{DGCpicture}
\DGCstrand[Green](0,0)(1,1)
\DGCstrand[Red](1,0)(0,1)
\end{DGCpicture}
~\right)
= 0.
\]
We are now ready to establish the following.

\begin{prop}
For each $\lambda=(0^a1^b0^c1^d)\in \mc{P}_n^{r,s}$, there is an isomorphism of right $p$-DG modules over $\nh_n^l$
\[
\HOM_{W_n(r,s)}(Q(\kappa_0),Q(\lambda))\cong \e_{\lambda}W_n(r,s)\e(\kappa_0)\cong e_{\lambda}G(\lambda).
\]
\end{prop}
\begin{proof}
By the discussion above, it suffices to show that the element
$e_{(b,d)}(y_1\cdots y_b)^s$ generates the same $p$-DG right ideal of $\nh_n^l$ as $e_\lambda y^\lambda$. It is clear that $e_\lambda G(\lambda) \subset e_{(b,d)}(y_1\cdots y_b)^s\nh_n^l $. The reverse inclusion holds because
\[
e_{(b,d)}(y_1\cdots y_b)^s=e_{(b,d)}\cdot e_{(b,d)} (y_1\cdots y_b)^s= e_{(b,d)}(y_1\dots y_b)^s e_{(b,d)}=e_{\lambda}y^\lambda \psi_{w_{(b,d)}},
\]
where $w_{(b,d)}$ is the longest element in the parabolic subgroup $\mf{S}_b\times \mf{S}_d\subset \mf{S}_{n}$.
\end{proof}

Summing over $\lambda\in \mc{P}_n^{r,s}$, we set
\begin{equation}
Q:=\bigoplus_{\lambda\in \mc{P}_n^{r,s}}Q(\lambda).
\end{equation}
We have shown that
\begin{equation}
\HOM_{W_n(r,s)}(Q(\kappa_0), Q)
\cong \bigoplus_{\lambda\in \mc{P}_n^{r,s}}e_\lambda G(\lambda)=G.
\end{equation}

\begin{thm}\label{thm-iso-to-Webster-block}
There is an isomorphism of $p$-DG algebras 
\begin{equation*}
S_n(r,s) \cong \END_{W_n(r,s)}(Q).
\end{equation*}
\end{thm}

\begin{proof}
The result follows by applying the general results of Section \ref{sec-double}. More specifically, observe that the $\nh_n^l$-module $G=\oplus_{\lambda\in \mc{P}_n^{r,s}}e_\lambda G(\lambda)$ is a faithful representation. Indeed, by Proposition \ref{prop-projective-as-summand}, the module $e_{(0^{l-n}1^n)} G(0^{l-n}1^n)$ contains $Y(0^{l-n}1^n)$. By Theorem \ref{thm-exactness-Soergel}, the functor
\[
\HOM_{S_n(r,s)}(G,-):(S_n(r,s),\dif)\dmod\lra (\nh_n^l,\dif)\dmod
\]
is fully faithful on cofibrant summands of $S_n(r,s)$. 
By Lemma \ref{lem-two-duals-isomorphic}, we have
\[
\HOM_{S_n(r,s)}(G,S_n(r,s))\cong \HOM_{\nh_n^l}(G,\nh_n^l)\cong G^* \cong G
\]
since $G$ is graded self-dual (Proposition \ref{prop-G-self-dual}). 
Now the result follows from the previous Proposition identifying the space with $\HOM_{W_n(r,s)}(Q(\kappa_0),Q)$.
\end{proof}

\subsection{Examples}
\label{examplesubsection}
In this subsection, we give some examples of the two-tensor quiver Schur algebras and point out a subtle difference with the $p$-DG (thin) Webster algebra. 

\paragraph{Example: $s=0$.} We first consider this very special case of the two-tensor quiver Schur algebra, and compare the results with those in Section \ref{subsec-cat-simples}.

In this situation, the set $\mc{P}_n^{r,0}=\mc{P}_n^{l,0}$ consists of a unique partition $\lambda_n=(0^{l-n}1^n)$ for each $n\in \{0,\dots, l\}$, associated with which are 
$$y^{\lambda_n}= y_1^{n-1}y_2^{n-2}\cdots y_n^0, \quad \quad \quad e_{\lambda_n}=y_1^{n-1}y_2^{n-2}\cdots y_n^0\psi_{w_0}=e_n.$$ 
We then have, by Lemma \ref{Yindecomspecial}, that
\[
Y(\lambda_n)\cong e_{\lambda_n}G(\lambda_n)= e_{n}y_1^{n-1}y_2^{n-2}\cdots y_n^0\nh_n^l.
\]
As in the proof of Lemma \ref{lem-eG-equals-EY}, we see that 
\begin{equation}
Y(\lambda_n)\cong e_{n}\nh_n^l
\end{equation}
is in fact, up to grading shifts, the unique indecomposable projective module over $\nh_n^l$ equipped with the right ideal $p$-DG module structure. It follows that
\begin{equation}
S_n(l,0)=\END_{\nh_n^l}(Y(\lambda_n))\cong \mH^*(\mathrm{Gr}(n,l)),
\end{equation}
with the latter ring identified with the center of $\nh_n^l$. The diagrammatics of the two-tensor quiver Schur algebra specializes into
\begin{equation}
S_n(l,0)\cong \left\{
~\begin{DGCpicture}
\DGCstrand[Red](0.2,0)(0.2,1)[$^l$`{\ }]
\DGCstrand[Green](1,0)(1,1)[$^n$`{\ }]
\DGCcoupon(0.6,0.3)(1.4,0.7){$x$}
\end{DGCpicture}
~\Bigg|
x\in \mH^*(\mathrm{Gr}(n,l))
\right\}.
\end{equation}

\begin{rem} 
Since the cohomology ring of $\mathrm{Gr}(n,l)$ is a positively graded local $p$-DG algebra, Theorem \ref{thm-K-group-positive} implies immediately that 
$$K_0(\mc{D}^c(S_n(l,0)))\cong \mathbb{O}_p[Y(\lambda_n)],$$
where $[Y(\lambda_n)]$ stands for the symbol of the $p$-DG module $Y(\lambda)$ in the Grothendieck group. One should compare this with Remark \ref{rmk-restriction-on-l} for cyclotomic nilHecke algebras. 

As we will see as in the next section (Theorem \ref{mainthm}), the direct sum of the derived categories
\[
\mc{D}^c(S(l,0)):=\bigoplus_{n=0}^l\mc{D}^c(S_n(l,0))
\]
categorifies the Weyl module for $\dot{U}_{\mathbb{O}_p}$ at a prime root of unity, while $\mc{D}^c(\nh^l)$ only categorifies the submodule generated by the highest weight vector.
\end{rem}

\paragraph{Example: $r=2$ and $s=1$.}

Consider the data $r=2, s=1$ and $n=2$, which is a continuation of Example \ref{l=3n=2}.  In this case $\mathcal{P}_2^{2,1}$ contains only two elements
\begin{equation*}
\mu=(~\yng(1)~,~\yng(1)~,~\emptyset~) \hspace{.5in} \lambda=(~\emptyset~,~\yng(1)~,~\yng(1)~).
\end{equation*}
There is a third element (Example \ref{eg-23partitions})
$$\gamma=(~\yng(1)~,\emptyset,~\yng(1)~) \in \mathcal{P}_2^3$$ 
which is not in 
$\mathcal{P}_2^{2,1}$. The corresponding generating objects of the category $\mc{NH}_2^{2,1}$ are:
\begin{equation*}
e_{2}G(\mu) \cong G(\mu) \hspace{.5in} e_{(1,1)}G(\lambda)=G(\lambda).
\end{equation*}
A basis of the endomorphism algebra of $G(\mu) \oplus G(\lambda)$ is given by
\begin{align*}
\HOM_{\nh_2^3}(G(\mu),G(\mu)) &= \{ \Psi^{\mu \mu}_{ee \mu} \} \\
\HOM_{\nh_2^3}(G(\mu),G(\lambda)) &= \{ \Psi^{\lambda \mu}_{ee \mu}, \Psi^{\lambda \mu}_{s_1 e \mu} \} \\
\HOM_{\nh_2^3}(G(\lambda),G(\mu)) &= \{ \Psi^{\mu \lambda}_{ee \mu}, \Psi^{\mu \lambda}_{e s_1\mu} \} \\
\HOM_{\nh_2^3}(G(\lambda),G(\lambda)) &= 
\{ \Psi^{\lambda \lambda}_{ee \lambda}, \Psi^{\lambda \lambda}_{ee \gamma}, \Psi^{\lambda \lambda}_{ee \mu}, \Psi^{\lambda \lambda}_{s_1 e \mu},
\Psi^{\lambda \lambda}_{e s_1 \mu}, \Psi^{\lambda \lambda}_{s_1 s_1 \mu}  \}.
\end{align*}

The subalgebra $\END_{\nh_2^3}(G(\mu) \oplus G(\lambda))$ is isomorphic to the subalgebra $\END_{W_2(2,1)}(Q(\mu) \oplus Q(\lambda))$ inside the Webster algebra via the isomorphism of Theorem \ref{thm-iso-to-Webster-block}. It is explicitly given as follows. Let us recall from equation \eqref{eqn-thick-idemp} that a thickness-$2$ strand represents the idempotent $e_2$, and with the induced endomorphism differential acting by zero: 
\[
\begin{DGCpicture}
\DGCstrand[Green](0,0)(0,1)[$^2$`{\ }]
\end{DGCpicture}
=
\begin{DGCpicture}
\DGCstrand(0,0)(1,1)[{\ }`{\ }]
\DGCstrand(1,0)(0,1)
\DGCdot{0.75}
\end{DGCpicture} \ ,
\quad \quad \quad \quad \quad 
\dif\left(~
\begin{DGCpicture}
\DGCstrand[Green](0,0)(0,1)[$^2$`{\ }]
\end{DGCpicture}
~\right)=0.
\]
Strands of thickness $1$ are depicted by thin strands.

Then we have the following identifications defined on bases:
\[
\bullet~\HOM_{\nh_2^3}(G(\mu),G(\mu)) \cong \HOM_{W_2(2,1)}(Q(\mu),Q(\mu)):
\]

\[
\Psi_{ee \mu}^{\mu \mu} \mapsto
\begin{DGCpicture}[scale=0.5]
\DGCstrand[Red](0,0)(0,2)[$^{2}$`{\ }]
\DGCstrand[Red](3,0)(3,2)[$^{1}$`{\ }]
\DGCstrand[Green](1.5,0)(1.5,2)[$^{2}$`{\ }]
\end{DGCpicture} \ .
\]

\[
\bullet~\HOM_{\nh_2^3}(G(\lambda),G(\mu)) \cong \HOM_{W_2(2,1)}(Q(\lambda),Q(\mu)):
\]

\[
\Psi^{\mu \lambda}_{e s_1 \mu}  \mapsto
\begin{DGCpicture}[scale=0.5]
\DGCstrand[Red](0,0)(0,2)[`$_{2}$]
\DGCstrand[Green](1.5,1)(1.5,2)[`$_2$]
\DGCstrand/u/(1,0)(1.5,1)/u/
\DGCstrand/u/(3,0)(1.5,1)/u/
\DGCstrand[Red](2,0)(3,2)[`$_{1}$]
\end{DGCpicture}
\ , \quad \quad \quad
\Psi^{\mu \lambda}_{e e \mu}  \mapsto
\begin{DGCpicture}[scale=0.5]
\DGCstrand[Red](0,0)(0,2)[`$_{2}$]
\DGCstrand[Green](1.5,1)(1.5,2)[`$_2$]
\DGCstrand/u/(1,0)(1.5,1)/u/
\DGCdot{0.5}
\DGCstrand/u/(3,0)(1.5,1)/u/
\DGCstrand[Red](2,0)(3,2)[`$_{1}$]
\end{DGCpicture} \ .
\]

\[
\bullet~\HOM_{\nh_2^3}(G(\mu),G(\lambda)) \cong \HOM_{W_2(2,1)}(Q(\mu),Q(\lambda)):
\]

\[
\Psi^{\lambda \mu}_{s_1 e \mu} \mapsto
\begin{DGCpicture}[scale=0.5]
\DGCstrand[Red](0,0)(0,2)[$^{2}$`{\ }]
\DGCstrand[Green](1.5,0)(1.5,1)[$^2$]
\DGCstrand/u/(1.5,1)(1,2)/u/
\DGCstrand/u/(1.5,1)(3,2)/u/
\DGCstrand[Red](3,0)(2,2)[$^{1}$`{\ }]
\end{DGCpicture}
\ , \quad \quad 
\Psi^{\lambda \mu}_{e e \mu} \mapsto
\begin{DGCpicture}[scale=0.5]
\DGCstrand[Red](0,0)(0,2)[$^{2}$`{\ }]
\DGCstrand[Green](1.5,0)(1.5,1)[$^2$]
\DGCstrand/u/(1.5,1)(1,2)/u/
\DGCdot{1.5}
\DGCstrand/u/(1.5,1)(3,2)/u/
\DGCstrand[Red](3,0)(2,2)[$^{1}$`{\ }]
\end{DGCpicture} \ .
\]

\[
\bullet~\HOM_{\nh_2^3}(G(\lambda),G(\lambda)) \cong \HOM_{W_2(2,1)}(Q(\lambda),Q(\lambda)):
\]

\[
\Psi^{\lambda \lambda}_{ee \lambda} \mapsto
\begin{DGCpicture}[scale=0.5]
\DGCstrand[Red](0,0)(0,2)[$^{2}$`{\ }]
\DGCstrand(1,0)(1,2)
\DGCstrand(3,0)(3,2)
\DGCstrand[Red](2,0)(2,2)[$^{1}$`{\ }]
\end{DGCpicture}
\ , \quad \quad
\Psi^{\lambda \lambda}_{ee \gamma} \mapsto
\begin{DGCpicture}[scale=0.5]
\DGCstrand[Red](0,0)(0,2)[$^{2}$`{\ }]
\DGCstrand(1,0)(1,2)
\DGCdot{1}
\DGCstrand(3,0)(3,2)
\DGCstrand[Red](2,0)(2,2)[$^{1}$`{\ }]
\end{DGCpicture} \ ,
\]

\[
\Psi^{\lambda \lambda}_{s_1 s_1 \mu} \mapsto
\begin{DGCpicture}[scale=0.5]
\DGCstrand[Red](0,0)(0,2)[$^{2}$`{\ }]
\DGCstrand(1,0)(3,2)
\DGCstrand(3,0)(1,2)
\DGCstrand[Red](2,0)(3,1)(2,2)[$^{1}$`{\ }]
\end{DGCpicture}
\ , \quad \quad 
\Psi^{\lambda \lambda}_{e s_1 \mu} \mapsto
\begin{DGCpicture}[scale=0.5]
\DGCstrand[Red](0,0)(0,2)[$^{2}$`{\ }]
\DGCstrand(1,0)(3,2)
\DGCstrand(3,0)(1,2)
\DGCdot{1.5}
\DGCstrand[Red](2,0)(3,1)(2,2)[$^{1}$`{\ }]
\end{DGCpicture}
\ , \quad \quad 
\Psi^{\lambda \lambda}_{s_1 e \mu} \mapsto
\begin{DGCpicture}[scale=0.5]
\DGCstrand[Red](0,0)(0,2)[$^{2}$`{\ }]
\DGCstrand(1,0)(3,2)
\DGCdot{.5}
\DGCstrand(3,0)(1,2)
\DGCstrand[Red](2,0)(3,1)(2,2)[$^{1}$`{\ }]
\end{DGCpicture}
\ , \quad \quad 
\Psi^{\lambda \lambda}_{ee \mu} \mapsto
\begin{DGCpicture}[scale=0.5]
\DGCstrand[Red](-2,0)(-2,2)[$^{2}$`{\ }]
\DGCstrand(-1,0)(-1,2)
\DGCdot{1}
\DGCstrand(1,0)(0,1)(1,2)
\DGCstrand[Red](0,0)(1,1)(0,2)[$^{1}$`{\ }]
\end{DGCpicture}
\ .
\]

It is straightforward to check that this vector space isomorphism is in fact an algebra isomorphism as well. We leave the details as an exercise to the reader.

\section{A categorification of a tensor product}
\label{sec-main-thm}
\subsection{The main theorem}

We start by collecting some easily deduced consequences of the previous sections. For convenience, we define
\begin{equation}
(S(r,s),\dif)\dmod:=\bigoplus_{n=0}^l \left(S_n(r,s),\dif\right)\dmod, \quad \quad (S^b(r,s),\dif)\dmod:=\bigoplus_{n=0}^l (S_n^b(r,s),\dif)\dmod
\end{equation}
where we recall that $S_n^b(r,s)$ is the basic algebra introduced in Definition \ref{def-basic-two-tensor-algebra}.
Likewise, we define their $p$-DG derived categories as
\begin{equation}
\mc{D}(S(r,s)):=\bigoplus_{n=0}^l \mc{D}(S_n(r,s)),\quad \quad \quad \mc{D}(S^b(r,s)):=\bigoplus_{n=0}^l \mc{D}(S^b_n(r,s)).
\end{equation}

We consider a collection of cofibrant $p$-DG modules over $S^b_n(r,s)$ as follows. Let $\mu\in \mc{P}_n^{r,s}$. Then the composition of the natural projection and inclusion defines an idempotent
\begin{equation}
\xi_\mu: \bigoplus_{\lambda\in \mc{P}_n^{r,s}}Y(\lambda)\rightarrow Y(\mu)\hookrightarrow \bigoplus_{\lambda\in \mc{P}_n^{r,s}}Y(\lambda)
\end{equation}
in the $p$-DG algebra $S_n^{b}(r,s)$. Thus we define the left $(S_n^b(r,s),\dif)$-module
\begin{equation}
P^b(\mu):= S_n^b(r,s)\xi_{\mu}\cong \HOM_{\nh_n^l}\left(Y(\mu), \bigoplus_{\lambda\in \mc{P}_n^{r,s}}Y(\lambda)\right).
\end{equation}
It is clearly a direct summand in $S_n^b(r,s)$, and hence is a cofibrant $p$-DG module.

\begin{prop}\label{prop-main-thm-basic-version}
\begin{enumerate}
\item[(i)] There is an action of the $p$-DG $2$-category $(\dot{\mathcal{U}},\dif)\dmod$ on 
$(S^b(r,s),\dif)\dmod$. Under localization, the action induces an action of the derived
$p$-DG $2$-category $\mc{D}(\dot{\mc{U}})$ on $\mc{D}(S^b(r,s))$.
\item[(ii)] On the level of Grothendieck groups, there is an isomorphism of modules over $K_0(\mc{D}^c(\dot{\mc{U}}))\cong\dot{U}_{\mathbb{O}_p}$: 
\begin{equation*}
K_0(\mathcal{D}^c(\oplus_{n=0}^l S^b_n(r,s))) \cong V_r \otimes_{\mathbb{O}_p} V_s
\end{equation*}
where $V_r$ and $V_s$ are the Weyl modules of rank
$r+1$ and $s+1$ over $\mathbb{O}_p$ respectively.
\item[(iii)]For each $\lambda=(0^a 1^b 0^c 1^d) \in \mc{P}_n^{r,s}$, the indecomposable module $P^b(\lambda)$ descends to the canonical basis element $v_b \diamond v_d$ in the Grothendieck group.
\end{enumerate}
\end{prop}

\begin{proof}
By Proposition \ref{prop-pdg-extension} and Theorem \ref{EFact}, the action of $(\dot{\mathcal{U}},\dif)$ on $\oplus_{n=0}^l(\nh_n^l,\dif)\dmod $ (Theorem \ref{catofVl}) extends to the category $\oplus_{n=0}^l(S_n^b(r,s),\dif)\dmod$.

The second statement follows by applying Theorem \ref{homspaceY1Y2}. The positivity of the $p$-DG algebra $S_n^b(r,s)$ allows one to compute the Grothendieck group directly from Theorem \ref{thm-K-group-positive}.

The third statement is a consequence of the construction of 
 $Y(0^a 1^b 0^c 1^d)$ as a step-by-step categorification of the canonical basis elements.  Recall from Remark \ref{canbasisremark} that 
\begin{equation*}
v_b \diamond v_d =
\begin{cases}
F^{(d)} E^{(a)} (v_r \otimes v_0) & \text{ if } b \leq c, \\
E^{(a)} F^{(d)} (v_r \otimes v_0) & \text{ if } b \geq c.
\end{cases}
\end{equation*}
Also recall from Theorem \ref{classificationprop} that 
\begin{equation*}
Y(\lambda) =
\begin{cases}
\mf{F}^{(d)} \mf{E}^{(a)} Y(1^r 0^s) & \text{ if } b \leq c, \\
\mf{E}^{(a)} \mf{F}^{(d)} Y(1^r 0^s) & \text{ if } b \geq c.
\end{cases}
\end{equation*}
Thus the map 
\begin{equation*}
K_0(\mathcal{D}^c(\oplus_{n=0}^l S^b_n(r,s))) \rightarrow V_r \otimes_{\mathbb{O}_p} V_s
\end{equation*}
sending $[P^b(\lambda)]$ to $v_b \otimes v_d$ intertwines the action of $[\mf{E}]$ and $[\mf{F}]$ on $K_0(\mathcal{D}^c(\oplus_{n=0}^l S^b_n(r,s)))$ with the action of
$E$ and $F$ on $V_r \otimes_{\mathbb{O}_p} V_s$.
\end{proof}

We then extend the theorem to the full two-tensor quiver Schur algebra. This requires us to set up a chain of intermediate $p$-DG algebras between $S_n^b(r,s)$ and $S_n(r,s)$, and show that they are all $p$-DG Morita equivalent. First we make a definition.

\begin{defn}\label{def-canonical-basis-mod}
For each $\mu=(0^a1^b0^c1^d)\in\mc{P}_n^{r,s}$, the $p$-DG module
\[
P(\mu):=\HOM_{\nh_n^l}\left(Y(\mu),\bigoplus_{\lambda\in \mc{P}_n^{r,s}}e_\lambda G(\lambda)\right)
\]
will be referred to as the \emph{canonical module} associated to the partition $\mu$.
\end{defn}

Choose an arbitrary total order ``$>$'' which refines the partial order of Definition \ref{def-order-on-nh-partition}. For each fixed $\mu\in \mc{P}_n^{r,s}$, we take the module
\[
G^{\geq \mu}:=\left(\bigoplus_{\lambda < \mu} Y(\lambda)\right)\bigoplus \left(\bigoplus_{\lambda \geq \mu} e_\lambda G(\lambda)\right).
\]
In other words, if $\mu > \nu$ are two neighboring terms in the total order, then $G^{\geq \nu}$ is obtained from $G^{\geq \mu}$ by replacing the summand $Y(\nu)\subset G^{\geq \mu}$ by $e_\nu G(\nu)$. By Proposition \ref{prop-multiplicity-Y-in-G}, $Y(\nu)$ appears in $e_\nu G(\nu)$ as a $p$-DG quotient module with multiplicity one. Furthermore, each time when making this replacement, we only introduce an extra filtered $p$-DG submodule $G^\prime$ whose subquotients are grading shifts of $Y(\lambda)$ with $\lambda > \nu$. Let us write 
\[
e_{\nu}G(\nu)=Y(\nu)\bigoplus \left(\bigoplus_{\lambda>\nu} g_{\lambda} Y(\lambda)\right)=Y(\nu)\oplus G^\prime,
\]
so that there is a filtered direct sum
\[
G^{\geq \nu} \cong  G^\prime \oplus G^{\geq \mu}.
\]
Here it is understood the summands are associated graded pieces of the $p$-DG filtration.

Now let us compare the endomorphism $p$-DG algebras of $G^{\geq \mu}$ and $G^{\geq \nu}$ over $\nh_n^l$. Call the former $S^{\geq \mu}$ and the latter $S^{\geq \nu}$. We may identify $S^{\geq \nu}$ as a ``block matrix'' involving the $S^{\geq \mu}$ as follows (the possible external differentials between blocks are indicated on the arrows):
\[
S^{\geq \nu}=\END(G^{\geq \nu})\cong 
\left(
\begin{gathered}
\xymatrix{
\END(G^\prime)\ar[r]^-\dif & \HOM(G^\prime, G^{\geq \mu})\\
 \HOM(G^{\geq \mu}, G^\prime)\ar[r]^-{\dif}\ar[u]^-{\dif} & S^{\geq \mu}\ar[u]_-{\dif}
}
\end{gathered}
\right)
\]
The last column of the formal matrix is isomorphic to $\HOM(G^{\geq \nu},G^{\geq \mu})$, and the last row may be identified with $\HOM(G^{\geq \mu},G^{\geq \nu})$, which are respectively $p$-DG bimodules over $(S^{\geq \mu},S^{\geq \nu})$ and $(S^{\geq \mu},S^{\geq \nu})$. We are then reduced to the situation of Proposition \ref{prop-p-dg-Morita}, which inductively allows us to conclude the following result.

\begin{thm}Fix two numbers $r,s\in \N$.
\label{mainthm}
\begin{enumerate}
\item[(i)] There is $p$-DG Morita equivalence between $(S(r,s),\dif)\dmod$ and $(S^b(r,s),\dif)\dmod$. For each fixed $n\in \{0,1,\dots, n\}$, the functor is given by tensor product with the $p$-DG bimodule over $(S_n(r,s),S_n^b(r,s))$:
\[
\HOM_{\nh_l}\left(\bigoplus_{\lambda\in \mc{P}_n^{r,s}}Y(\lambda),\bigoplus_{\mu\in \mc{P}_n^{r,s}}G(\mu)\right)\otimes_{S_n^b(r,s)}(-):(S_n^b(r,s),\dif)\dmod\lra (S_n(r,s),\dif)\dmod.
\]
Summing over $n$, the equivalences induce derived equivalence between $\mc{D}(S(r,s))$ and $\mc{D}(S^b(r,s))$.
\item[(ii)]The $p$-DG $2$-category $(\dot{\mathcal{U}},\dif)\dmod$ acts on 
$(S(r,s),\dif)\dmod$, inducing an action of the derived
$p$-DG $2$-category $\mc{D}(\dot{\mc{U}})$ on $\mc{D}(S(r,s))$. The derived action categorifies the action of $\dot{U}_{\mathbb{O}_p}$ on the tensor product representation $V_r\otimes_{\mathbb{O}_p}V_s$.
\item[(iii)] For each $\mu=(0^a1^b0^c1^d)\in\mc{P}_n^{r,s}$, the canonical module $P(\mu)$
descends in the Grothendieck group $K_0(S_n(r,s))$ to the canonical basis element $v_b\diamond v_d$.
\end{enumerate}
\end{thm}

\begin{proof}
The first statement now follows easily by an induction on the $\lambda \in \mc{P}_n^{r,s}$ via Proposition \ref{prop-p-dg-Morita} with respect to the total ordering chosen above. The induction also shows that the $(S_n(r,s),S_n^b(r,s))$-bimodule
\[
\HOM_{\nh_l}\left(\bigoplus_{\lambda\in \mc{P}_n^{r,s}}Y(\lambda),\bigoplus_{\mu\in \mc{P}_n^{r,s}}G(\mu)\right)
\]
is cofibrant as both a left and right $p$-DG module. Thus the derived tensor functor is equal to the underived tensor product.

The categorical $\dot{U}_{\mathbb{O}_p}$-action extension result follows, similar as in the proof of the previous proposition, by applying Proposition \ref{prop-pdg-extension} to the current situation.

Finally, the above functor, when applied to the cofibrant module $P^b(\mu)$ for a fixed $\mu\in \mc{P}_n^{r,s}$, truncates the factor $\oplus_{\lambda} Y(\lambda)$ in the definition of the bimodule through the idempotent $\xi_\mu$, i.e., it replaces $\oplus_{\lambda}Y(\lambda)$ just by $Y(\mu)$. The claim follows.
\end{proof}

\subsection{Stratified structure}
\label{subsec-strat}
The two-tensor quiver Schur algebra is naturally equipped with a stratified structure (see, for instance, \cite{KleAHW}) that is compatible with the differential. We sketch the construction here, starting with the basic algebra case. Recall that, for $\lambda\in \mc{P}_n^{r,s}$, we have associated with the cofibrant $p$-DG module over $S_n^b(r,s)$:
\[
P^b(\lambda):=\HOM_{\nh_n^l}\left(Y(\lambda),\bigoplus_{\mu\in \mc{P}_n^{r,s}}Y(\mu)\right).
\]

\begin{defn}\label{def-standard-mod}
Let $\lambda \in \mathcal{P}_n^{r,s}$. 
Define the standard left $S^b_n(r,s)$-module
$\Delta^b(\lambda)=P^b(\lambda)/P^{>\lambda}$
where $P^{>\lambda}$ is the left submodule generated by
\begin{equation*}
\HOM_{\nh_n^l}\left(Y(\lambda), \bigoplus_{\substack{\mu \in \mathcal{P}_n^{r,s} \\ \mu > \lambda}} Y(\mu)\right).
\end{equation*}
\end{defn}

We would like to construct a filtration on each projective module $P^b(\lambda)$ which has subquotients of the form $\Delta^b(\gamma)$ for $\gamma \geq \lambda$.

\begin{defn}\label{def-stratified-mod}
\begin{enumerate}
\item[(1)] Let $ \HOM_{\nh_n^l}^{\geq \gamma}(Y(\lambda),Y(\mu))$ be the left submodule of $P^b(\lambda)$ generated by all maps in  
$ \HOM_{\nh_n^l}(Y(\lambda),Y(\mu))$  
which factor through $Y(\gamma')$ for some $\gamma' \geq \gamma$.  
Then define the submodule $P^{\geq \gamma}(\lambda) \subset P^b(\lambda)$ to be the left submodule generated by all maps factoring through $Y(\gamma')$ for some $\gamma' \geq \gamma$.  That is,
\begin{equation*}
P^{\geq \gamma}(\lambda)=
\HOM^{\geq \gamma}_{\nh_n^l}\left(Y(\lambda), \bigoplus_{\substack{\mu \in \mathcal{P}_n^{r,s} \\ \mu > \lambda}} Y(\mu)\right).
\end{equation*}
\item[(2)] Let $ \HOM_{\nh_n^l}^{> \gamma}(Y(\lambda),Y(\mu))$ be the left submodule of $P(\lambda)$ generated by all maps in  
$ \HOM_{\nh_n^l}(Y(\lambda),Y(\mu))$  
which factor through $Y(\gamma')$ for some $\gamma' > \gamma$.  
Then define the submodule $P^{\geq \gamma}(\lambda) \subset P^b(\lambda)$ to be the left submodule generated by all maps factoring through $Y(\gamma')$ for some $\gamma' > \gamma$.  That is,
\begin{equation*}
P^{> \gamma}(\lambda)=
\HOM^{> \gamma}_{\nh_n^l}\left(Y(\lambda), \bigoplus_{\substack{\mu \in \mathcal{P}_n^{r,s} \\ \mu > \lambda}} Y(\mu)\right).
\end{equation*}
\end{enumerate}
\end{defn}

\begin{lem}
\label{filtpdgstable}
The submodules $P^{\geq \gamma}(\lambda)$ and $P^{> \gamma}(\lambda)$
of $P^b(\lambda)$ are stable under $\partial$.
\end{lem}

\begin{proof}
If $\phi \in P^{\geq \gamma}(\lambda) $ then $\phi = \phi_2 \circ \phi_1$ with
$ \phi_1 \colon Y(\lambda)  \rightarrow Y(\gamma')  $ and
$\phi_2 \colon Y(\gamma') \rightarrow Y(\mu)$
for some $\mu$ and some $\gamma' \geq \gamma$.
By the definition of $\partial$ on $\phi_1$ and $\phi_2$,
$ \partial(\phi_1) \colon Y(\lambda) \rightarrow Y(\gamma') $ and
$ \partial(\phi_2) \colon Y(\gamma') \rightarrow Y(\mu)$.
Thus $\partial(\phi)=\partial(\phi_2) \phi_1 + \phi_2 \partial(\phi_1)$
factors through $Y(\gamma')$.

The proof that $P^{> \gamma}(\lambda)$ is a $p$-DG submodule is the same.
\end{proof}

\begin{cor}
The standard module $\Delta(\lambda)$ is  a $p$-DG module over $S^b_n(r,s)$.
\end{cor}

\begin{proof}
The $p$-DG structure is determined by the $p$-DG structure on the $Y(\mu)$ and $Y(\lambda)$ using the formula $\partial f(x)=\partial(f(x))-f(\partial(x))$.
The fact that $P(>\lambda)$ is stable under $\partial$ follows from the proof of Lemma~\ref{filtpdgstable}.
\end{proof}

The cofibrant module $P^b(\lambda)$ should have a filtration $P^{\geq \lambda}$ with subquotients
$P^{\geq \gamma}(\lambda) /P^{> \gamma} (\lambda)\cong  \Delta(\gamma)^{f(\gamma)}$ for some graded multiplicities $f(\gamma)$.

The formula for $f(\gamma)$ should be determined by the coefficient of a standard basis in a canonical basis element $v_{b} \diamond v_d$ from Proposition \ref{canbasisprop} where $\lambda=(0^a 1^b 0^c 1^d)$.

We provide a couple of examples illustrating this.

\begin{example}
Consider $S_1^b(2,1)$.  
In this case 
$\nh_1^3=\Bbbk[y]/y^3$ 
and there are two elements of 
$ \mathcal{P}_1^{2,1} $:
\begin{equation*}
\mu=(~\emptyset~,~\yng(1)~,~\emptyset~) \hspace{.5in} \lambda=(~\emptyset~,~\emptyset~,~\yng(1)~).
\end{equation*}
Then we have 
\begin{equation*}
Y(\mu)=y \Bbbk[y]/y^3 \hspace{.5in}
Y(\lambda)=\Bbbk[y]/y^3.
\end{equation*}

\begin{equation*}
\Delta(\mu)=P^b(\mu)=\HOM_{\Bbbk[y]/y^3}(Y(\mu),Y(\lambda)) \oplus 
\HOM_{\Bbbk[y]/y^3}(Y(\mu),Y(\mu))
\end{equation*}
with
\begin{align*}
\HOM_{\Bbbk[y]/y^3}(Y(\mu),Y(\lambda)) &= \Bbbk \langle y \mapsto y, y \mapsto y^2    \rangle \\
\HOM_{\Bbbk[y]/y^3}(Y(\mu),Y(\mu)) &= \Bbbk \langle y \mapsto y, y \mapsto y^2    \rangle.
\end{align*}

\begin{equation*}
P^b(\lambda)=\HOM_{\Bbbk[y]/y^3}(Y(\lambda),Y(\lambda)) \oplus 
\HOM_{\Bbbk[y]/y^3}(Y(\lambda),Y(\mu))
\end{equation*}
with
\begin{align*}
\HOM_{\Bbbk[y]/y^3}(Y(\lambda),Y(\lambda)) &= \Bbbk \langle 1 \mapsto 1, 1 \mapsto y, 1 \mapsto y^2    \rangle \\
\HOM_{\Bbbk[y]/y^3}(Y(\lambda),Y(\mu)) &= \Bbbk \langle 1 \mapsto y, 1 \mapsto y^2    \rangle.
\end{align*}
All of the maps in $P^b(\lambda)$ factor through $Y(\mu)$ except for the identity map of $Y(\lambda)$.
Thus $P^{\geq \mu}(\lambda) \cong \Delta(\mu) $ and
$P^b(\lambda)/P^{\geq \mu}(\lambda) \cong  \Delta(\lambda)$ where $\Delta(\lambda)$ is a simple one-dimensional module.

Note that $S_1^b(2,1) \cong ((2)+(3))A_3^!((2)+(3))$.
\end{example}

\begin{example}
Consider $S_2^b(2,1)$.  
In this case there are two elements of 
$ \mathcal{P}_2^{2,1} $:
\begin{equation*}
\mu=(~\yng(1)~,~\yng(1)~,~\emptyset~) \hspace{.5in} \lambda=(~\emptyset~,~\yng(1)~,~\yng(1)~).
\end{equation*}
One easily checks that 
\begin{equation*}
Y(\mu)=y_1^2 y_2 \nh_2^3 \hspace{.5in}
Y(\lambda)=y_1 \psi_1 \nh_2^3.
\end{equation*}

\begin{equation*}
\Delta(\mu)=P^b(\mu)=\HOM_{\nh_2^3}(Y(\mu),Y(\lambda)) \oplus 
\HOM_{\nh_2^3}(Y(\mu),Y(\mu))
\end{equation*}
with
\begin{align*}
\HOM_{\nh_2^3}(Y(\mu),Y(\lambda)) &=
\Bbbk \langle y_1^2 y_2 \mapsto y_1^2 y_2 \psi_1 y_1   \rangle \\
\HOM_{\nh_2^3}(Y(\mu),Y(\mu)) &= 
\Bbbk \langle y_1^2 y_2 \mapsto y_1^2 y_2   \rangle.
\end{align*}

\begin{equation*}
P^b(\lambda)=\HOM_{\nh_2^3}(Y(\lambda),Y(\lambda)) \oplus 
\HOM_{\nh_2^3}(Y(\lambda),Y(\mu))
\end{equation*}
with
\begin{align*}
\HOM_{\nh_2^3}(Y(\lambda),Y(\lambda)) &= 
\Bbbk \langle y_1 \psi_1 \mapsto y_1^2 y_2 \psi_1, y_1 \psi_1 \mapsto y_1 \psi_1 y_1^2 \psi_1, y_1 \psi_1 \mapsto y_1 \psi_1  \rangle \\
\HOM_{\nh_2^3}(Y(\lambda),Y(\mu)) &= 
\Bbbk \langle y_1 \psi_1 \mapsto y_1^2 y_2 \psi_1  \rangle. 
\end{align*}
Notice that only the maps in $P^b(\lambda)$ that factor through $Y(\mu)$ are
\begin{equation*}
y_1 \psi_1 \mapsto y_1^2 y_2 \psi_1 \hspace{.5in}
y_1^2 y_2 \mapsto y_1^2 y_2 \psi_1 y_1.
\end{equation*}
This submodule is isomorphic to $\Delta(\mu)$ and the quotient of $P^b(\lambda)$ by this submodule is isomorphic to $\Delta(\lambda)$.
If $S_2^b(2,1) $ were cellular then $\Delta(\lambda)$ would be simple (which it is not).

Note that $S_2^b(2,1) \cong ((1)+(3))A_3^!((1)+(3))$.
\end{example}

\begin{rem}
\begin{enumerate}
\item[(1)] Definition \ref{def-stratified-mod} is equivalent, in the language of \cite{KleAHW}, to the following description on the algebra $S_n^b(r,s)$ itself. Pick an arbitrary total order refining the dominance order (Definition~\ref{def-order-on-nh-partition}) on $\mc{P}_n^{r,s}$, then there is a chain of two-sided ideals in $S_n^b(r,s)$ defined as
\[
I^b_{\geq \lambda}:=\sum_{\mu\geq \lambda} S_n^b(r,s)\xi_\mu S_n^b(r,s) \quad \quad \quad (\textrm{resp}.~I^b_{>\lambda}:=\sum_{\mu> \lambda} S_n^b(r,s)\xi_\mu S_n^b(r,s)).
\]
Since $\dif(\xi_\mu)=0$, it is clear that the chain is $\dif$-stable, which induces a differential on the quotient algebra. Then one can show that
$
I^b_{\geq \lambda}/I^b_{>\lambda}
$
is a $p$-DG matrix algebra with coefficients in $\END_{S_n^b(r,s)}(\Delta^b(\lambda))$.
\item[(2)] One can also make this description on the two-tensor quiver Schur algebra itself by replacing the idempotent $\xi_\mu$ above with the idempotent $\e_\mu\in S_n(r,s)$:
\[
\e_\mu: \bigoplus_{\lambda\in \mc{P}_n^{r,s}}e_\lambda G(\lambda)\rightarrow e_\mu G(\mu)\hookrightarrow \bigoplus_{\lambda\in \mc{P}_n^{r,s}}e_\lambda G(\lambda).
\]
Using the Webster diagrammatics of Theorem \ref{thm-iso-to-Webster-block}, the idempotent translates into ($\mu=(0^a1^b0^c1^d)$),
\[
\e_{\mu}= 
\begin{DGCpicture}
\DGCstrand[Red](0,0)(0,1)[$^r$]
\DGCstrand[Green](0.5,0)(0.5,1)[$^b$]
\DGCstrand[Red](1,0)(1,1)[$^s$]
\DGCstrand[Green](1.5,0)(1.5,1)[$^d$]
\end{DGCpicture} \ ,
\]
while the corresponding chain of ideals in $S_n(r,s)$ is taken to be
\[
I_{\geq \mu}:=\sum_{\lambda\geq \mu} S_n(r,s)\e_\lambda S_n(r,s) \quad \quad \quad (\textrm{resp}.~I_{>\lambda}:=\sum_{\mu> \lambda} S_n(r,s)\e_\mu S_n(r,s)).
\]
\end{enumerate}
\end{rem}

The $p$-DG stratified structure will play an important role in subsequent works.

\subsection{Future directions}
\label{subset-future}
As a conclusion, we formulate a conjecture for categorifying a general $m$-fold tensor product $V_{r_1}\otimes_{\mathbb{O}_p}V_{r_2}\otimes_{\mathbb{O}_p}\cdots\otimes_{\mathbb{O}_p} V_{r_m}$, where each $V_{r_i}$ is the rank-$(r_i+1)$ $\mathbb{O}_p$-integral Weyl module over $\dot{{U}}_{\mathbb{O}_p}$. 

To do this, we first generalize Definition \ref{speciallambda}. Set $\mathbf{r}=(r_1,r_2,\dots, r_m)$ and write $l=r_1+r_2+\cdots+r_m$.

\begin{defn}
\label{dfoldspeciallambda}
Let $\mathcal{P}_n^{\mathbf{r}}$ be the subset of all partitions $\lambda \in
\mathcal{P}_n^{l}$ of the form $\lambda=(0^{a_1} 1^{b_1} 0^{a_2} 1^{b_2}\dots 0^{a_m} 1^{b_m})$ satisfying 
\begin{equation*}
a_i+b_i=r_i, \quad (i=1,2,\dots, m) \quad  \quad \quad \textrm{and} \quad \quad \quad \sum_{i=1}^m b_i=n.
\end{equation*}
We think of such a sequence also as a partition
\[
\lambda=(\underbrace{~\emptyset~,\dots, ~\emptyset~}_{a_1},\underbrace{~\yng(1)~,\dots, ~\yng(1)~}_{b_1}|\underbrace{~\emptyset~,\dots, ~\emptyset~}_{a_2},\underbrace{~\yng(1)~,\dots, ~\yng(1)~}_{b_2}|\dots|\underbrace{~\emptyset~,\dots, ~\emptyset~}_{a_m},\underbrace{~\yng(1)~,\dots, ~\yng(1)~}_{b_m}).
\]
Note that the minimal partition $\lambda_0\in \mc{P}_n^l$ (Definition \ref{def-order-on-nh-partition}) always belongs to $\mc{P}_n^{\mathbf{r}}$.
\end{defn}

To any partition $\lambda\in \mc{P}_n^{\mathbf{r}}$ we have associated a thick idempotent in $\nh_n^l$ as in equation \eqref{eqn-sequence-idempotent}
\[
e_\lambda =e_{b_1}\otimes e_{b_2}\otimes \cdots \otimes e_{b_m}.
\]
We can then construct the right $\nh_n^l$-module $e_{\lambda}G(\lambda)$ as the idemptoent truncation of $G(\lambda)$ in Definition \ref{def-G-lambda} by the idemptotent $e_\lambda$. It is a right $p$-DG module by Proposition \ref{genYindecomp}.

Likewise, for the Webster algebra $W_n(\mathbf{r})$ (Definition \ref{def-Webster-algebra}), we associate the idempotent
\[
\e_{\lambda}= 
\begin{DGCpicture}
\DGCstrand[Red](0,0)(0,1)[$^{r_1}$]
\DGCstrand[Green](0.5,0)(0.5,1)[$^{b_1}$]
\DGCstrand[Red](1,0)(1,1)[$^{r_2}$]
\DGCstrand[Green](1.5,0)(1.5,1)[$^{b_2}$]
\DGCstrand[Red](2.5,0)(2.5,1)[$^{r_m}$]
\DGCstrand[Green](3,0)(3,1)[$^{b_m}$]
\DGCcoupon*(1.6,0.2)(2.4,0.8){$\cdots$}
\end{DGCpicture} 
\]
to each $\lambda\in \mc{P}_n^{\mathbf{r}}$, and the corresponding projective module $Q(\lambda):=\e_\lambda\cdot W_n(\mathbf{r})$. It is clear that $Q(\lambda)$ is a right $p$-DG module over $W_n(\mathbf{r})$.

As in Proposition \ref{prop-projective-as-summand}, the minimal partition $\lambda_0$ gives rise to $e_{\lambda_0}G(\lambda_0)$, which is a projective module of $\nh_n^l$. This implies that the action of $\nh_n^l$ on $\oplus_{\lambda\in \mc{P}_n^{\mathbf{r}}}e_\lambda G(\lambda)$ is faithful.

\begin{defn}
For any $n\in \{0,1,\dots, l\}$, the $p$-DG \emph{$m$-tensor quiver Schur algebra} is 
\[
S_n(\mathbf{r}):=\END_{\nh_n^l}\left(\bigoplus_{\lambda\in \mc{P}_n^{\mathbf{r}}}e_\lambda G(\lambda)\right),
\]
equipped with the natural $p$-differential as endomorphism algebra of a $p$-DG module.
\end{defn}

The faithfulness of the $\nh_n^l$ representation exhibits the commuting $S_n(\mathbf{r})$ and $\nh_n^l$ actions on $\oplus_{\lambda\in \mc{P}_n^{\mathbf{r}}}e_\lambda G(\lambda)$ as double commutants, so that the framework of Section \ref{sec-double} applies.

Similar to what we have seen in Section \ref{subsec-conn-Web}, the collection of modules $\oplus_{\lambda\in \mc{P}_n^{\mathbf{r}}}e_\lambda G(\lambda)$ may be reconstructed using the Webster idempotent
\[
\e(\kappa_0)= 
\begin{DGCpicture}
\DGCstrand[Red](0,0)(0,1)[$^{r_1}$]
\DGCstrand[Red](.5,0)(.5,1)[$^{r_2}$]
\DGCstrand[Red](1.5,0)(1.5,1)[$^{r_m}$]
\DGCstrand(2,0)(2,1)[$^{1}$]
\DGCstrand(3,0)(3,1)[$^{n}$]
\DGCcoupon*(.6,0.2)(1.4,0.8){$\cdots$}
\DGCcoupon*(2.1,0.2)(2.9,0.8){$\cdots$}
\end{DGCpicture} 
\]
as
\begin{equation}
\bigoplus_{\lambda\in \mc{P}_n^{\mathbf{r}}} e_\lambda G(\lambda) \cong 
\HOM_{W_n(\mathbf{r})}\left(Q(\kappa_0),\bigoplus_{\lambda\in \mc{P}_n^{\mathbf{r}}}Q(\lambda)\right),
\end{equation}
where $Q(\kappa_0):=\e(\kappa_0)\cdot W_n(\mathbf{r})$. Parallel to Theorem \ref{thm-iso-to-Webster-block}, we have the following.

\begin{thm}
There is an isomorphism of $p$-DG algebras
\[
\END_{\nh_n^l}\left(\bigoplus_{\lambda\in \mc{P}_n^{\mathbf{r}}} e_\lambda G(\lambda)\right)
\cong 
\END_{W_{\mathbf{r}}}\left(\bigoplus_{\lambda\in \mc{P}_n^{\mathbf{r}}}Q(\lambda)\right).
\]
\end{thm}
\begin{proof}
The proof is almost identical to that of Theorem \ref{thm-iso-to-Webster-block}. We leave the details to the reader.
\end{proof}

The theorem gives rise to a diagrammatic description of the $m$-tensor quiver Schur algebra in terms of Webster diagrammatics.

By the framework in Section \ref{sec-double}, there exists a categorical action of $(\dot{\mc{U}},\dif)$ on the $p$-DG category
\[
(S(\mathbf{r}),\dif)\dmod:=\bigoplus_{n=0}^l (S_n(\mathbf{r}),\dif)\dmod
\] 
induced from that on $\oplus_{n=0}^l(\nh_n^l,\dif)\dmod$. We propose the following.

\begin{conj}
The compact derived category $\mc{D}^c(S_n(\mathbf{r}))$ categorifies the weight $l-2n$ subspace in the $m$-fold tensor product representation $V_{r_1}\otimes_{\mathbb{O}_p}V_{r_2}\otimes_{\mathbb{O}_p}\cdots\otimes_{\mathbb{O}_p} V_{r_m}$ of $\dot{U}_{\mathbb{O}_p}(\mf{sl}_2)$.
\end{conj}

So far, we have seen that the conjecture holds in the case when $m=1$ (see the first example of Section~\ref{examplesubsection}) and $m=2$ (Theorem \ref{mainthm}). 

In a sequel to this work, we will prove the conjecture for the extreme case when $m=l$, which corresponds to categorifying $V_1^{\otimes m}$. We expect that the stratified ($p$-DG) structure of the previous subsection generalizes to the $m$-tensor case, and will play an important role towards proving the conjecture. In particular, for the case of categorifying $V_1^{\otimes m}$, the stratified $p$-DG structure upgrades into a quasi-hereditary $p$-DG structure. We will develop further the related machinery in a sequel to this work, which will enable us to identify the Grothendieck groups of $\mc{D}^c(S(\mathbf{r}))$ explicitly.


\addcontentsline{toc}{section}{References}


\bibliographystyle{alpha}
\bibliography{qy-bib}

%

\vspace{0.1in}

\noindent M.~K.: { \sl \small Department of Mathematics, Columbia University, New York, NY 10027, USA} \newline \noindent {\tt \small email: khovanov@math.columbia.edu}

\vspace{0.1in}

\noindent Y.~Q.: { \sl \small Department of Mathematics, University of Virginia, Charlottesville, VA 22904, USA} \newline \noindent {\tt \small email: yq2dw@virginia.edu}

\vspace{0.1in}

\noindent J.~S.:  { \sl \small Department of Mathematics, CUNY Medgar Evers, Brooklyn, NY, 11225, USA}\newline \noindent {\tt \small email: jsussan@mec.cuny.edu}

%
\end{document}